\documentclass[bj,authoryear]{imsart}

\usepackage{mathtools}
\usepackage{pgfplots}
\pgfplotsset{compat=1.18}
\usepackage{booktabs}

\usepackage{tikz}

\usepackage{amsmath, amssymb, amsthm}
\usepackage{float}   
\usepackage{natbib}

\usepackage{subfiles}

\newcommand{\vir}[1]{``#1''}

\newcommand{\We}{W_1^{(\ell^2)}}              
\newcommand{\Wl}{W_1^{(\ell^1)}}   

\usepackage{enumitem}    

\PassOptionsToPackage{unicode}{hyperref}

\startlocaldefs
\theoremstyle{plain}

\newtheorem{theorem}{Theorem}[section]
\newtheorem{lemma}[theorem]{Lemma}
\newtheorem{proposition}[theorem]{Proposition}
\newtheorem{corollary}[theorem]{Corollary}

\DeclareMathOperator*{\argmin}{arg\,min}

\theoremstyle{definition}
\newtheorem{definition}[theorem]{Definition}

\newtheorem{remark}{Remark}
\newcommand{\di}[1]{\mathrm{d}#1}

\providecommand{\di}{\,\mathrm{d}}

\allowdisplaybreaks

\usepackage{bm}

\usepackage{enumitem}

\usepackage{array,booktabs}

\newenvironment{steps}[1][$\bullet$]{%
    \begin{itemize}[label=#1]
}{%
    \end{itemize}
}

\usepackage[sectionbib]{bibunits}
\usepackage{nameref}
\usepackage[nameinlink,noabbrev,capitalise]{cleveref}

\endlocaldefs

\begin{document}

\begin{frontmatter}
\title{Adaptive minimax multivariate \(L^1\)-deconvolution: harmonic-mean rates under noise filtering, and a \(1\)-Wasserstein equivalence}
\runtitle{Multivariate \(L^1\)-deconvolution and a \(W_1\)-equivalence}

\begin{aug}
\author[A]{\inits{F.}\fnms{Catia}~\snm{Scricciolo}
\ead[label=e1]{catia.scricciolo@univr.it}}
\address[A]{Dipartimento di Scienze Economiche,
Università di Verona, Via Cantarane no. 24, 37129, Verona, Italy\printead[presep={,\ }]{e1}}
\end{aug}

\begin{abstract}
We study the multivariate convolution model with additive independent noise, multidimensional signal and noise, aiming to recover the signal's cumulative distribution function from contaminated observations. The noise is ordinary smooth with anisotropic regularity, and the signal belongs to an anisotropic Nikol'skii density class. We extend to the multivariate setting an approximate minimum \(L^1\)-distance estimator based on integrated kernel density estimation, and establish matching upper and lower bounds for the \(L^1\)-risk over anisotropic Nikol'skii classes. To our knowledge, this is the first minimax rate, not a sharpening of existing bounds. Although the lower-bound scheme is classical, the test-function construction is novel: it distinguishes no active component from at least one active component. The rate is governed by a censoring mechanism through the positive part, a noise filter in a harmonic-mean representation of the exponent. We further propose a fully data-driven, rate-adaptive procedure selecting an optimal bandwidth vector over the full Nikol'skii scale. On Nikol'skii product density classes, the \(L^1\)-distance between cumulative distribution functions and the \(1\)-Wasserstein distance share the same minimax rate, though not equivalent in general, via coordinate-wise decoupling of the \(1\)-Wasserstein cost on product measures. These findings reveal a link between the two distances, whose full understanding remains an open question.
\end{abstract}


\end{frontmatter}

\section{Introduction}\label{sec:1}
We study the multivariate deconvolution model with additive noise. We observe $n$ independent and identically distributed (i.i.d.) random vectors $\mathbf{Y}_1,\,\dots,\,\mathbf{Y}_n$ in $\mathbb{R}^d$, generated as
\begin{equation}\label{eq:model}
\mathbf{Y}_i=\mathbf{X}_i+\boldsymbol{\varepsilon}_i,\quad i=1,\,\dots,\,n,
\end{equation}
where the signal vectors $\mathbf{X}_i\in\mathbb{R}^d$ are i.i.d. with unknown probability measure $\mu_X$ and cumulative distribution function $F_X$. The noise vectors $\boldsymbol{\varepsilon}_i\in\mathbb{R}^d$ are i.i.d. with known probability measure $\mu_\varepsilon$ having density $f_\varepsilon$ with respect to Lebesgue measure. The sequences $(\mathbf{X}_i)_{i\in\mathbb N}$ and $(\boldsymbol{\varepsilon}_i)_{i\in\mathbb N}$ are independent of each other. Only the $\mathbf{Y}_i$ are observed. The goal is to recover $F_X$ from the indirect observations $\mathbf{Y}_1,\,\dots,\,\mathbf{Y}_n$. The nonparametric estimation of $F_X$ has received renewed attention in recent years, driven by its role in shape-constrained inference and in data regimes where the dimension grows with $n$, where observations are functional or censored, or where privacy constraints restrict the information released; in such regimes, classical methods often do not extend straightforwardly. This function is a natural target: unlike the density, it always exists and is bounded and monotone, and in the direct-observation setting it can be estimated without smoothing, since the empirical distribution function converges uniformly at the rate $n^{-1/2}$, independent of the dimension.

A further reason for studying the cumulative distribution function is its connection to Wasserstein distances, a standard tool for comparing distributions, from generative modeling, \emph{e.g.}, Wasserstein Generative Adversarial Networks (GANs) \citep{goodfellow2014generative, Arjovsky2017, Gulrajani2017}, to two-sample testing. In the univariate setting, the \(1\)-Wasserstein distance coincides with the \(L^1\)-distance between cumulative distribution functions, a consequence of the dual Kantorovich formulation \citep{Kantorovitch1942} and of the one-dimensional structure of the problem. The identity fails when \(d\ge 2\), so we work with the more tractable \(L^1\)-distance between distribution functions. Although the two distances are not equivalent in general, on the product Nikol'skii classes considered here they share the same minimax rate, as shown in Section~\ref{sec:equivalence}.

Estimating the distribution function rather than the density is also motivated by the fact that the former, being an integral of the density, is a smoother functional of the underlying distribution. This regularity, however, does not by itself guarantee a better rate than density estimation: as shown by \citet{HallLahiri2008}, cumulative distribution function deconvolution may exhibit unexpected local behaviour, with a convergence rate that can be slower at the origin than away from it. We study the deconvolution problem with known error distribution, assuming ordinary smooth noise with anisotropic regularity and a signal density in an anisotropic Nikol'skii class, thereby extending standard smoothness conditions under which the deconvolution problem admits finite minimax rates.

While the univariate case is well understood, to our knowledge multivariate cumulative distribution function deconvolution with known error distribution has not been previously addressed in the form presented here.
The main technical difficulty is that the anisotropic minimax rate is governed by a \emph{bottleneck effect}: a single unfavorable coordinate direction can dominate the overall difficulty, so that the convergence exponent is a harmonic mean of the coordinate-wise exponents (Section~\ref{sec:structure}). Capturing this phenomenon sharply requires a lower-bound construction that distinguishes the case of \emph{no active component} from that of \emph{at least one active component} (Section~\ref{sec:lowerbound}); this appears to be a device with no counterpart in the existing deconvolution literature.

The estimator we study is kernel-based, in the sense that it is an approximate $L^1$-projection of the integral of a kernel density estimator onto the space of cumulative distribution functions, following the univariate construction of \citet{Dedecker2015}. The kernel is anisotropic, with a Fourier transform compactly supported away from the origin; this design makes the integrated bias and the stochastic term simultaneously controllable, which is why Section~\ref{sec:assumptions} is devoted to the choice of the kernel and its interplay with the signal regularity. Our contributions are the following.
\begin{enumerate}
\item[(i)] \emph{An approximate multivariate minimum-\(L^1\)-distance estimator.} We extend to the multivariate setting the approximate minimum-\(L^1\)-distance estimator of \citet{Dedecker2015}, constructed as the \(L^1\)-projection of an integrated kernel density estimator onto the space of distribution functions. This projection guarantees that the estimator is a genuine cumulative distribution function, a property not automatically inherited from the integrated kernel (Section~\ref{sec:assumptions}).\\[-8pt]

\item[(ii)] \emph{Minimax-optimal rate and its structure.} We provide an asymptotic lower bound matching the upper bound, thereby obtaining the minimax-optimal rate. In the anisotropic Nikol'skii scale, the rate takes the form
\[
\psi_n^{\boldsymbol{\alpha},\boldsymbol{\beta}}
=
n^{-\frac{1}{2}\left[1+\sum_{j=1}^d
\frac{(\beta_j-\frac12)_+}{\alpha_j+1}\right]^{-1}},
\]
up to logarithmic factors, where \(\alpha_j\) and \(\beta_j\) quantify the anisotropic smoothness of the signal and the anisotropic decay of the noise characteristic function along coordinate \(j\), respectively.
The exponent admits a harmonic-mean representation in terms of coordinate-wise difficulties, revealing a \emph{bottleneck effect}: a single unfavorable coordinate direction can severely degrade the overall estimation efficiency (Section~\ref{sec:structure}).\\[-8pt]

\item[(iii)] \emph{A novel lower-bound construction.} The matching lower bound relies on a new construction of test functions that distinguishes the case of \emph{no active component} from that of \emph{at least one active component}. While the general lower-bound scheme is classical, this device is, to the best of our knowledge, what makes the anisotropic bound sharp, and no analogous construction seems to have been used previously in the deconvolution literature (Section~\ref{sec:lowerbound}).\\[-8pt]

\item[(iv)] \emph{A \(W_1\)--\(L^1\) structural link.} On Nikol'skii product density classes, the \(L^1\)-distance between cumulative distribution functions and the \(1\)-Wasserstein distance share the same minimax rate, although the two distances are not equivalent in general for \(d\geq2\). This follows from the coordinate-wise decoupling of the \(1\)-Wasserstein cost on product measures and contributes to a partial understanding of the structural link between the two distances, whose full nature remains to be investigated (Section~\ref{sec:equivalence}).\\[-8pt]

\item[(v)] \emph{A fully data-driven, rate-adaptive procedure.} We propose a penalized model-selection procedure, inspired by \citet{ComteLacour2013} and itself adapted from \citet{GoldenshlugerLepski2011} and \citet{Chagny2013}, that selects an optimal bandwidth vector across the full Nikol'skii scale without prior knowledge of the signal smoothness. We establish an asymptotic oracle inequality for the integrated kernel density estimator, and, as far as we are aware, this is the first adaptive procedure for multivariate cumulative distribution function deconvolution (Section~\ref{sec:adaptivity}).
\end{enumerate}

We now relate our contributions to the existing deconvolution literature, which has primarily addressed density rather than cumulative distribution function recovery in the univariate setting, and discuss how our work fits into it.
Deconvolving a density \(f_X\) from observations corrupted by independent additive noise with known distribution has been extensively studied. Kernel methods \citep{Carroll1988, Zhang1990} and wavelet-based strategies \citep{Fan2002} are among the most prominent. In the univariate case, a standard approach is the kernel deconvolution estimator of \citet{StefanskiCarroll1990}, whose theoretical properties depend on the smoothness of both the target density \(f_X\) and the error density \(f_\varepsilon\). A pivotal framework, introduced by \citet{Fan1991}, distinguishes between \textit{ordinary smooth} and \textit{supersmooth} distributions based on the decay rate of their characteristic functions, and classifies deconvolution problems into four regimes; \citet{Butucea2004v} showed that kernel estimators achieve minimax-optimal rates when \(f_X\) is supersmooth and the noise is ordinary smooth, while the supersmooth--supersmooth case has been examined from a minimax perspective by \citet{ButuceaTsybakov2008a, ButuceaTsybakov2008b}.

A separate line of research concerns \emph{adaptive} estimators, methods that attain optimal convergence rates without prior knowledge of the target's smoothness. Adaptivity in kernel-based deconvolution is typically addressed through data-driven bandwidth selection: \citet{Delaigle2004} proposed a consistent bootstrap method, while \citet{HazeltonTurlach2009} introduced a weighted kernel estimator with automated weight selection, extendable to multivariate settings. Projection and penalization provide an alternative approach: \citet{Pensky1999} developed an adaptive wavelet estimator for Sobolev spaces, \citet{Comte2006} devised an adaptive estimator using a Shannon-type basis and a penalization device, and more recent work employs non-compactly supported bases for specific spaces, such as the Laguerre and Hermite bases \citep{Mabon2017, ComteGenonCatalot2018}. 


The literature on multivariate density deconvolution is sparser. \citet{Masry1991} generalized the kernel estimator to stationary random processes, focusing on dependencies while assuming isotropic noise, but did not study an adaptive strategy. \citet{YoundjeWells2008} proposed a cross-validation strategy to select the optimal bandwidth under isotropic, ordinary smooth noise and established its asymptotic optimality. \citet{ComteLacour2013} introduced a bandwidth selection procedure inspired by \citet{GoldenshlugerLepski2011} accommodating anisotropic noises, with ordinary and supersmooth components, and provided convergence rates for the pointwise and \(L^2\)-risks when \(f_X\) belongs to anisotropic H\"older, Nikol'skii, or Sobolev classes. For ordinary smooth noise and \(f_X\) in anisotropic Nikol'skii classes, \citet{Rebelles2016} provided an adaptive kernel estimator that is minimax-optimal for the \(L^p\)-loss, and \citet{LepskiWiller2019} studied a more general model, with direct and indirect observations of \(X\), and constructed one that is optimal over anisotropic Nikol'skii classes under the \(L^p\)-loss.

A parallel line of work concerns the cumulative distribution function rather than the density. This is motivated by the role of the \(W_1\)-distance in comparing distributions, since in the univariate case \(W_1\) coincides with the \(L^1\)-distance between cumulative distribution functions, so estimating the latter amounts to estimating the former. For known ordinary smooth error densities, \citet{Dedecker2015} introduced an approximate minimum-$L^1$-distance cumulative distribution function estimator based on integrating a kernel density estimator, attaining minimax-optimal rates up to logarithmic factors without smoothness conditions on the mixing density. These results were later extended to Sobolev smoothness classes by \citet{Rousseau&Scricciolo}, and, for unknown error distribution, \citet{Scricciolo25b} derived minimax-optimal convergence rates and proposed adaptive procedures. The present work differs in two respects: it is multivariate (\(d\ge 2\)), where the univariate identity between \(W_1\) and the \(L^1\)-cumulative distribution function distance fails, and it establishes the first matching upper and lower bounds over anisotropic Nikol'skii classes, together with a fully data-driven adaptive procedure.

Finally, the classical convolution model assumes the error density to be completely known, a requirement often regarded as a major practical limitation. When repeated measurements of the same variable are available, each corrupted by independent noise, information on the error distribution can be obtained without knowing it a priori \citep{Delaigle2008,
Neumann2007}; alternatively, given a separate calibration experiment, one observes an independent noise sample and replaces the unknown noise characteristic function by its empirical counterpart \citep{Diggle1993,
Johannes2009, ComteLacour2011}. The present setting assumes the noise to be known, so these directions are not pursued.

The rest of the paper is structured as follows. Subsection~\ref{sec:preliminaries} introduces the notation. Section~\ref{sec:assumptions} states the assumptions on noise and kernel. Section~\ref{sec:estimator} presents the cumulative distribution function estimator. Section~\ref{sec:convergence} derives the non-asymptotic MIAE bound and establishes rates over anisotropic Nikol'skii classes. Section~\ref{sec:lowerbound} provides the matching lower bound. Section~\ref{sec:structure} analyzes the minimax-optimal rate structure, the influence of the dimension \(d\), and the bottleneck effect, comparing them with the density estimation rates of \citet{ComteLacour2013} and illustrating on a specific family of error densities. Section~\ref{sec:adaptivity} introduces the model-selection procedure and the corresponding oracle inequality. Section~\ref{sec:equivalence} discusses the \(W_1\)--\(L^1\) link on product Nikol'skii classes and outlines future research.


\subsection{Notation}\label{sec:preliminaries}
Fix the index sets $\mathbb{N} \coloneqq \{1,\,2,\,3,\,\dots\}$ and $\mathbb{N}_0 \coloneqq \mathbb{N} \cup \{0\}$. Let \(\mathbb{R}_+ \coloneqq (0,\,+\infty)\). For \(a \in \mathbb{R}\), the positive part is \((a)_+ \coloneqq \max\{a,\,0\}\). We denote by \(\mathbf{0}\) and \(\mathbf{1}\) the zero vector and the vector of ones in \(\mathbb{R}^d\). For vectors \(\boldsymbol\alpha\) and \(\boldsymbol\beta\) in \(\mathbb{R}^d\), the inequality \(\boldsymbol\beta \ge \boldsymbol\alpha\) is understood componentwise, \emph{i.e.}, \(\beta_j \ge \alpha_j\) for all \(j = 1,\,\dots,\,d\).

For positive sequences \((a_n)_{n\in\mathbb N}\) and \((b_n)_{n\in\mathbb N}\), we write \(a_n \lesssim b_n\) if there exists a constant \(C > 0\) such that \(a_n \le C b_n\) for all sufficiently large \(n\), \(a_n \gtrsim b_n\) if \(b_n \lesssim a_n\), and \(a_n \asymp b_n\) if there exist constants \(c, \,C > 0\) such that \(c b_n \le a_n \le C b_n\) for all sufficiently large \(n\).

Let \(C^\infty(\mathbb R)\) denote the space of infinitely differentiable functions \(f :\, \mathbb R \to \mathbb R\). The subspace of compactly supported functions is \(C_c^\infty(\mathbb R) \coloneqq \{ f \in C^\infty(\mathbb R) :\, \operatorname{supp}(f) \text{ is compact} \}\). For \(s \ge 1\), the space of \(s\)-times continuously differentiable functions on \(\mathbb R\) is denoted by \(C^s(\mathbb R)\). For \(x \in \mathbb R\), let \(\lceil x \rceil\) denote the ceiling function, \emph{i.e.}, \(\lceil x \rceil \coloneqq \min\{ n \in \mathbb Z :\, n \ge x \}\). Analogously, let \(\lfloor \cdot \rfloor\) stand for the floor function, \emph{i.e.}, \(\lfloor x \rfloor \coloneqq \max\{ m \in \mathbb Z :\, m \le x \}\).

For vectors \(\mathbf{t}\) and \(\mathbf{x}\) in \(\mathbb{R}^d\), the standard Euclidean inner product is \(\langle \mathbf{t},\, \mathbf{x} \rangle \coloneqq \sum_{j=1}^d t_j x_j\). For functions \(f,\,g \in L^2(\mathbb{R}^d)\), the scalar product is \(\langle f,\, g \rangle_2 \coloneqq \int_{\mathbb{R}^d} f(\mathbf{x})\,\overline{g(\mathbf{x})}\,\mathrm{d}\mathbf{x}\), with norm \(\|f\|_{L^2(\mathbb{R}^d)} \coloneqq \sqrt{\langle f,\, f \rangle_2}\). 

For \(f \in L^1(\mathbb{R})\), the Fourier transform is \(\mathcal F[f](t) \coloneqq \int_{\mathbb{R}} \mathrm{e}^{\mathrm{i} t x} f(x)\,\mathrm{d}x\), \(t \in \mathbb{R}\). If \(f \in L^1(\mathbb{R})\) and \(\mathcal F[f] \in L^1(\mathbb{R})\), the Fourier inversion formula holds \(f(x) = (2\pi)^{-1} \int_{\mathbb{R}} \mathrm{e}^{-\mathrm{i} t x} \mathcal F[f](t)\,\mathrm{d}t\). On \(L^2(\mathbb{R})\), the Fourier transform extends uniquely via the Plancherel theorem, preserving the \(L^2\)-norm up to a constant, \(\|\mathcal F[f]\|_{L^2(\mathbb{R})} = \sqrt{2\pi} \,\|f\|_{L^2(\mathbb{R})}\). For \(f \in L^1(\mathbb{R}^d) \cap L^2(\mathbb{R}^d)\), the Fourier transform on \(\mathbb{R}^d\) is defined analogously as \(\mathcal F[f](\mathbf{t}) \coloneqq \int_{\mathbb{R}^d} \mathrm{e}^{\mathrm{i} \langle \mathbf{t},\, \mathbf{x} \rangle} f(\mathbf{x})\,\mathrm{d}\mathbf{x}\), \(\mathbf{t} \in \mathbb{R}^d\). The convolution of \(f,\,g :\, \mathbb{R}^d \to \mathbb{R}\) is \((f \ast g)(\mathbf{x}) \coloneqq \int_{\mathbb{R}^d} f(\mathbf{x} - \mathbf{u})\,g(\mathbf{u})\,\mathrm{d}\mathbf{u}\), \(\mathbf{x} \in \mathbb{R}^d\), provided the integral exists.
The characteristic function of \(\mu_{\varepsilon}\) is $\varphi_{\varepsilon}(\mathbf{t}) \coloneqq \int_{\mathbb{R}^d} \mathrm{e}^{\mathrm{i} \langle \mathbf{t},\, \mathbf{x} \rangle} \,\mu_{\varepsilon}(\mathrm{d}\mathbf{x})$. For \(j = 1,\,\dots,\,d\), let \(\mu_{\varepsilon_j}\) denote the marginal distribution of the \(j\)-th component \(\varepsilon_{1,j}\) of \(\boldsymbol{\varepsilon}_1\), with Lebesgue density \(f_{\varepsilon_j}\). Its characteristic function is \(\varphi_{\varepsilon_j}(t) \coloneqq \int_{\mathbb{R}} \mathrm{e}^{\mathrm{i} t x} \,\mu_{\varepsilon_j}(\mathrm{d}x)\), \(t \in \mathbb{R}\).

\section{Assumptions on the noise and the signal}\label{sec:assumptions}
This section details the conditions on the noise distribution (Section~\ref{sec:noise}), the signal density (Section~\ref{sec:signal}), and the choice of the kernel function (Section~\ref{sec:kern}) that will be used throughout the article.

\subsection{Assumptions on the noise}\label{sec:noise}
We impose the following conditions on the error distribution:\\[-10pt]
\begin{enumerate}[label=(\textbf{E\arabic*}), leftmargin=*]
    \item The components \(\varepsilon_{1,1},\,\ldots,\,\varepsilon_{1,d}\) of the error vector \(\boldsymbol{\varepsilon}_1\in\mathbb{R}^d\) are independent random variables.\label{eq:E1indep}
    \item For each \(j=1,\,\ldots,\,d\), the characteristic function \(\varphi_{\varepsilon_j}\) of \(\mu_{\varepsilon_j}\) is such that there exist an exponent \(\beta_j>0\) and constants \(c_{0,j},\,c_{1,j}>0\) for which either or both of the following hold:\label{eq:E2}\\[-8pt]
    \begin{enumerate}[label=(\textbf{E2.\arabic*}), leftmargin=*]
        \item \(|\varphi_{\varepsilon_j}(t)|\geq c_{0,j}(1+|t|)^{-\beta_j}\) for all \(t\in\mathbb{R}\),\label{it:E21}\\[-8pt]
        \item \(|\varphi_{\varepsilon_j}(t)|\leq c_{1,j}(1+|t|)^{-\beta_j}\) for all \(t\in\mathbb{R}\).\label{it:E22}\\[-8pt]
    \end{enumerate}
    \item For each \(j=1,\,\ldots,\,d\), let \(r_{\varepsilon_j}(t)\coloneqq1/\varphi_{\varepsilon_j}(t)\), \(t\in\mathbb{R}\), be the reciprocal of the characteristic function \(\varphi_{\varepsilon_j}\). We assume that \(r_{\varepsilon_j}\in C^2(\mathbb{R})\). Moreover, for \(\ell=1,\,2\), there exist constants \(d_{\ell,j}>0\) such that
    \(
    |r^{(\ell)}_{\varepsilon_j}(t)|\leq d_{\ell,j}(1+|t|)^{\beta_j}\), $t\in\mathbb{R}$.
    \label{eq:E3}
\end{enumerate}

\noindent\emph{Consequences of Assumptions \emph{\ref{eq:E1indep}}--\emph{\ref{eq:E3}}}\\[2pt]
By Assumption \ref{eq:E1indep}, we have
\(\boldsymbol{\varepsilon}_1\sim \mu_{\varepsilon}=\bigotimes_{j=1}^d \mu _{\varepsilon_j}.\footnote{Here \(\bigotimes_{j=1}^d \mu_{\varepsilon_j} = \mu_{\varepsilon_1}\otimes\cdots\otimes\mu_{\varepsilon_d}\) denotes the product measure on \(\mathbb{R}^d\) induced by the independence of the components \(\varepsilon_{1,1},\,\ldots,\,\varepsilon_{1,d}\) stated in Assumption \ref{eq:E1indep}. The joint distribution of \(\boldsymbol{\varepsilon}_1\) factorizes as the product of its marginal distributions.}\)
Hence, the characteristic function of $\boldsymbol{\varepsilon}_1$ factorizes into the product of the characteristic functions of its components 
\(\varphi _{\varepsilon }(\mathbf{t})=\prod _{j=1}^d\varphi _{\varepsilon_j}(t_j)\), 
$\mathbf{t}\in\mathbb{R}^d$.
Combined with Assumption \ref{eq:E2}, this factorization implies the existence of constants $R_0\coloneqq\prod_{j=1}^d c_{0,j}>0$ and $R_1\coloneqq\prod_{j=1}^d c_{1,j}>0$ such that either or both of the following hold:\\[-15pt]
\begin{align}
R_0\prod _{j=1}^d(1+|t_j|)^{-\beta _j}\leq| \varphi _{\varepsilon }(\mathbf{t})|&,\quad 
\mathbf{t}\in\mathbb{R}^d,\label{eq:OS}\\[-10pt]
|\varphi _{\varepsilon }(\mathbf{t})|&
\leq R_1\prod _{j=1}^d(1+|t_j|)^{-\beta_j},\quad 
\mathbf{t}\in\mathbb{R}^d.\label{eq:OSS}
\end{align}
Condition \eqref{eq:OS} implies $|\varphi _{\varepsilon }(\mathbf{t})|\neq 0$ for every $\mathbf{t}\in\mathbb{R}^d$, that is, \(|\varphi_{\varepsilon}|\) never vanishes. If both \ref{it:E21} and \ref{it:E22} hold, then both the lower and upper bounds hold, describing a polynomial decay in each coordinate direction, with exponent \(\beta_j\) governing the tail decay rate of \(|\varphi_{\varepsilon_j}|\) along the \(j\)-th coordinate. Consequently, \eqref{eq:OS} and \eqref{eq:OSS} hold, and \(\varphi_{\varepsilon}\) satisfies the \emph{anisotropic ordinary smooth condition}
\(
R_0\prod_{j=1}^d (1+|t_j|)^{-\beta_j} \leq |\varphi_{\varepsilon}(\mathbf{t})| \leq R_1\prod_{j=1}^d (1+|t_j|)^{-\beta_j}\), $ \mathbf{t}\in\mathbb{R}^d$. Assumption \ref{eq:E3} imposes regularity and growth conditions on the derivatives of the reciprocal \(r_{\varepsilon_j} = 1/\varphi_{\varepsilon_j}\). For \(\ell=0\), we may take \(d_{0,j} \coloneqq 1/c_{0,j}\), so that
\(
|r^{(\ell)}_{\varepsilon_j}(t)| \leq d_{\ell,j}(1+|t|)^{\beta_j}\), $t \in \mathbb{R}$,
holds for \(\ell = 0,\,1,\,2\), extending the decay information from \ref{it:E21} to the derivatives of the reciprocal.
Several classical families satisfy Assumption \ref{eq:E3}. A prominent example is the \emph{Laplace} distribution, with characteristic function $\varphi_{\varepsilon}(t) = (1 + t^2)^{-1}$. The reciprocal $r_{\varepsilon}(t) = 1 + t^2$ is a polynomial of degree two, hence $|r_{\varepsilon}^{(\ell)}(t)| \le d_{\ell} (1 + |t|)^2$ for $\ell =1,\,2$. More generally, the \emph{Gamma} distribution with shape parameter $a>0$ and rate $\lambda>0$ has characteristic function $\varphi_{\varepsilon}(t) = (1-\mathrm{i}t/\lambda)^{-a}$. The reciprocal $r_{\varepsilon}(t) = (1 - \mathrm{i}t/\lambda)^{a}$ grows polynomially of order $a$. For \(a \ge 2\), it satisfies Assumption \ref{eq:E3} (with $j=1$), taking \(\beta = a\) if \(a\) is an integer, and \(\beta = \lceil a \rceil\) otherwise; in the integer case, \(r_{\varepsilon}(t)\) is a polynomial of degree $a$, ensuring the required differentiability and growth conditions.

Another family satisfying \ref{eq:E3} is the \emph{Student's t} with $\nu > 0$ degrees of freedom. Its characteristic function decays like $|\varphi_{\varepsilon}(t)| \sim C |t|^{-\nu}$ as $|t| \to \infty$, so the reciprocal behaves as $|r_{\varepsilon}(t)| \sim C |t|^{\nu}$. For $\nu \ge 4$, the reciprocal is twice continuously differentiable and satisfies $|r_{\varepsilon}^{(\ell)}(t)| \le d_\ell (1 + |t|)^{\nu}$ for $\ell = 1,\,2$.

We emphasize that the \emph{symmetric stable} distributions, with characteristic function $\varphi_{\varepsilon}(t) = \mathrm{e}^{-|t|^s}$, $0 < s \le 2$, satisfy neither \ref{eq:E2} nor \ref{eq:E3}. Indeed, the reciprocal $r_{\varepsilon}(t) = \mathrm{e}^{|t|^s}$ grows exponentially rather than polynomially, violating \ref{eq:E3}. Likewise, the lower bound in \ref{eq:E2} fails because $|\varphi_{\varepsilon}(t)| = \mathrm{e}^{-|t|^s}$ decays exponentially, not polynomially. Thus, stable distributions are excluded from the ordinary smooth noise framework. Distributions whose characteristic functions have zeros, such as the \emph{uniform}, are also excluded because the reciprocal is not defined everywhere.

\subsection{Regularity assumptions on the signal density}\label{sec:signal}
If the signal distribution \(\mu_X\) admits a Lebesgue density \(f_X\) on \(\mathbb{R}^d\), different regularity assumptions on \(f_X\) can be considered. The choice of the kernel depends on the specific regularity of \(f_X\). Here, we focus on anisotropic Nikol'skii classes with respect to the \(L^1\)-norm, because this norm matches the \(L^1\)-loss used for cumulative distribution functions.

\begin{definition}[Anisotropic Nikol'skii class of densities]\label{def:anisotropy}
For \(\boldsymbol{\alpha}=(\alpha_1,\,\ldots,\,\alpha_d)^\top\in\mathbb{R}_+^d\footnote{Here \(\mathbb{R}_+^d=(0,\,\infty)^d\), so each $\alpha_j>0$.}\) and a constant \(L>0\), the anisotropic Nikol'skii density class \(\mathcal{N}_1^d(\boldsymbol{\alpha},\,L)\) consists of all Lebesgue probability densities \(f_X\) on \(\mathbb{R}^d\) satisfying the following conditions.

\begin{enumerate}[label=(\textbf{N\arabic*}), leftmargin=*]
    \item \textbf{Differentiability.} \label{it:diff}
    For each \(j=1,\,\ldots,\,d\), let \(m_j\coloneqq \lceil \alpha_j \rceil -1\), where 
    \(\lceil \alpha_j \rceil\) denotes the smallest integer that is greater than or equal to \(\alpha_j\). Then, \(f_X\) admits partial derivatives with respect to \(x_j\) up to order \(m_j\), with the convention that \(\partial_j^0 f_X = f_X\) when \(m_j = 0\).\\[-8pt]

\item \textbf{\(L^1\)-bounds on derivatives.} \label{it:L1}
    For every \(j=1,\,\ldots,\,d\) and every \(k=0,\,\ldots,\,m_j\), we have \(\|\partial_j^{k} f_X\|_{L^1(\mathbb R^d)} \le L\), where \(\partial_j^{k} f_X\) denotes the partial derivative of order \(k\) with respect to the \(j\)-th coordinate.\\[-8pt]

    \item \textbf{Global smoothness of derivatives.} \label{it:smooth}
    For every \(j=1,\,\ldots,\,d\) and all \(\delta\in\mathbb{R}\), we have \(\|\Delta_{\delta\mathbf{e}_j}\partial_j^{m_j} f_X\|_{L^1(\mathbb R^d)} \le L\,|\delta|^{\alpha_j-m_j}\), where \(\mathbf{e}_j \coloneqq (0,\,\ldots,\,0,\,1,\,0,\,\ldots,\,0)^\top\) is the \(j\)-th standard basis vector of \(\mathbb{R}^d\), and the forward finite difference operator is defined by \(\Delta_{\delta\mathbf{e}_j}\partial_j^{m_j} f_X(\mathbf{x})\coloneqq \partial_j^{m_j} f_X(\mathbf{x}+\delta\mathbf{e}_j) - \partial_j^{m_j} f_X(\mathbf{x})\), \(\mathbf{x}\in\mathbb{R}^d\).
\end{enumerate}
\end{definition}

We briefly comment on the parameters involved. For \(\alpha_j\in(0,\,1)\), we have \(m_j=0\); then condition \ref{it:diff} requires no differentiability (\emph{i.e.}, only the density $f_X$ itself), and condition \ref{it:smooth} reduces to a Hölder condition of order \(\alpha_j\) on \(f_X\). For integer \(\alpha_j\ge 1\), we have \(m_j=\alpha_j-1\), and the exponent in \ref{it:smooth} becomes \(1\), \emph{i.e.}, a Lipschitz condition on the \((\alpha_j-1)\)-th derivative. More generally, for any $\alpha_j > 0$, the quantity $\alpha_j - m_j$ satisfies $0 < \alpha_j - m_j \le 1$: if $\alpha_j$ is not an integer, then 
$m_j = \lfloor \alpha_j \rfloor$ and $\alpha_j - m_j \in (0,\,1)$, so the condition is H\"{o}lder of order $\alpha_j - \lfloor \alpha_j \rfloor$; if $\alpha_j$ is an integer, then $m_j = \alpha_j - 1$ and 
$\alpha_j - m_j = 1$, yielding a Lipschitz condition on the derivative of order $\alpha_j - 1$. Condition \ref{it:smooth} thus controls the oscillation of the highest-order partial derivative of $f_X$ in each coordinate direction. 
For a given coordinate $j$, it requires that the $L^1$-norm of the difference of $\partial_j^{m_j} f_X$ over a step of size $\delta$ does not exceed $L|\delta|^{\alpha_j-m_j}$. Together with \ref{it:L1}, this ensures that 
$\partial_j^{m_j} f_X$ satisfies an $L^1$-H\"{o}lder condition of order $\alpha_j-m_j$ along the $j$-th coordinate direction. The anisotropic nature of the class is captured by the fact that the smoothness exponent $\alpha_j$ 
and the order of differentiability may differ across coordinates. 

The above definition is a multivariate extension of the one-dimensional Nikol'skii density 
class; for the univariate case, see, \emph{e.g.}, \citet{tsybakov2009introduction}, p.~13.


\subsection{Choice of the kernel}\label{sec:kern}
Since the deconvolution cumulative distribution function estimator is obtained by integrating the kernel density estimator, the kernel must satisfy suitable regularity conditions to ensure that the density estimator is well-defined (\emph{i.e.}, integrable) and that the approximation error has the correct order.

The specific choice of the kernel depends on the regularity of the density \(f_X\); different regularity assumptions require different bias-control strategies.  For densities belonging to an anisotropic Nikol'skii class \(\mathcal{N}_1^d(\boldsymbol{\alpha},\, L)\), we take a \emph{kernel \(K\) of order} \(\boldsymbol{\ell} \coloneqq (\ell_1,\,\dots,\,\ell_d)^{\!\top} \in \mathbb{N}^d\), with \(\boldsymbol{\ell}=\lceil\boldsymbol{\alpha}\rceil\) componentwise. This means that the following two conditions hold:
\begin{enumerate}[label=(\textbf{K\arabic*}), leftmargin=*, align=left]
\item \label{cond:K1} \textbf{Moment conditions}: \(\displaystyle \int_{\mathbb R^d}x_j^rK(\mathbf{x})\,\di\mathbf{x}=\delta_{0,r}\) for all \(r=0,\,\ldots,\,\ell_j\) and \(j=1,\,\ldots,\,d\), where \(\delta_{0,r}\) is the Kronecker delta (equal to $1$ if $r=0$ and $0$ otherwise). Thus, the kernel integrates to $1$ (when $r=0$) and its \vir{moments} of orders $1$ through $\ell_j$ vanish.
\item \label{cond:K2} \textbf{Integrability condition}: \(\displaystyle \int_{\mathbb R^d}(1+|x_j|)^{\ell_j+1}|K(\mathbf{x})|\,\di\mathbf{x}<\infty\) for all \(j=1,\,\ldots,\,d\). 
\end{enumerate}


For technical convenience, we employ a product kernel structure  
\begin{equation}\label{eq:kernstruc}
K(\mathbf{x}) \coloneqq \prod_{j=1}^d K_j(x_j), \quad \mathbf{x}\in \mathbb{R}^d,
\end{equation}
where each \(K_j\) is a univariate kernel. For a bandwidth vector \(\mathbf{h} = (h_1,\,\dots,\,h_d)^{\!\top} \in \mathbb{R}_+^d\), the corresponding multivariate rescaled kernel function is defined by  
\begin{equation}\label{eq:productkern}
K_{\mathbf{h}}(\mathbf{x}) \coloneqq \prod_{j=1}^d \frac{1}{h_j} K_j\left(\frac{x_j}{h_j}\right), \quad \mathbf{x} \in \mathbb{R}^d.
\end{equation}
Although the product form is not strictly necessary, it simplifies the proofs. 
We further require that \(K \in L^1(\mathbb{R}^d) \cap L^2(\mathbb{R}^d)\) and that its Fourier transform (which exists)  
\(
\varphi_K(\mathbf{t}) \coloneqq \mathcal{F}[K](\mathbf{t}) = \int_{\mathbb{R}^d} \mathrm{e}^{\mathrm{i}\langle\mathbf{t},\,\mathbf{x}\rangle} K(\mathbf{x}) \, \di\mathbf{x}=\prod_{j=1}^d\int_{\mathbb R}\mathrm{e}^{\mathrm{i}t_jx_j}K_j(x_j)\,\di x_j\), $\mathbf{t} \in \mathbb{R}^d$, 
has compact support and satisfies $\varphi_K(\mathbf{t})=\varphi_K(-\mathbf{t})$ for all $\mathbf{t}\in\mathbb R^d$. Since $\varphi_K(-\mathbf{t})=\overline{\varphi_K(\mathbf{t})}$, this implies that 
$\varphi_K(\mathbf{t})=\overline{\varphi_K(\mathbf{t})}$, \emph{i.e.}, $\varphi_K$ is real-valued. 

If in \eqref{eq:kernstruc} or \eqref{eq:productkern} we take all univariate kernels equal, then it is sufficient to take a univariate kernel of order $\ell\coloneqq \max_{1\leq j\leq d}\ell_j=\max_{1\leq j\leq d}\lceil {\alpha}_j \rceil$, although this may not be optimal for anisotropic classes.

\citet{GoldenLepski2011} construct kernels of order \(\boldsymbol{\ell}\) satisfying conditions \ref{cond:K1} and \ref{cond:K2} as follows. This is a classical finite-difference construction of higher-order kernels; see, \emph{e.g.}, \citet{tsybakov2009introduction}, Section~1.2.2, pp.~10--11. Consider functions \(u_j \colon \mathbb{R} \to \mathbb{R}\) such that, for each \(j=1,\,\dots,\,d\), we have
$\int_{\mathbb{R}} x^r u_j(x) \, \di x = \delta_{0,r}$ for all $r=0,\,\dots,\,\ell_j$,
and $\int_{\mathbb{R}} (1+|x|)^{\ell_j+1} |u_j(x)| \, \di x <\infty$ (each a univariate kernel of order \(\ell_j\), corresponding to the case \(d=1\) in \ref{cond:K1} and \ref{cond:K2}).
Then,  
\[
K_j(x_j) \coloneqq \sum_{k=1}^{\ell_j} \binom{\ell_j}{k} (-1)^{k+1} \frac{1}{k}u_j\!\left(\frac{x_j}{k}\right), \quad x_j \in \mathbb{R},
\]  
is a univariate kernel of order \(\ell_j\), and the product kernel defined in \eqref{eq:kernstruc} is a kernel of order \(\boldsymbol{\ell}\). A concrete example uses \(u_j(x_j) = v_{p_j}(x_j)\), with \(p_j \in \mathbb{N}\) chosen large enough to ensure the integrability condition, \emph{e.g.}, \(p_j > \ell_j+2\). For any \(p \in \mathbb{N}\), define \(v_p(x) \coloneqq c_p [\sin(x/p)/(x/p)]^p\), \(x \in \mathbb{R}\), with the constant \(c_p\) chosen such that \(\int_{\mathbb{R}} v_p(x) \, \di x = 1\). This choice ensures \(|v_{p_j}(x)| = O(|x|^{-p_j})\) as \(|x| \to\infty\), which guarantees both the moment and integrability conditions. 
Moreover, the Fourier transform of \(v_p\) is a compactly supported function with support contained in \([-1,\,1]\) (indeed, the convolution vanishes for \(|pt| > p\), \emph{i.e.}, \(|t| > 1\)), given by \(\mathcal{F}[v_p](t) = c_p(\pi p)^p (\mathbf{1}_{[-1,\,1]} \ast \cdots \ast \mathbf{1}_{[-1,\,1]})(p t)\), for \(t \in \mathbb{R}\), with \(p\) copies of $\mathbf{1}_{[-1,\,1]}$. Taking \(p_j = \ell_j+3\) (or any larger integer) guarantees the moment and integrability conditions, as the decay $|v_{p_j}(x)| = O(|x|^{-(\ell_j+3)})$ 
ensures $\int_{\mathbb R} (1+|x|)^{\ell_j+1} |v_{p_j}(x)|\, \di x < \infty$.

\section{Deconvolution cumulative distribution function estimator}\label{sec:estimator}
This section introduces an approximate minimum-\(L^1\)-distance estimator for the cumulative distribution function in the multivariate convolution model. The definition extends, to the multivariate setting, the univariate deconvolution estimator proposed by \citet{Dedecker2015}, pp.~240--241. The estimator is constructed from a preliminary estimator, defined as the integral of a kernel density estimator. We begin by presenting this preliminary estimator in a general form.
For a symmetric positive definite bandwidth matrix \(\mathbf{H}\) of size \(d \times d\), 
the kernel density estimator is defined as
\begin{equation}\label{eq:prelestim}
\hat f_{n,\mathbf{H}}(\mathbf{x}) \coloneqq \frac{1}{n}\sum_{i=1}^n \frac{1}{|\mathbf{H}|}\,\tilde K_{\mathbf{H}}(\mathbf{H}^{-1}(\mathbf{x}-\mathbf{Y}_i)),\quad \mathbf{x}\in\mathbb{R}^d,
\end{equation}
where \(|\mathbf{H}| \coloneqq |\det(\mathbf{H})|\) is the absolute value of the determinant of \(\mathbf{H}\) and the 
kernel \(\tilde K_{\mathbf{H}}\) is given by
\begin{equation}\label{eq:tildekernelH}
\tilde K_{\mathbf{H}}(\mathbf{x}) \coloneqq 
\mathcal{F}^{-1}\!\left[\frac{\varphi_K(\mathbf{t})}{\varphi_\varepsilon(\mathbf{H}^{-1}\mathbf{t})}\right](\mathbf{x}),\quad \mathbf{x}\in\mathbb{R}^d,
\end{equation}
where the Fourier transform \(\varphi_K\) of a kernel \(K\) has compact support and satisfies \(\varphi_K(-\mathbf{t}) = \varphi_K(\mathbf{t})\) for all \(\mathbf{t}\in\mathbb{R}^d\). This condition ensures that the kernel is real-valued. The kernel $\tilde K_{\mathbf{H}}$ is defined via its Fourier transform; it depends on $\mathbf{H}$ through the argument in the denominator of the Fourier transform and on the error distribution through its characteristic function $\varphi_\varepsilon$.
The definition is well-posed because \(\varphi_K\) has compact support and, under Assumption \ref{it:E21}, \(\varphi_\varepsilon\) is continuous and strictly positive on that support. Hence, the minimum of $\varphi_\varepsilon$ on that support is strictly positive, so the quotient is bounded. Its inverse Fourier transform therefore exists as a continuous, bounded function. The preliminary cumulative distribution function estimator is then defined as 
\(
\widehat F_{n,\mathbf{H}}(\mathbf{x}) \coloneqq \int_{-\infty}^{x_1}\!\!\cdots\!\!\int_{-\infty}^{x_d} \hat f_{n,\mathbf{H}}(\mathbf{u})\,\mathrm{d}\mathbf{u}\), $\mathbf{x}\in\mathbb{R}^d$.
A key property, which will be used repeatedly, is that the expectation of \(\widehat F_{n,\mathbf{H}}\) equals the convolution of the true cumulative distribution function \(F_X\) with the rescaled kernel \(K_{\mathbf{H}}\), defined as
\begin{equation}\label{eq:scaledkern}
K_{\mathbf{H}}(\mathbf{x})\coloneqq\frac{1}{|\mathbf{H}|}K(\mathbf{H}^{-1}\mathbf{x}),\quad \mathbf{x}\in\mathbb R^d.
\end{equation}
The result follows from the Fourier inversion formula and the definition of \(\tilde{K}_{\mathbf{H}}\); a rigorous derivation is given in the following proposition (the proof is deferred to \hyperref[sec:S1]{Section~S1 of the Supplement}).
\begin{proposition}[Unbiasedness of the estimator $\widehat F_{n,\mathbf{H}}(\mathbf{x})$]\label{prop:unbiasedness}
Assume that \(K \in L^1(\mathbb{R}^d)\), that \(\varphi_K\) has compact support, 
and that \(\varphi_\varepsilon(\mathbf{t}) \neq 0\) for all \(\mathbf{t}\) in the support of \(\varphi_K(\mathbf{H}\,\cdot\,)\).
Then, for every \(\mathbf{x}\in\mathbb{R}^d\),
\begin{equation}\label{eq:bias}
\mathsf E_{(\mu_X\ast\mu_\varepsilon)^{\otimes n}}[\widehat F_{n,\mathbf{H}}(\mathbf{x})]
=(K_{\mathbf{H}}\ast F_X)(\mathbf{x}),
\end{equation}
where the expectation is taken over the \(n\)-fold product of the convolution measure \(\mu_X\ast \mu_\varepsilon\).
\end{proposition}
The assumptions stated in Proposition~\ref{prop:unbiasedness} are precisely those used in the proof. 
They are satisfied under the kernel conditions of Section~\ref{sec:kern} and Assumption~\ref{it:E21} on the noise.

In the special case of a diagonal bandwidth matrix $\mathbf{H}=\operatorname{diag}(h_1,\,\dots,\, h_d)$ and a product kernel $K(\mathbf{x})=\prod_{j=1}^d K_j(x_j)$ as in \eqref{eq:kernstruc}, where $\varphi_K(\mathbf{t}) = \prod_{j=1}^d \varphi_{K_j}(t_j)$, if in addition the noise components are independent (Assumption \ref{eq:E1indep}) so that $\varphi_\varepsilon(\mathbf{t}) = \prod_{j=1}^d \varphi_{\varepsilon_j}(t_j)$, then the deconvolution kernel in \eqref{eq:tildekernelH} factorizes as
$
\tilde{K}_{\mathbf{H}}(\mathbf{x}) = \prod_{j=1}^d \tilde{K}_{j,h_j}(x_j)$, $\mathbf{x}\in\mathbb R^d$,
where each univariate component is
\[
\tilde{K}_{j,h_j}(x_j)\coloneqq\mathcal F^{-1}\!\left[\frac{\varphi_{K_j}(t)}{\varphi_{\varepsilon_j}(t/h_j)}\right](x_j),
\quad x_j\in\mathbb R.
\]  
Consequently, for the bandwidth vector $\mathbf{h} = (h_1,\,\ldots,\,h_d)^\top \in \mathbb{R}_+^d$, whose components are the diagonal entries of $\mathbf{H} = \operatorname{diag}(h_1,\,\ldots,\,h_d)$, the preliminary kernel density estimator in \eqref{eq:prelestim} simplifies to
\[
\hat{f}_{n,\mathbf{h}}(\mathbf{x}) = \frac{1}{n} \sum_{i=1}^n \prod_{j=1}^d \frac{1}{h_j} \tilde{K}_{j,h_j}\!\left(\frac{x_j - Y_{i,j}}{h_j}\right),\quad \mathbf{x}\in\mathbb{R}^d.
\]
The corresponding integrated estimator $\widehat{F}_{n,\mathbf{h}}$ is defined by
\begin{equation}\label{eq:defcumestim}
\widehat{F}_{n,\mathbf{h}}(\mathbf{x}) \coloneqq \int_{-\infty}^{x_1} \!\cdots \int_{-\infty}^{x_d} \hat{f}_{n,\mathbf{h}}(\mathbf{u})\,\mathrm{d}\mathbf{u}=\frac{1}{n} \sum_{i=1}^n\prod_{j=1}^d\int_{-\infty}^{x_j}
\frac{1}{h_j} \tilde{K}_{j,h_j}\!\left(\frac{u_j - Y_{i,j}}{h_j}\right)\,\di u_j,\quad \mathbf{x}\in \mathbb{R}^d.
\end{equation}
In this setting, the expectation in \eqref{eq:bias} becomes 
\begin{equation}\label{eq:biasprod}
\mathsf{E}_{(\mu_X \ast \mu_\varepsilon)^{\otimes n}}[\widehat{F}_{n,\mathbf{h}}(\mathbf{x})] = (K_{\mathbf{h}} \ast F_X)(\mathbf{x}),\quad \mathbf{x}\in\mathbb{R}^d,
\end{equation}
where, with a slight abuse of notation, $K_{\mathbf{h}}$ denotes the product kernel with bandwidth vector $\mathbf{h}=(h_1,\,\dots,\, h_d)^\top$, as given in \eqref{eq:scaledkern} for $\mathbf{H}=\operatorname{diag}(h_1,\,\dots,\, h_d)$.

Because the deconvolution kernel density estimator $\hat{f}_{n,\mathbf{h}}$ is not guaranteed to be non-negative, the integrated estimator $\widehat{F}_{n,\mathbf{h}}$ in \eqref{eq:defcumestim}
may fail to be monotone and therefore is not, in general, a proper cumulative distribution function. To obtain an estimator that is a cumulative distribution function, we project $\widehat{F}_{n,\mathbf{h}}$ onto the set of all cumulative distribution functions via an approximate minimum-\(L^1\)-distance approach, as proposed by \citet{Dedecker2015} for the univariate case, pp.~240--241. Let $\varrho_n$ be a positive sequence such that $\varrho_n \to 0$ as $n \to \infty$, and let $\mathfrak{F}$ denote the set of all cumulative distribution functions on $\mathbb{R}^d$. 
A $\varrho_n$-approximate $L^1$-minimizer $\widehat{F}_{n,\mathbf{h}}^{\ast}$ is defined by the condition
\begin{equation}\label{eq:defminimizer}
\|\widehat{F}_{n,\mathbf{h}}^{\ast} - \widehat{F}_{n,\mathbf{h}}\|_{L^1(\mathbb R^d)} = \inf_{F \in \mathfrak{F}} \|F - \widehat{F}_{n,\mathbf{h}}\|_{L^1(\mathbb R^d)} + \varrho_n.
\end{equation}
As will be shown later, the choice $\varrho_n = O({n}^{-1/2})$ is sufficient to preserve the minimax-optimal convergence rates of the kernel-based estimator $\widehat{F}_{n,\mathbf{h}}$.


\section{Risk upper bound for the approximate minimum-\(L^1\) estimator}\label{sec:convergence}
In this section, we derive an upper bound on the convergence rates of the \(L^1\)-risk of the \(\varrho_n\)-approximate minimum-\(L^1\)-distance estimator \(\widehat F^{\ast}_{n,\mathbf{h}}\) defined in \eqref{eq:defminimizer}. As will be shown in Theorem~\ref{thm:ubrate}, the \(L^1\)-risk is bounded above by the sum of an approximation error and a stochastic error; we therefore treat these two components separately. We first establish an upper bound for the bias term, namely the \(L^1\)-norm of \(B_{\mathbf{h}}\coloneqq F_X - (K_{\mathbf{h}} \ast F_X)\). Recall from \eqref{eq:biasprod} that the pointwise expectation of \(\widehat F_{n,\mathbf{h}}\) is \(K_{\mathbf{h}} \ast F_X\).

\begin{proposition}[Upper bound on the $L^1$-bias for anisotropic Nikol'skii classes]\label{prop:bias}
Let \(F_X\) be a cumulative distribution function whose Lebesgue density \(f_X\) belongs to the anisotropic Nikol'skii class \(\mathcal{N}_1^d(\boldsymbol{\alpha},\,L)\), where \(\alpha_j>0\) for each \(j=1,\,\dots,\,d\) and \(L>0\) is a constant. 
For every \(j=1,\,\dots,\,d\), let \(K_j\in L^1(\mathbb R)\cap L^2(\mathbb R)\) be a univariate kernel satisfying the following moment conditions:
\begin{equation}\label{eq:kernelcond}
\int_{\mathbb R} z^r K_j(z)\,\mathrm{d}z = \delta_{0,r},\quad r=0,\,\dots,\,\lceil\alpha_j\rceil,
\end{equation}
which correspond to the \(d=1\) version of \ref{cond:K1}. We also impose the integrability condition
$\int_{\mathbb R} |z|^{\alpha_j+1}|K_j(z)|\,\mathrm{d}z<\infty$.\footnote{This is the natural condition for controlling the bias of the cumulative distribution function. Note that when \(\alpha_j\) is an integer, this coincides with the \(d=1\) version of \ref{cond:K2}; when \(\alpha_j\) is not an integer, it is slightly weaker than \ref{cond:K2}, but it is sufficient for the bound established below.}
For \(\mathbf h\in(0,\,1)^d\), let \(K_{\mathbf h}\) be the rescaled product kernel defined in \eqref{eq:productkern}.
Then,
\[
\|F_X-(K_{\mathbf h}\ast F_X)\|_{L^1(\mathbb R^d)}
\le L\sum_{j=1}^d C_j h_j^{\alpha_j+1},
\]
where, for every \(j=1,\,\dots,\,d\), with the convention that the empty product equals \(1\),
\begin{equation}\label{eq:constantc_j}
C_j\coloneqq
\prod_{\substack{r=1 \\ r\ne j}}^d \|K_r\|_{L^1(\mathbb R)}
\,
\frac{C_{\alpha_j}}{\alpha_j+1}
\int_{\mathbb R}|z|^{\alpha_j+1}|K_j(z)|\,\mathrm{d}z<\infty .
\end{equation}
Here, \(C_{\alpha_j}>0\) is a constant depending only on \(\alpha_j\), whose explicit expression is given in the proof.
\end{proposition}

The proof of the bias bound, which is provided in \hyperref[sec:S2]{Section~S2 of the Supplement}, exploits the anisotropic structure of the Nikol'skii class. The key idea is to decompose the pointwise bias $B_{\mathbf{h}}(\mathbf x)$ into a telescoping sum of one-dimensional increments along each coordinate direction. By applying a Taylor expansion with integral remainder in each coordinate direction and using the vanishing moment conditions of the kernel, only the highest-order remainder terms survive. These are then controlled componentwise by virtue of the anisotropic smoothness of $f_X$.

\begin{remark}\label{rem:samebias}
Although Proposition~\ref{prop:bias} is stated for the \(L^1\)-norm of the bias $B_{\mathbf{h}}$, the same proof carries over to the \(L^2\)-norm of $B_{\mathbf{h}}$, provided that the regularity assumption on the density \(f_X\) is modified accordingly: the \(L^1\)-based Nikol'skii class \(\mathcal N_1^d(\boldsymbol{\alpha},\,L)\) is replaced by its \(L^2\)-based counterpart \(\mathcal N_2^d(\boldsymbol{\alpha},\, L)\). In particular, condition \ref{it:smooth} is replaced by its \(L^2\)-analogue: for each \(j=1,\,\ldots,\,d\), there exists a constant \(L>0\) such that, for all \(\delta\in\mathbb R\), we have
\(
\bigl\| \Delta_{\delta\mathbf{e}_j}\partial_j^{m_j} f_X \bigr\|_{L^2(\mathbb R^d)}\le L |\delta |^{\alpha_j - m_j}\),
where \(m_j \coloneqq \lceil \alpha_j\rceil - 1\) as in \ref{it:diff}. Under this strengthened assumption, the bound is
$
\|F_X-(K_{\mathbf h}\ast F_X)\|_{L^2(\mathbb R^d)}
\le
L \sum_{j=1}^d C_j h_j^{\alpha_j+1}$, 
with the same constants \(C_j>0\) as in \eqref{eq:constantc_j}. Further details are provided in Remark \ref{rem:adjusted_proof} in \hyperref[sec:S2]{Section~S2 of the Supplement}.
\end{remark}


We now give a result for controlling the stochastic error term.
For notational convenience, we define
\begin{equation}\label{eq:defV}
V(\mathbf{h}) \coloneqq \frac{1}{n}\prod_{j=1}^d h_j^{-(\beta_j-\frac{1}{2})_+},
\quad \mathbf{h} \in (0,\,1)^d.
\end{equation}
The following bound follows from the fact that the stochastic error term gives rise to a product of coordinatewise factors, while the $L^1$-norm introduces a square root. The proof is deferred to \hyperref[sec:S2]{Section~S2 of the Supplement}.


\begin{proposition}[Upper bound on the stochastic error]\label{prop:variance}
Assume that the noise conditions \ref{eq:E1indep}, \ref{it:E21} and \ref{eq:E3} hold with $\beta_j>0$ for all $j=1,\,\ldots,\,d$. Suppose also that, for some \(p > 1/2\), the following moment conditions are satisfied:
\begin{align}
&\mathsf E_{\mu_X}\Biggl[\prod_{j=1}^{d}(1\vee |X_{1,j}|^{2p})\Biggr] < \infty, \label{eq:assonX}\\
&\mathsf E_{\mu_{\varepsilon_j}}[|\varepsilon_{1,j}|^{2p}] < \infty \quad\text{for every } j=1,\,\dots,\,d. \label{eq:assoneps}
\end{align}
Consider the kernel-based estimator $\widehat F_{n,\mathbf{h}}$ constructed with a product kernel $K(\mathbf{x})=\prod_{j=1}^d K_j(x_j)$, $\mathbf{x}\in\mathbb R^d$, as in \eqref{eq:kernstruc}, where, for each $j=1,\,\ldots,\,d$,
\begin{enumerate}[label=(\emph{\roman*}), align=left, leftmargin=*]
\item $K_j\in L^1(\mathbb R)\cap L^2(\mathbb R)$,
\item $K_j$ is real-valued and symmetric about $0$ (consequently its Fourier transform $\varphi_{K_j}\coloneqq \mathcal F[K_j]$ is real and even),\footnote{The symmetry ensures that the deconvolution kernel $\tilde K_{1,j,h_j}$ is real-valued, which is essential for the proof to proceed using standard variance arguments for real-valued random variables.}
\item $\varphi_{K_j}$ is supported on $[-1,\,1]$ and is twice continuously differentiable on $\mathbb R$.\footnote{An explicit example of a kernel $K_j$ satisfying all the requirements in (\emph{i})--(\emph{iii}) is exhibited in the proof of Proposition~\ref{prop:variance}.}
\end{enumerate}
Then, for every $\mathbf{h}\in(0,\,1)^d$ and $q=1,\,2$, the expected $L^q$-deviation of $\widehat F_{n,\mathbf{h}}$ from $K_{\mathbf{h}}\ast F_X$ satisfies
$$\mathsf E_{(\mu_X\ast\mu_\varepsilon)^{\otimes n}}[\|\widehat F_{n,\mathbf{h}}-(K_{\mathbf{h}}\ast F_X)\|_{L^q(\mathbb R^d)}]\lesssim \mathcal V(\mathbf{h}),$$
where 
\begin{equation}\label{eq:varprox}
\mathcal V(\mathbf{h})\coloneqq V(\mathbf{h})\, \prod_{j=1}^d |\log h_j|^{\mathbf{1}_{\{\beta_j=1/2\}}/2}
\end{equation}
with $V(\mathbf{h})$ defined in \eqref{eq:defV}.
\end{proposition}


We combine the upper bounds for the approximation error (Proposition~\ref{prop:bias}) and the stochastic error (Proposition~\ref{prop:variance}) to obtain, in Theorem~\ref{thm:ubrate}, an upper bound for the mean integrated absolute error of the estimator \(\widehat{F}^\ast_{n,\mathbf{h}}\) defined in \eqref{eq:defminimizer}. The proof is deferred to \hyperref[sec:S2]{Section~S2 of the Supplement}.

\begin{theorem}[Upper bound on the $L^1$-risk of $\widehat F_{n,\mathbf{h}}^\ast$]\label{thm:ubrate}
Assume that the noise conditions \ref{eq:E1indep}, \ref{it:E21}, and \ref{eq:E3} hold with $\beta_j>0$ for all $j=1,\,\dots,\,d$. 
Suppose further that the moment conditions \eqref{eq:assonX} and \eqref{eq:assoneps} are satisfied for some $p>1/2$, and that the density $f_X$ belongs to an anisotropic Nikol'skii class $\mathcal{N}_1^d(\boldsymbol{\alpha},\,L)$ for some $\boldsymbol{\alpha}\in\mathbb{R}_+^d$ and $L>0$.

Let $\widehat F_{n,\mathbf{h}}$ be the kernel‑based estimator constructed with a product kernel as in \eqref{eq:kernstruc} of order $\boldsymbol{\ell}= \lceil\boldsymbol{\alpha}\rceil$, where each univariate kernel $K_j$ satisfies the conditions (i)--(iii) listed in Proposition~\ref{prop:variance}. 
Let $\widehat F^{\ast}_{n,\mathbf{h}}$ be the $\varrho_n$-approximate minimum-$L^1$-distance estimator with $\varrho_n = O({n}^{-1/2})$. Then, for the bandwidth vector $\mathbf{h}$ chosen as 
$h_j \asymp n^{-{E}/[2(\alpha_j+1)]}$, $j=1,\,\dots,\,d$, 
the $L^1$-risk of $\widehat F^{\ast}_{n,\mathbf{h}}$ satisfies
\[
\mathsf{E}_{(\mu_X \ast \mu_\varepsilon)^{\otimes n}}[\|\widehat F^{\ast}_{n,\mathbf{h}} - F_X\|_{L^1(\mathbb R^d)}]
\lesssim\psi_n^{\boldsymbol{\alpha},\boldsymbol{\beta}} \; (\log n)^{\frac12\sum_{j=1}^d \mathbf{1}_{\{\beta_j = 1/2\}}},
\]
where
\begin{equation}\label{eq:rate}
\psi_n^{\boldsymbol{\alpha},\boldsymbol{\beta}} 
\coloneqq n^{-\frac{E}{2}},
\qquad\text{with }
E \coloneqq \Biggl[1 + \sum_{j=1}^d \frac{(\beta_j-\frac12)_+}{\alpha_j+1}\Biggr]^{-1}.
\end{equation}
\end{theorem}

The following theorem provides the $L^2$-analogue of Theorem~\ref{thm:ubrate}, under the corresponding $L^2$-based regularity condition on the density.
\begin{theorem}[Extension of Theorem~\ref{thm:ubrate} to the $L^2$-risk]\label{thm:ubrateL2}
Under the same assumptions as in Theorem~\ref{thm:ubrate}, except that $f_X$ belongs to the $L^2$-based Nikol'skii class $\mathcal N_2^d(\boldsymbol{\alpha},\, L)$ instead of $\mathcal N_1^d(\boldsymbol{\alpha},\, L)$, the $L^2$-risk of the $\varrho_n$-approximate minimum-$L^2$-distance estimator $\widehat F^{\ast}_{n,\mathbf{h}}$, with $\varrho_n = O({n}^{-1/2})$, satisfies
\(
\mathsf E_{(\mu_X \ast \mu_\varepsilon)^{\otimes n}}
[\| \widehat F^{\ast}_{n,\mathbf{h}} - F_X \|_{L^2(\mathbb R^d)}]
\lesssim
\psi_n^{\boldsymbol{\alpha},\boldsymbol{\beta}}
(\log n)^{\frac12 \sum_{j=1}^d \mathbf{1}_{\{\beta_j = 1/2\}}}\),
with rate \(\psi_n^{\boldsymbol{\alpha},\boldsymbol{\beta}}\) as in \eqref{eq:rate}.
\end{theorem}
\begin{proof}
The proof follows the same structure as that of Theorem~\ref{thm:ubrate}, with Proposition~\ref{prop:bias} replaced by its $L^2$-counterpart established in Remark~\ref{rem:samebias}. Moreover, as shown in Proposition~\ref{prop:variance}, the stochastic error bound remains valid in the $L^2$-norm.
\end{proof}

In practice, the regularity vector $\boldsymbol{\alpha}$ of the density $f_X$ is typically unknown. This poses a problem when choosing a kernel, since the bias bounds derived in Proposition~\ref{prop:bias} and Remark~\ref{rem:samebias} require the kernel order to be at least $\lceil \alpha_j \rceil$ along each coordinate direction. However, there exist kernels that are simultaneously of every order, and thus can be used in the absence of any knowledge of $\boldsymbol{\alpha}$. Such kernels are particularly appealing for adaptive estimation procedures, which will be studied in Section~\ref{sec:adaptivity}. We formalize this notion in the following definition.
\begin{definition}[Superkernel]\label{def:superkernel}
A \emph{superkernel} is a function $K \in L^1(\mathbb R) \cap L^2(\mathbb R)$ whose Fourier transform $\mathcal F[K]$ is of class $C^\infty(\mathbb R)$, compactly supported, symmetric about zero, and satisfies $\mathcal F[K]= 1$ on $[-1,\,1]$.
\end{definition}

The following proposition, asserting the existence of a superkernel, is a standard result in the literature; its proof is provided in \hyperref[sec:S2]{Section~S2 of the Supplement} for completeness and reader's convenience.

\begin{proposition}[Existence of a superkernel]\label{prop:exsuperkern}
A superkernel $K$ in the sense of Definition~\ref{def:superkernel} exists.
\end{proposition}

Since a superkernel has infinite order (\emph{i.e.}, $\int_{\mathbb R} x^r K(x)\,\di x = \delta_{0,r}$ for all $r\ge0$), it is of order $\lceil \alpha \rceil$ for every $\alpha>0$. Let $f_X \in \mathcal N_q^d(\boldsymbol{\alpha},\, L)$, for $q = 1,\, 2$. Applying Proposition~\ref{prop:bias} or Remark~\ref{rem:samebias} to a product kernel $K_{\mathbf{h}}$ as in \eqref{eq:productkern}, where each univariate kernel is a superkernel, the bias for the cumulative distribution function is of order $\sum_{j=1}^d h_j^{\alpha_j+1}$, regardless of $\boldsymbol{\alpha}$, and without requiring knowledge of the density regularity. This makes the superkernel particularly well-suited for adaptive estimation procedures, as it can achieve the optimal rate of convergence for any regularity level without prior knowledge of $\boldsymbol{\alpha}$. 

We conclude with a remark on the case where the signal density $f_X$ is non-smooth, yet the noise decays enough to cancel the deconvolution effect, yielding the noise-free density estimation rate \(n^{-{1}/{(2+d)}}\).

\begin{remark}
When the density $f_X$ exists and no further smoothness is assumed (this may be regarded as corresponding to $\alpha_j = 0$ for all $j = 1,\,\dots,\,d$) and the noise satisfies $\beta_j = 1$ for all $j = 1,\,\dots,\,d$, the optimal convergence rate is \(n^{-{1}/{(2+d)}}\).
Indeed, substituting \(\alpha_j=0\) and \(\beta_j=1\) into \eqref{eq:rate} gives \((\beta_j-1/2)_+ = 1/2\), hence $
E =[1 + d/2]^{-1} = {2}/{(2+d)}$,
so \(\psi_n^{\boldsymbol{\alpha},\boldsymbol{\beta}} = n^{-E/2} = n^{-{1}/{(2+d)}}\).
This rate coincides with the classical \emph{noise-free} density estimation rate, where the random vectors $\mathbf{X}_1,\,\dots,\,\mathbf{X}_n\in\mathbb{R}^d$ are observed directly without measurement errors.
The rate follows from the standard bias-variance trade-off. Indeed, the bias is of order $\sum_{j=1}^d h_j$, while the stochastic error term is of order $(n\prod_{j=1}^d h_j)^{-1/2}$. Balancing these two quantities yields $n^{-{1}/{(2+d)}}$. To verify the order of the stochastic error, consider the bound on the squared dominant term:
\[
\frac{1}{n}\prod_{j=1}^d I_{2,j,h_j} \;\lesssim\; \frac{1}{n}\prod_{j=1}^d \int_{1<|t|\le 1/h_j} \frac{|r_{\varepsilon_j}(t)|^2}{|t|^2}\,\mathrm{d}t
\lesssim\frac{1}{n}\prod_{j=1}^d \int_{1<|t|\le 1/h_j} \frac{(1+|t|)^{2\beta_j}}{|t|^2}\,\mathrm{d}t
\lesssim \frac{1}{n}\prod_{j=1}^d h_j^{-1},
\]
where we have used $|r_{\varepsilon_j}(t)| \lesssim (1+|t|)^{\beta_j}$ for all $j$, and the last inequality follows from $(1+|t|)^{2\beta_j}/|t|^2 \lesssim 1$ for $|t|>1$ when $\beta_j=1$. Consequently, the square root of the dominant stochastic term is of order $(n\prod_{j=1}^d h_j)^{-1/2}$. Balancing this with the bias term $\sum_{j=1}^d h_j$ by choosing the bandwidth vector $\mathbf{h}$ optimally gives the rate $n^{-{1}/{(2+d)}}$.

The key insight why this matches the noise-free density estimation rate is that when $\beta_j = 1$, the factor $|r_{\varepsilon_j}(t)|^2$ cancels the $|t|^2$ term in the denominator. Estimating the cumulative distribution function introduces a factor $|t|^{-2}$ in the Fourier domain, which corresponds to a smoothness gain of order $1$ compared to density estimation. At the same time, the measurement error with $\beta_j = 1$ contributes a factor $|r_{\varepsilon_j}(t)|^2 \lesssim (1+|t|)^2$, which grows at the same rate. Hence, the ratio $|r_{\varepsilon_j}(t)|^2/|t|^2$ is asymptotically constant for $|t|> 1$, making the squared stochastic error scale as $(n\prod_{j=1}^d h_j)^{-1}$, exactly as in the noise-free density estimation case; see, \emph{e.g.}, \citet{ComteLacour2013}, p.~575, where $\varphi_{\varepsilon_j}(t)=1$ (\emph{i.e.}, no noise). The smoothness gain from estimating the cumulative distribution function therefore compensates perfectly for the measurement error when $\beta_j = 1$, explaining why the optimal convergence rate $n^{-{1}/{(2+d)}}$ coincides with the classical rate for nonparametric density estimation from direct observations.
\end{remark}


\section{Minimax optimality of the approximate minimum-\(L^1\) estimator}\label{sec:lowerbound}
In the previous section, we established that the estimator \(\widehat{F}_{n,\mathbf{h}}^{\ast}\) achieves, with an appropriate choice of the bandwidth vector \(\mathbf{h}\), an upper bound of order \(\psi_n^{\boldsymbol{\alpha},\boldsymbol{\beta}}\) (up to logarithmic factors) on the \(L^1\)-risk. To prove that this rate is minimax-optimal, we now derive a matching lower bound. The argument proceeds by constructing densities within the anisotropic Nikol'skii class for which the \(L^1\)-distance between the corresponding cumulative distribution functions is at least a constant multiple of \(\psi_n^{\boldsymbol{\alpha},\boldsymbol{\beta}}\), while the \(\chi^2\)-divergence between their convolved versions with the noise density \(f_\varepsilon\) is \(o(n^{-1})\). A two-point testing argument then yields the lower bound.

Define the set of active coordinates
\begin{equation}\label{eq:setactivecoord}
\mathcal P \coloneqq \bigl\{j\in\{1,\dots,d\}:\,\beta_j>\tfrac12\bigr\}.
\end{equation}
The lower bound construction is standard when no coordinates are active (\(|\mathcal P|=0\)), following the translation perturbation approach of \citet{ComteLacour2013}. However, when multiple coordinates are active (\(|\mathcal P|\ge 1\)), the translation perturbation no longer yields a density in the anisotropic Nikol'skii class with a constant independent of \(n\). A novel construction of the test densities is therefore required, which constitutes the main technical contribution of the lower bound proof.

Combining this lower bound with the upper bound from Theorem~\ref{thm:ubrate}, we conclude that the estimator \(\widehat{F}_{n,\mathbf{h}}^{\ast}\) is minimax-optimal (up to logarithmic factors) for the \(L^1\)-risk over the anisotropic Nikol'skii class. We now state the lower bound theorem; its proof, along with preparatory lemmas, is given in \hyperref[sec:S3]{Section~S3 of the Supplement}.


\begin{theorem}[Lower bound on the minimax $L^1$-risk]\label{th:LB}
Consider an i.i.d. sample $\mathbf{Y}_1,\,\dots,\,\mathbf{Y}_n$ from the multivariate convolution model in \eqref{eq:model}. 
Suppose that noise conditions \ref{eq:E1indep}, \ref{it:E22}, and \ref{eq:E3} hold with $\beta_j > 0$ for $j = 1,\, \dots,\, d$. Then, for any $\boldsymbol{\alpha}\in\mathbb R_+^d$ and $L>0$,
\[
\inf_{\hat{F}_n}\sup_{f_X\in\mathcal{N}_1^d(\boldsymbol{\alpha},L)}
\mathsf{E}_{(\mu_X\ast \mu_\varepsilon)^{\otimes n}}[\|\hat{F}_n-F_X\|_{L^1(\mathbb R^d)}]
\;\gtrsim\; \psi_n^{\boldsymbol{\alpha},\boldsymbol{\beta}},
\]
where $\hat{F}_n$ denotes a distribution function estimator based on $n$ observations,
and $\psi_n^{\boldsymbol{\alpha},\boldsymbol{\beta}}$ is as in \eqref{eq:rate}.
\end{theorem}

Some remarks are in order. The upper bound on the risk of the \(\varrho_n\)-approximate minimum-\(L^1\)-distance estimator \(\widehat F_{n,\mathbf{h}}^\ast\) requires the noise conditions \ref{eq:E1indep}, \ref{it:E21}, and \ref{eq:E3}, together with the finite‑moment conditions \eqref{eq:assonX} and \eqref{eq:assoneps}.
In contrast, the lower bound (Theorem~\ref{th:LB}) relies on the noise conditions \ref{eq:E1indep}, \ref{it:E22}, and \ref{eq:E3}, and does not require any moment conditions on the signal. Consequently, there exist heavy-tailed distributions that satisfy the lower-bound assumptions but not those of the upper bound.

To establish minimax optimality, the noise distribution must satisfy all the required noise conditions \ref{eq:E1indep}, \ref{it:E21}, \ref{it:E22}, and \ref{eq:E3}, as well as the moment conditions \eqref{eq:assonX} and \eqref{eq:assoneps}. Under these assumptions, the upper bound of Theorem~\ref{thm:ubrate} and the lower bound of Theorem~\ref{th:LB} hold on the same class 
\(\mathcal{N}_{\mathrm{mom}} \coloneqq \{f_X \in \mathcal{N}_1(\boldsymbol{\alpha},L):\, \text{\eqref{eq:assonX} and \eqref{eq:assoneps} hold} \}\).
Therefore, the minimax \(L^1\)-risk over this class satisfies
\[
\inf_{\hat{F}_n} \sup_{\mu_X \in \mathcal{N}_{\mathrm{mom}}} 
\mathsf{E}_{(\mu_X \ast \mu_\varepsilon)^{\otimes n}}[\|\hat{F}_n - F_X\|_{L^1(\mathbb R^d)}] 
\asymp \psi_n^{\boldsymbol{\alpha},\boldsymbol{\beta}}.
\]
Thus, the rate \(\psi_n^{\boldsymbol{\alpha},\boldsymbol{\beta}}\) is minimax-optimal up to logarithmic factors for this class.

\section{Noise-filtered harmonic mean representation of the rate exponent}\label{sec:structure}
This section analyzes the exponent of the convergence rate $\psi_n^{\boldsymbol{\alpha},\boldsymbol{\beta}}$ and its harmonic-mean representation. It also discusses the bottleneck effect and the curse of dimensionality, and compares it with the exponent in density deconvolution. We show that the positive-part threshold acts as a noise filter, and that this filter is naturally embedded in the harmonic mean structure of the exponent.

Recall that the $L^1$-minimax convergence rate for a $d$-dimensional deconvolution problem (up to a logarithmic factor) is
$\psi_n^{\boldsymbol{\alpha}, \boldsymbol{\beta}}= n^{-{E}/{2}}$,
where the exponent $E$ is defined in~\eqref{eq:rate}.
Here $\alpha_j > 0$ represents the smoothness of the density along the $j$-th coordinate direction,
and $\beta_j > 0$ characterizes the decay rate of the noise characteristic function in that direction. Only coordinate directions with $\beta_j > 1/2$ contribute to the sum; directions with $\beta_j\leq1/2$ do not affect the rate (up to a logarithmic factor).

For the univariate case ($d=1$) with $\beta>1/2$, the convergence rate becomes
\(
 n^{-{(\alpha+1)}{[2(\alpha+\beta) + 1]}}
\).
This rate is known to be minimax-optimal up to logarithmic factors for the $1$-Wasserstein distance. By the identity between the $1$-Wasserstein distance for probability measures on the real line and the $L^1$-distance between their corresponding cumulative distribution functions, see \cite{Rousseau&Scricciolo}, p.~1708, the same optimality holds for the $L^1$-distance between distribution functions.

For $d \ge 2$, when $\alpha_j = 0$ for all $j$ and the noise is isotropic, \emph{i.e.}, $\beta_1 =\, \dots \,= \beta_d = \beta$, the rate simplifies to
$
n^{-1/{\{2[1+d(\beta-1/2)_+]\}}}$,
which deteriorates as $d$ increases, the \emph{curse of dimensionality}.

The exponent $E$ admits a harmonic-mean representation that makes the filter explicit. Restrict attention to the directions that survive the filter, that is, those with $\beta_j > 1/2$; directions with $\beta_j \le 1/2$ simply disappear from the sum. For these directions, define
\(a_j \coloneqq \alpha_j+1 > 1\) and \(b_j \coloneqq \beta_j - 1/2 > 0\).
Let $H_{\textnormal{CF}}$ be the harmonic mean of the ratios $a_j/b_j$; in symbols,
\begin{equation}\label{eq:H_b}
H_{\textnormal{CF}} \coloneqq \operatorname{HM}\!\left\{\frac{a_j}{b_j}\right\}_{j=1}^d 
      = \frac{d}{\sum_{j=1}^d \frac{b_j}{a_j}}
      = \frac{d}{\sum_{j=1}^d \frac{\beta_j-\frac{1}{2}}{\alpha_j+1}}.
\end{equation}
Then, the exponent can be written as $E =(1 + {d}/{H_{\textnormal{CF}}})^{-1}$.
The noise filter operates on two levels.
\begin{itemize}
    \item \emph{Entry filter:} directions with $\beta_j \le 1/2$ are absent from the harmonic mean sum because their contribution $(\beta_j-1/2)_+$ is zero, so they do not appear in $H_{\textnormal{CF}}$;\\[-12pt]
    \item \emph{Bottleneck effect:} because $H_{\textnormal{CF}}$ is a harmonic mean, it is dominated by the smallest ratio $(\alpha_j+1)/(\beta_j-1/2)$. A single direction with poor characteristics (small $\alpha_j$ or large $\beta_j$) reduces $H_{\textnormal{CF}}$, reducing $E$ and slowing the convergence rate. This acts as a filter that amplifies the worst direction.
\end{itemize}
Thus, the harmonic mean representation not only encodes the threshold filter, but also reveals the bottleneck effect, a sensitivity to the most unfavourable coordinate.

The ratio \(d/H_{\textnormal{CF}}\) measures the \emph{effective difficulty}
of the deconvolution problem, combining anisotropic smoothness \(\alpha_j\) of the
density with the anisotropic decay $\beta_j$ of the noise characteristic 
function along the $j$-th coordinate direction. Since \(E = (1 + d/H_{\textnormal{CF}})^{-1}\) decreases as \(d/H_{\textnormal{CF}}\) increases, raising the dimension $d$ or lowering $H_{\textnormal{CF}}$ reduces \(E\). As shown in Figure~\ref{fig:convergence}, a smaller exponent \(E\) slows the minimax convergence rate \(\psi_n^{\boldsymbol{\alpha},\boldsymbol{\beta}} = n^{-E/2}\).
A large value of \(H_{\textnormal{CF}}\) indicates that in every direction either 
\(\alpha_j\) is large, or \(\beta_j\) is close to the 
threshold \(1/2\), or both. The ratio \((\alpha_j+1)/(\beta_j-1/2)\)
measures the ``advantage'' of smoothness over the excess decay \(\beta_j-1/2\): a large ratio means that smoothness dominates.

\begin{figure}[H]
\centering
\begin{minipage}{0.48\textwidth}
\centering
\begin{tikzpicture}
\begin{axis}[
    width=\textwidth,
    height=0.50\textwidth,
    xlabel={$\log_{10} n$},
    ylabel={$\log_{10}\psi_n^{\boldsymbol{\alpha},\boldsymbol{\beta}}$},
    grid=major,
    legend pos=north east,
    title={Impact of $E$ on convergence rate},
    xmajorgrids=true,
    ymajorgrids=true,
    xtick={1,2,3,4},
    xticklabels={$1$ ($n=10$), $2$ ($n=10^2$), $3$ ($n=10^3$), $4$ ($n=10^4$)},
    ytick={-0.5,-0.4,-0.3,-0.2,-0.1,0},
    yticklabels={$-0.5$,$-0.4$,$-0.3$,$-0.2$,$-0.1$,$0$},
    domain=1:4,
    xmin=1, xmax=4,
    ymin=-0.55, ymax=0.05,
    grid style={line width={0.2pt}, draw=gray!30},
    axis lines=left,
    tick label style={font=\footnotesize},
    label style={font=\footnotesize},
    title style={font=\small},
    legend style={
        font=\scriptsize,
        row sep=0.5pt,
        inner sep=2pt,
    },
]
\addplot[blue, thick, samples=50] {-0.025*x};
\addplot[red, thick, samples=50] {-0.075*x};
\addplot[green, thick, samples=50] {-0.125*x};
\addplot[orange, thick, samples=50] {-0.175*x};
\addplot[black, thick, dashed, samples=50] {-0.25*x};
\legend{$E=0.05$, $E=0.15$, $E=0.25$, $E=0.35$, $E=0.5$}
\end{axis}
\end{tikzpicture}
\caption{$\log_{10}\psi_n^{\boldsymbol{\alpha},\boldsymbol{\beta}} = -(E/2)\log_{10} n$ for different $E$.}
\label{fig:convergence}
\end{minipage}\hfill
\begin{minipage}{0.48\textwidth}
\centering
\begin{tikzpicture}
\begin{axis}[
    width=\textwidth,
    height=0.50\textwidth,
    xlabel={Direction $j$},
    ylabel={$\frac{\alpha_j+1}{\beta_j-\frac12}$},
    ymin=0, ymax=6,
    xtick={1,2,3,4,5},
    ytick={0,2,4,6},
    grid=major,
    title={Bottleneck effect},
    bar width=1pt,
    ymajorgrids=true,
    tick label style={font=\footnotesize},
    label style={font=\footnotesize},
    title style={font=\small},
]
\addplot[ybar, fill=blue!30, draw=blue!50] coordinates {
    (1,5.0) (2,5.0) (3,5.0) (4,5.0) (5,0.9)
};
\addplot[red, dashed, thick] coordinates {(0,2.62) (6,2.62)};
\node[red, font=\footnotesize] at (3,2.9) {$H_{\textnormal{CF}} \approx 2.62$};
\addplot[blue, dashed, thick] coordinates {(0,5.0) (6,5.0)};
\node[blue, font=\footnotesize] at (3,5.2) {$H_{\textnormal{CF}}$ if all good};
\end{axis}
\end{tikzpicture}
\caption{One coordinate with small $(\alpha_j+1)/(\beta_j-1/2)$ drastically reduces \(H_{\textnormal{CF}}\), and hence \(E\).}
\label{fig:bottleneck}
\end{minipage}
\end{figure}

\vspace{-0.3cm}

Because the harmonic mean \(H_{\textnormal{CF}}\) is dominated by the smallest ratio \((\alpha_j+1)/(\beta_j-1/2)\), a single direction with poor characteristics, either low \(\alpha_j\) or a large difference \(\beta_j-1/2\), can drastically reduce \(H_{\textnormal{CF}}\) and hence \(E\), creating a \emph{bottleneck effect}: one ``bad'' direction can destroy the overall estimation efficiency. Figure~\ref{fig:bottleneck} illustrates how one single unfavorable direction can dominate the convergence rate through the harmonic mean structure.

\vspace{-0.4cm}

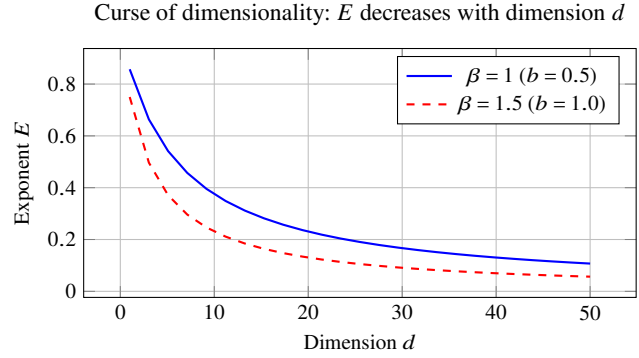
\begin{figure}[H]
\centering
\begin{tikzpicture}
\begin{axis}[
    width=0.62\textwidth,
    height=0.34\textwidth,
    xlabel={Dimension $d$},
    ylabel={Exponent $E$},
    domain=1:50,
    grid=major,
    legend pos=north east,
    title={Curse of dimensionality: $E$ decreases with dimension $d$},
    xmajorgrids=true,
    ymajorgrids=true,
    tick label style={font=\footnotesize},
    label style={font=\footnotesize},
    title style={font=\small},
    legend style={font=\footnotesize},
]
\addplot[blue, thick] {3/(3+0.5*x)};
\addplot[red, thick, dashed] {3/(3+1.0*x)};
\legend{$\beta=1$ ($b=0.5$), $\beta=1.5$ ($b=1.0$)}
\end{axis}
\end{tikzpicture}
\caption{In the isotropic case, for $\alpha=2$, the exponent $E = 3/[3 + d(\beta-1/2)]$ decays approximately as $1/d$.}
\label{fig:dimcurse}
\end{figure}

\vspace{-0.4cm}

In the \emph{isotropic case} where $\alpha_j \equiv \alpha$ and $\beta_j \equiv \beta > 1/2$ for all $j$, the exponent becomes
\[
E = \frac{\alpha+1}{\alpha+1 + d(\beta-\frac12)}=\frac{1}{1 + d\frac{(\beta-\frac12)}{\alpha+1}}\sim \frac{C}{d} \quad \text{as } d\to\infty,
\]
with $C \coloneqq (\alpha+1)/(\beta-1/2)$. This polynomial decay explains the \emph{curse of dimensionality}: the convergence rate becomes very slow in high dimensions. Figure~\ref{fig:dimcurse} illustrates this decay for $\alpha=2$ and two $\beta$ values.

To conclude, the harmonic-mean structure of $E$ reveals that both smoothness
and noise characteristics must be favorable in every coordinate to achieve fast 
convergence rates. The positive-part operation $(\beta_j-1/2)_+$ provides 
a mechanism for handling non-identifiable directions. The exponent 
suffers from both a bottleneck effect (the worst direction governs \(H_{\textnormal{CF}}\)) and 
a curse of dimensionality ($E \sim C/d$). Fast convergence therefore requires favorable smoothness-to-noise ratios 
$(\alpha_j+1)/(\beta_j-1/2)$ in each coordinate. Table~\ref{tab:effects} 
summarizes how different conditions affect $E$ and the resulting convergence rate.

\vspace{-0.35cm}

\begin{table}[H]
\caption{Effects on convergence exponent $E = (1 + d/H_{\textnormal{CF}})^{-1}$.}
\label{tab:effects}
\centering
\footnotesize
\setlength{\tabcolsep}{4pt}
\renewcommand{\arraystretch}{1.13}
\begin{tabular}{@{}lll@{}}
\toprule
\textbf{Condition} & \textbf{Effect on $E$} & \textbf{Rate consequence} \\
\midrule
Large \(H_{\textnormal{CF}}\) & Increases & Faster convergence \\
One small $a_j/b_j$ & Drastically decreases & Bottleneck effect \\
Increasing dimension \(d\) & Decreases & Curse of dimensionality \\
$\beta_j \leq \frac12$ for some \(j\) & Term excluded & Rate independent of that direction \\
All \(\alpha_j\) large & Increases \(H_{\textnormal{CF}}\) & Faster convergence \\
All \(\beta_j\) near \(\frac12\) & Increases \(H_{\textnormal{CF}}\) & Faster convergence \\
\bottomrule
\end{tabular}
\end{table}

\subsection{Comparison with density estimation rates}\label{rem:rate-comparison}
The convergence rate obtained for the cumulative distribution function estimator allows a transparent comparison with the minimax rate for anisotropic density deconvolution when both are expressed in terms of harmonic means. For density estimation under ordinary smooth noise, \citet{ComteLacour2013}, p.~575, obtained the rate $\psi_n^{\mathrm{dens}} \coloneqq n^{-E_{\mathrm{dens}}/2}$, where
\[
E_{\mathrm{dens}} \coloneqq \left[1+\sum_{j=1}^d\frac{\beta_j+ \frac{1}{2}}{\alpha_j}\right]^{-1}=\left[1 + \frac{d}{H_{\mathrm{dens}}}\right]^{-1}
\quad
\text{and}
\quad
H_{\mathrm{dens}} \coloneqq \operatorname{HM}\!\left\{\frac{\alpha_j}{\beta_j+\frac{1}{2}}\right\}_{j=1}^d 
      = \frac{d}{\sum_{j=1}^d \frac{\beta_j+\frac{1}{2}}{\alpha_j}},
\]
denotes the harmonic mean of the ratios $\alpha_j/(\beta_j+{1}/{2})$ for $j=1,\,\ldots,\,d$.

For the cumulative distribution function we have $
\psi_n^{\boldsymbol\alpha,\boldsymbol\beta} = n^{- E/2}$, with $E =(1 + {d}/ H_{\textnormal{CF}})^{-1}$,
where \(H_{\mathrm{CF}}\) is defined in \eqref{eq:H_b}.
If \(\alpha_j>0\) and \(\beta_j>{1}/{2}\) for all $j$, we have
\((\beta_j-1/2)/(\alpha_j+1)<(\beta_j+1/2)/\alpha_j\), $j=1,\,\ldots,\,d$.
Consequently, \(H_{\textnormal{CF}} > H_{\mathrm{dens}}\) and therefore \(E > E_{\mathrm{dens}}\). 
Since both rates are of the form \(n^{-(\cdot)/2}\), it follows that
$
\psi_n^{\boldsymbol\alpha,\boldsymbol\beta}=o(\psi_n^{\mathrm{dens}})$
as $n\to\infty$,
that is, the rate for the cumulative distribution function estimator is faster than the corresponding minimax rate for anisotropic density estimation in ordinary smooth deconvolution problems, as is typical when estimating integrated functionals. 
The improvement stems from two combined effects:\\[-15pt]
\begin{itemize}
    \item \emph{Smoothing by integration:} estimating the cumulative distribution function 
    \(F_X\) rather than the density \(f_X\) increases the regularity 
    of the target object, reflected in the shift from 
    \(\alpha_j\) to \(\alpha_j+1\) in each coordinate direction;\\[-7pt]
    \item \emph{Reduced sensitivity to noise:} estimating \(F_X\) rather than \(f_X\) reduces the effective noise exponent from $\beta_j$ to $\beta_j-1/2$ in each coordinate direction for which $\beta_j>1/2$. This reduction follows from the Fourier relationship between a cumulative distribution function and its density: the characteristic function of the cumulative distribution function equals that of the density divided by $\mathrm{i} t_j$ (up to a constant). Thus, its Fourier transform decays as $|t_j|^{-1}$ relative to that of the density, cancelling one power of the high-frequency amplification factor per coordinate. Consequently, the variance contribution for each coordinate involves $h_j^{-(\beta_j-1/2)_+}$ instead of $h_j^{-(\beta_j+1/2)}$. This shift modifies the overall rate exponent, yielding a strictly faster rate than the density version. Taking the product over all $d$ coordinates gives a global factor $\prod_{j=1}^d (\mathrm{i} t_j)^{-1}$. Hence, the full variance of the cumulative distribution function estimator is proportional to $n^{-1/2}\prod_{j=1}^d h_j^{-(\beta_j-1/2)_+}$, whereas for density estimation the analogous term is $n^{-1/2}\prod_{j=1}^d h_j^{-(\beta_j+1/2)}$, polynomially larger by a factor $\prod_{j=1}^d h_j^{-1}$. This is captured by the shift from \(\beta_j+{1}/{2}\) to \(\beta_j-{1}/{2}\) in the numerator.
\end{itemize}
Therefore, cumulative distribution function deconvolution achieves uniformly faster convergence 
than density deconvolution in anisotropic ordinary smooth models when \(\alpha_j>0\) and \(\beta_j>{1}/{2}\) for all $j$. If some $\alpha_j=0$ or $\beta_j\leq1/2$, the comparison may differ.


\subsection{Example}
In this section, we consider a family of distributions for the error density used by \citet{Caillerie2011}, pp.~1407--1408, in the context of $2$-Wasserstein ($W_2$) multivariate deconvolution. This family provides a tractable framework for analysing how the smoothness of an error distribution affects convergence rates.
For a shape parameter $\gamma \in (0,\,2]$, consider the density defined by
\(
f_{\gamma}(u) = \gamma/[2\,\Gamma(1/\gamma)]\,\mathrm{e}^{-|u|^{\gamma}}\), $u\in\mathbb{R}$.
This density is a symmetric exponential power density, which includes the Laplace ($\gamma=1$) and Gaussian ($\gamma=2$) cases. Its Fourier transform $\mathcal{F}[f_{\gamma}]$ has closed forms in these two cases. For $\gamma=1$ we obtain the Laplace density \(f_{1}(u)=\mathrm{e}^{-|u|}/2\) with Fourier transform \(\mathcal{F}[f_{1}](t)=1/(1+t^{2})\), while for $\gamma=2$ we obtain the Gaussian density \(f_{2}(u)=\mathrm{e}^{-u^{2}}/\sqrt{\pi}\) with Fourier transform \(\mathcal{F}[f_{2}](t)=\mathrm{e}^{-t^{2}/4}\), for \(u,\,t\in\mathbb{R}\). For $\gamma>2$, $f_{\gamma}$ does not have a polynomially decaying Fourier transform, so the “ordinary smooth” assumption restricts $\gamma\in(0,\,2)$. For $\gamma=2$ the decay is exponential (super-smooth). The function $\mathcal{F}[f_{\gamma}]$ is positive, symmetric and infinitely differentiable because
$\int_{\mathbb{R}} |x|^{j} \mathrm{e}^{-|x|^{\gamma}}\,\mathrm{d}x<\infty$ for every $j\in\mathbb{N}_{0}$. Its derivatives are given by
\((\mathcal{F}[f_{\gamma}])^{(j)}(t) = \gamma /[2\Gamma(1/\gamma)] \int_{\mathbb{R}} \mathrm{e}^{\mathrm{i}tx}\,(\mathrm{i}x)^{j}\,\mathrm{e}^{-|x|^{\gamma}}\,\mathrm{d}x\), $t\in\mathbb{R}$.
By the asymptotic expansion of \citet{Bergstrom1952}, for every $\gamma\in(0,\,2)$ there is a constant $C(\gamma)>0$ such that
$\mathcal{F}[f_{\gamma}](t)=C(\gamma)\,t^{-(\gamma+1)}+O(t^{-(2\gamma+1)})$ as $t\to\infty$.
Differentiating term-by-term yields, for any $j\in\mathbb{N}$,
\[
(\mathcal{F}[f_{\gamma}])^{(j)}(t)=(-1)^{j}\,C(\gamma)\,\frac{\Gamma(\gamma+1+j)}{\Gamma(\gamma+1)}\,t^{-(\gamma+1+j)}+O(t^{-(2\gamma+1+j)})\quad\text{as }t\to\infty.
\]

We now consider a setting where each component $\varepsilon_j$ of the error vector $\boldsymbol{\varepsilon}=(\varepsilon_1,\,\dots,\,\varepsilon_d)^\top$ is distributed according to the $k_j$-fold convolution of $f_{\gamma_j}$. In other words, the characteristic function of $\varepsilon_j$ is
$\varphi_{\varepsilon_j}(t)=(\mathcal{F}[f_{\gamma_j}](t))^{k_j}$, $t\in\mathbb{R}$.
Using the asymptotic expansion above, valid for $\gamma_j\in(0,\,2)$, we obtain the decay rate
$\varphi_{\varepsilon_j}(t)=(C(\gamma_j))^{k_j}\,t^{-(\gamma_j+1)k_j}+O(t^{-(\gamma_j+1)k_j-\gamma_j})$ as $t\to\infty$. This decay defines the smoothness parameter of the error distribution as
$\beta_j \coloneqq (\gamma_j+1)k_j$, $j=1,\,\dots,\,d$. Since we require ordinary smooth noise, we take \(\gamma_j\in(0,\,2)\) and \(k_j\ge 1\) so that \(\beta_j>0\) with no further upper bound (the exponent can be arbitrarily large, \emph{e.g.}, $\gamma_j=0.5$, $k_j=4$ gives $\beta_j=(0.5+1)4=6$).

If $f_X$ belongs to the anisotropic Nikol'skii class with $\alpha_j=0$ for all $j$ (by convention, this means that $f_X\in L^1(\mathbb R)$ with no further differentiability; the class can be extended to include $\alpha_j=0$ as the $L^1$-space), then the convergence rate under the present noise model becomes
\(
n^{-1/\{2\sum_{j=1}^d[(\gamma_j+1)k_j]-(d-2)\}}
\).
Define the \emph{total smoothness index} $I(d)\coloneqq\sum_{j=1}^d\beta_j=\sum_{j=1}^d(\gamma_j+1)k_j$.
Assuming $k_j\ge 1$ for all $j$, so that each error component is at least once convolved, the rate simplifies to $n^{-{1}/{[2I(d)-(d-2)]}}$.


\vspace*{0.2cm}

The following table summarises the convergence rates for several configurations, illustrating the interplay between dimension $d$, convolution orders $k_j$ and shape parameters $\gamma_j$, with all $\alpha_j=0$.

\vspace*{-0.15cm}

\begin{table}[H]
\centering
\scriptsize
\setlength{\tabcolsep}{4pt}
\renewcommand{\arraystretch}{1.0}
\caption{Convergence rates for different noise configurations under the convolution model.}
\label{tab:rates}
\begin{tabular}{@{}lcc@{}}
\toprule
Case & $I(d)$ & Rate \\
\midrule
Univariate ($d=1$) & $(\gamma_1+1)k_1$ & $n^{-1/[2(\gamma_1+1)k_1+1]}$ \\
Isotropic & $d(\gamma+1)$ & $n^{-1/[2d(\gamma+1)-(d-2)]}$ \\
Single component & $d(\gamma+1)$ & $n^{-1/[2d(\gamma+1)+1]}$ \\
General anisotropic & $\sum_j(\gamma_j+1)k_j$ & $n^{-1/\{2\sum_j[(\gamma_j+1)k_j]-(d-2)\}}$ \\
\bottomrule
\end{tabular}
\end{table}

\vspace*{-0.5cm}

Table \ref{tab:rates} reveals the \emph{curse of dimensionality}: the exponent of $n$ deteriorates as $d$ increases unless the total smoothness $I(d)$ grows sufficiently fast. The rate is not determined solely by the total smoothness \(I(d)\); it also depends on the distribution of the noise across coordinates. For instance, while the isotropic case and the single-component case both have \(I(d)=d(\gamma+1)\), they yield different rates. The isotropic case gives \(n^{-1/[2d(\gamma+1)-(d-2)]}\), whereas the single-component case gives \(n^{-1/[2d(\gamma+1)+1]}\) because the effective number of ``active'' coordinates (those with \(\beta_j>1/2\)) enters the denominator. Thus, concentrating noise smoothness in a single coordinate is less favorable than spreading it isotropically across all coordinates, since in the latter case the denominator contains $d$ instead of a constant.

\section{Adaptive distribution function deconvolution via oracle bounds}\label{sec:adaptivity}
In this section, we develop a fully data-driven procedure for recovering \(F_X\) that automatically attains the optimal bias-variance trade-off without any prior knowledge of the mixing measure \(\mu_X\) or the smoothness regularity of \(F_X\). Since this regularity is unknown, the optimal bandwidth vector \(\mathbf{h}_{\text{opt}}\) from Theorem~\ref{thm:ubrate}, which depends on the unknown exponents \(\alpha_j\), cannot be computed in practice.
Our goal is therefore to construct a data-driven bandwidth selector \(\hat{\mathbf{h}}\) over a finite collection \(\mathcal H_{M_n}\) such that the resulting kernel-based estimator \(\widehat{F}_{n,\hat{\mathbf{h}}}\) performs as well as an oracle that knows the best bandwidth choice in \(\mathcal H_{M_n}\), up to logarithmic factors. Specifically, we aim to select \(\hat{\mathbf{h}} \in \mathcal H_{M_n}\) so that the expected $L^2$-error of \(\widehat{F}_{n,\hat{\mathbf{h}}}\) in estimating \(F_X\) is close to the oracle risk
$
\inf_{\mathbf{h} \in \mathcal H_{M_n}} \mathsf{E}_{(\mu_X\ast \mu_\varepsilon)^{\otimes n}}[\|\widehat{F}_{n,\mathbf{h}}-F_X\|_{L^2(\mathbb R^d)}]$.
An estimator satisfying an inequality of the form
\begin{equation*}\label{eq:oracle_ineq}
\mathsf{E}_{(\mu_X\ast \mu_\varepsilon)^{\otimes n}}[\|\widehat{F}_{n,\hat{\mathbf{h}}}-F_X\|_{L^2(\mathbb R^d)}] \leq C \inf_{\mathbf{h} \in \mathcal H_{M_n}} \mathsf{E}_{(\mu_X\ast \mu_\varepsilon)^{\otimes n}}\bigl[\|\widehat{F}_{n,\mathbf{h}}-F_X\|_{L^2(\mathbb R^d)}\bigr] + R_n
\end{equation*}
for some constant \(C>0\) and a remainder term 
$R_n = o(\inf_{\mathbf{h}\in\mathcal H_{M_n}} \mathsf{E}_{(\mu_X\ast \mu_\varepsilon)^{\otimes n}}[\|\widehat{F}_{n,\mathbf{h}}-F_X\|_{L^2(\mathbb R^d)}])$,
is said to satisfy an \emph{oracle inequality}. When such an inequality holds uniformly over a collection of function classes (\emph{e.g.}, anisotropic H\"older, Nikol'skii, or Sobolev balls), the procedure is called \emph{adaptive}, as its convergence rate automatically adjusts to the unknown regularity of \(F_X\).

Our approach is analogous to the model-selection methodology developed by \cite{GoldenshlugerLepski2011}, later adapted by \cite{Chagny2013} to conditional density estimation. We extend this strategy to the present multivariate cumulative distribution function deconvolution setting and prove an oracle inequality for the resulting estimator.
To describe the procedure, we introduce some notation. 
For a bandwidth vector \(\mathbf{h}=(h_1,\,\ldots,\,h_d)^\top \in \mathbb{R}_+^d\), recall the definition of a rescaled product kernel $K_{\mathbf{h}}(\mathbf{x}) \coloneqq \prod_{j=1}^d h_j^{-1} K_j({x_j}/{h_j})$, $\mathbf{x}=(x_1,\,\ldots,\,x_d)^\top \in \mathbb{R}^d$, cf. \eqref{eq:productkern}. 
The cumulative distribution function estimator is defined as the integral of the corresponding kernel density estimator
\[
\widehat{F}_{n,\mathbf{h}}(\mathbf{x}) \coloneqq
\int_{-\infty}^{x_1} \cdots \int_{-\infty}^{x_d} 
\mathcal{F}^{-1}\!\left[\frac{\varphi_K(\mathbf{h} \odot \mathbf{t}
) \,\varphi_n(\mathbf{t})}{\varphi_\varepsilon(\mathbf{t})}\right](\mathbf{u}) \; \mathrm{d}\mathbf{u}, \quad \mathbf{x} \in \mathbb{R}^d,
\]
where \(\varphi_K\), \(\varphi_\varepsilon\), and \(\varphi_n\) denote the characteristic functions of the kernel, the error distribution, and the empirical measure, respectively; the latter is defined as
$\varphi_n(\mathbf{t}) \coloneqq {n}^{-1} \sum_{k=1}^{n} \mathrm{e}^{\mathrm{i} \langle \mathbf{t},\, Y_k \rangle}$, $\mathbf{t}\in\mathbb R^d$.
Also, \(\mathbf{h}\odot\mathbf{t} \coloneqq (h_1 t_1,\, \ldots,\, h_d t_d)^\top\) denotes the componentwise (Hadamard) product, which corresponds to the action of the diagonal bandwidth matrix \(\mathbf{H}=\operatorname{diag}(h_1,\,\ldots,\,h_d)\) on the vector \(\mathbf{t}\), \emph{i.e.}, \(\mathbf{H}\mathbf{t} = \mathbf{h}\odot\mathbf{t}\).
The bandwidth selector \(\hat{\mathbf{h}}\) is chosen from a finite collection \(\mathcal H_{M_n}\) (defined below in \eqref{eq:setH_M}):
\begin{equation}\label{eq:hopt}
\hat{\mathbf{h}} \coloneqq \argmin_{\mathbf{h} \in \mathcal H_{M_n}} \bigl\{ D(\mathbf{h}) + \tilde V(\mathbf{h}) \bigr\},
\end{equation}
where the term
\begin{equation}\label{eq:Ah}
D(\mathbf{h}) \coloneqq \sup_{\mathbf{h}' \in \mathcal H_{M_n}} \Bigl(\|\widehat{F}_{n,\mathbf{h}'} - (K_{\mathbf{h}'}\ast \widehat{F}_{n,\mathbf{h}})\|_{L^2(\mathbb R^d)} - \tilde V(\mathbf{h}')\Bigr)_+
\end{equation}
serves as an empirical proxy of the bias contribution, and the squared penalty term is given by
\begin{equation}\label{eq:proxyvar}
\tilde V^2(\mathbf{h}) \coloneqq C_{d,\boldsymbol{\beta}} \mathcal V^2(\mathbf{h}),
\end{equation}
with $\mathcal V(\mathbf{h})$ defined in \eqref{eq:varprox},
and $C_{d,\boldsymbol{\beta}}>0$ a constant depending only on \(d\) and the exponents \(\beta_j\), see \eqref{eq:defvtilde} in \hyperref[sec:S4]{Section~S4 of the Supplement} for the exact expression. 

To obtain sharp control on the expected value of \(D(\mathbf{h})\), we first establish in Proposition \ref{prop:oracle} a concentration inequality for the empirical process. Applying this result then yields a non-asymptotic upper bound for \(\mathsf{E}_{(\mu_X\ast \mu_\varepsilon)^{\otimes n}}[D(\mathbf{h})]\). That proposition provides a \emph{fixed-bandwidth} concentration inequality for \(\widehat F_{n,\mathbf{h}}\) and forms the core technical ingredient from which the adaptive oracle inequality of Theorem~\ref{thm:oracle_ineq} is derived. It controls the stochastic fluctuations of the \(L^2\)-error uniformly over \(\mathbf{h}\) in terms of the standard deviation proxy \(\mathcal V(\mathbf{h})\), enabling a data-driven bandwidth selector to mimic the oracle performance.

For a given sequence of positive integers $M_n\to\infty$ such that $M_n \le n$, the adaptive procedure selects the bandwidth vector from the anisotropic grid
\begin{equation}\label{eq:setH_M}
\mathcal{H}_{M_n}\coloneqq
\Bigl\{
\mathbf{h}^{(\mathbf{k})}\in\mathbb{R}_+^d:\; 
h_j^{(\mathbf{k})}=k_j^{-1},\; j=1,\,\dots,\,d,\; 
\mathbf{k}\in\{1,\,\dots,\,M_n\}^d,\; 
V(\mathbf{h}^{(\mathbf{k})})\le 1
\Bigr\}.
\end{equation}
We give conditions on $(M_n)_{n\in\mathbb N}$ ensuring the grid is rich enough to approximate the optimal bandwidth vector $\mathbf{h}_{\mathrm{opt}}$. The optimal bandwidth coordinates satisfy $h_{\mathrm{opt},j}\asymp n^{-\zeta_j}$, where $\zeta_j\coloneqq E/[2(\alpha_j+1)]\in(0,\,1/2]$. To approximate $\mathbf{h}_{\mathrm{opt}}$ by a grid point, it suffices to choose, for each coordinate $j$, an integer $k_j \asymp n^{\zeta_j}$. This yields $h_j^{(k_j)} = k_j^{-1} \asymp n^{-\zeta_j} \asymp h_{\mathrm{opt},j}$ for every coordinate $j$ simultaneously. Since $\zeta_j \le 1/2$ for all $j$, the approximation condition holds if $\liminf_{n\to\infty} (M_n/n^{1/2})>0$. Since $\mathbf{h}_{\mathrm{opt}}$ satisfies $V(\mathbf{h}_{\mathrm{opt}}) \le 1$ up to constants, the approximating grid point $\mathbf{h}^{(\mathbf{k})}$ can be chosen to satisfy $V(\mathbf{h}^{(\mathbf{k})}) \le 1$ by adjusting the proportionality constant, absorbed into the universal constants of the oracle inequality.
The cardinality of the grid must also be controlled for the exponential sums in the proof of Theorem~\ref{thm:oracle_ineq}. Since $\operatorname{card}(\mathcal{H}_{M_n}) = M_n^d$, we require $\log M_n = O(\log n)$. The simplest choice satisfying both conditions is $M_n = \lfloor n^{1/2} \rfloor$.\footnote{Note that \(M_n\) as either \(\lfloor n^{1/2} \rfloor\) or \(\lceil n^{1/2} \rceil\) differs by at most \(1\), asymptotically negligible. Both choices yield \(M_n \asymp n^{1/2}\), all that is required for the approximation condition \(\liminf_{n\to\infty} (M_n/n^{1/2})>0\) and the summability condition \(\log M_n = O(\log n)\).} Indeed, for each coordinate, setting $k_j = \lfloor n^{\zeta_j}\rfloor$ gives $k_j \le n^{1/2} \asymp M_n$.
A rigorous justification that the infimum over the discrete grid $\mathcal{H}_{M_n}$ is equivalent, up to universal constants, to the infimum over $(0,\,1)^d$ is deferred to the proof of Theorem~\ref{thm:oracle_ineq}, where the grid approximation is shown to lose only a constant factor, absorbed into the universal constant of the oracle inequality.

Recalling the definition of $\hat{\mathbf{h}}$ in \eqref{eq:hopt},
the kernel-based estimator is defined as
\begin{equation}\label{eq:defKestim}
\check{F}_n \coloneqq \widehat{F}_{n,\hat{\mathbf{h}}}.
\end{equation}
When combined with the deterministic bias bound (Proposition~\ref{prop:bias}) and the definition of $\hat{\mathbf{h}}$ as the minimizer of the empirical criterion $D(\mathbf{h})+\tilde V(\mathbf{h})$, the concentration inequality of Proposition~\ref{prop:oracle} implies the oracle inequality
\(
\|\check{F}_n - F_X\|_{L^2(\mathbb R^d)} 
\;\lesssim\; 
\inf_{\mathbf{h}\in\mathcal{H}_{M_n}}
\{ \|B_{\mathbf{h}}\|_{L^2(\mathbb R^d)} 
+ \mathcal V(\mathbf{h})\}
+ {n}^{-1/2}.
\)
Thus, $\check{F}_n$ achieves the optimal bias-variance trade-off adaptively, without prior knowledge of the smoothness parameters. Proposition~\ref{prop:oracle} is a \emph{fixed-bandwidth} deviation bound, while the term ``oracle inequality'' is reserved for the adaptive result of Theorem~\ref{thm:oracle_ineq} below, which compares the data-driven estimator to the oracle infimum over $\mathbf{h}$. Thus, the adaptive kernel-based estimator \(\check{F}_n\) achieves the oracle rate. Then, 
\begin{equation}\label{eq:reductionstep}
\|\widehat F_{n,\hat{\mathbf{h}}}^\ast - F_X\|_{L^2(\mathbb R^d)} \le 2\|\check{F}_n - F_X\|_{L^2(\mathbb R^d)} + \varrho_n,\end{equation}
with $\varrho_n=O({n}^{-1/2})$, the rate is transferred to the minimum-\(L^2\)-distance estimator \(\widehat F_{n,\hat{\mathbf{h}}}^\ast\).

\begin{theorem}[Oracle inequality]\label{thm:oracle_ineq}
Assume that the noise conditions  \ref{eq:E1indep}, \ref{it:E21} and \ref{eq:E3} hold with \(\beta_j>0\) for all \(j=1,\,\ldots,\,d\). Consider the kernel-based estimator \(\widehat F_{n,\mathbf{h}}\) constructed with a product kernel \(K(\mathbf{x})=\prod_{j=1}^d K_j(x_j)\), \(\mathbf{x}\in\mathbb R^d\), where each \(K_j\) satisfies the conditions (i)--(iii) listed in Proposition~\ref{prop:variance}.
Then, for any cumulative distribution function \(F_X\) for which the bias \(B_{\mathbf{h}} \coloneqq F_X - (K_{\mathbf{h}}\ast F_X)\) belongs to \(L^2(\mathbb R^d)\) for all \(\mathbf{h}\in\mathcal H_{M_n}\), where \(\mathcal H_{M_n}\) is the bandwidth grid defined in \eqref{eq:setH_M} with $M_n=\lfloor n^{1/2} \rfloor$, the minimum-\(L^2\)-distance estimator \(\widehat F_{n,\hat{\mathbf{h}}}^\ast\) (defined in \eqref{eq:defminimizer} but with respect to the \(L^2\)-distance instead of the \(L^1\)-distance) satisfies, with constants \(D_1,\,D_2>0\) depending only on the kernels, the noise distribution, and the dimension \(d\),
\begin{equation}\label{eq:discretegrid}
\mathsf E_{(\mu_X\ast\mu_\varepsilon)^{\otimes n}}\Bigl[\|\widehat F_{n,\hat{\mathbf{h}}}^\ast-F_X\|_{L^2(\mathbb R^d)}\Bigr]
\leq D_1 \inf_{\mathbf{h}\in\mathcal H_{M_n}} \bigl\{\|B_{\mathbf{h}}\|_{L^2(\mathbb R^d)} + \mathcal V(\mathbf{h})\bigr\} + \frac{D_2}{\sqrt{n}},
\end{equation}
where \(\mathcal V^2(\mathbf{h})\) is the variance proxy (cf. \eqref{eq:varprox}), and the approximation error \(\varrho_n\) arising from the projection of the kernel-based estimator onto the class of all cumulative distribution functions $\mathfrak F$ on $\mathbb R^d$ satisfies \(\varrho_n = O({n}^{-1/2})\), cf. \eqref{eq:reductionstep}; this term has been absorbed into the constant \(D_2/\sqrt{n}\).

Since the grid $\mathcal H_{M_n}$ is chosen so that the discrete infimum is equivalent, up to constants, to the continuous infimum over $(0,\,1)^d$ (see the grid approximation argument in the proof), inequality \eqref{eq:discretegrid} implies the oracle inequality
\(
\mathsf E_{(\mu_X\ast\mu_\varepsilon)^{\otimes n}}[\|\widehat F_{n,\hat{\mathbf{h}}}^\ast-F_X\|_{L^2(\mathbb R^d)}]
\lesssim
\mathcal R_n
\), where 
\begin{equation}\label{eq:defRn}
\mathcal R_n \coloneqq \inf_{\mathbf{h} \in (0,1)^d} \bigl\{ \|B_{\mathbf{h}}\|_{L^2(\mathbb R^d)} + \mathcal V(\mathbf{h}) \bigr\} + \frac{1}{\sqrt{n}}.
\end{equation}
\end{theorem}
 
While the proof of Theorem~\ref{thm:oracle_ineq}, given in \hyperref[sec:S4]{Section~S4 of the Supplement}, uses the same Goldenshluger--Lepski (G--L) type of selector $D(\mathbf{h})$ as \citet{ComteLacour2013}, it differs in three main respects. First, we estimate the cumulative distribution function rather than the density. This changes the Fourier deconvolution kernel and the bias-variance trade-off, requiring care with the logarithmic factors appearing when $\beta_j=1/2$. Second, the estimator of interest is not the kernel-based estimator $\widehat F_{n,\hat{\mathbf{h}}}$, but the minimum-$L^2$-distance estimator $\widehat F_{n,\hat{\mathbf{h}}}^\ast$. This adds an extra step to the proof, bounding the $L^2$-error of $\widehat F_{n,\hat{\mathbf{h}}}^\ast$ by that of $\widehat F_{n,\hat{\mathbf{h}}}$ via the inequality in \eqref{eq:reductionstep}.
Third, the bandwidth grid \(\mathcal H_{M_n}\) is anisotropic, in contrast to the isotropic grid in \citet{ComteLacour2013}. Consequently, the summability analysis over this grid requires carefully chosen functions \(C_{0j}(h_j')\) depending on the individual bandwidth components \(h_j'\), tailored to the anisotropic exponents, so that no artificial losses are introduced.
In practice, however, the regularity parameters $\alpha_j$ are unknown, preventing the direct use of a kernel of a prescribed finite order tailored to $\boldsymbol{\alpha}$. As discussed in Section~\ref{sec:convergence}, this difficulty can be circumvented by a superkernel, whose infinite order ensures that the bias bound holds uniformly over all regularity levels, rendering the adaptive estimator fully data-driven and free from any prior knowledge of the unknown smoothness.


We now establish that the G--L selector, designed for \(L^2\)-estimation, automatically leads to optimal rates simultaneously in both the \(L^1\)- and \(L^2\)-norms, and that this robustness extends to both approximate \(L^1\)- and \(L^2\)-projection estimators. The proof is given in \hyperref[sec:S4]{Section~S4 of the Supplement}.

\begin{theorem}[Unified oracle inequalities for \(L^1\)- and \(L^2\)-norms]\label{thm:unified_oracle}
Under the assumptions of Theorem~\ref{thm:oracle_ineq}, assume further that there is a constant \(0<C_1<\infty\) such that, for all \(\mathbf{h} \in (0,\,1)^d\), the bias comparability condition 
\begin{equation}\label{eq:l1l2bias}
\|B_{\mathbf{h}}\|_{L^1(\mathbb R^d)} \le C_1 \|B_{\mathbf{h}}\|_{L^2(\mathbb R^d)}
\end{equation}
holds. 
Let \(\check{F}_n \coloneqq \widehat{F}_{n,\hat{\mathbf{h}}}\) be the kernel-based estimator defined in \eqref{eq:defKestim} with bandwidth vector $\hat{\mathbf{h}}$ selected by the G--L method. For \(q \in \{1,\,2\}\), let
$\widehat F_{n,\hat{\mathbf{h}}}^{\ast,q}$
be a $\varrho_{n,q}$-approximate \(L^q\)-minimizer, with $\varrho_{n,q}=O({n}^{-1/2})$, defined in \eqref{eq:defminimizer} for $q=1$ and in \eqref{eq:reductionstep} for $q=2$. Then,
$\mathsf{E}_{(\mu_X\ast\mu_\varepsilon)^{\otimes n}}[\|\check{F}_n - F_X\|_{L^q(\mathbb R^d)}]
\lesssim
\mathcal R_n$ and
$\mathsf{E}_{(\mu_X\ast\mu_\varepsilon)^{\otimes n}}[\|\widehat F_{n,\hat{\mathbf{h}}}^{\ast,q} - F_X\|_{L^q(\mathbb R^d)}]
\lesssim
\mathcal R_n$ for $q= 1,\, 2$, where $\mathcal R_n$ is defined in \eqref{eq:defRn}.
\end{theorem}

The uniform bias comparability condition \eqref{eq:l1l2bias} holds, for instance, if \(F_X\) has compact support or sufficiently fast decaying tails uniformly in \(\mathbf{h}\). Theorem~\ref{thm:unified_oracle} conveys a unified message for adaptive deconvolution of distribution functions. First, it shows that the G--L selector, despite being designed for squared error loss, is not confined to that metric: the very same bandwidth vector achieves optimal rates for the \(L^1\)-risk as well, under only the bias comparability assumption \eqref{eq:l1l2bias}. Second, it carries over to the approximate \(L^1\)- and \(L^2\)-minimum distance estimators: the projection step preserves the corresponding oracle inequality without loss in the rate. This unification reveals that the G--L method is robust across norms, so the practitioner need not choose between \(L^1\)- and \(L^2\)-estimation at the bandwidth selection stage. In practice, one can safely use the \(L^2\)-G--L selector, confident of good performance in both norms, and, depending on the application, project the estimator onto the space of distribution functions in either norm without compromising the rate. In all cases the same oracle rate \(\mathcal R_n\) is achieved, so the projection norm can be chosen for computational convenience rather than statistical trade-offs. This robustness, with the simple bandwidth selection rule, highlights the practical appeal of the G--L methodology for nonparametric deconvolution of cumulative distribution functions.


\section{Minimax \(1\)-Wasserstein projection rate for product densities}\label{sec:equivalence}
In this section, we establish minimax optimality of a projection estimator for the Kantorovich--Wasserstein distance on product signal probability measures with anisotropic Nikol'skii marginals. Building on the $L^1$-cumulative distribution function deconvolution theory of the preceding sections, we prove matching lower and upper bounds, showing that an approximate minimum-$W_1$-distance estimator achieves the minimax $L^1$-rate $\psi_n^{\boldsymbol{\alpha},\boldsymbol{\beta}}$ in \eqref{eq:rate}. The estimator is tailored to $W_1$, which in dimension one coincides with the \(L^1\)-distance between cumulative distribution functions, motivating the \(L^1\)-based construction. Although these distances are not equivalent in general when $d\geq2$, the minimax rate on product measures reveals that they share the same optimal rate, pointing to a deeper structural connection and shedding light on why these metrics exhibit the same asymptotic behavior in the product setting.

We begin by recalling the definition of the \(1\)-Wasserstein distance. For \(d \ge 1\), the \(1\)-Wasserstein distance between two probability measures \(\mu\) and \(\nu\) on \(\mathbb R^d\) with finite first moments is defined by
\[
\We(\mu,\,\nu)
\coloneqq \inf_{\pi \in \Pi(\mu,\nu)} \int_{\mathbb R^d \times \mathbb R^d} \|\mathbf x - \mathbf y\|_{\ell^2} \, \di\pi(\mathbf x,\,\mathbf y),
\]
where \(\Pi(\mu,\,\nu)\) denotes the set of all couplings of \(\mu\) and \(\nu\), and \(\|\cdot\|_{\ell^2}\) is the Euclidean norm on \(\mathbb R^d\), see, \emph{e.g.}, \citet{villani}, for a comprehensive treatment of Wasserstein distances. The cost function in the definition above is the Euclidean distance; however, one could equally well consider the $\ell^1$-distance as cost. For the proofs below, it is in fact more convenient to work with the \(\ell^1\)-norm \(\|\mathbf z\|_{\ell^1}\coloneqq \sum_{j=1}^d|z_j|\), defining the corresponding distance
$\Wl(\mu,\,\nu) \coloneqq \inf_{\pi \in \Pi(\mu,\nu)} \int_{\mathbb R^d \times \mathbb R^d} \|\mathbf x - \mathbf y\|_{\ell^1} \, \di\pi(\mathbf x,\,\mathbf y)$. Since $
\|\mathbf z\|_{\ell^2}\le \|\mathbf z\|_{\ell^1}\le \sqrt{d}\,\|\mathbf z\|_{\ell^2}$, for $\mathbf z \in \mathbb R^d$,
the two distances satisfy 
\(
\We(\mu,\,\nu) \le \Wl(\mu,\,\nu) \le \sqrt{d}\,\We(\mu,\,\nu)\),
so the minimax rates obtained for \(\Wl\) coincide with those for \(\We\) up to constants depending only on \(d\).
In the case \(d = 1\), both norms coincide with the absolute value, and the above definitions reduce to the well-known representation; see \citet{Dall'Aglio},
\begin{equation}\label{eq:identity}
\We(\mu,\,\nu) = \Wl(\mu,\,\nu) = \int_{\mathbb R} |F_\mu(x) - F_\nu(x)| \, \di x,
\end{equation}
where \(F_\mu\) and \(F_\nu\) denote the cumulative distribution functions of \(\mu\) and \(\nu\), respectively.
We next introduce the class of signal probability measures for which we establish minimax optimality.

\begin{definition}[Product class with finite first moment]\label{def:prodclass}
Let \(\mathcal P_{\text{prod}}\) denote the class of signal probability measures \(\mu_X\) on \(\mathbb R^d\) with product densities
\(
f_X(\mathbf x) \coloneqq \prod_{j=1}^d f_{X,j}(x_j)\), $\mathbf x = (x_1,\,\dots,\,x_d)^\top \in \mathbb R^d$,
where each marginal density \(f_{X,j}\) belongs to the one-dimensional anisotropic Nikol'skii class \(\mathcal N_1^1(\alpha_j,\, L_j)\) as in Definition~\ref{def:anisotropy} (with \(d=1\)), for some \(\alpha_j > 0\) and \(L_j > 0\), and additionally satisfies
$\int_{\mathbb R} |x| f_{X,j}(x) \, \di x < \infty$.
\end{definition}
The finite first moment condition ensures that the \(1\)-Wasserstein distance between any two probability measures in \(\mathcal P_{\text{prod}}\) is finite. This class is motivated by both practical and theoretical considerations: it arises, for instance, in applications with independent signal components, such as multi-channel signal processing and image analysis, and its product structure enables a coordinate-wise analysis via the factorization \(F_{X}(\mathbf x) = \prod_{j=1}^d F_{X,j}(x_j)\). A more detailed discussion of the role and relevance of this class is deferred to after the statement of Theorem~\ref{th:W1minimax}. We now define the estimator. For each coordinate \(j = 1,\,\dots,\,d\), let \(\widehat{F}_{j,n,\mathbf{h}}\) be the univariate integrated estimator. Since this estimator may lack a finite first moment, required for the $1$-Wasserstein distance to be finite, we project each \(\widehat{F}_{j,n,\mathbf{h}}\) onto the set of univariate cumulative distribution functions with compact support \([-M_n,\,M_n]\), obtaining the restricted \(L^1\)-projection estimator \(\widehat{F}_{j,n,\mathbf{h}}^\ast\), with \(\varrho_{j,n} = O({n}^{-1/2})\) and \(M_n\) chosen so that it achieves the same convergence rate as the unrestricted estimator; see Corollary~\ref{cor:finiteMomentRestricted} for details. The multivariate estimator is then the product of the univariate restricted estimators
\(
\widehat{F}_{n,\mathbf{h}}^*(\mathbf x) \coloneqq \prod_{j=1}^d \widehat{F}_{j,n,\mathbf{h}}^\ast(x_j)\), $\mathbf x = (x_1,\,\dots,\,x_d)^\top \in \mathbb R^d$.
Let \(\widehat{\mu}_{n,\mathbf{h}}^\ast \coloneqq \bigotimes_{j=1}^d \widehat{\mu}_{j,n,\mathbf{h}}^\ast\) denote the associated product probability measure. We now state the main result of this section, which establishes the minimax optimality of the estimator \(\widehat{F}_{n,\mathbf{h}}^\ast\) for \(\We\) on \(\mathcal P_{\text{prod}}\). The proof is deferred to \hyperref[sec:S5]{Section~S5 of the Supplement}.

\begin{theorem}[Minimax optimality of the projection estimator for \(\We\) on \(\mathcal P_{\text{prod}}\)]\label{th:W1minimax}
Let \(\mathcal P_{\textnormal{prod}}\) be the class of probability measures introduced in Definition~\ref{def:prodclass}, and suppose there is at least one active coordinate, i.e., \(|\mathcal P| \ge 1\), where \(\mathcal P\) is defined in \eqref{eq:setactivecoord}. Let \(\widehat{\mu}_{n,\mathbf{h}}^\ast\) be the probability measure above. Then,
\begin{enumerate}[label={(\emph{\roman*})}]
    \item \emph{Lower bound} \quad \(\displaystyle 
    \inf_{\hat{\mu}_n} \sup_{\mu_X \in \mathcal P_{\textnormal{prod}}} 
    \mathsf{E}_{(\mu_X\ast \mu_\varepsilon)^{\otimes n}} \big[ \We(\hat{\mu}_n,\, \mu_X) \big] \gtrsim \psi_n^{\boldsymbol{\alpha},\boldsymbol{\beta}}.
    \)

    \item \emph{Upper bound} \quad \(\displaystyle 
    \sup_{\mu_X \in \mathcal P_{\textnormal{prod}}} 
    \mathsf{E}_{(\mu_X\ast \mu_\varepsilon)^{\otimes n}} \big[\We(\widehat{\mu}_{n,\mathbf{h}}^\ast,\, \mu_X) \big] \lesssim \psi_n^{\boldsymbol{\alpha},\boldsymbol{\beta}}.
    \)

    \item \emph{Minimax rate} \quad \(\displaystyle 
    \inf_{\hat{\mu}_n} \sup_{\mu_X \in \mathcal P_{\textnormal{prod}}} 
    \mathsf{E}_{(\mu_X\ast \mu_\varepsilon)^{\otimes n}} \big[ \We(\hat{\mu}_n,\, \mu_X) \big] \asymp \psi_n^{\boldsymbol{\alpha},\boldsymbol{\beta}}.
    \)
\end{enumerate}
\end{theorem}
The lower and upper bounds together show that \(\widehat{\mu}_{n,\mathbf{h}}^\ast\) attains the rate \(\psi_n^{\boldsymbol{\alpha},\boldsymbol{\beta}}\), minimax-optimal for \(\We\) on \(\mathcal P_{\text{prod}}\). Notably, this coincides with the rate for \(L^1\)-deconvolution of cumulative distribution functions on anisotropic Nikol'skii classes \(\mathcal N_1^d(\boldsymbol\alpha,\,L)\).
In dimension one, the \(1\)-Wasserstein distance admits the representation \eqref{eq:identity}. In dimension \(d \ge 2\), this identity fails, since optimal transport can exploit couplings between coordinates. On product measures, however, the transport problem decouples: since the product structure rules out cross-couplings, the total cost is the sum of the one-dimensional costs,
\begin{equation}\label{eq:W1decouple}
\Wl(\mu,\,\nu) = \sum_{j=1}^d \Wl(\mu_j,\,\nu_j)
= \sum_{j=1}^d \|F_{\mu_j} - F_{\nu_j}\|_{L^1(\mathbb R)},
\end{equation}
where \(\mu_j\), \(\nu_j\) denote the \(j\)-th marginals of \(\mu\) and \(\nu\), respectively. Unlike the general tensorisation inequality, which gives an upper bound up to multiplicative constants, the exact identity \eqref{eq:W1decouple} follows from the \(\ell^1\)-cost and the product structure. This decoupling yields sharp minimax rates with no loss in the reduction from \(d\) dimensions to one. It is also algorithmically valuable, reducing a \(d\)-dimensional transport problem to \(d\) scalar problems, each treatable with the one-dimensional theory of cumulative distribution function deconvolution; see \citet{Dedecker2015} and \citet{Rousseau&Scricciolo}. The agreement of the minimax rates follows from two facts. First, Nikol'skii regularity is separable and adapts to the product structure. Second, the optimal rate for recovering a cumulative distribution function in \(L^1(\mathbb R^d)\) with anisotropic regularity equals that for estimating \(\We\) on product measures, since in both cases the difficulty is governed by the aggregate of one-dimensional contributions across the active coordinates. Hence, although \(\We\) and the \(L^1\)-distance between cumulative distribution functions are not equivalent in general, their minimax rates coincide on \(\mathcal P_{\text{prod}}\), revealing a connection between optimal transport and \(L^1\)-deconvolution arising from the product structure and \eqref{eq:identity}.

We now discuss the class $\mathcal P_{\textnormal{prod}}$. The product structure describes real scenarios in multi-channel signal processing and image analysis, where independent components can have different degrees of smoothness. For instance, in medical imaging a smooth tissue-density channel may coexist with a highly irregular edge-detection channel.
This structure makes the product class theoretically useful: the deconvolution problem separates into one-dimensional subproblems, and explicit rates become derivable. It does not reduce the ambient dimension, but it makes the high-dimensional problem solvable, computationally and statistically. The resulting rates depend on the active coordinates rather than the full ambient dimension. An alternative mechanism is Wasserstein GANs, where \citet{Stephanovitch2024} exploit the intrinsic dimension of the support of the target measure rather than a product factorisation, the rate being thus governed by the intrinsic dimension rather than the ambient one. Thus, while the product structure \emph{decomposes} the ambient dimension, the manifold approach \emph{compresses} it.
While the results cover known error distributions, important extensions remain open, such as handling unknown noise densities and inhomogeneous settings where ordinary and supersmooth noise components coexist, as in \citet{ComteLacour2013}. The multivariate \(L^1\)-distance can be regarded as an extension of the \(1\)-Wasserstein distance in dimension one, even though the two distances do not coincide in higher dimensions. While we have shown that their minimax rates coincide on \(\mathcal P_{\textnormal{prod}}\), whether this agreement extends to general multivariate measures remains an open question.



\bibliographystyle{imsart-nameyear} 
\bibliography{bibliography.bib}       

\begin{thebibliography}{43}

\bibitem[\protect\citeauthoryear{Arjovsky, Chintala and
  Bottou}{2017}]{Arjovsky2017}
\begin{binproceedings}[author]
\bauthor{\bsnm{Arjovsky},~\bfnm{Martin}\binits{M.}},
  \bauthor{\bsnm{Chintala},~\bfnm{Soumith}\binits{S.}} \AND
  \bauthor{\bsnm{Bottou},~\bfnm{L{\'e}on}\binits{L.}}
(\byear{2017}).
\btitle{Wasserstein generative adversarial networks}.
In \bbooktitle{International {C}onference on {M}achine {L}earning}
\bpages{214--223}.
\bpublisher{Pmlr}.
\end{binproceedings}
\endbibitem

\bibitem[\protect\citeauthoryear{Bergstr\"om}{1952}]{Bergstrom1952}
\begin{barticle}[author]
\bauthor{\bsnm{Bergstr\"om},~\bfnm{Harald}\binits{H.}}
(\byear{1952}).
\btitle{On some expansions of stable distribution functions}.
\bjournal{Ark. Mat.}
\bvolume{2}
\bpages{375--378}.
\end{barticle}
\endbibitem

\bibitem[\protect\citeauthoryear{Butucea}{2004}]{Butucea2004v}
\begin{barticle}[author]
\bauthor{\bsnm{Butucea},~\bfnm{Cristina}\binits{C.}}
(\byear{2004}).
\btitle{Deconvolution of supersmooth densities with smooth noise}.
\bjournal{Canad. J. Statist.}
\bvolume{32}
\bpages{181--192}.
\end{barticle}
\endbibitem

\bibitem[\protect\citeauthoryear{Butucea and
  Tsybakov}{2007a}]{ButuceaTsybakov2008a}
\begin{barticle}[author]
\bauthor{\bsnm{Butucea},~\bfnm{C.}\binits{C.}} \AND
  \bauthor{\bsnm{Tsybakov},~\bfnm{A.~B.}\binits{A.~B.}}
(\byear{2007}a).
\btitle{Sharp optimality in density deconvolution with dominating bias. {I}}.
\bjournal{Teor. Veroyatn. Primen.}
\bvolume{52}
\bpages{111--128}.
\end{barticle}
\endbibitem

\bibitem[\protect\citeauthoryear{Butucea and
  Tsybakov}{2007b}]{ButuceaTsybakov2008b}
\begin{barticle}[author]
\bauthor{\bsnm{Butucea},~\bfnm{C.}\binits{C.}} \AND
  \bauthor{\bsnm{Tsybakov},~\bfnm{A.~B.}\binits{A.~B.}}
(\byear{2007}b).
\btitle{Sharp optimality in density deconvolution with dominating bias. {II}}.
\bjournal{Teor. Veroyatn. Primen.}
\bvolume{52}
\bpages{336--349}.
\end{barticle}
\endbibitem

\bibitem[\protect\citeauthoryear{Caillerie et~al.}{}]{Caillerie2011}
\begin{barticle}[author]
\bauthor{\bsnm{Caillerie},~\bfnm{Claire}\binits{C.}},
  \bauthor{\bsnm{Chazal},~\bfnm{Fr{\'e}d{\'e}ric}\binits{F.}},
  \bauthor{\bsnm{Dedecker},~\bfnm{J{\'e}r{\^o}me}\binits{J.}} \AND
  \bauthor{\bsnm{Michel},~\bfnm{Bertrand}\binits{B.}}
\btitle{{Deconvolution for the Wasserstein metric and geometric inference}}.
\bjournal{Electron. J. Stat.}
\bvolume{5}.
\end{barticle}
\endbibitem

\bibitem[\protect\citeauthoryear{Carroll and Hall}{1988}]{Carroll1988}
\begin{barticle}[author]
\bauthor{\bsnm{Carroll},~\bfnm{Raymond~J.}\binits{R.~J.}} \AND
  \bauthor{\bsnm{Hall},~\bfnm{Peter}\binits{P.}}
(\byear{1988}).
\btitle{Optimal rates of convergence for deconvolving a density}.
\bjournal{J. Amer. Statist. Assoc.}
\bvolume{83}
\bpages{1184--1186}.
\end{barticle}
\endbibitem

\bibitem[\protect\citeauthoryear{Chagny}{2013}]{Chagny2013}
\begin{barticle}[author]
\bauthor{\bsnm{Chagny},~\bfnm{G.}\binits{G.}}
(\byear{2013}).
\btitle{Warped bases for conditional density estimation}.
\bjournal{Math. Methods Statist.}
\bvolume{22}
\bpages{253--282}.
\end{barticle}
\endbibitem

\bibitem[\protect\citeauthoryear{Comte and
  Genon-Catalot}{2018}]{ComteGenonCatalot2018}
\begin{barticle}[author]
\bauthor{\bsnm{Comte},~\bfnm{F.}\binits{F.}} \AND
  \bauthor{\bsnm{Genon-Catalot},~\bfnm{V.}\binits{V.}}
(\byear{2018}).
\btitle{Laguerre and {H}ermite bases for inverse problems}.
\bjournal{J. Korean Statist. Soc.}
\bvolume{47}
\bpages{273--296}.
\end{barticle}
\endbibitem

\bibitem[\protect\citeauthoryear{Comte and Lacour}{2011}]{ComteLacour2011}
\begin{barticle}[author]
\bauthor{\bsnm{Comte},~\bfnm{F.}\binits{F.}} \AND
  \bauthor{\bsnm{Lacour},~\bfnm{C.}\binits{C.}}
(\byear{2011}).
\btitle{Data-driven density estimation in the presence of additive noise with
  unknown distribution}.
\bjournal{J. R. Stat. Soc. Ser. B Stat. Methodol.}
\bvolume{73}
\bpages{601--627}.
\end{barticle}
\endbibitem

\bibitem[\protect\citeauthoryear{Comte and Lacour}{2013}]{ComteLacour2013}
\begin{barticle}[author]
\bauthor{\bsnm{Comte},~\bfnm{F.}\binits{F.}} \AND
  \bauthor{\bsnm{Lacour},~\bfnm{C.}\binits{C.}}
(\byear{2013}).
\btitle{Anisotropic adaptive kernel deconvolution}.
\bjournal{Ann. Inst. Henri Poincar\'e{} Probab. Stat.}
\bvolume{49}
\bpages{569--609}.
\end{barticle}
\endbibitem

\bibitem[\protect\citeauthoryear{Comte and Mabon}{2017}]{Mabon2017}
\begin{barticle}[author]
\bauthor{\bsnm{Comte},~\bfnm{F.}\binits{F.}} \AND
  \bauthor{\bsnm{Mabon},~\bfnm{G.}\binits{G.}}
(\byear{2017}).
\btitle{Laguerre deconvolution with unknown matrix operator}.
\bjournal{Math. Methods Statist.}
\bvolume{26}
\bpages{237--266}.
\end{barticle}
\endbibitem

\bibitem[\protect\citeauthoryear{Comte, Rozenholc and Taupin}{2006}]{Comte2006}
\begin{barticle}[author]
\bauthor{\bsnm{Comte},~\bfnm{Fabienne}\binits{F.}},
  \bauthor{\bsnm{Rozenholc},~\bfnm{Yves}\binits{Y.}} \AND
  \bauthor{\bsnm{Taupin},~\bfnm{Marie-Luce}\binits{M.-L.}}
(\byear{2006}).
\btitle{Penalized contrast estimator for adaptive density deconvolution}.
\bjournal{Canad. J. Statist.}
\bvolume{34}
\bpages{431--452}.
\end{barticle}
\endbibitem

\bibitem[\protect\citeauthoryear{Dall'Aglio}{1956}]{Dall'Aglio}
\begin{barticle}[author]
\bauthor{\bsnm{Dall'Aglio},~\bfnm{Giorgio}\binits{G.}}
(\byear{1956}).
\btitle{Sugli estremi dei momenti delle funzioni di ripartizione doppia}.
\bjournal{Ann. Scuola Norm. Sup. Pisa Cl. Sci. (3)}
\bvolume{10}
\bpages{35--74}.
\end{barticle}
\endbibitem

\bibitem[\protect\citeauthoryear{Dedecker and
  Michel}{2013}]{dedecker:michel:2013}
\begin{barticle}[author]
\bauthor{\bsnm{Dedecker},~\bfnm{J.}\binits{J.}} \AND
  \bauthor{\bsnm{Michel},~\bfnm{B.}\binits{B.}}
(\byear{2013}).
\btitle{{M}inimax rates of convergence for {W}asserstein deconvolution with
  supersmooth errors in any dimension}.
\bjournal{J. {M}ultivar. {A}nal.}
\bvolume{122}
\bpages{278--291}.
\end{barticle}
\endbibitem

\bibitem[\protect\citeauthoryear{Dedecker and Michel}{2015}]{Dedecker2015}
\begin{barticle}[author]
\bauthor{\bsnm{Dedecker},~\bfnm{J.}\binits{J.}} \AND
  \bauthor{\bsnm{Michel},~\bfnm{B.}\binits{B.}}
(\byear{2015}).
\btitle{Improved rates for {W}asserstein deconvolution with ordinary smooth
  error in dimension one}.
\bjournal{Electron. J. Stat.}
\bvolume{9}
\bpages{234--265}.
\end{barticle}
\endbibitem

\bibitem[\protect\citeauthoryear{Delaigle and Gijbels}{2004}]{Delaigle2004}
\begin{barticle}[author]
\bauthor{\bsnm{Delaigle},~\bfnm{A.}\binits{A.}} \AND
  \bauthor{\bsnm{Gijbels},~\bfnm{I.}\binits{I.}}
(\byear{2004}).
\btitle{Practical bandwidth selection in deconvolution kernel density
  estimation}.
\bjournal{Comput. Statist. Data Anal.}
\bvolume{45}
\bpages{249--267}.
\end{barticle}
\endbibitem

\bibitem[\protect\citeauthoryear{Delaigle, Hall and
  Meister}{2008}]{Delaigle2008}
\begin{barticle}[author]
\bauthor{\bsnm{Delaigle},~\bfnm{Aurore}\binits{A.}},
  \bauthor{\bsnm{Hall},~\bfnm{Peter}\binits{P.}} \AND
  \bauthor{\bsnm{Meister},~\bfnm{Alexander}\binits{A.}}
(\byear{2008}).
\btitle{On deconvolution with repeated measurements}.
\bjournal{Ann. Statist.}
\bvolume{36}
\bpages{665--685}.
\end{barticle}
\endbibitem

\bibitem[\protect\citeauthoryear{Diggle and Hall}{1993}]{Diggle1993}
\begin{barticle}[author]
\bauthor{\bsnm{Diggle},~\bfnm{Peter~J.}\binits{P.~J.}} \AND
  \bauthor{\bsnm{Hall},~\bfnm{Peter}\binits{P.}}
(\byear{1993}).
\btitle{A {F}ourier approach to nonparametric deconvolution of a density
  estimate}.
\bjournal{J. Roy. Statist. Soc. Ser. B}
\bvolume{55}
\bpages{523--531}.
\end{barticle}
\endbibitem

\bibitem[\protect\citeauthoryear{Fan}{1991}]{Fan1991}
\begin{barticle}[author]
\bauthor{\bsnm{Fan},~\bfnm{Jianqing}\binits{J.}}
(\byear{1991}).
\btitle{On the optimal rates of convergence for nonparametric deconvolution
  problems}.
\bjournal{Ann. Statist.}
\bvolume{19}
\bpages{1257--1272}.
\end{barticle}
\endbibitem

\bibitem[\protect\citeauthoryear{Fan and Koo}{2002}]{Fan2002}
\begin{barticle}[author]
\bauthor{\bsnm{Fan},~\bfnm{Jianqing}\binits{J.}} \AND
  \bauthor{\bsnm{Koo},~\bfnm{Ja-Yong}\binits{J.-Y.}}
(\byear{2002}).
\btitle{Wavelet deconvolution}.
\bjournal{IEEE Trans. Inform. Theory}
\bvolume{48}
\bpages{734--747}.
\end{barticle}
\endbibitem

\bibitem[\protect\citeauthoryear{Fan and Truong}{1993}]{Fan&Truong1993}
\begin{barticle}[author]
\bauthor{\bsnm{Fan},~\bfnm{Jianqing}\binits{J.}} \AND
  \bauthor{\bsnm{Truong},~\bfnm{Young~K.}\binits{Y.~K.}}
(\byear{1993}).
\btitle{Nonparametric regression with errors in variables}.
\bjournal{Ann. Statist.}
\bvolume{21}
\bpages{1900--1925}.
\end{barticle}
\endbibitem

\bibitem[\protect\citeauthoryear{Goldenshluger and
  Lepski}{2011a}]{GoldenshlugerLepski2011}
\begin{barticle}[author]
\bauthor{\bsnm{Goldenshluger},~\bfnm{Alexander}\binits{A.}} \AND
  \bauthor{\bsnm{Lepski},~\bfnm{Oleg}\binits{O.}}
(\byear{2011}a).
\btitle{Bandwidth selection in kernel density estimation: oracle inequalities
  and adaptive minimax optimality}.
\bjournal{Ann. Statist.}
\bvolume{39}
\bpages{1608--1632}.
\end{barticle}
\endbibitem

\bibitem[\protect\citeauthoryear{Goldenshluger and
  Lepski}{2011b}]{GoldenLepski2011}
\begin{barticle}[author]
\bauthor{\bsnm{Goldenshluger},~\bfnm{Alexander}\binits{A.}} \AND
  \bauthor{\bsnm{Lepski},~\bfnm{Oleg}\binits{O.}}
(\byear{2011}b).
\btitle{Uniform bounds for norms of sums of independent random functions}.
\bjournal{Ann. Probab.}
\bvolume{39}
\bpages{2318--2384}.
\end{barticle}
\endbibitem

\bibitem[\protect\citeauthoryear{Goodfellow
  et~al.}{2014}]{goodfellow2014generative}
\begin{barticle}[author]
\bauthor{\bsnm{Goodfellow},~\bfnm{Ian~J}\binits{I.~J.}},
  \bauthor{\bsnm{Pouget-Abadie},~\bfnm{Jean}\binits{J.}},
  \bauthor{\bsnm{Mirza},~\bfnm{Mehdi}\binits{M.}},
  \bauthor{\bsnm{Xu},~\bfnm{Bing}\binits{B.}},
  \bauthor{\bsnm{Warde-Farley},~\bfnm{David}\binits{D.}},
  \bauthor{\bsnm{Ozair},~\bfnm{Sherjil}\binits{S.}},
  \bauthor{\bsnm{Courville},~\bfnm{Aaron}\binits{A.}} \AND
  \bauthor{\bsnm{Bengio},~\bfnm{Yoshua}\binits{Y.}}
(\byear{2014}).
\btitle{Generative adversarial nets}.
\bjournal{Advances in {N}eural {I}nformation {P}rocessing {S}ystems}
\bvolume{27}.
\end{barticle}
\endbibitem

\bibitem[\protect\citeauthoryear{Gulrajani et~al.}{2017}]{Gulrajani2017}
\begin{barticle}[author]
\bauthor{\bsnm{Gulrajani},~\bfnm{Ishaan}\binits{I.}},
  \bauthor{\bsnm{Ahmed},~\bfnm{Faruk}\binits{F.}},
  \bauthor{\bsnm{Arjovsky},~\bfnm{Martin}\binits{M.}},
  \bauthor{\bsnm{Dumoulin},~\bfnm{Vincent}\binits{V.}} \AND
  \bauthor{\bsnm{Courville},~\bfnm{Aaron}\binits{A.}}
(\byear{2017}).
\btitle{Improved training of {W}asserstein {GAN}s}.
\bjournal{Advances in {N}eural {I}nformation {P}rocessing {S}ystems}
\bvolume{30}.
\end{barticle}
\endbibitem

\bibitem[\protect\citeauthoryear{Hall and Lahiri}{2008}]{HallLahiri2008}
\begin{barticle}[author]
\bauthor{\bsnm{Hall},~\bfnm{Peter}\binits{P.}} \AND
  \bauthor{\bsnm{Lahiri},~\bfnm{Soumendra~N.}\binits{S.~N.}}
(\byear{2008}).
\btitle{Estimation of distributions, moments and quantiles in deconvolution
  problems}.
\bjournal{Ann. Statist.}
\bvolume{36}
\bpages{2110--2134}.
\end{barticle}
\endbibitem

\bibitem[\protect\citeauthoryear{Hazelton and
  Turlach}{2009}]{HazeltonTurlach2009}
\begin{barticle}[author]
\bauthor{\bsnm{Hazelton},~\bfnm{Martin~L.}\binits{M.~L.}} \AND
  \bauthor{\bsnm{Turlach},~\bfnm{Berwin~A.}\binits{B.~A.}}
(\byear{2009}).
\btitle{Nonparametric density deconvolution by weighted kernel estimators}.
\bjournal{Stat. Comput.}
\bvolume{19}
\bpages{217--228}.
\end{barticle}
\endbibitem

\bibitem[\protect\citeauthoryear{Johannes}{2009}]{Johannes2009}
\begin{barticle}[author]
\bauthor{\bsnm{Johannes},~\bfnm{Jan}\binits{J.}}
(\byear{2009}).
\btitle{Deconvolution with unknown error distribution}.
\bjournal{Ann. Statist.}
\bvolume{37}
\bpages{2301--2323}.
\end{barticle}
\endbibitem

\bibitem[\protect\citeauthoryear{Kantorovitch}{1942}]{Kantorovitch1942}
\begin{barticle}[author]
\bauthor{\bsnm{Kantorovitch},~\bfnm{L.}\binits{L.}}
(\byear{1942}).
\btitle{On the translocation of masses}.
\bjournal{C. R. (Doklady) Acad. Sci. URSS (N.S.)}
\bvolume{37}
\bpages{199--201}.
\end{barticle}
\endbibitem

\bibitem[\protect\citeauthoryear{Lepski and Willer}{2019}]{LepskiWiller2019}
\begin{barticle}[author]
\bauthor{\bsnm{Lepski},~\bfnm{O.~V.}\binits{O.~V.}} \AND
  \bauthor{\bsnm{Willer},~\bfnm{T.}\binits{T.}}
(\byear{2019}).
\btitle{Oracle inequalities and adaptive estimation in the convolution
  structure density model}.
\bjournal{Ann. Statist.}
\bvolume{47}
\bpages{233--287}.
\end{barticle}
\endbibitem

\bibitem[\protect\citeauthoryear{Masry}{1991}]{Masry1991}
\begin{barticle}[author]
\bauthor{\bsnm{Masry},~\bfnm{Elias}\binits{E.}}
(\byear{1991}).
\btitle{Multivariate probability density deconvolution for stationary random
  processes}.
\bjournal{IEEE Trans. Inform. Theory}
\bvolume{37}
\bpages{1105--1115}.
\end{barticle}
\endbibitem

\bibitem[\protect\citeauthoryear{Neumann}{2007}]{Neumann2007}
\begin{barticle}[author]
\bauthor{\bsnm{Neumann},~\bfnm{Michael~H.}\binits{M.~H.}}
(\byear{2007}).
\btitle{Deconvolution from panel data with unknown error distribution}.
\bjournal{J. Multivariate Anal.}
\bvolume{98}
\bpages{1955--1968}.
\end{barticle}
\endbibitem

\bibitem[\protect\citeauthoryear{Pensky and Vidakovic}{1999}]{Pensky1999}
\begin{barticle}[author]
\bauthor{\bsnm{Pensky},~\bfnm{Marianna}\binits{M.}} \AND
  \bauthor{\bsnm{Vidakovic},~\bfnm{Brani}\binits{B.}}
(\byear{1999}).
\btitle{Adaptive wavelet estimator for nonparametric density deconvolution}.
\bjournal{Ann. Statist.}
\bvolume{27}
\bpages{2033--2053}.
\end{barticle}
\endbibitem

\bibitem[\protect\citeauthoryear{Rebelles}{2016}]{Rebelles2016}
\begin{barticle}[author]
\bauthor{\bsnm{Rebelles},~\bfnm{G.}\binits{G.}}
(\byear{2016}).
\btitle{Structural adaptive deconvolution under {$\Bbb{L}_p$}-losses}.
\bjournal{Math. Methods Statist.}
\bvolume{25}
\bpages{26--53}.
\end{barticle}
\endbibitem

\bibitem[\protect\citeauthoryear{Rousseau and
  Scricciolo}{2024}]{Rousseau&Scricciolo}
\begin{barticle}[author]
\bauthor{\bsnm{Rousseau},~\bfnm{Judith}\binits{J.}} \AND
  \bauthor{\bsnm{Scricciolo},~\bfnm{Catia}\binits{C.}}
(\byear{2024}).
\btitle{{Wasserstein convergence in Bayesian and frequentist deconvolution
  models}}.
\bjournal{Ann. Statist.}
\bvolume{52}
\bpages{1691--1715}.
\end{barticle}
\endbibitem

\bibitem[\protect\citeauthoryear{Scricciolo}{2026}]{Scricciolo25b}
\begin{barticle}[author]
\bauthor{\bsnm{Scricciolo},~\bfnm{Catia}\binits{C.}}
(\byear{2026}).
\btitle{Adaptive minimax-optimal Wasserstein deconvolution with unknown error
  distributions}.
\bjournal{Statistics \& Probability Letters}
\bvolume{230}
\bpages{110589}.
\end{barticle}
\endbibitem

\bibitem[\protect\citeauthoryear{Stefanski and
  Carroll}{1990}]{StefanskiCarroll1990}
\begin{barticle}[author]
\bauthor{\bsnm{Stefanski},~\bfnm{Leonard}\binits{L.}} \AND
  \bauthor{\bsnm{Carroll},~\bfnm{Raymond~J.}\binits{R.~J.}}
(\byear{1990}).
\btitle{Deconvoluting kernel density estimators}.
\bjournal{Statistics}
\bvolume{21}
\bpages{169--184}.
\end{barticle}
\endbibitem

\bibitem[\protect\citeauthoryear{St\'ephanovitch, Aamari and
  Levrard}{2024}]{Stephanovitch2024}
\begin{barticle}[author]
\bauthor{\bsnm{St\'ephanovitch},~\bfnm{Arthur}\binits{A.}},
  \bauthor{\bsnm{Aamari},~\bfnm{Eddie}\binits{E.}} \AND
  \bauthor{\bsnm{Levrard},~\bfnm{Cl\'ement}\binits{C.}}
(\byear{2024}).
\btitle{Wasserstein generative adversarial networks are minimax optimal
  distribution estimators}.
\bjournal{Ann. Statist.}
\bvolume{52}
\bpages{2167--2193}.
\end{barticle}
\endbibitem

\bibitem[\protect\citeauthoryear{Tsybakov}{2009}]{tsybakov2009introduction}
\begin{bbook}[author]
\bauthor{\bsnm{Tsybakov},~\bfnm{Alexandre~B.}\binits{A.~B.}}
(\byear{2009}).
\btitle{Introduction to nonparametric estimation}.
\bseries{Springer Series in Statistics}.
\bpublisher{Springer, New York}
\bnote{Revised and extended from the 2004 French original, Translated by
  Vladimir Zaiats}.
\end{bbook}
\endbibitem

\bibitem[\protect\citeauthoryear{Villani}{2009}]{villani}
\begin{bbook}[author]
\bauthor{\bsnm{Villani},~\bfnm{C\'edric}\binits{C.}}
(\byear{2009}).
\btitle{Optimal transport}.
\bseries{Grundlehren der mathematischen Wissenschaften [Fundamental Principles
  of Mathematical Sciences]}
\bvolume{338}.
\bpublisher{Springer-Verlag, Berlin}
\bnote{Old and new}.
\end{bbook}
\endbibitem

\bibitem[\protect\citeauthoryear{Youndj\'e and Wells}{2008}]{YoundjeWells2008}
\begin{barticle}[author]
\bauthor{\bsnm{Youndj\'e},~\bfnm{\'Elie}\binits{E.}} \AND
  \bauthor{\bsnm{Wells},~\bfnm{Martin~T.}\binits{M.~T.}}
(\byear{2008}).
\btitle{Optimal bandwidth selection for multivariate kernel deconvolution
  density estimation}.
\bjournal{TEST}
\bvolume{17}
\bpages{138--162}.
\end{barticle}
\endbibitem

\bibitem[\protect\citeauthoryear{Zhang}{1990}]{Zhang1990}
\begin{barticle}[author]
\bauthor{\bsnm{Zhang},~\bfnm{Cun-Hui}\binits{C.-H.}}
(\byear{1990}).
\btitle{Fourier methods for estimating mixing densities and distributions}.
\bjournal{Ann. Statist.}
\bvolume{18}
\bpages{806--831}.
\end{barticle}
\endbibitem

\end{thebibliography}









\clearpage
\markboth{}{}

\begin{supplement}
\stitle{Supplement to ``Title of the Main Article''}\label{Sup:A}
\sdescription{This supplement contains all proofs and auxiliary results for the main article.}

\setcounter{page}{1}

\emph{\Large\bf{Supplement to ``Adaptive minimax multivariate \(L^1\)-deconvolution: harmonic-mean rates under noise filtering, and a \(1\)-Wasserstein equivalence'' by \normalsize{CATIA SCRICCIOLO}}}\\[-5pt]

This supplement contains all proofs and auxiliary results for the main article. Throughout, equation, section, subsection, and theorem numbers are preceded by an ``S'' (\emph{e.g.}, Section~S1, Equation~(S2.1), Lemma~S2.1) to distinguish them from those of the main article.

\section{Proof of the unbiasedness of the estimator $\widehat F_{n,\mathbf{H}}$}\phantomsection\label{sec:S1}
We establish the key unbiasedness property of the preliminary kernel-based estimator $\widehat F_{n,\mathbf{H}}$ for the cumulative distribution function defined in \hyperref[eq:defcumestim]{\eqref{eq:defcumestim} of the Main Article}, namely, that its expectation equals the convolution of the true cumulative distribution function with the rescaled kernel.

\begin{proof}[Proof of Proposition~\ref{prop:unbiasedness} (Unbiasedness of the estimator $\widehat F_{n,\mathbf{H}}$)]
Let \(\mathbf{x}\in\mathbb{R}^d\) be fixed. By the definition of \(\widehat{F}_{n,\mathbf{H}}\) as the integral of \(\hat{f}_{n,\mathbf{H}}\) and Fubini's theorem,
\[
\mathsf{E}_{(\mu_X \ast \mu_\varepsilon)^{\otimes n}}[\widehat{F}_{n,\mathbf{H}}(\mathbf{x})]
= \int_{-\infty}^{x_1}\!\!\cdots\!\!\int_{-\infty}^{x_d} \mathsf{E}_{(\mu_X \ast \mu_\varepsilon)^{\otimes n}}[\hat{f}_{n,\mathbf{H}}(\mathbf{u})] \,\mathrm{d}\mathbf{u}.
\]
The expectation of the deconvolution kernel density estimator is
\[
\mathsf{E}_{(\mu_X \ast \mu_\varepsilon)^{\otimes n}}[\hat{f}_{n,\mathbf{H}}(\mathbf{u})]
= \int_{\mathbb{R}^d} \frac{1}{|\mathbf{H}|} \tilde{K}_{\mathbf{H}}(\mathbf{H}^{-1}(\mathbf{u}-\mathbf{y})) \, (\mu_X \ast \mu_\varepsilon)(\mathrm{d}\mathbf{y}).
\]
Writing the convolution \(\mu_X \ast \mu_\varepsilon\) in terms of characteristic functions and applying the Fourier inversion formula, we obtain
\begin{align*}
\mathsf{E}_{(\mu_X \ast \mu_\varepsilon)^{\otimes n}}[\hat{f}_{n,\mathbf{H}}(\mathbf{u})]
&= \frac{1}{(2\pi)^d} \int_{\mathbb{R}^d} \mathrm{e}^{-\mathrm{i}\langle \mathbf{t},\, \mathbf{u}\rangle} 
\frac{\varphi_K(\mathbf{H}\mathbf{t})}{\varphi_\varepsilon(\mathbf{t})} \,
\varphi_\varepsilon(\mathbf{t}) \varphi_X(\mathbf{t}) \,\mathrm{d}\mathbf{t} \\
&= \frac{1}{(2\pi)^d} \int_{\mathbb{R}^d}  \mathrm{e}^{-\mathrm{i}\langle \mathbf{t},\, \mathbf{u}\rangle}  
\varphi_K(\mathbf{H}\mathbf{t}) \varphi_X(\mathbf{t}) \,\mathrm{d}\mathbf{t} \\
&= \int_{\mathbb{R}^d} K_{\mathbf{H}}(\mathbf{u}-\mathbf{z}) \, \mu_X(\mathrm{d}\mathbf{z}),
\end{align*}
where the last equality uses the fact that \(\mathcal{F}[K_{\mathbf{H}}](\mathbf{t})=\varphi_K(\mathbf{H}\mathbf{t})\), $\mathbf{t}\in\mathbb R^d$ and \(\mathcal{F}^{-1}[\varphi_K(\mathbf{H}\mathbf{t})\varphi_X(\mathbf{t})](\mathbf{u}) = \int_{\mathbb{R}^d} K_{\mathbf{H}}(\mathbf{u}-\mathbf{z}) \mu_X(\mathrm{d}\mathbf{z})\), $\mathbf{u}\in\mathbb R^d$. Substituting this expression into the integrated expectation and applying Fubini's theorem (since $K_{\mathbf{H}}\in L^1(\mathbb R^d)$ and $\mu_X$ is a finite measure) gives
\begin{align*}
\mathsf{E}_{(\mu_X \ast \mu_\varepsilon)^{\otimes n}}[\widehat{F}_{n,\mathbf{H}}(\mathbf{x})]
&= \int_{-\infty}^{x_1}\!\!\cdots\!\!\int_{-\infty}^{x_d} \int_{\mathbb{R}^d} K_{\mathbf{H}}(\mathbf{u}-\mathbf{z}) \, \mu_X(\mathrm{d}\mathbf{z}) \,\mathrm{d}\mathbf{u}\\
&=\int_{\mathbb{R}^d} \left( \int_{-\infty}^{x_1}\!\!\cdots\!\!\int_{-\infty}^{x_d} K_{\mathbf{H}}(\mathbf{u}-\mathbf{z}) \,\mathrm{d}\mathbf{u} \right) \mu_X(\mathrm{d}\mathbf{z}).
\end{align*}
For a fixed \(\mathbf{z}\), change variables \(\mathbf{v} = \mathbf{u}-\mathbf{z}\). Then, \(\mathrm{d}\mathbf{u} = \mathrm{d}\mathbf{v}\) and the condition \(\mathbf{u}\le \mathbf{x}\) (understood coordinatewise) becomes \(\mathbf{v}\le \mathbf{x}-\mathbf{z}\). Thus,
\[
\int_{\{\mathbf{u}\le \mathbf{x}\}} K_{\mathbf{H}}(\mathbf{u}-\mathbf{z}) \,\mathrm{d}\mathbf{u}
= \int_{\mathbb{R}^d} K_{\mathbf{H}}(\mathbf{v})\,\mathbf{1}_{\{\mathbf{v}\le \mathbf{x}-\mathbf{z}\}}\,\mathrm{d}\mathbf{v},
\]
where \[
\mathbf{1}_{\{\mathbf{v}\le \mathbf{x}-\mathbf{z}\}} =
\begin{cases}
1, & \text{if } v_j \le x_j - z_j \text{ for all } j = 1,\,\dots,\,d, \\
0, & \text{otherwise.}
\end{cases}
\]
Inserting this into the integral and applying Fubini's theorem once more to swap the order of integration, we obtain
\begin{align*}
\mathsf{E}_{(\mu_X \ast \mu_\varepsilon)^{\otimes n}}[\widehat{F}_{n,\mathbf{H}}(\mathbf{x})]
&= \int_{\mathbb{R}^d} \left( \int_{\mathbb{R}^d} K_{\mathbf{H}}(\mathbf{v})\,\mathbf{1}_{\{\mathbf{v}\le \mathbf{x}-\mathbf{z}\}}\,\mathrm{d}\mathbf{v} \right) \mu_X(\mathrm{d}\mathbf{z})\\
&= \int_{\mathbb{R}^d} K_{\mathbf{H}}(\mathbf{v}) \left( \int_{\mathbb{R}^d} \mathbf{1}_{\{\mathbf{z}\le \mathbf{x}-\mathbf{v}\}} \,\mu_X(\mathrm{d}\mathbf{z}) \right) \mathrm{d}\mathbf{v}.
\end{align*}
The inner integral is the measure of the set \(\{\mathbf{z}\in\mathbb R^d:\, \mathbf{z}\le \mathbf{x}-\mathbf{v}\}\), which is precisely the cumulative distribution function \(F_X(\mathbf{x}-\mathbf{v})\). Therefore,
\[
\mathsf{E}_{(\mu_X \ast \mu_\varepsilon)^{\otimes n}}[\widehat{F}_{n,\mathbf{H}}(\mathbf{x})]
= \int_{\mathbb{R}^d} K_{\mathbf{H}}(\mathbf{v}) \, F_X(\mathbf{x}-\mathbf{v}) \,\mathrm{d}\mathbf{v}
= (K_{\mathbf{H}} \ast F_X)(\mathbf{x}).
\]
This completes the proof.
\end{proof}

\section{Proofs for the upper bound on the $L^1$-risk of the estimator $\widehat{F}^\ast_{n,\mathbf{h}}$}\phantomsection\label{sec:S2}
This section collects the proofs of the bias bound, the stochastic error control, and the resulting upper bound on the $L^1$-risk of the $\varrho_n$-approximate minimum-$L^1$-distance estimator defined by condition \eqref{eq:defminimizer} \hyperref[eq:defminimizer]{of the Main Article}. A construction of a superkernel is also provided.

The following proof establishes the bias bound for the multivariate cumulative distribution function under the anisotropic Nikol'skii conditions \ref{it:diff}--\ref{it:smooth} of \hyperref[def:anisotropy]{Definition~\ref*{def:anisotropy} of the Main Article}. The extension to the \(L^2\)-setting is discussed in Remark~\ref{rem:adjusted_proof}.

\begin{proof}[Proof of Proposition~\ref{prop:bias} (Upper bound on the $L^1$-bias for anisotropic Nikol'skii classes).]
We begin by expressing the bias in an integral form and then decomposing it into a sum of coordinatewise differences via a telescoping identity. For \(\mathbf{x}\in\mathbb R^d\), the bias is
\[
B_{\mathbf{h}}(\mathbf{x})
= [F_X- (K_{\mathbf{h}}\ast F_X)](\mathbf{x})=
\int_{\mathbb R^d} K(\mathbf{z})
[F_X(\mathbf{x})-F_X(\mathbf{x}-\mathbf{h}\odot\mathbf{z})]\,\mathrm{d}\mathbf{z},
\]
where \(\mathbf{h}\odot\mathbf{z}:=(h_1z_1,\,\dots,\,h_dz_d)^\top\) and \(K(\mathbf{z})=\prod_{j=1}^d K_j(z_j)\).
For fixed \(\mathbf{x},\,\mathbf{z}\in\mathbb R^d\), define the points
\[
\mathbf{x}^{(0)}\coloneqq\mathbf{x},\qquad 
\mathbf{x}^{(j)}\coloneqq(x_1-h_1z_1,\,\dots,\,x_j-h_jz_j,\,x_{j+1},\,\dots,\,x_d)^\top.
\]
These points form a chain from \(\mathbf{x}\) to \(\mathbf{x}-\mathbf{h}\odot\mathbf{z}\). The difference between the values of the cumulative distribution function at the endpoints of the chain can be decomposed, by telescoping, as
\[
F_X(\mathbf{x})-F_X(\mathbf{x}-\mathbf{h}\odot\mathbf{z})
=
\sum_{j=1}^d
[F_X(\mathbf{x}^{(j-1)})-F_X(\mathbf{x}^{(j)})].
\]
Using the triangle inequality and Fubini's theorem, we obtain
\begin{equation}\label{eq:biasnorm}
\|B_{\mathbf{h}}\|_{L^1(\mathbb R^d)}
\le
\sum_{j=1}^d
\int_{\mathbb R^d}|K(\mathbf{z})|
\int_{\mathbb R^d}
|F_X(\mathbf{x}^{(j-1)})-F_X(\mathbf{x}^{(j)})|\,\mathrm{d}\mathbf{x}\,\mathrm{d}\mathbf{z}.
\end{equation}
Next, we isolate the contribution of the \(j\)-th coordinate by performing a change of variables that shifts the \(j\)-th coordinate only. For a fixed \(j\in\{1,\,\dots,\,d\}\), define
\[
v_r\coloneqq
\begin{cases}
x_r-h_rz_r,& r<j,\\
x_r,& r\ge j.
\end{cases}
\]
Then, \(\mathbf{x}^{(j-1)}=\mathbf{v}\) and \(\mathbf{x}^{(j)}=\mathbf{v}-h_jz_j\mathbf e_j\), and the Jacobian of the transformation is \(1\). The inner integral in \eqref{eq:biasnorm} therefore becomes
\[
\int_{\mathbb R^d}
|F_X(\mathbf{x}^{(j-1)})-F_X(\mathbf{x}^{(j)})|\,\mathrm{d}\mathbf{x}=\int_{\mathbb R^d}
|F_X(\mathbf{v})-F_X(\mathbf{v}-h_jz_j\mathbf e_j)|\,\mathrm{d}\mathbf{v}.
\]
The integrals over \(z_r\) for \(r\ne j\) factor out as the constants \(\|K_r\|_{L^1(\mathbb R)}\), yielding
\begin{equation}\label{eq:coord}
\|B_{\mathbf{h}}\|_{L^1(\mathbb R^d)}
\le
\sum_{j=1}^d
\left(\prod_{\substack{r=1 \\ r\ne j}}^d \|K_r\|_{L^1(\mathbb R)}\right)
\int_{\mathbb R}|K_j(z_j)|
\int_{\mathbb R^d}
|F_X(\mathbf{v})-F_X(\mathbf{v}-h_jz_j\mathbf e_j)|\,\mathrm{d}\mathbf{v}\,\mathrm{d}z_j.
\end{equation}
The difference of the cumulative distribution function along the \(j\)-th coordinate admits the exact integral representation
\[
F_X(\mathbf{v})-F_X(\mathbf{v}-h_jz_j\mathbf e_j)
=\int_{v_j - h_j z_j}^{v_j}
f_X(v_1,\,\dots,\,v_{j-1},\, t,\,v_{j+1},\,\dots,\,v_d)\,\mathrm{d}t=
\int_{v_j-h_jz_j}^{v_j} f_X(\mathbf{v}_{-j},\, t)\,\mathrm{d}t,
\]
where, in the joint density \(f_X(\mathbf{v}_{-j},\, t)\), the notation \(\mathbf{v}_{-j}\) indicates that all non-\(j\) coordinates are held fixed, while \(t\) varies along the \(j\)-th axis. Substituting this into \eqref{eq:coord}, we find that
\begin{equation}\label{eq:biasint}
\|B_{\mathbf{h}}\|_{L^1(\mathbb R^d)}
\le
\sum_{j=1}^d
\left(\prod_{\substack{r=1 \\ r\ne j}}^d \|K_r\|_{L^1(\mathbb R)}\right)
\int_{\mathbb R}|K_j(z_j)|
\int_{\mathbb R^d}
\bigg|
\int_{v_j-h_jz_j}^{v_j} f_X(\mathbf{v}_{-j},\, t)\,\mathrm{d}t
\bigg|
\mathrm{d}\mathbf{v}\,\mathrm{d}z_j.
\end{equation}
We now expand the density \(f_X\) in the \(j\)-th coordinate using Taylor's theorem with integral remainder. For each fixed \(\mathbf{v}_{-j}\), the function \(t \mapsto f_X(\mathbf{v}_{-j},\, t)\) is expanded around \(t=v_j\) up to order \(m_j \coloneqq \lceil \alpha_j\rceil - 1\). This is justified because the Nikol'skii condition guarantees the existence of derivatives up to order \(m_j\), and the remainder can be controlled in \(L^1(\mathbb R^d)\). For any \(t\in\mathbb R\), writing \(s \coloneqq t - v_j\), we have
\begin{equation}\label{eq:dev}
f_X(\mathbf{v}_{-j},\, v_j+s)
=
\sum_{k=0}^{m_j}
\frac{s^k}{k!}\partial_j^k f_X(\mathbf{v}_{-j},\, v_j)
+
R_j(\mathbf{v},\, s),
\end{equation}
with the remainder
\[
R_j(\mathbf{v},\, s)\coloneqq
\frac{1}{(m_j-1)!}
\int_0^{s} (s-u)^{m_j-1}
[
\partial_j^{m_j} f_X(\mathbf{v}_{-j},\, v_j+u)
-
\partial_j^{m_j} f_X(\mathbf{v}_{-j},\, v_j)
]\, \mathrm{d}u
\]
for \(m_j\ge1\). For \(m_j=0\), the sum over \(k=0,\,\dots,\,m_j\) reduces to the term \(k=0\), and the remainder is defined as
\[
R_j(\mathbf{v},\, s) \coloneqq f_X(\mathbf{v}_{-j},\, v_j+s) - f_X(\mathbf{v}_{-j},\,v_j),
\]
so that the expansion takes the form \(f_X(\mathbf{v}_{-j},\, v_j+s) = f_X(\mathbf{v}_{-j},\, v_j) + R_j(\mathbf{v},\, s)\). 

The remainder satisfies the \(L^1\)-bound
\begin{equation}\label{eq:remainder2}
\int_{\mathbb R^d} |R_j(\mathbf{v},\, s)|\,\mathrm{d}\mathbf{v}
\le C_{\alpha_j} L |s|^{\alpha_j},
\end{equation}
where 
\[C_{\alpha_j} \coloneqq
\begin{cases}
\frac{\Gamma(\alpha_j - m_j + 1)}{\Gamma(\alpha_j + 1)}, & m_j \ge 1,\\[5pt]
1, & m_j = 0.
\end{cases}
\]
Indeed, for \(m_j\ge1\), taking the \(L^1\)-norm over \(\mathbf{v}_{-j}\) and \(v_j\) yields
\begin{align*}\label{eq:remainderterm}
\int_{\mathbb R^d} |R_j(\mathbf{v},\, s)|\,\mathrm{d}\mathbf{v}
&\le
\frac{1}{(m_j-1)!}
\int_0^{|s|} (|s|-u)^{m_j-1}
\|
\Delta_{u \mathbf{e}_j} \partial_j^{m_j} f_X
\|_{L^1(\mathbb R^d)}
\,\mathrm{d}u \nonumber\\
&\le
\frac{L}{(m_j-1)!}
\int_0^{|s|} (|s|-u)^{m_j-1} |u|^{\eta_j}\,\mathrm{d}u \nonumber\\
&=
C_{\alpha_j} L |s|^{\alpha_j},
\end{align*}
where the second inequality follows from condition \ref{it:smooth} \hyperref[it:smooth]{of the Main Article} with \(\eta_j \coloneqq \alpha_j - m_j\). For \(m_j=0\), the same bound is a direct consequence of \ref{it:smooth} \hyperref[it:smooth]{of the Main Article} with \(\eta_j=\alpha_j\). This establishes \eqref{eq:remainder2}.
Substituting the Taylor expansion \eqref{eq:dev} into the integral over \(t\) in \eqref{eq:biasint}, we obtain
\[\begin{split}
\int_{v_j-h_jz_j}^{v_j} f_X(\mathbf{v}_{-j},\, t)\,\mathrm{d}t
&=
-\sum_{k=0}^{m_j}
\frac{(-h_jz_j)^{k+1}}{(k+1)!}
\partial_j^k f_X(\mathbf{v}_{-j},\, v_j)
+
\int_{v_j-h_jz_j}^{v_j} R_j(\mathbf{v},\, t-v_j)\,\mathrm{d}t\\
&=
-\sum_{k=0}^{m_j}
\frac{(-h_jz_j)^{k+1}}{(k+1)!}
\partial_j^k f_X(\mathbf{v})
+
\int_{v_j-h_jz_j}^{v_j} R_j(\mathbf{v},\, t-v_j)\,\mathrm{d}t.
\end{split}
\]
Multiplying by \(K_j(z_j)\), integrating over \(z_j\), taking the \(L^1\)-norm in \(\mathbf{v}\), and using assumption \ref{it:L1} \hyperref[it:L1]{of the Main Article} that $\|\partial_j^k f_X\|_{L^1(\mathbb R^d)}<\infty$ for all $k=0,\,\ldots,\,m_j$, the polynomial part becomes
\[
\sum_{k=0}^{m_j}
\frac{h_j^{k+1}}{(k+1)!}
\left|\int_{\mathbb R} z_j^{k+1} K_j(z_j)\,\mathrm{d}z_j\right|
\int_{\mathbb R^d} |\partial_j^k f_X(\mathbf{v})|\,\mathrm{d}\mathbf{v}=
\sum_{k=0}^{m_j}
\frac{h_j^{k+1}}{(k+1)!}
\left|\int_{\mathbb R} z_j^{k+1} K_j(z_j)\,\mathrm{d}z_j\right|
\|\partial_j^k f_X\|_{L^1(\mathbb R^d)}.
\]
For each \(k=0,\,\dots,\,m_j\), the exponent satisfies \(k+1 \le m_j+1 = \lceil \alpha_j\rceil\). By the moment condition in \eqref{eq:kernelcond} \hyperref[eq:kernelcond]{of the Main Article}, the integrals \(\int_{\mathbb R} z_j^{k+1} K_j(z_j)\,\mathrm{d}z_j\) vanish for all such \(k\). Hence, all polynomial terms disappear identically, leaving only the remainder.
Using the remainder bound \eqref{eq:remainder2}, we estimate the contribution of the remainder as
\begin{align}\label{eq:intremaind}
\int_{\mathbb R} |K_j(z_j)|
\int_{\mathbb R^d}
\bigg|
\int_{v_j-h_jz_j}^{v_j} R_j(\mathbf{v},\, t-v_j)\,\mathrm{d}t
\bigg|
\mathrm{d}\mathbf{v}\,\mathrm{d}z_j
&\le
\int_{\mathbb R} |K_j(z_j)|
\int_0^{|h_jz_j|}
\int_{\mathbb R^d} |R_j(\mathbf{v}, \,u)|\,\mathrm{d}\mathbf{v}\,\mathrm{d}u\,\mathrm{d}z_j \nonumber\\
&\le
C_{\alpha_j} L
\int_{\mathbb R} |K_j(z_j)|
\int_0^{|h_jz_j|} |u|^{\alpha_j}\,\mathrm{d}u\,\mathrm{d}z_j \nonumber\\
&=
\frac{C_{\alpha_j}}{\alpha_j+1} L h_j^{\alpha_j+1}
\int_{\mathbb R} |z_j|^{\alpha_j+1} |K_j(z_j)|\,\mathrm{d}z_j.
\end{align}
Finally, combining \eqref{eq:biasint}, the cancellation of the polynomial terms, and the remainder estimate \eqref{eq:intremaind}, we arrive at the bound
\[
\|B_{\mathbf{h}}\|_{L^1(\mathbb R^d)}
\le
L\sum_{j=1}^d
\left(\prod_{\substack{r=1 \\ r\ne j}}^d \|K_r\|_{L^1(\mathbb R)}\right)
\frac{C_{\alpha_j}}{\alpha_j+1}
\left(\int_{\mathbb R} |z_j|^{\alpha_j+1} |K_j(z_j)|\,\mathrm{d}z_j\right)
h_j^{\alpha_j+1}.
\]
Recalling the definition of $C_j$ in \eqref{eq:constantc_j} \hyperref[eq:constantc_j]{of the Main Article}
as
\[
C_j\coloneqq
\left(\prod_{\substack{r=1 \\ r\ne j}}^d \|K_r\|_{L^1(\mathbb R)}\right)
\frac{C_{\alpha_j}}{\alpha_j+1}
\int_{\mathbb R} |z_j|^{\alpha_j+1} |K_j(z_j)|\,\mathrm{d}z_j<\infty,
\]
we obtain the bound
\[
\|F_X-(K_{\mathbf{h}}\ast F_X)\|_{L^1(\mathbb R^d)}
\le
L\sum_{j=1}^d C_j h_j^{\alpha_j+1}.
\]
This completes the proof.
\end{proof}


\begin{remark}\label{rem:adjusted_proof}
As noted in Remark~\ref{rem:samebias}, under the \(L^2\)-based Nikol'skii condition $f_X\in\mathcal N_2^d(\boldsymbol{\alpha},\, L)                    
$, the proof proceeds identically up to the remainder estimate in \eqref{eq:intremaind}. Consider the remainder contribution in the \(j\)-th coordinate. By applying Minkowski's integral inequality twice, we obtain
\[
\begin{aligned}
\left\| \int_{\mathbb R} |K_j(z_j)| 
\int_{0}^{h_j z_j} R_j(\cdot,\, u)\,\mathrm{d}u \,\mathrm{d}z_j \right\|_{L^2(\mathbb R^d)}
&\le \int_{\mathbb R} |K_j(z_j)| 
\left\| \int_{0}^{h_j z_j} R_j(\cdot,\, u)\,\mathrm{d}u \right\|_{L^2(\mathbb R^d)} \,\mathrm{d}z_j \\
&\le \int_{\mathbb R} |K_j(z_j)| 
\int_{0}^{|h_j z_j|} \| R_j(\cdot,\, u) \|_{L^2(\mathbb R^d)} \,\mathrm{d}u \,\mathrm{d}z_j\\
&\le
\frac{C_{\alpha_j} L}{\alpha_j+1} h_j^{\alpha_j+1}
\int_{\mathbb R} |z_j|^{\alpha_j+1} |K_j(z_j)|\,\mathrm{d}z_j,
\end{aligned}
\]
where, in the last line, we have substituted the \(L^2\)-remainder bound $
\| R_j(\cdot,\, u) \|_{L^2(\mathbb R^d)} \le C_{\alpha_j} L |u|^{\alpha_j}$,
which follows from the Nikol'skii condition in the \(L^2\)-case. 
The cancellation of the polynomial terms is independent of the chosen norm. Therefore, we arrive at the bound
\[
\|B_{\mathbf{h}}\|_{L^2(\mathbb R^d)}
\le
L \sum_{j=1}^d C_j h_j^{\alpha_j+1},
\]
where the constants \(C_j>0\) are exactly those defined in \eqref{eq:constantc_j} \hyperref[eq:constantc_j]{of the Main Article} because the \(C_{\alpha_j}\) (and hence \(C_j\)) are determined solely by the regularity exponent \(\alpha_j\), the kernel moments, and the \(L^1\)-norms of the univariate kernels: they are independent of whether the underlying norm is the \(L^1\) or the \(L^2\). Thus, replacing the \(L^1\)-assumption by the \(L^2\)-assumption does not introduce any new constants into the final bound and the same inequality holds in both settings.
\end{remark}



In what follows, we establish the key technical results needed to control the stochastic error and to prove the upper bound on the risk of the $\varrho_n$-approximate minimum-$L^1$-distance estimator. We first prove a moment bound for the observations $\mathbf{Y}_i$ in terms of the moments of $\mathbf{X}_i$ and $\boldsymbol{\varepsilon}_i$ (Lemma~\ref{lem:boundmomY}). This bound is then used in the proof of the stochastic error control (Proposition~\ref{prop:variance}). Finally, we prove Theorem~\ref{thm:ubrate}, which combines the approximation and stochastic error bounds to obtain the convergence rate of the estimator. We conclude this part with the construction of a superkernel (Proposition~\ref{prop:exsuperkern}), which is useful for adaptive estimation procedures.

\begin{lemma}\label{lem:boundmomY}
Consider the $d$-dimensional convolution model \eqref{eq:model} \hyperref[eq:model]{of the Main Article}
$\mathbf{Y}_i = \mathbf{X}_i + \boldsymbol{\varepsilon}_i$, $i=1,\,\dots,\,n$,
where $(\mathbf{X}_i)_{i\in\mathbb N}$ are i.i.d.~random vectors in $\mathbb{R}^d$ with unknown distribution $\mu_X$, and $(\boldsymbol{\varepsilon}_i)_{i\in\mathbb N}$ are i.i.d.~random vectors in $\mathbb{R}^d$ with known distribution $\mu_\varepsilon$; the two sequences are independent. 
Assume further that the components \(\varepsilon_{1,1},\,\ldots,\,\varepsilon_{1,d}\) of the error vector \(\boldsymbol{\varepsilon}_1\in\mathbb{R}^d\) are independent random variables (this is assumption \ref{eq:E1indep} \hyperref[eq:E1indep]{of the Main Article}). Assume the following moment conditions hold for some $p>0$:
\begin{align}
&\mathsf E_{\mu_X}\Biggl[\prod_{j=1}^{d}(1\vee |X_{1,j}|^{2p})\Biggr] < \infty,\label{(A1)}\\
&\mathsf E_{\mu_{\varepsilon_j}}[|\varepsilon_{1,j}|^{2p}] < \infty \quad\text{for every } j=1,\,\dots,\,d,\label{(A2)}
\end{align}
(these are conditions \eqref{eq:assonX} and \eqref{eq:assoneps} \hyperref[eq:assonX]{of the Main Article}, respectively, except that no constraint on $p$ is assumed here).
Then, there exists a constant $C(d,\,p)$, depending only on $d$ and $p$, such that
\begin{equation}\label{eq:boundprodY}
\mathsf E_{\mu_X\ast\mu_\varepsilon}\Biggl[\prod_{j=1}^{d}(1\vee |Y_{1,j}|^{2p})\Biggr]
\le C(d,\,p)\;
\mathsf E_{\mu_X}\Biggl[\prod_{j=1}^{d}(1\vee |X_{1,j}|^{2p})\Biggr]\;
\prod_{j=1}^{d}\bigl(1\vee \mathsf E_{\mu_{\varepsilon_j}}[|\varepsilon_{1,j}|^{2p}]\bigr).
\end{equation}
\end{lemma}

\begin{proof}
By the independence of the vectors \(\mathbf{X}_1=(X_{1,1},\,\dots,\,X_{1,d})^{\!\top}\) and \(\boldsymbol{\varepsilon}_1=(\varepsilon_{1,1},\,\dots,\,\varepsilon_{1,d})^{\!\top}\), together with assumptions \eqref{(A1)} and \eqref{(A2)},
\begin{align}\label{eq:dismax}
\mathsf E_{\mu_X\ast\mu_\varepsilon}\Biggl[\prod_{j=1}^{d}(1\vee |Y_{1,j}|^{2p})\Biggr]&=\mathsf E_{\mu_X\ast\mu_\varepsilon}\Biggl[\prod_{j=1}^d(1\vee |Y_{1,j}|^{p})^2\Biggr]\nonumber\\
&\lesssim 
\mathsf E_{\mu_X}\Biggl[\prod_{j=1}^d(1\vee |X_{1,j}|^{2p})\Biggr]\,
\prod_{j=1}^d\bigl(1\vee \mathsf E_{\mu_{\varepsilon_j}}[|\varepsilon_{1,j}|^{2p}]\bigr)<\infty,
\end{align}
where the implicit constant in ``\(\lesssim\)'' depends on \(d\) and \(p\) and is finite under the given assumptions. 

We now verify the inequality in~\eqref{eq:dismax} with an explicit constant. For any real numbers \(a,\,b\) and \(p>0\), the inequality \(|a+b|^{p}\le 2^{p}(|a|^{p}+|b|^{p})\) holds. Applying it with \(a=X_{1,j}\) and \(b=\varepsilon_{1,j}\) gives
$
|Y_{1,j}|^{p}=|X_{1,j}+\varepsilon_{1,j}|^{p}\le 2^{p}(|X_{1,j}|^{p}+|\varepsilon_{1,j}|^{p})$.
Hence,
$
1\vee |Y_{1,j}|^{p}\le 1\vee[2^{p}(|X_{1,j}|^{p}+|\varepsilon_{1,j}|^{p})]
\le 2^{p}[(1\vee |X_{1,j}|^{p})+(1\vee |\varepsilon_{1,j}|^{p})]$.
Squaring yields
$
(1\vee |Y_{1,j}|^{p})^{2}\le 2^{2p}[(1\vee |X_{1,j}|^{p})^{2}+(1\vee |\varepsilon_{1,j}|^{p})^{2}
+2(1\vee |X_{1,j}|^{p})(1\vee |\varepsilon_{1,j}|^{p})]$.
Taking the product over \(j=1,\,\dots,\,d\), we obtain
\[
\prod_{j=1}^{d}(1\vee |Y_{1,j}|^{p})^{2}
\le 2^{2pd}
\prod_{j=1}^{d}[(1\vee |X_{1,j}|^{p})^{2}+(1\vee |\varepsilon_{1,j}|^{p})^{2}
+2(1\vee |X_{1,j}|^{p})(1\vee |\varepsilon_{1,j}|^{p})].
\]
The right‑hand side expands into a sum of \(3^{d}\) terms. For each index \(j\), we choose one of the three summands inside the brackets. If we choose the cross term (the third summand) for exactly \(k\) of the \(d\) coordinates, the corresponding coefficient becomes
\(2^{2pd}2^{k}\). Summing over all choices, the total coefficient is
\[
2^{2pd}\sum_{k=0}^{d}\binom{d}{k}2^{k}=2^{2pd}(1+2)^{d}=3^{d}2^{2pd}.
\]
Each term in the expansion is of the form
\[
2^{2pd} 2^{k}\prod_{j=1}^{d}(1\vee |X_{1,j}|^{p})^{\nu_j}(1\vee |\varepsilon_{1,j}|^{p})^{q_j},
\]
where \(\nu_j,\,q_j\in\{0,\,1,\,2\}\), with \(\nu_j+q_j\le 2\), and \(k\) is the number of indices for which \(\nu_j=q_j=1\). 
Because \((1\vee |X_{1,j}|^{p})^{\nu_j}\le 1\vee |X_{1,j}|^{2p}\) and \((1\vee |\varepsilon_{1,j}|^{p})^{q_j}\le 1\vee |\varepsilon_{1,j}|^{2p}\) for every \(j\), each term is bounded by
\[
2^{2pd}2^k\prod_{j=1}^{d}(1\vee |X_{1,j}|^{2p})\,
        \prod_{j=1}^{d}(1\vee |\varepsilon_{1,j}|^{2p}).
\]
Taking expectation and using the independence between the vectors \(\mathbf{X}_1\) and \(\boldsymbol{\varepsilon}_1\) gives 
\[\begin{split}
&\mathsf E_{\mu_X\ast\mu_\varepsilon}\Biggl[2^{2pd}2^k\prod_{j=1}^{d}(1\vee |X_{1,j}|^{p})^{\nu_j}(1\vee |\varepsilon_{1,j}|^{p})^{q_j}\Biggr]\\
&\hspace*{5cm}\le 2^{2pd}2^k\,
   \mathsf E_{\mu_X}\Biggl[\prod_{j=1}^{d}(1\vee |X_{1,j}|^{2p})\Biggr]\,
   \prod_{j=1}^{d}\mathsf E_{\mu_{\varepsilon_j}}[1\vee |\varepsilon_{1,j}|^{2p}].
\end{split}
\]
Summing over all \(3^{d}\) terms gives
\[
\mathsf E_{\mu_X\ast\mu_\varepsilon}\Biggl[\prod_{j=1}^{d}(1\vee |Y_{1,j}|^{p})^{2}\Biggr]
\le 3^{d}2^{2pd}\,
   \mathsf E_{\mu_X}\Biggl[\prod_{j=1}^{d}(1\vee |X_{1,j}|^{2p})\Biggr]\,
   \prod_{j=1}^{d}\mathsf E_{\mu_{\varepsilon_j}}[1\vee |\varepsilon_{1,j}|^{2p}]<\infty
\]
because, by \eqref{(A1)}, \(\mathsf E_{\mu_X}[\prod_{j=1}^{d}(1\vee |X_{1,j}|^{2p})]<\infty\), and, by \eqref{(A2)}, \(\mathsf E_{\mu_{\varepsilon_j}}[|\varepsilon_{1,j}|^{2p}]<\infty\) for every \(j\). Hence, the right-hand side is finite.
Finally, we relate \(\mathsf E_{\mu_{\varepsilon_j}}[1\vee |\varepsilon_{1,j}|^{2p}]\) to \((1\vee\mathsf E_{\mu_{\varepsilon_j}}[|\varepsilon_{1,j}|^{2p}])\). For any non‑negative random variable \(Z\),
\[
\mathsf E[1\vee Z]\le 1+\mathsf E[Z]\le 2(1\vee\mathsf E[Z]).
\]
Applying this with \(Z=|\varepsilon_{1,j}|^{2p}\) yields $
\mathsf E_{\mu_{\varepsilon_j}}[1\vee |\varepsilon_{1,j}|^{2p}]\le 2(1\vee \mathsf E_{\mu_{\varepsilon_j}}[|\varepsilon_{1,j}|^{2p}])$.
Hence,
\[
\prod_{j=1}^{d}\mathsf E_{\mu_{\varepsilon_j}}[1\vee |\varepsilon_{1,j}|^{2p}]
\le 2^{d}\prod_{j=1}^{d}(1\vee \mathsf E_{\mu_{\varepsilon_j}}[|\varepsilon_{1,j}|^{2p}]).
\]
Thus, with the constant $C(d,\,p):=3^{d}2^{2pd} 2^{d}=3^{d}2^{d(2p+1)}$,
\begin{align*}
\mathsf E_{\mu_X\ast\mu_\varepsilon}\Biggl[\prod_{j=1}^{d}(1\vee |Y_{1,j}|^{2p})\Biggr]&\le C(d,\,p)\,
   \mathsf E_{\mu_X}\Biggl[\prod_{j=1}^{d}(1\vee |X_{1,j}|^{2p})\Biggr]\,
   \prod_{j=1}^{d}(1\vee \mathsf E_{\mu_{\varepsilon_j}}[|\varepsilon_{1,j}|^{2p}]),
\end{align*}
which is \eqref{eq:boundprodY} with a constant depending only on \(d\) and \(p\). The proof is thus complete.
\end{proof}


We now use this bound to derive an upper bound on the stochastic error control. 

\begin{proof}[Proof of Proposition~\ref{prop:variance} (Upper bound on the stochastic error)]
We begin by exhibiting an example of a kernel $K_j$ satisfying requirements (\emph{i})--(\emph{iii}) from the statement \hyperref[prop:variance]{in the Main Article}. Let $q_j \ge 2$ be a positive integer parameter. Define the Fourier transform of $K_j$ as
$\varphi_{K_j}(t) \coloneqq (1-t^2)^{q_j}\mathbf{1}_{[-1,1]}(t)$, $t\in\mathbb R$.
Then, $K_j$ is given by the inverse Fourier transform
$K_j(x) \coloneqq (2\pi)^{-1}\int_{\mathbb R} \mathrm{e}^{-\mathrm{i}tx}\varphi_{K_j}(t)\,\di t$, $x\in\mathbb R$.
This choice of $K_j$ is precisely the explicit example referenced in the footnote of Proposition~\ref{prop:variance}.
We now verify that $K_j$ satisfies all the required properties.
\begin{enumerate}[label=(\emph{\roman*}), align=left, leftmargin=*]
\item \emph{Integrability:} Since $\varphi_{K_j}$ is compactly supported and $q_j\ge 2$, it is twice continuously differentiable on $\mathbb R$. Therefore, its inverse Fourier transform $K_j$ decays as $|x|^{-2}$ as $|x|\to\infty$. This follows by integrating by parts twice in the Fourier inversion formula, using that $\varphi_{K_j}\in C_c^2(\mathbb R)$. In particular, $K_j\in L^1(\mathbb R)\cap L^2(\mathbb R)$.\\[-8pt]
\item \emph{Symmetry:} Since $\varphi_{K_j}$ is real and even, its inverse Fourier transform $K_j$ is real and even.\\[-8pt]
\item \emph{Support and regularity of $\varphi_{K_j}$:} By construction, $\varphi_{K_j}$ is supported on $[-1,\,1]$ and, for $q_j\ge 2$, is twice continuously differentiable on $\mathbb R$.
\end{enumerate}
Thus, $K_j$ satisfies requirements (\emph{i})--(\emph{iii}) \hyperref[prop:variance]{in the Main Article}.

\smallskip

We now introduce an auxiliary cutoff function $\phi$. Consider a bump function $\phi \colon \mathbb{R} \to [0,\,1]$ satisfying the following properties:
\begin{itemize}
    \item[(a)] $\phi$ is even, \emph{i.e.}, $\phi(-t) = \phi(t)$ for all $t \in \mathbb{R}$;
    \item[(b)] $\phi(t) = 1$ for $|t| \leq 1$;
    \item[(c)] $\phi(t) = 0$ for $|t| \geq 2$;
    \item[(d)] $\phi \in C^2(\mathbb{R})$.
\end{itemize}
Such a function exists. For instance, define the smoothing polynomial 
$\vartheta(u) = 6u^5 - 15u^4 + 10u^3$, $u\in[0,\,1]$,
which satisfies $\vartheta(0)=0$, $\vartheta^{(1)}(0)=0$, $\vartheta^{(2)}(0)=0$
and $\vartheta(1)=1$, $\vartheta^{(1)}(1)=0$, $\vartheta^{(2)}(1)=0$.
Extend \(\vartheta\) to \(\mathbb R\) by setting \(\vartheta(u)=0\) for \(u\le 0\) and \(\vartheta(u)=1\) for \(u\ge 1\). Then, for \(t\ge 0\), set
\[
\phi(t) \coloneqq
\begin{cases}
1, & 0\le t\le 1,\\[2pt]
1 - \vartheta(t-1), & 1\le t\le 2,\\[2pt]
0, & t\ge 2,
\end{cases}
\]
and extend \(\phi\) to all real \(t\) by evenness, \emph{i.e.}, \(\phi(-t)=\phi(t)\) for \(t\in\mathbb R\). This function \(\phi\) is of class \(C^2(\mathbb{R})\), takes values in \([0,\,1]\), equals \(1\) on \([-1,\,1]\), and vanishes outside \((-2,\,2)\). Properties (a)--(d) are thus verified.

\smallskip

Fix $\mathbf{h}\in(0,\,1)^d$. For each $j=1,\,\ldots,\,d$, define 
\begin{align*}
\tilde K_{1,j,h_j}(u)&\coloneqq\mathcal F^{-1}\!\left[\frac{\varphi_{K_j}(t)}{\varphi_{\varepsilon_j}(t/h_j)}\phi(t/h_j)\right]\!(u),\quad u\in\mathbb R,\\
\shortintertext{and}
\tilde K_{2,j,h_j}(u)&\coloneqq\mathcal F^{-1}\!\left[\frac{\varphi_{K_j}(t)}{\varphi_{\varepsilon_j}(t/h_j)}[1-\phi(t/h_j)]\right]\!(u),\quad u\in\mathbb R.
\end{align*}
The full kernel $$\tilde K_{j,h_j}(u)\coloneqq\mathcal F^{-1}\!\left[\frac{\varphi_{K_j}(t)}{\varphi_{\varepsilon_j}(t/h_j)}\right]\!(u),\quad u\in\mathbb R,$$ satisfies the decomposition 
$\tilde K_{j,h_j}(u)=\tilde K_{1,j,h_j}(u)+\tilde K_{2,j,h_j}(u)$, $u\in\mathbb R$.
Define the integrated kernels
$$G_{1,j,h_j}(x)\coloneqq\int_{-\infty}^x
\tilde K_{1,j,h_j}(u)\,\di u, \quad 
G_{2,j,h_j}(x)\coloneqq\int_{-\infty}^x
\tilde K_{2,j,h_j}(u)\,\di u,\quad x\in\mathbb R,
$$
and their sum
\begin{equation}\label{eq:gfun}
G_{j,h_j}(x)\coloneqq\int_{-\infty}^x\tilde K_{j,h_j}(u)\,\di u=G_{1,j,h_j}(x)+G_{2,j,h_j}(x),\quad x\in\mathbb R.
\end{equation}
Because the function \([\varphi_{K_j}(\cdot)\phi(\cdot/h_j)/\varphi_{\varepsilon_j}(\cdot/h_j)]\) is compactly supported and twice continuously differentiable, its inverse Fourier transform \(\tilde K_{1,j,h_j}\) belongs to \(L^1(\mathbb R)\) and
\[
\int_{\mathbb R} \tilde K_{1,j,h_j}(u) \, \mathrm{d}u 
= \mathcal F[\tilde K_{1,j,h_j}](0) 
= \left. \frac{\varphi_{K_j}(t)}{\varphi_{\varepsilon_j}(t/h_j)} \phi(t/h_j) \right|_{t=0} 
= 1,
\]
so that
\[
\lim_{x\to\infty} G_{1,j,h_j}(x) = \int_{\mathbb R} \tilde K_{1,j,h_j}(u) \, \mathrm{d}u = 1.\]

For a real number $p>1/2$ as in the statement, we define the integral 
$$I_{2p}\coloneqq\int_{\mathbb R}(1\vee |x|^{p})^{-2}\,\di x=
\int_{\mathbb R}(1\vee |x|^{2p})^{-1}\,\di x=\frac{4p}{2p-1}<\infty.$$ 
Then,
\begin{align}
&\mathsf E_{(\mu_X\ast\mu_\varepsilon)^{\otimes n}}[\|\widehat F_{n,\mathbf{h}}-(K_{\mathbf {h}}\ast F_X)\|_{L^1(\mathbb R^d)}]\nonumber\\
&\,\,\,=
\int_{\mathbb R^d}\mathsf E_{(\mu_X\ast\mu_\varepsilon)^{\otimes n}}[|\widehat F_{n,\mathbf{h}}(\mathbf x)-(K_{\mathbf {h}}\ast F_X)(\mathbf x)|]\,\mathrm{d}{\mathbf{x}}\nonumber\\
&\,\,\,=
\int_{\mathbb R^d}\mathsf E_{(\mu_X\ast\mu_\varepsilon)^{\otimes n}}[|\widehat F_{n,\mathbf{h}}(\mathbf x)-\mathsf E_{(\mu_X\ast\mu_\varepsilon)^{\otimes n}}[\widehat F_{n,\mathbf{h}}(\mathbf x)]|]\,\mathrm{d}{\mathbf{x}}\nonumber\\
&\,\,\,
\leq
\int_{\mathbb R^d}\sqrt{\mathrm{Var}[\widehat F_{n,\mathbf{h}}(\mathbf x)]}\,\mathrm{d}{\mathbf{x}}\nonumber\tag{\text{by Lyapunov's inequality}}\\[5pt]
&\,\,\,\leq
\frac{1}{\sqrt{n}}\int_{\mathbb R^d}\sqrt{\mathrm{Var}\Biggl[\prod_{j=1}^d\int_{-\infty}^{(x_j-Y_{1,j})/h_j}\tilde K_{j,h_j}(u_j)\,\mathrm{d} u_j\Biggr]}\,\mathrm{d}{\mathbf{x}}\nonumber\tag{\text{variance of an empirical mean of i.i.d. terms and change of variable}}\\[5pt]
&\,\,\,=\frac{1}{\sqrt{n}}\int_{\mathbb R^d}\sqrt{\mathrm{Var}\Biggl[\prod_{j=1}^d
G_{j,h_j}\left(\frac{x_j-Y_{1,j}}{h_j}\right)\Biggr]}\,\mathrm{d}{\mathbf{x}}\nonumber\\
&\,\,\,=\frac{1}{\sqrt{n}}\int_{\mathbb R^d}\sqrt{\mathrm{Var}\Biggl[\prod_{j=1}^d
(G_{1,j,h_j}+G_{2,j,h_j})\left(\frac{x_j-Y_{1,j}}{h_j}\right)\Biggr]}\,\mathrm{d}{\mathbf{x}}\nonumber\tag{\text{by identity \eqref{eq:gfun}}}\\[5pt]
&\,\,\,=\frac{1}{\sqrt{n}}\int_{\mathbb R^d}\sqrt{\mathrm{Var}\Biggl[\prod_{j=1}^d
\bigl[G_{1,j,h_j}(s)\mathbf{1}_{\{s\leq0\}}-(1-G_{1,j,h_j}(s))\mathbf{1}_{\{s>0\}}+G_{2,j,h_j}(s)
\bigl]_{s=\frac{x_j-Y_{1,j}}{h_j}}\Biggr]}\,\mathrm{d}{\mathbf{x}}\nonumber\\
&\,\,\,\leq
\frac{1}{\sqrt{n}}\int_{\mathbb R^d}\sqrt{\mathsf E_{\mu_X\ast\mu_\varepsilon}\Biggl[\prod_{j=1}^d\Biggl(
\bigl[G_{1,j,h_j}(s)\mathbf{1}_{\{s\leq0\}}-(1-G_{1,j,h_j}(s))\mathbf{1}_{\{s>0\}}+G_{2,j,h_j}(s)\bigr]_{s=\frac{x_j-Y_{1,j}}{h_j}}\Biggr)^2\Biggr]}\,\mathrm{d}{\mathbf{x}}\nonumber\\
&\tag{\text{by the inequality $\mathsf{V}\textrm{ar}[\cdot]\leq\mathsf E[|\cdot|^2]$}}\\[5pt]
&\,\,\,\leq\frac{1}{\sqrt{n}}
\Biggl(\int_{\mathbb R^d}\prod_{j=1}^d(1\vee |x_j|^{p})^{-2}\,{\mathbf{\di x}}\Biggr)^{1/2}\nonumber\\
&\qquad\qquad\,\,\times\Biggl\{\mathsf{E}_{\mu_X\ast\mu_\varepsilon}\Biggl[\prod_{j=1}^d\int_{\mathbb R}(1\vee |x_j|^{p})^2\nonumber\\
&\qquad\qquad\qquad\quad\times
\Biggl(\bigl[G_{1,j,h_j}(s)\mathbf{1}_{\{s\leq0\}}-(1-G_{1,j,h_j}(s))\mathbf{1}_{\{s>0\}}+G_{2,j,h_j}(s)\bigr]_{s=\frac{x_j-Y_{1,j}}{h_j}}\Biggr)^2{\di x_j}\Biggr]\Biggr\}^{1/2}\nonumber\\
&\tag{\text{by the Cauchy--Schwarz inequality}}\\[5pt]
&\,\,\,=\frac{1}{\sqrt{n}}I_{2p}^{d/2}
\Biggl\{\mathsf{E}_{\mu_X\ast\mu_\varepsilon}\Biggl[\prod_{j=1}^d\int_{\mathbb R}(1\vee |x_j|^{2p})
\Biggl(\biggl[\mathbf{1}_{\{v\leq0\}}\int_{-\infty}^{v}\frac{1}{h_j}\tilde K_{1,j,h_j}\Bigl(\frac{z_j}{h_j}\Bigr)\,\di z_j\nonumber\\
&\qquad\qquad\qquad\qquad\qquad\qquad\qquad\qquad\qquad+\mathbf{1}_{\{v>0\}}\int_{v}^{\infty}\frac{1}{h_j}\tilde K_{1,j,h_j}\Bigl(\frac{z_j}{h_j}\Bigr)\,\di z_j\nonumber\\
&\qquad\qquad\qquad\qquad\qquad\qquad\qquad\qquad\qquad+\int_{-\infty}^{v}\frac{1}{h_j}\tilde K_{2,j,h_j}\Bigl(\frac{z_j}{h_j}\Bigr)\,\di z_j\biggr]_{v=x_j-Y_{1,j}}\Biggr)^2\,{\di x_j}\Biggr]\Biggr\}^{1/2}\nonumber\\
&\tag{\text{by the change of variable $z_j=h_j u_j$}}\\
&\,\,\,\lesssim\frac{1}{\sqrt{n}}
\Biggl\{\mathsf{E}_{\mu_X\ast\mu_\varepsilon}\Biggl[\prod_{j=1}^d\int_{\mathbb R}(1\vee |x_j|^{2p})
\Biggl[\mathbf{1}_{\{v\leq0\}}\biggl(\int_{-\infty}^{v}\frac{1}{h_j}\tilde K_{1,j,h_j}\Bigl(\frac{z_j}{h_j}\Bigr)\,\di z_j\biggr)^2\nonumber\\
&\qquad\qquad\qquad\qquad\qquad\qquad\qquad\qquad\qquad+\mathbf{1}_{\{v>0\}}\biggl(\int_{v}^{\infty}\frac{1}{h_j}\tilde K_{1,j,h_j}\Bigl(\frac{z_j}{h_j}\Bigr)\,\di z_j\biggr)^2\nonumber\\
&\qquad\qquad\qquad\qquad\qquad\qquad\qquad\qquad\qquad+\biggl(\int_{-\infty}^{v}\frac{1}{h_j}\tilde K_{2,j,h_j}\Bigl(\frac{z_j}{h_j}\Bigr)\,\di z_j\biggr)^2\Biggr]_{v=x_j-Y_{1,j}}\,{\di x_j}\Biggr]\Biggr\}^{1/2}\nonumber\\
&
&\tag{\text{neglecting $I_{2p}^{d/2}$ and squaring}}\\[5pt]
&\,\,\,\lesssim\frac{1}{\sqrt{n}}\Biggl\{\mathsf{E}_{\mu_X\ast\mu_\varepsilon}\Biggl[\prod_{j=1}^d\int_{\mathbb R}(1\vee |w_j+Y_{1,j}|^{2p})
\Biggl[\mathbf{1}_{\{w_j\leq0\}}\biggl(\int_{-\infty}^{w_j}\frac{1}{h_j}\tilde K_{1,j,h_j}\Bigl(\frac{z_j}{h_j}\Bigr)\,\di z_j\biggr)^2\nonumber\\&\qquad\qquad\qquad\qquad\qquad\qquad\qquad\qquad\qquad+\mathbf{1}_{\{w_j>0\}}\biggl(\int_{w_j}^{\infty}\frac{1}{h_j}\tilde K_{1,j,h_j}\Bigl(\frac{z_j}{h_j}\Bigr)\,\di z_j\biggr)^2\nonumber\\
&\qquad\qquad\qquad\qquad\qquad\qquad\qquad\qquad\qquad+\biggl(\int_{-\infty}^{w_j}\frac{1}{h_j}\tilde K_{2,j,h_j}\Bigl(\frac{z_j}{h_j}\Bigr)\,\di z_j\biggr)^2\Biggr]\,\di w_j\Biggr]\Biggr\}^{1/2}\nonumber\\
&\tag{by the change of variable {$w_j=x_j-Y_{1,j}$}}\\[5pt]
&\,\,\,\lesssim\frac{1}{\sqrt{n}}\sqrt{\mathsf{E}_{\mu_X\ast\mu_\varepsilon}\Biggl[\prod_{j=1}^d(1\vee |Y_{1,j}|^{2p})\Biggr]}
\Biggl\{\prod_{j=1}^d\int_{\mathbb R}(1\vee |w_j|^{p})^2
\Biggl[\mathbf{1}_{\{w_j\leq0\}}\biggl(\int_{-\infty}^{w_j}\frac{1}{h_j}\tilde K_{1,j,h_j}\Bigl(\frac{z_j}{h_j}\Bigr)\,\di z_j\biggr)^2\nonumber
\\&\qquad\qquad\qquad\qquad\qquad\qquad\qquad\qquad\qquad+\mathbf{1}_{\{w_j>0\}}\biggl(\int_{w_j}^{\infty}\frac{1}{h_j}\tilde K_{1,j,h_j}\Bigl(\frac{z_j}{h_j}\Bigr)\,\di z_j\biggl)^2\nonumber\\
&\qquad\qquad\qquad\qquad\qquad\qquad\qquad\qquad\qquad+\biggl(\int_{-\infty}^{w_j}\frac{1}{h_j}\tilde K_{2,j,h_j}\Bigl(\frac{z_j}{h_j}\Bigr)\,\di z_j\biggr)^2
\Biggr]\,\di w_j
\Biggr\}^{1/2}\nonumber\\
&\,\,\,\lesssim\frac{1}{\sqrt{n}}
\Biggl\{\prod_{j=1}^d\int_{\mathbb R}(1\vee |w_j|^{p})^2
\Biggl[\mathbf{1}_{\{w_j\leq0\}}\biggl(\int_{-\infty}^{w_j}\frac{1}{h_j}\tilde K_{1,j,h_j}\Bigl(\frac{z_j}{h_j}\Bigr)\,\di z_j\biggr)^2\\
&\qquad\qquad\qquad\qquad\qquad\qquad\qquad\qquad\qquad+\mathbf{1}_{\{w_j>0\}}\biggl(\int_{w_j}^{\infty}\frac{1}{h_j}\tilde K_{1,j,h_j}\Bigl(\frac{z_j}{h_j}\Bigr)\,\di z_j\biggl)^2\nonumber\\
&\qquad\qquad\qquad\qquad\qquad\qquad\qquad\qquad\qquad+\biggl(\int_{-\infty}^{w_j}\frac{1}{h_j}\tilde K_{2,j,h_j}\Bigl(\frac{z_j}{h_j}\Bigr)\,\di z_j\biggr)^2\Biggr]\,\di w_j\Biggr\}^{1/2}\nonumber\\
&\nonumber\tag{\text{by inequality \eqref{eq:boundprodY} of Lemma \ref{lem:boundmomY}}}\\[5pt]
&\,\,\,\lesssim\frac{1}{\sqrt{n}}\prod_{j=1}^d\Biggl\{\underbrace{\int_{-\infty}^0(1\vee |w_j|^{p})^2
\biggl(\int_{-\infty}^{w_j}\frac{1}{h_j}\tilde K_{1,j,h_j}\Bigl(\frac{z_j}{h_j}\Bigr)\,\di z_j\biggr)^2\,\di w_j}_{\eqqcolon I^-_{1,j,h_j}}\nonumber\\&\qquad\qquad\qquad\qquad\qquad\qquad\qquad\qquad\qquad
+\underbrace{\int_0^{\infty}(1\vee |w_j|^{p})^2
\biggl(\int_{w_j}^{\infty}\frac{1}{h_j}\tilde K_{1,j,h_j}\Bigl(\frac{z_j}{h_j}\Bigr)\,\di z_j\biggr)^2\,\di w_j}_{\eqqcolon I^+_{1,j,h_j}}\nonumber\\
&\qquad\qquad\qquad\qquad\qquad\qquad\qquad\qquad\qquad+\underbrace{\int_{\mathbb R}(1\vee |w_j|^{p})^2\biggl(\int_{-\infty}^{w_j}\frac{1}{h_j}\tilde K_{2,j,h_j}\Bigl(\frac{z_j}{h_j}\Bigr)\,\di z_j\biggr)^2\,\di w_j}_{\eqqcolon I_{2,j,h_j}}\Biggr\}^{1/2}\nonumber\\
&\,\,\,\lesssim \frac{1}{\sqrt{n}}\prod_{j=1}^d(I^-_{1,j,h_j}+
I^+_{1,j,h_j}+I_{2,j,h_j})^{1/2}.\label{eq:var}
\end{align}
Under noise conditions \ref{eq:E1indep}, \ref{it:E21}, and \ref{eq:E3} \hyperref[eq:E1indep]{of the Main Article}, the terms $I^-_{1,j,h_j}$, $I^+_{1,j,h_j}$, and $I_{2,j,h_j}$ can be bounded in the same way as the terms $I_2^-$, $I^+$ and $J$, respectively, in the proof of Proposition~3.1 in \citet{Dedecker2015}, pp.~248--253. Specifically, for each $j=1,\,\dots,\,d$, the following asymptotic bounds hold as $h_j \to 0^+$:
\begin{equation}\label{eq:ub}
I^-_{1,j,h_j}=O(1),\quad I^+_{1,j,h_j}=O(1),\quad I_{2,j,h_j}=O(h_j^{-2(\beta_j-\frac{1}{2})_+}|\log h_j|^{\mathbf{1}_{\{\beta_j=1/2\}}}).
\end{equation}
From \eqref{eq:var} and \eqref{eq:ub}, recalling \eqref{eq:defV} \hyperref[eq:defV]{of the Main Article}, we obtain\footnote{The symmetry of the kernel $K_j$ (equivalently, the fact that $\varphi_{K_j}$ is real and even) ensures that the quotient $Q(t) \coloneqq \varphi_{K_j}(t)\phi(t/h_j)/\varphi_{\varepsilon_j}(t/h_j)$ satisfies $Q(-t) = \overline{Q(t)}$, which guarantees that the inverse Fourier transform $\tilde K_{1,j,h_j}$ is real-valued. This is essential because the variance of a complex-valued random variable involves the modulus squared, which would introduce additional cross-terms and prevent the factorization of the variance into the three separate integral terms $I^-_{1,j,h_j}$, $I^+_{1,j,h_j}$, and $I_{2,j,h_j}$ used in the proof. Without this real-valued property, the standard variance arguments for real-valued random variables would not apply.}
\begin{align*}
\mathsf{E}_{(\mu_X\ast\mu_\varepsilon)^{\otimes n}}[\|\widehat F_{n,\mathbf{h}}-(K_{\mathbf{h}}\ast F_X)\|_{L^1(\mathbb R^d)}]
&\lesssim \frac{1}{\sqrt{n}}\prod_{j=1}^d(I^-_{1,j,h_j}+I^+_{1,j,h_j}+I_{2,j,h_j})^{1/2}\\
&\lesssim \frac{1}{\sqrt{n}}\prod_{j=1}^d \sqrt{I_{2,j,h_j}} \tag{\text{because $I_{2,j,h_j}$ is the dominant term
for small $h_j$}}\\
&\lesssim \frac{1}{\sqrt{n}}\prod_{j=1}^d h_j^{-(\beta_j-\frac12)_+}|\log h_j|^{\mathbf{1}_{\{\beta_j=1/2\}}/2}\\
&=V(\mathbf{h})\prod_{j=1}^d |\log h_j|^{\mathbf{1}_{\{\beta_j=1/2\}}/2}\\
&=\mathcal V(\mathbf{h}),
\end{align*}
which concludes the proof for the $L^1$-case. 

\smallskip

For the $L^2$-extension, by Jensen's inequality,
\[\begin{split}
\mathsf E_{(\mu_X\ast\mu_\varepsilon)^{\otimes n}}[\|\widehat F_{n,\mathbf{h}} - (K_{\mathbf{h}} \ast F_X)\|_{L^2(\mathbb R^d)}]
&\le
\Bigl( \mathsf E_{(\mu_X\ast\mu_\varepsilon)^{\otimes n}}[\|\widehat F_{n,\mathbf{h}} - (K_{\mathbf{h}} \ast F_X)\|_{L^2(\mathbb R^d)}^2] \Bigr)^{1/2}\\
&=
\left( \int_{\mathbb R^d} \mathrm{Var}[\widehat F_{n,\mathbf{h}}(\mathbf{x})] \, \di\mathbf{x} \right)^{1/2}.
\end{split}
\]
The same variance bound as in the $L^1$-case yields
$
\mathsf E_{(\mu_X\ast\mu_\varepsilon)^{\otimes n}}[\|\widehat F_{n,\mathbf{h}} - (K_{\mathbf{h}} \ast F_X)\|_{L^2(\mathbb R^d)}]
\lesssim
\mathcal V(\mathbf{h})$,
which completes the proof.
\end{proof}


We now combine the bias and stochastic error upper bounds to derive an upper bound on the convergence rate for the $L^1$-risk of the $\varrho_n$-approximate minimum-$L^1$-distance estimator $\widehat F_{n,\mathbf{h}}^\ast$.

\begin{proof}[Proof of Theorem~\ref{thm:ubrate} (Upper bound on the 
$L^1$-risk of $\widehat F_{n,\mathbf{h}}^\ast$)]
The proof proceeds in two steps. First, we verify that the kernel constructed in Proposition~\ref{prop:variance} also satisfies the moment conditions required for the bias bound. Second, we combine the resulting bias and stochastic error bounds and optimize over the bandwidth vector $\mathbf{h}$.

Since the kernel estimator must simultaneously satisfy the conditions required by the bias Proposition~\ref{prop:bias} and the variance Proposition~\ref{prop:variance}, the kernel $K_j$ must be constructed so that both requirements (\emph{i})--(\emph{iii}) \hyperref[prop:variance]{of the Main Article} and the moment conditions in \eqref{eq:kernelcond} \hyperref[eq:kernelcond]{of the Main Article} are simultaneously verified. We now show that the kernel constructed above fulfills all these conditions. Set $\kappa_j \coloneqq \lceil \alpha_j\rceil$ and let $q_j \coloneqq \kappa_j + 3$. Since $\alpha_j > 0$, we have $\kappa_j = \lceil \alpha_j\rceil \ge 1$, hence $q_j = \kappa_j + 3 \ge 4 > 2$, which satisfies the requirement $q_j \ge 2$ of Proposition~\ref{prop:variance}. 

For a function $K_j\in L^1(\mathbb R)$, the \vir{moments} are related to the derivatives of its Fourier transform at the origin by
$\int_{\mathbb R} z^r K_j(z)\,\di z = \mathrm{i}^r \varphi_{K_j}^{(r)}(0)$, $r=0,\,1,\,\dots$,
where $\varphi_{K_j}(t)=(1-t^2)^{q_j}\mathbf{1}_{[-1,\,1]}(t)$.
Since $\varphi_{K_j}$ is a polynomial of degree $2q_j$ on $[-1,\,1]$, its derivatives at $0$ satisfy
$\varphi_{K_j}^{(r)}(0) = 0$ for $r=1,\,\dots,\,q_j-1$,
and $\varphi_{K_j}(0)=1$. In particular, with $q_j=\kappa_j+3$, we have
$\varphi_{K_j}^{(r)}(0)=0$ for $r=1,\,\dots,\,\kappa_j+2$.
Therefore, we have $\int_{\mathbb R} z^r K_j(z)\,\di z = 0$ for $r=1,\,\dots,\,\kappa_j+2$,
and $\int_{\mathbb R} K_j(z)\,\di z = 1$. 
Thus, in particular, 
\begin{equation}\label{(M2)}
\int_{\mathbb R} z^r K_j(z)\,\di z = \delta_{0,r},\quad r=0,\,\dots,\,\lceil \alpha_j\rceil.
\end{equation}
This verifies the moment condition in \eqref{eq:kernelcond} \hyperref[eq:kernelcond]{of the Main Article}.

We now verify that 
\begin{equation}\label{(M3)}
\int_{\mathbb R} |z|^{\alpha_j+1}|K_j(z)|\,\di z < \infty. 
\end{equation}
Since $\varphi_{K_j}\in C^{q_j-1}(\mathbb R) = C^{\kappa_j+2}(\mathbb R)$, its inverse Fourier transform $K_j$ satisfies the decay estimate $
|K_j(x)| \lesssim (1+|x|)^{-(\kappa_j+2)}$, $x\in\mathbb R$.
Since $\kappa_j+2 > \alpha_j+2$ (because $\kappa_j = \lceil \alpha_j\rceil > \alpha_j$), we get
\[
\int_{\mathbb R} |z|^{\alpha_j+1}|K_j(z)|\,\di z
\lesssim \int_{\mathbb R} |z|^{\alpha_j+1}(1+|z|)^{-(\kappa_j+2)}\,\di z < \infty.
\]
Thus, \eqref{(M3)} is satisfied.

The parameter $q_j$ is therefore chosen sufficiently large to guarantee that conditions \eqref{(M2)} and \eqref{(M3)} are satisfied. Consequently, the kernel $K_j$ is a valid univariate kernel of order $\lceil \alpha_j\rceil$ for the anisotropic Nikol'skii regularity $\alpha_j$ in direction $j$.

Since $\widehat F_{n,\mathbf{h}}^\ast$ is, by definition, a $\varrho_n$-approximate $L^1$-minimizer of $\|F - \widehat F_{n,\mathbf{h}}\|_{L^1(\mathbb R^d)}$ over all cumulative distribution functions $F$ on $\mathbb R^d$, see condition \eqref{eq:defminimizer} \hyperref[eq:defminimizer]{of the Main Article}, for any $\mathbf{h} \in (0,\,1)^d$ and any sequence $\varrho_n \to 0$ as $n\to\infty$,
we have $\|\widehat F_{n,\mathbf{h}}^\ast - \widehat F_{n,\mathbf{h}}\|_{L^1(\mathbb R^d)} = \inf_{F \in\mathfrak F} \|F - \widehat F_{n,\mathbf{h}}\|_{L^1(\mathbb R^d)} + \varrho_n$.
Because $F_X$ is a cumulative distribution function, $\inf_{F \in\mathfrak F} \|F - \widehat F_{n,\mathbf{h}}\|_{L^1(\mathbb R^d)}\leq \|F_X - \widehat F_{n,\mathbf{h}}\|_{L^1(\mathbb R^d)}$, hence 
$$\|\widehat F_{n,\mathbf{h}}^\ast - \widehat F_{n,\mathbf{h}}\|_{L^1(\mathbb R^d)} \leq \|F_X - \widehat F_{n,\mathbf{h}}\|_{L^1(\mathbb R^d)} + \varrho_n.$$
Consequently,
\begin{align}\label{eq:stimdist}
\|\widehat F_{n,\mathbf{h}}^\ast - F_X\|_{L^1(\mathbb R^d)} 
&\leq \|\widehat F_{n,\mathbf{h}}^\ast - \widehat F_{n,\mathbf{h}}\|_{L^1(\mathbb R^d)} + \|\widehat F_{n,\mathbf{h}} - F_X\|_{L^1(\mathbb R^d)} \nonumber\\
&\leq \|F_X - \widehat F_{n,\mathbf{h}}\|_{L^1(\mathbb R^d)} + \varrho_n + \|\widehat F_{n,\mathbf{h}} - F_X\|_{L^1(\mathbb R^d)} \nonumber\\
&= 2\|\widehat F_{n,\mathbf{h}} - F_X\|_{L^1(\mathbb R^d)} + \varrho_n.
\end{align}
Thus, the $L^1$-error of $\widehat F_{n,\mathbf{h}}^\ast$ is bounded by twice the $L^1$-error of $\widehat F_{n,\mathbf{h}}$ plus $\varrho_n$. Hence, analyzing $\|\widehat F_{n,\mathbf{h}}^\ast - F_X\|_{L^1(\mathbb R^d)}$ reduces to studying $\|\widehat F_{n,\mathbf{h}} - F_X\|_{L^1(\mathbb R^d)}$ and then choosing $\mathbf{h}$ to minimize this upper bound. The choice 
$\varrho_n=O({n}^{-1/2})$ is seen to preserve the minimax-optimal rates of $\widehat F_{n,\mathbf{h}}$. Noting that $n^{-1/2}\lesssim\psi_n^{\boldsymbol{\alpha}, \boldsymbol{\beta}}$, we have $\varrho_n=O(\psi_n^{\boldsymbol{\alpha}, \boldsymbol{\beta}})$. Taking expectations in \eqref{eq:stimdist} and applying the bound (which will be proved below)
\begin{equation}\label{eq:ratebound}
\mathsf{E}_{(\mu_X \ast \mu_\varepsilon)^{\otimes n}}[\|\widehat F_{n,\mathbf{h}} - F_X\|_{L^1(\mathbb R^d)}] \lesssim \psi_n^{\boldsymbol{\alpha}, \boldsymbol{\beta}}(\log n)^{\frac12\sum_{j=1}^d\mathbf{1}_{\{\beta_j=1/2\}}},
\end{equation}
we obtain
\[\begin{split}
\mathsf E_{(\mu_X \ast \mu_\varepsilon)^{\otimes n}}[ \|\widehat F^\ast_{n,\mathbf{h}} - F_X\|_{L^1(\mathbb R^d)}] &\leq 2\mathsf E_{(\mu_X \ast \mu_\varepsilon)^{\otimes n}}[ \|\widehat F_{n,\mathbf{h}} - F_X\|_{L^1(\mathbb R^d)}] + \varrho_n\\ &\lesssim \psi_n^{\boldsymbol{\alpha}, \boldsymbol{\beta}}\;(\log n)^{\frac12\sum_{j=1}^d\mathbf{1}_{\{\beta_j=1/2\}}}.
\end{split}
\]

We now prove inequality \eqref{eq:ratebound}.
By Proposition~\ref{prop:bias} (approximation error) and Proposition~\ref{prop:variance} (stochastic error), for every \(\mathbf{h}\in(0,\,1)^d\),
\begin{align*}
&\mathsf{E}_{(\mu_X\ast\mu_\varepsilon)^{\otimes n}}[\|\widehat F_{n,\mathbf{h}}-F_X\|_{L^1(\mathbb R^d)}]\\
&\hspace*{3cm}\leq 
\underbrace{\mathsf{E}_{(\mu_X\ast\mu_\varepsilon)^{\otimes n}}[\|\widehat F_{n,\mathbf{h}}-(K_{\mathbf{h}}\ast F_X)\|_{L^1(\mathbb R^d)}]}_{\text{stochastic error}} 
\;+\;
\underbrace{\|F_X-(K_{\mathbf{h}}\ast F_X)\|_{L^1(\mathbb R^d)}}_{\text{approximation error}}\\
&\hspace*{3cm}\lesssim V(\mathbf{h})\;\prod_{j=1}^d |\log h_j|^{\mathbf{1}_{\{\beta_j=1/2\}}/2}
\;+\;
\sum_{j=1}^d h_j^{\alpha_j+1}\\
&\hspace*{3cm}=\mathcal V(\mathbf{h})
\;+\;
\sum_{j=1}^d h_j^{\alpha_j+1}
\eqqcolon\psi_n(\mathbf{h}),
\end{align*}
where $V(\mathbf{h})$ and $\mathcal V(\mathbf{h})$ are defined in \eqref{eq:defV} and  \eqref{eq:varprox} \hyperref[eq:defV]{of the Main Article}, respectively.

We first ignore the logarithmic factors. 
Let \(\mathbf{h}_{\text{opt}}\) be a minimizer of \(\psi_n(\mathbf{h})\). 
For coordinates with \(\beta_j > 1/2\), the first-order optimality conditions give
\[
(\alpha_j+1)h_{\text{opt},j}^{\alpha_j+1} 
= \frac{\beta_j-1/2}{\sqrt{n}}\prod_{k=1}^d h_{\text{opt},k}^{-(\beta_k-1/2)}.
\]
To derive this, we set
\[
\psi_n(\mathbf{h}) =  \frac{1}{\sqrt{n}}\prod_{k=1}^d h_k^{-(\beta_k-1/2)_+} + \sum_{k=1}^d h_k^{\alpha_k+1},
\]
where for coordinates with \(\beta_k \le 1/2\) we have \((\beta_k-1/2) = 0\) so that the corresponding factor is \(h_k^0 = 1\). For coordinates with \(\beta_j > 1/2\), differentiating with respect to \(h_j\) gives
\[
\frac{\partial \psi_n(\mathbf{h})}{\partial h_j}
= -\frac{\beta_j-1/2}{\sqrt{n}} h_j^{-(\beta_j-1/2)-1}
\prod_{\begin{subarray}{l} k=1 \\ k\ne j \end{subarray}}^d h_k^{-(\beta_k-1/2)}
+ (\alpha_j+1)h_j^{\alpha_j}.
\]
Setting this equal to zero gives
\[
(\alpha_j+1)h_j^{\alpha_j+1}
= \frac{\beta_j-1/2}{\sqrt{n}}\prod_{k=1}^d h_k^{-(\beta_k-1/2)}.
\]
Solving this system yields
\[
h_{\text{opt},j} \asymp n^{-\frac{E}{2(\alpha_j+1)}}, \quad j=1,\,\dots,\,d,
\]
where the exponent $E$ is as defined in \eqref{eq:rate} \hyperref[eq:rate]{of the Main Article}. For these bandwidths, the bias and stochastic terms balance:
\[
\sum_{j=1}^d h_{\text{opt},j}^{\alpha_j+1} \asymp V(\mathbf{h}_{\text{opt}}).
\]
Substituting \(\mathbf{h}_{\text{opt}}\) into \(V(\mathbf{h})\) gives
\[
V(\mathbf{h}_{\text{opt}}) = n^{-E/2} 
= \psi_n^{\boldsymbol{\alpha},\boldsymbol{\beta}}.
\]

When $\beta_j = 1/2$ for some $j$, the factor $|\log h_j|$ satisfies $|\log h_j| \asymp \log n$, since $h_j \asymp n^{-E/[2(\alpha_j+1)]}$ and the exponent is positive. Therefore, the rate acquires the extra logarithmic factor
\[
(\log n)^{\frac12\sum_{j=1}^d\mathbf{1}_{\{\beta_j=1/2\}}}.
\]
This completes the proof. 
\end{proof}

We now explicitly construct a superkernel, thereby demonstrating the existence of a kernel of infinite order. The construction proceeds via the Fourier transform: we choose a smooth, compactly supported, even function that equals $1$ near the origin. Its inverse Fourier transform then inherits the required regularity, integrability, and moment-vanishing properties. Such kernels are particularly useful for adaptive estimation procedures, as they avoid the need to know the regularity of the density in advance. Recall that $C^\infty(\mathbb{R})$ denotes the space of infinitely differentiable functions $f:\mathbb{R}\to\mathbb{R}$.

\begin{proof}[Proof of Proposition \ref{prop:exsuperkern} (Existence of a superkernel)]
The assertion is proved by explicit construction. 
Let $\Phi:\, \mathbb R \to [0,\,1]$ be a function of class $C^\infty(\mathbb R)$ such that
\[
\Phi(t) \coloneqq
\begin{cases}
1, &\quad\,\,\,\, |t| \le 1, \\
\psi(|t|), & 1 < |t| < 2, \\
0, & \quad\,\,\,\, |t| \ge 2,
\end{cases}
\]
where $\psi$ is a function of class $C^\infty(\mathbb R)$ that smoothly connects $1$ to $0$ in a monotone way, with $\psi^{(1)}(1) = \psi^{(1)}(2) = 0$ and all derivatives vanishing at $1$ and $2$. Define $\mathcal F[K](t) \coloneqq \Phi(t)$ for $t\in\mathbb R$. By construction, $\mathcal F[K]$ is even, of class $C^\infty(\mathbb R)$, compactly supported in $[-2,\,2]$, and satisfies $\mathcal F[K](t) = 1$ for $|t| \le 1$.

The kernel $K$ is then defined as the inverse Fourier transform of $\Phi$
\[
K(x) \coloneqq \mathcal F^{-1}[\Phi(t)](x)= \frac{1}{2\pi}\int_{\mathbb R}\mathrm{e}^{-\mathrm{i} tx} \Phi(t) \,\di t,\quad x\in\mathbb R.
\]
Since $\Phi$ is even, $K$ is symmetric about zero:
\[
K(-x) = \frac{1}{2\pi}\int_{\mathbb R} \mathrm{e}^{-\mathrm{i} t(-x)} \Phi(t) \, \di t
= \frac{1}{2\pi}\int_{\mathbb R}\mathrm{e}^{-\mathrm{i} ux} \Phi(u)  \, \di u
= K(x),\quad x\in\mathbb R,
\]
where we used the change of variable $u = -t$ and the evenness of $\Phi$.
Since $\mathcal F[K]=\Phi$ and $\Phi\in L^2(\mathbb R)$ (because it is bounded and compactly supported), by Plancherel's theorem we have $K \in L^2(\mathbb R)$ and
\[
\|K\|_{L^2(\mathbb R)}^2 =\frac{1}{2\pi} \|\mathcal F[K]\|_{L^2(\mathbb R)}^2  =\frac{1}{2\pi} \|\Phi\|_{L^2(\mathbb R)}^2 < \infty.
\]
Moreover, since $\Phi$ is of class $C^\infty(\mathbb R)$ with compact support, integration by parts shows that its inverse Fourier transform $K$ decays faster than any power of $1/|x|$, that is, for every $N \in \mathbb N_0$ there exists $C_N > 0$ such that $
|K(x)| \le C_N (1 + |x|)^{-N}$. In particular, $K \in L^1(\mathbb R)$.

The condition $\Phi(t) = 1$ for $|t| \le 1$ implies that $\mathcal F[K](0) =\Phi(0)= 1$ and 
$(\mathcal F[K])^{(j)}(0) = \Phi^{(j)}(0)=0$ for all $j \ge 1$. 
From the relation between moments and derivatives of the Fourier transform,
$\int_{\mathbb R} x^j K(x) \, \di x =\mathrm{i}^j (\mathcal F[K])^{(j)}(0) = 0$ for all $j \ge 1$, and
$\int_{\mathbb R} K(x) \, \di x = \mathcal F[K](0) = 1$. 
Thus, $K$ annihilates all positive moments, that is, it is a kernel of infinite order. 
A kernel $K$ has thus been constructed satisfying all the required properties and the proof is complete.
\end{proof}

\section{Proofs for the minimax lower bound (Theorem~\ref{th:LB})}\phantomsection\label{sec:S3}

To establish a lower bound on the minimax $L^1$-risk for estimating the cumulative distribution function of a density belonging to an anisotropic Nikol'skii class, we introduce some further notation and preliminary results. Following the classical approach, we define a function $H$ which will be used to construct an additive perturbation of a reference density $f_0$.
\begin{definition}\label{defHH}
Let $H:\,\mathbb R\rightarrow\mathbb R$ be a bounded function such that $H\in C^\infty(\mathbb R)\cap L^1(\mathbb R)$ and the following conditions hold:
\begin{enumerate}[label={(\textbf{B\arabic*})}, leftmargin=*, itemsep=0pt, parsep=0pt]
    \item \label{B1} \(\int_{\mathbb{R}} H(x)\,\di x = 0;\)\\[-8pt]
    \item \label{B2} \(\int_0^1 |H^{(-1)}(z)|\,\di z > 0\), where 
          \(H^{(-1)}(z)\coloneqq \int_{-\infty}^z H(x)\,\di x,\; z\in\mathbb{R}\), is a primitive of \(H;\)\\[-8pt]
    \item \label{B3} \(\mathcal{F}[H](t) = 0\) for \(|t| \notin [1,\,2];\)\\[-8pt]
    \item \label{B4} \((\mathcal{F}[H])^{(\ell)}\) is continuous and \(\|(\mathcal{F}[H])^{(\ell)}\|_\infty\coloneqq \sup_{t\in\mathbb R}|(\mathcal{F}[H])^{(\ell)}(t)|<\infty\) 
          for every \(\ell \in \mathbb{N}_0;\)\\[-8pt]
    \item \label{B5} for a given \(J \in \mathbb{N}\), \(|H(x)| = O(|x|^{-2J})\) as \(|x| \to \infty\).
\end{enumerate}
\end{definition}
For the reader's convenience, Lemma~3 of the Supplementary Material in \citet{Scricciolo25b}, pp.~13--14, provides a detailed proof of the existence of a function $H$ satisfying Definition~\ref{defHH}. The proof follows the approach of \citet{Fan1991}, p.~1268, and \citet{Fan&Truong1993}, p.~1921. This result will be useful in our subsequent analysis.

We also consider the density $g_s$ of a symmetric stable distribution with index $s \in (0,\,1]$, scale parameter $1$ and location parameter $0$. Its characteristic function is given by $\mathcal{F}[g_s](t) = \mathrm{e}^{-|t|^s}$ for $t \in \mathbb{R}$. The relevance of these densities is illustrated by the following result, which appears in the proof of Theorem~1 in \cite{ComteLacour2013}, pp.~593--594.
\begin{lemma}[\citet{ComteLacour2013}]\label{lem:chiquad}
Suppose that noise conditions~\ref{eq:E1indep} and~\ref{it:E22} \hyperref[eq:E1indep]{of the Main Article} hold with $\beta_j > 0$ for $j = 1,\, \dots,\, d$. Moreover, assume that each $\varphi_{\varepsilon_j}$ is differentiable and satisfies, for some exponent $\beta'_j \geq \beta_j$ and a constant $B_j>0$,
\begin{equation*}\label{eq:condder}
|t|^{\beta'_j}|\varphi^{(1)}_{\varepsilon_j}(t)| \leq B_j, \quad t \in \mathbb{R},
\end{equation*}
where $\varphi^{(1)}_{\varepsilon_j}$ denotes the first derivative of $\varphi_{\varepsilon_j}$.

For each $j = 1,\, \dots,\, d$, define
\[
q_j(x_j) \coloneqq \left(\frac{1}{c_j} g_{s_j}\!\left(\frac{\cdot}{c_j}\right) \ast f_{\varepsilon_j}\right)(x_j)
= \int_{\mathbb{R}} \frac{1}{c_j} g_{s_j}\!\left(\frac{x_j-u}{c_j}\right) f_{\varepsilon_j}(u)\,\di u, \quad x_j \in \mathbb{R},
\]
for $s_j \in (0,\,1)$ and $c_j > 0$.
Then, for any $h_j \in (0,\,1)$,
\[
\int_{\mathbb{R}} \left[\int_{\mathbb{R}} H\!\left(\frac{x_j - u}{h_j}\right) f_{\varepsilon_j}(u) \, \mathrm{d}u \right]^2 q^{-1}_j(x_j) \, \mathrm{d}x_j = O(h_j^{2\beta_j + 1}).
\]
\end{lemma}

We begin by establishing an algebraic expansion that will be used repeatedly in the construction of the test densities and in the analysis of their Fourier transforms.

\begin{lemma}[First-order expansion of a perturbed product]\label{lem:foexp}
Let $m \ge 2$ be an integer. For each $\ell \in \{1,\, \dots,\, m\}$, let $A_\ell$ and $B_\ell$ be real numbers (or real-valued functions on some domain $\mathcal X$) with
$\max_{1 \le \ell \le m} |A_\ell| \le M_A$ and 
$\max_{1 \le \ell \le m} |B_\ell| \le M_B$
for some constants $0<M_A,\, M_B < \infty$. Let $\delta \in \mathbb R$ satisfy $|\delta| \le 1$. Then,
\[
\prod_{\ell=1}^m (A_\ell + \delta B_\ell)
= \prod_{\ell=1}^m A_\ell
+ \delta \sum_{r=1}^m B_r \prod_{\substack{\ell=1 \\ \ell \ne r}}^m A_\ell
+ R_m(\delta),
\]
where the remainder 
\begin{equation}\label{eq:remainder}
R_m(\delta)
\coloneqq \delta^2\sum_{k=2}^m \,\sum_{\substack{S \subset \{1,\dots,m\} \\ |S| = k}}\, \delta^{k-2} \Bigg( \prod_{j \in S} B_j \Bigg) \prod_{\substack{\ell=1 \\ \ell \notin S}}^m A_\ell,
\end{equation}
with the convention that products over an empty set are equal to $1$. 
Moreover, the remainder satisfies
\begin{equation}\label{eq:remainder1}
|R_m(\delta)| \le \delta^2 \sum_{k=2}^m \binom{m}{k} M_B^k M_A^{m-k}.
\end{equation}
In particular, $R_m(\delta) = O(\delta^2)$ as $\delta \to 0$ for fixed $m,\, M_A,\, M_B$. If the $A_\ell$ and $B_\ell$ are functions, the bound holds pointwise for every $x \in \mathcal X$.
\end{lemma}

\begin{proof}
We prove the lemma by induction on \(m\), using the implicit definition of the remainder \(R_m(\delta)\) as the difference between the full product and the first two terms of the expansion. We begin with the base case $m = 2$. 
We have
\[
\begin{aligned}
\prod_{\ell=1}^2 (A_\ell + \delta B_\ell) 
&= (A_1 + \delta B_1)(A_2 + \delta B_2) \\
&= A_1 A_2 + \delta A_1 B_2 + \delta B_1 A_2 + \delta^2 B_1 B_2 \\
&= \prod_{\ell=1}^2 A_\ell
+ \delta \sum_{r=1}^2 B_r \prod_{\substack{\ell=1 \\ \ell \ne r}}^2 A_\ell
+ \delta^2 B_1 B_2.
\end{aligned}
\]
Thus, $R_2(\delta) = \delta^2 B_1 B_2$. This agrees with the general formula for the remainder term in \eqref{eq:remainder}.
Using the bound $\max_{1 \le \ell \le 2} |B_\ell| \le M_B$, we obtain 
$|R_2(\delta)| = \delta^2 |B_1 B_2| \le \delta^2 M_B^2$, which 
matches the bound claimed in \eqref{eq:remainder1} for $m=2$.

We now consider the inductive step. Assume the lemma holds for $m \ge 2$, \emph{i.e.}, the expansion and the bound \eqref{eq:remainder1} hold for $m$. We prove it for $m+1$.
Let $P_m \coloneqq \prod_{\ell=1}^m (A_\ell + \delta B_\ell)$. 
Multiplying by the $(m+1)$-th factor and using the induction hypothesis, we obtain
\[
\begin{aligned}
P_{m+1} &= P_m (A_{m+1} + \delta B_{m+1})\\
&=\biggl(\prod_{\ell=1}^m A_\ell
+ \delta \sum_{r=1}^m B_r \prod_{\substack{\ell=1 \\ \ell \ne r}}^m A_\ell
+ R_m(\delta)\biggr)(A_{m+1} + \delta B_{m+1})\\
&= \left( \prod_{\ell=1}^m A_\ell \right) A_{m+1}
+ \delta \left( \prod_{\ell=1}^m A_\ell \right) B_{m+1} + \delta \left( \sum_{r=1}^m B_r \prod_{\substack{\ell=1 \\ \ell \ne r}}^m A_\ell \right) A_{m+1}
+ \delta^2 \left( \sum_{r=1}^m B_r \prod_{\substack{\ell=1 \\ \ell \ne r}}^m A_\ell \right) B_{m+1} \\
&\hspace*{8.7cm} + R_m(\delta) A_{m+1} + \delta R_m(\delta) B_{m+1}\\
&=\prod_{\ell=1}^{m+1} A_\ell+\delta \left[ B_{m+1} \prod_{\substack{\ell=1 \\ \ell \ne m+1}}^{m+1} A_\ell + \sum_{r=1}^m B_r \prod_{\substack{\ell=1 \\ \ell \ne r}}^{m+1} A_\ell\right]+
 \delta^2 \left( \sum_{r=1}^m B_r \prod_{\substack{\ell=1 \\ \ell \ne r}}^m A_\ell \right) B_{m+1}\\
&\hspace*{8.7cm}+ R_m(\delta) A_{m+1} + \delta R_m(\delta) B_{m+1}\\
&=\prod_{\ell=1}^{m+1} A_\ell+\delta\sum_{r=1}^{m+1} B_r \prod_{\substack{\ell=1 \\ \ell \ne r}}^{m+1} A_\ell+
\underbrace{ \delta^2 \left( \sum_{r=1}^m B_r \prod_{\substack{\ell=1 \\ \ell \ne r}}^m A_\ell \right) B_{m+1}+ R_m(\delta) A_{m+1} + \delta R_m(\delta) B_{m+1}}_{\eqqcolon R_{m+1}(\delta)},
\end{aligned}
\]
that is, the desired expansion
\[
P_{m+1}
= \prod_{\ell=1}^{m+1} A_\ell
+ \delta \sum_{r=1}^{m+1} B_r \prod_{\substack{\ell=1 \\ \ell \ne r}}^{m+1} A_\ell
+ R_{m+1}(\delta).
\]
We now bound the remainder term $R_{m+1}(\delta)$. Using the bounds on $A_\ell$ and $B_r$, we obtain
\[
\left| \delta^2 \left( \sum_{r=1}^m B_r \prod_{\substack{\ell=1 \\ \ell \ne r}}^m A_\ell \right) B_{m+1} \right|
\le \delta^2 m M_B^2 M_A^{m-1}.
\]
Also, using the induction bound on $R_m(\delta)$, we get
\[
|R_m(\delta) A_{m+1}| \le \delta^2 \sum_{k=2}^m \binom{m}{k} M_B^k M_A^{m+1-k}
\quad\text{and}\quad
|\delta R_m(\delta) B_{m+1}| 
\le \delta^2 \sum_{k=2}^m \binom{m}{k} M_B^{k+1} M_A^{m-k},
\]
where we used $|\delta| \le 1$ in the last inequality on the right-hand side. Combining these bounds,
\[
\begin{aligned}
|R_{m+1}(\delta)| 
&\le \delta^2 \Biggl[ m M_B^2 M_A^{m-1} + \sum_{k=2}^m \binom{m}{k} M_B^k M_A^{m+1-k} + \sum_{k=2}^m \binom{m}{k} M_B^{k+1} M_A^{m-k} \Biggr] \\
&\le \delta^2 \Biggl[ \binom{m+1}{2} M_B^2 M_A^{m-1} + \sum_{k=2}^m \binom{m}{k} M_B^k M_A^{m+1-k} + \sum_{k'=3}^{m+1} \binom{m}{k'-1} M_B^{k'} M_A^{m+1-k'} \Biggr] \\
&= \delta^2 \sum_{k=2}^{m+1} \binom{m+1}{k} M_B^k M_A^{m+1-k},
\end{aligned}
\]
where
\[
m \le \binom{m+1}{2} \quad \text{for } m \ge 2
\text{ and }
\binom{m}{k} + \binom{m}{k-1} = \binom{m+1}{k}.
\] 
This completes the induction.

Since $m$, $M_A$ and $M_B$ are fixed (independent of $\delta$), the remainder satisfies $R_m(\delta) = O(\delta^2)$ as $\delta \to 0$. The bound is pointwise if the $A_\ell$ and $B_\ell$ are functions, because the inequalities hold for each fixed $x\in\mathcal X$.

We now prove relationship \eqref{eq:remainder}.
The remainder is defined implicitly as
\[
R_m(\delta) = \prod_{\ell=1}^m (A_\ell + \delta B_\ell)
- \prod_{\ell=1}^m A_\ell
- \delta \sum_{r=1}^m B_r \prod_{\substack{\ell=1 \\ \ell \ne r}}^m A_\ell.
\]
Expanding the product, we obtain
\[
\prod_{\ell=1}^m (A_\ell + \delta B_\ell)
= \sum_{S \subseteq \{1,\dots,m\}} \delta^{|S|}
\Bigg( \prod_{j \in S} B_j \Bigg)
\prod_{\substack{\ell=1 \\ \ell \notin S}}^m A_\ell,
\]
which follows by induction on $m\geq 2$.
Separating the terms according to $|S|$, the terms corresponding to $|S|=0$ and $|S|=1$ exactly cancel with the first two terms in the definition of $R_m(\delta)$ (the product of the \(A_\ell\)'s and the first-order correction). Hence,
\[
R_m(\delta)
= \sum_{k=2}^m \sum_{\substack{S \subset \{1,\dots,m\} \\ |S| = k}} \delta^k
\Bigg( \prod_{j \in S} B_j \Bigg)
\prod_{\substack{\ell=1 \\ \ell \notin S}}^m A_\ell.
\]
Factoring out $\delta^2$, we get the expression in \eqref{eq:remainder}.
\end{proof}
We are now ready to prove Theorem~\ref{th:LB}.

\begin{proof}[Proof of Theorem~\ref{th:LB} (Lower bound on the minimax $L^1$-risk)]
The strategy is to construct two probability densities \(f_0\) and \(f_1\) in the anisotropic Nikol'skii class \(\mathcal N_1^d(\boldsymbol\alpha,\, L)\) such that their cumulative distribution functions are well separated in the \(L^1\)-distance, while their convolutions with the noise density, \(f_0 \ast f_\varepsilon\) and \(f_1 \ast f_\varepsilon\), are close in \(\chi^2\)-divergence. This makes the two resulting observation distributions difficult to distinguish, allowing Le Cam's two-point method to yield the minimax rate.

We distinguish two cases depending on the cardinality of $\mathcal P$.
If \(|\mathcal P| = 0\), there are no active coordinates; we follow the approach in the literature based on a translation perturbation of the baseline density \(f_0\). The proof is an adaptation to the cumulative distribution function setting of that of \citet{ComteLacour2013}. If \(|\mathcal P| \ge 1\), the translation perturbation used in the previous case does not extend to the case where multiple active coordinates are present, as it would not yield a perturbed density belonging to the anisotropic Nikol'skii class with a constant that is independent of \(n\). A novel construction of the two test densities \(f_0\) and \(f_1\) is therefore required. 

The two cases are treated separately: we first handle \(|\mathcal P| = 0\), and then address \(|\mathcal P| \ge 1\).\\

\noindent\textbf{Case $|\mathcal P| = 0$.}\\[5pt]
The proof follows the general strategy of \cite{ComteLacour2013} for density deconvolution, with adaptations required for the cumulative distribution function setting. Following \citet{dedecker:michel:2013} for the supersmooth case, we exploit the primitive \(H^{(-1)}\) of the kernel \(H\) to control the \(L^1\)-distance between cumulative distribution functions. 
However, the present setting involves ordinary smooth noise and anisotropic Nikol'skii classes; the verification that the perturbed densities belong to \(\mathcal{N}_1^d(\boldsymbol{\alpha},\,L)\) is new and does not appear in either \citet{ComteLacour2013} or \citet{dedecker:michel:2013}. These modifications yield a lower bound for the \(L^1\)-risk over anisotropic Nikol’skii classes.

We now proceed with the detailed construction. As in the proof of Theorem~3 in \citet
{ComteLacour2013}, pp.~599--603, we define the density
\[f_0(\mathbf{x})\coloneqq\prod_{j=1}^d\frac{1}{c_j}g_{s_j}\biggl(\frac{x_j}{c_j}\biggr), \quad\mathbf{x}\in\mathbb R^d,\]
where $g_{s_j}$ is the density of a symmetric stable law with characteristic function $\mathcal F [g_{s_j}](t)=\mathrm{e}^{-|t|^{s_j}}$, $t\in\mathbb R$, with parameters ${s_j}\in(0,\,1)$ and $c_j>0$.

Let $H$ be an even function satisfying Definition~\ref{defHH} with $J\in\mathbb N$ such that \(2J > \sum_{j=1}^d(s_j+1)\), \emph{e.g.}, \(J=d+1\). Then, by condition \ref{B5},
\(|H(x)| = O(|x|^{-2J})\) as \(|x| \to \infty\). 

Since \(\mathcal P = \emptyset\), we have \(\beta_j \le 1/2\) for all \(j=1,\,\dots,\,d\). 
We set \(h_j = 1\) for every \(j=1,\,\dots,\,d\).
For $\mathbf{k}_0=(1,\,\dots,\,1)^\top\in\mathcal{K}$, define the grid points
$x_{\mathbf{k}_0j} \coloneqq a_j + {k_0}_j(b_j-a_j)h_j=b_j$, $j=1,\,\dots,\,d$,
where $a_j,\,b_j$ are chosen so that $b_j-a_j>0$ and
\[
\int_0^{(b_j-a_j)/2} |H^{(-1)}(-z)|\,\mathrm{d}z > 0.
\]
If \(H\) is even, since \(\int_{\mathbb{R}} H(x)\,\mathrm{d}x = 0\) by condition \ref{B1}, then its primitive \(H^{(-1)}\) is odd. Indeed, for any \(z\in\mathbb{R}\),
\[
\begin{aligned}
H^{(-1)}(-z) &= \int_{-\infty}^{-z} H(x)\,\mathrm{d}x \\
&= \int_{\infty}^{z} H(-u)\,(-\mathrm{d}u) \qquad (x = -u,\; \mathrm{d}x = -\mathrm{d}u)\\
&= \int_{z}^{\infty} H(-u)\,\mathrm{d}u \\
&= \int_{z}^{\infty} H(u)\,\mathrm{d}u \qquad (\text{since } H \text{ is even})\\
&= \int_{\mathbb R} H(u)\,\mathrm{d}u - \int_{-\infty}^{z} H(u)\,\mathrm{d}u \\
&= 0 - H^{(-1)}(z) = -H^{(-1)}(z).
\end{aligned}
\]
Thus, \(H^{(-1)}(-z) = -H^{(-1)}(z)\), so \(|H^{(-1)}(-z)| = |H^{(-1)}(z)|\) for all $z\in\mathbb R$.
Condition \ref{B2} gives \(\int_0^1 |H^{(-1)}(z)|\,\mathrm{d}z > 0\). Hence, by choosing \(a_j,\, b_j\) such that \((b_j-a_j)/2= 1\), \emph{e.g.}, \(a_j = 0,\; b_j = 2\), we obtain
\[
\int_{0}^{(b_j-a_j)/2} |H^{(-1)}(-z)|\,\mathrm{d}z =\int_{0}^{(b_j-a_j)/2} |H^{(-1)}(z)|\,\mathrm{d}z= \int_0^1 |H^{(-1)}(z)|\,\mathrm{d}z > 0,
\]
which verifies the required positivity condition. 
The choice \(a_j=0\) and \(b_j=2\) simplifies the construction: indeed, \(x_{\mathbf{k}_0j} = b_j=2\). The shifted arguments become
\begin{equation}\label{eq:translate}
\frac{x_j - x_{\mathbf{k}_0j}}{2h_j} = \frac{x_j - x_{\mathbf{k}_0j}}{2}=\frac{x_j}{2} - {k_0}_j=\frac{x_j}{2} - 1.
\end{equation}
For a constant $c>0$ small enough to be specified later, define the function
\[
f_1(\mathbf{x}) \coloneqq f_0(\mathbf{x}) + cn^{-\frac12}\prod_{j=1}^d  H\!\left(\frac{x_j}{2}-1\right),\quad \mathbf{x}\in\mathbb{R}^d,
\]
which will be proved to be a perturbed probability density. The constant \(c>0\) is chosen sufficiently small, independently of \(n\), to satisfy two requirements: 
first, to ensure that \(f_1\ge 0\) for all sufficiently large \(n\) (this is verified below); 
second, to guarantee that the constant in the \(\chi^2\)-divergence bound 
\(\chi^2(f_1 \ast f_\varepsilon \,\|\, f_0 \ast f_\varepsilon) \le C c^2n^{-1}\) 
can be made arbitrarily small. 
The independence of \(c\) from \(n\) is crucial: it ensures that the separation rate 
\(\|F_1 - F_0\|_{L^1(\mathbb R^d)} \asymp c n^{-1/2}\) 
is of the correct order \(n^{-1/2}\), with a constant that does not vanish as \(n\to\infty\).

Since $H$ integrates to zero by condition \ref{B1}, which implies 
$$\int_{\mathbb R^d}\prod_{j=1}^d  H\!\left(\frac{x_j}{2}-1\right)\,\di \mathbf x=0,$$ 
we have $\int_{\mathbb{R}^d} f_1(\mathbf{x})\,\mathrm{d}\mathbf{x} = \int_{\mathbb{R}^d} f_0(\mathbf{x})\,\mathrm{d}\mathbf{x} = 1$. 
We now show that \(f_1\) is strictly positive for all sufficiently large \(n\).
We have
\[
|f_1(\mathbf{x})-f_0(\mathbf{x})|
\le c n^{-1/2}\prod_{j=1}^d \left|H\!\left(\frac{x_j-2}{2}\right)\right|.
\]
For any fixed \(R>0\), the density \(f_0\) attains a positive minimum on the compact set \(\{\|\mathbf{x}\|_2\le R\}\), where \(\|\mathbf{x}\|_2 \coloneqq (\sum_{j=1}^d x_j^2)^{1/2}\) denotes the Euclidean norm of \(\mathbf{x}\). Thus, for \(c>0\) small enough depending on \(R\), we have 
$
c\, n^{-1/2} (C_H)^d \le (1/2) \inf_{\|\mathbf{x}\|_2\le R} f_0(\mathbf{x})$,
where \(C_H \coloneqq \sup_{z\in\mathbb R} |H(z)| < \infty\).
Fix \(R_0 = 4\sqrt{d}+1\). For \(\|\mathbf{x}\|_2 > R_0\), since \(\|\mathbf{x}\|_2 \le \sqrt{d}\max_{j=1,\dots,d}|x_j|\), we have \(\max_{j=1,\dots,d}|x_j| \ge R_0/\sqrt{d} > 4\). Choose \(j_0\) such that \(|x_{j_0}| = \max_{j=1,\dots,d}|x_j|\). Then, \(|x_{j_0}| > 4\). Therefore,
\[
\left|\frac{x_{j_0}-2}{2}\right| \ge \frac{|x_{j_0}|-2}{2} \ge \frac{|x_{j_0}|}{4}.
\]
Using the decay property of \(H\) from condition \ref{B5}, \(|H(z)| \le C_J (1+|z|)^{-2J}\), we obtain
\[
\left|H\!\left(\frac{x_{j_0}-2}{2}\right)\right|
\le C_J \left(1+\frac{|x_{j_0}|}{4}\right)^{-2J}
\le C_J' |x_{j_0}|^{-2J}.
\]
For the other coordinates \(j \neq j_0\), we use the boundedness of \(H\).
Hence,
\[
\prod_{j=1}^d \left|H\!\left(\frac{x_j-2}{2}\right)\right|
\le C_J' C_H^{d-1} |x_{j_0}|^{-2J}.
\]
Consequently, $|f_1(\mathbf{x})-f_0(\mathbf{x})| \lesssim n^{-1/2} |x_{j_0}|^{-2J}$. Since $|x_{j_0}| \ge \|\mathbf{x}\|_2/\sqrt{d}$, we have $$|x_{j_0}|^{-2J} \le (\sqrt{d})^{2J} \|\mathbf{x}\|_2^{-2J}.$$ Hence, $|f_1(\mathbf{x})-f_0(\mathbf{x})| \lesssim n^{-1/2} \|\mathbf{x}\|_2^{-2J}$.
Now, note that there exists a constant $c_0>0$ such that for all $\|\mathbf{x}\|_2 > R_0$, we have $f_0(\mathbf{x}) \ge c_0 (1+\|\mathbf{x}\|_2)^{-\sum_{j=1}^d (s_j+1)}$. Indeed, for $s_j \in (0,1)$, the symmetric stable density satisfies $g_{s_j}(x) \sim C_{s_j} |x|^{-1-s_j}$ as $|x|\to\infty$.
Therefore,
\[
\frac{|f_1(\mathbf{x})-f_0(\mathbf{x})|}{f_0(\mathbf{x})}
\lesssim n^{-1/2}
\|\mathbf{x}\|_2^{-2J+\sum_{j=1}^d (s_j+1)}.
\]
Since \(J\in\mathbb N\) is chosen so that \(2J > \sum_{j=1}^d (s_j+1)\), the exponent of \(\|\mathbf{x}\|_2\) is negative. Hence, the right-hand side is maximal at \(\|\mathbf{x}\|_2 = R_0\). Thus, for all \(\|\mathbf{x}\|_2 > R_0\),
\[
\frac{|f_1(\mathbf{x})-f_0(\mathbf{x})|}{f_0(\mathbf{x})}
\lesssim n^{-1/2} R_0^{-2J+\sum_{j=1}^d (s_j+1)}.
\]
Since \(n^{-1/2} \le 1\) for \(n\ge 1\), by choosing \(c>0\) sufficiently small (depending on \(C_J'\), \(d\), \(R_0\), \(C_H\), and \(\inf_{\|\mathbf{x}\|_2\le R_0} f_0(\mathbf{x})\)), we can ensure that
$
|f_1(\mathbf{x})-f_0(\mathbf{x})| \le f_0(\mathbf{x})/2$ for all $\|\mathbf{x}\|_2 > R_0$,
for all sufficiently large \(n\). For the compact set \(\{\|\mathbf{x}\|_2 \le R_0\}\), the uniform bound
$
|f_1(\mathbf{x})-f_0(\mathbf{x})|
\le c n^{-1/2} (C_H)^d$
ensures that, for \(c>0\) small enough depending on \(\inf_{\|\mathbf{x}\|_2\le R_0} f_0(\mathbf{x})\) and for all sufficiently large \(n\),
\[
|f_1(\mathbf{x})-f_0(\mathbf{x})|
\le \frac12 \inf_{\|\mathbf{x}\|_2\le R_0} f_0(\mathbf{x})
\le \frac12 f_0(\mathbf{x}),
\]
because \(\inf_{\|\mathbf{x}\|_2\le R_0} f_0(\mathbf{x}) > 0\) is a fixed positive constant.

Combining the two regimes, for all sufficiently large \(n\) and all \(\mathbf{x}\in\mathbb R^d\),
$
f_1(\mathbf{x}) = f_0(\mathbf{x}) + [f_1(\mathbf{x})-f_0(\mathbf{x})]
\ge f_0(\mathbf{x}) - f_0(\mathbf{x})/2
= f_0(\mathbf{x})/2 > 0$.
Thus, \(f_1\) is a strictly positive probability density for all sufficiently large \(n\).

\smallskip

The rest of the proof can be split into three steps.\\[5pt]
\noindent{\em{$\bullet$ $f_0$ and $f_1$ belong to the anisotropic Nikol'skii class $\mathcal N_1^d(\boldsymbol{\alpha},\,L)$.}}\\[3pt]
We begin by showing that $f_0 \in \mathcal N_1^d(\boldsymbol\alpha,\, L')$ for some suitable constant $L'>0$. 
Recalling that \(\partial_j^{m_j} f_0\) denotes the partial derivative of $f_0$ of order \(m_j\coloneqq \lceil \alpha_j \rceil -1\) with respect to the \(j\)-th coordinate, for any $\delta\in\mathbb{R}$, we have
\begin{align*}
\bigl\|\Delta_{\delta\mathbf e_j}\partial^{m_j}_j f_0
\bigr\|_{L^1(\mathbb R^d)}
&= \int_{\mathbb{R}^d}\big| \Delta_{\delta\mathbf e_j}
\partial^{m_j}_j f_0
(\mathbf{x})\big| \, \mathrm{d}\mathbf{x} \\
&=\frac{1}{c_j^{m_j}}\int_{\mathbb{R}^d}
\biggl| g^{(m_j)}_{s_j}\biggl(\frac{x_j+\delta}{c_j}\biggr)-g^{(m_j)}_{s_j}\biggl(\frac{x_j}{c_j}\biggr)\biggr| \prod_{\substack{r=1 \\ r \neq j}}^d\frac{1}{c_r}\biggl|g_{s_r}\!\left(\frac{x_r}{c_r}\right)\biggr|\, \mathrm{d}\mathbf{x}
\\&=\frac{1}{c_j^{m_j}}c_j\bigg\| g^{(m_j)}_{s_j}\bigg(\cdot+\frac{\delta}{c_j}\bigg)-g^{(m_j)}_{s_j}(\cdot)\bigg\|_{L^1(\mathbb R)}\underbrace{\prod_{\substack{r=1 \\ r \neq j}}^d\|g_{s_r}\|_{L^1(\mathbb R)}}_{=1}\\
&=\frac{1}{c_j^{m_j-1}}\bigg\| g^{(m_j)}_{s_j}\bigg(\cdot+\frac{\delta}{c_j}\bigg)-g^{(m_j)}_{s_j}(\cdot)\bigg\|_{L^1(\mathbb R)}.
\end{align*}
We prove that, for every \(j=1,\,\dots,\,d\), setting \(\delta_j\coloneqq \delta/c_j\),
\begin{equation}\label{eq:main-est}
\bigl\| g^{(m_j)}_{s_j}(\cdot+\delta_j)-g^{(m_j)}_{s_j}(\cdot)\bigr\|_{L^1(\mathbb R)}\le C(\alpha_j,\,c_j,\,s_j)\,|\delta|^{\alpha_j-m_j}.
\end{equation}
We now distinguish two cases: \(|\delta|\leq 1\) and \(|\delta|>1\).
\begin{itemize}
\item Case \(|\delta|\leq1\). The proof uses the elementary fact that all derivatives of a symmetric stable density are in \(L^1(\mathbb R)\).
Indeed, for \(s\in(0,\,1)\),
\[
g_s^{(k)}(x) \sim C_{s,k} |x|^{-(1+s+k)} \quad \text{as }|x|\to\infty,
\]
and \(g_s^{(k)}\) is bounded near \(0\), so \(g_s^{(k)}\in L^1(\mathbb R)\) for every \(k\ge 0\).
Using the mean-value theorem,
\[
\bigl\| g^{(m_j)}_{s_j}(\cdot+\delta_j)-g^{(m_j)}_{s_j}(\cdot)\bigr\|_{L^1(\mathbb R)}
\le |\delta_j|\, \bigl\| g^{(m_j+1)}_{s_j}\bigr\|_{L^1(\mathbb R)}
= c_j^{-1}\bigl\| g^{(m_j+1)}_{s_j}\bigr\|_{L^1(\mathbb R)}\, |\delta|.
\]
Because \(|\delta|\le 1\) and \(\alpha_j-m_j\le 1\), we have \(|\delta|\le |\delta|^{\alpha_j-m_j}\). Hence
\[
\bigl\| g^{(m_j)}_{s_j}(\cdot+\delta_j)-g^{(m_j)}_{s_j}(\cdot)\bigr\|_{L^1(\mathbb R)}
\le C(\alpha_j,\,c_j,\,s_j)\, |\delta|^{\alpha_j-m_j},
\]
with 
\[
C(\alpha_j,\,c_j,\,s_j)\coloneqq c_j^{-1}\bigl\| g^{(m_j+1)}_{s_j}\bigr\|_{L^1(\mathbb R)} < \infty.
\]
\item Case \(|\delta|>1\). Recall that, for \(s\in(0,\,1)\), 
\[
g_s(x) \sim C_s \, |x|^{-(1+s)}\,\,\,\text{ as }|x|\to\infty,
\]
and, for any integer \(k\ge 0\), the derivatives behave as
\(
g_s^{(k)}(x) \sim C_{s,k} \, |x|^{-(1+s+k)}\) for \(|x|\to\infty\).
This follows from term-by-term differentiation of the asymptotic expansion.
The function \(|x|^{-(1+s+k)}\) is integrable at infinity when \(
1+s+k > 1\), which is always true because \(s>0\) and \(k\ge0\).  
Near the origin, \(g_s^{(k)}\) is continuous and hence bounded on compact sets, so it is locally integrable.
Consequently, \(\|g_s^{(k)}\|_{L^1(\mathbb R)} <\infty\) for every \(k\ge0\).
In particular, \(\|g_{s_j}^{(m_j)}\|_{L^1(\mathbb R)} < \infty\), which yields
\[
\bigl\| g^{(m_j)}_{s_j}(\cdot+\delta_j)-g^{(m_j)}_{s_j}(\cdot)\bigr\|_{L^1(\mathbb R)}
\le 2\bigl\|g_{s_j}^{(m_j)}\bigr\|_{L^1(\mathbb R)}.
\]
Because \(|\delta|>1\) implies \(|\delta|^{\alpha_j-m_j}> 1\), we obtain
\[
\bigl\| g^{(m_j)}_{s_j}(\cdot+\delta_j)-g^{(m_j)}_{s_j}(\cdot)\bigr\|_{L^1(\mathbb R)}
< 2\bigl\|g_{s_j}^{(m_j)}\bigr\|_{L^1(\mathbb R)}\,|\delta|^{\alpha_j-m_j}.
\]
\end{itemize}
Thus, \eqref{eq:main-est} holds for all \(\alpha_j>0\). We conclude that \(f_0\in\mathcal N_1^d(\boldsymbol\alpha,\,L')\), with \[L'\geq \max_{1\leq j\leq d}\max\big\{C(\alpha_j,\,c_j,\,s_j),\,2\|g_{s_j}^{(m_j)}\|_{L^1(\mathbb R)}\big\}.\]

\smallskip

We now prove that \((f_1-f_0)\in\mathcal N_1^d(\boldsymbol\alpha,\,L'')\) for some suitable constant \(L''>0\). For every \(j=1,\,\dots,\,d\), set
\[
H_{j}(x_j)\coloneqq H\!\left(\frac{x_j - 2}{2}\right)
\quad\text{and}\quad
H_{j}^{(m_j)}(x_j)\coloneqq 2^{-m_j} H^{(m_j)}\!\left(\frac{x_j - 2}{2}\right), \quad x_j\in\mathbb R.
\]
Since the partial derivative \(\partial_j^{m_j}\) acts only on the \(j\)-th factor and \(f_1-f_0 = c n^{-1/2} \prod_{r=1}^d H_r(x_r)\), we have
\[
\partial_j^{m_j} (f_1-f_0)(\mathbf{x})
= c n^{-1/2} H_{j}^{(m_j)}(x_j) 
   \prod_{\substack{r=1 \\ r \neq j}}^d H_{r}(x_r),\quad \mathbf{x}\in\mathbb R^d.
\]
For any \(\delta\in \mathbb R\), 
\[
\Delta_{\delta\mathbf e_j}\partial_j^{m_j} (f_1-f_0)(\mathbf{x})
= c n^{-1/2}
\bigl[ H_{j}^{(m_j)}(x_j + \delta) - H_{j}^{(m_j)}(x_j) \bigr]
\prod_{\substack{r=1 \\ r \neq j}}^d H_{r}(x_r),\quad \mathbf{x}\in\mathbb R^d.
\]
Because the variables separate, the \(L^1\)-norm is bounded by
\begin{equation}\label{eq:diff}
\big\|\Delta_{\delta\mathbf e_j}\partial_j^{m_j} (f_1-f_0)\big\|_{L^1(\mathbb R^d)} \le 
c n^{-1/2}
\big\| H_{j}^{(m_j)}(\cdot + \delta) - H_{j}^{(m_j)}(\cdot)\big\|_{L^1(\mathbb R)}
\prod_{\substack{r=1 \\ r \neq j}}^d \| H_{r} \|_{L^1(\mathbb R)}.
\end{equation}
For \(r \neq j\), we have \(\| H_{r} \|_{L^1(\mathbb R)}= 2\|H\|_{L^1(\mathbb R)} <\infty\).
Because \(\mathcal F[H]\) has compact support (condition \ref{B3}) and is \(C^\infty\) (condition \ref{B4}), its inverse Fourier transform \(H\) is smooth, bounded, and decays faster than any polynomial. The same holds for all derivatives \(H^{(k)}\), so \(H^{(k)}\in L^1(\mathbb R)\) for every \(k\ge 0\). For any \(\delta\in\mathbb R\),
\begin{itemize}
\item if \(|\delta|\le 1\),
    \[
    \|H^{(m_j)}(\cdot+\delta)-H^{(m_j)}(\cdot)\|_{L^1(\mathbb R)}
    \le |\delta|\,\|H^{(m_j+1)}\|_{L^1(\mathbb R)}
    \le \|H^{(m_j+1)}\|_{L^1(\mathbb R)} |\delta|^{\alpha_j-m_j},
    \]
    because \(|\delta|\le |\delta|^{\alpha_j-m_j}\);\\
\item if \(|\delta|> 1\),
    \[
    \|H^{(m_j)}(\cdot+\delta)-H^{(m_j)}(\cdot)\|_{L^1(\mathbb R)}
    \le 2\|H^{(m_j)}\|_{L^1(\mathbb R)} < 2\|H^{(m_j)}\|_{L^1(\mathbb R)} |\delta|^{\alpha_j-m_j}.
    \]
\end{itemize}
Thus, there exists \(L_j>0\) such that
\[
\|H^{(m_j)}(\cdot+\delta)-H^{(m_j)}(\cdot)\|_1 \le L_j|\delta|^{\alpha_j-m_j}.
\]
Hence, for each coordinate direction \(j\), the function \(H\) satisfies the one-dimensional Nikol'skii condition 
\(H\in \mathcal N_1^1(\alpha_j,\, L_j)\). For the \(j\)-th factor, the change of variable \(u_j = x_j - 2\) gives
\[\begin{split}
\big\| H_{j}^{(m_j)}(\cdot + \delta) - H_{j}^{(m_j)}(\cdot) \big\|_{L^1(\mathbb R)} 
&= \frac{2}{2^{m_j}} 
   \bigg\| H^{(m_j)}\bigg(\cdot + \frac{\delta}{2}\bigg) - H^{(m_j)}(\cdot) \bigg\|_{L^1(\mathbb R)}\\
&\le \frac{2}{2^{m_j}} \, L_j 
   \biggl(\frac{|\delta|}{2}\biggr)^{\alpha_j - m_j}\\&= L_j 2^{1-\alpha_j} |\delta|^{\alpha_j - m_j}.
\end{split}
\]
Inserting these bounds in \eqref{eq:diff}, we obtain
\[\begin{split}
\bigl\|\Delta_{\delta\mathbf e_j}\partial_j^{m_j} (f_1-f_0)\bigr\|_{L^1(\mathbb R^d)} 
&\le c n^{-1/2} L_j 2^{1-\alpha_j} |\delta|^{\alpha_j - m_j}
   \prod_{\substack{r=1 \\ r \neq j}}^d 2 \|H\|_{L^1(\mathbb R)}\\
&= c  n^{-1/2}2^{d-\alpha_j} L_j \|H\|_{L^1(\mathbb R)}^{d-1}
   |\delta|^{\alpha_j - m_j}.
\end{split}
\]
Define $
D_j \coloneqq c\,2^{d-\alpha_j} L_j \|H\|_{L^1(\mathbb R)}^{d-1}$.
Then, $
\bigl\|\Delta_{\delta\mathbf e_j}\partial_j^{m_j} (f_1-f_0)\bigr\|_{L^1(\mathbb R^d)} 
\le D_j n^{-1/2} |\delta|^{\alpha_j - m_j}$. Since \(n^{-1/2} \le 1\) for \(n\ge 1\), we have
\[
\bigl\|\Delta_{\delta\mathbf e_j}\partial_j^{m_j} (f_1-f_0)\bigr\|_{L^1(\mathbb R^d)} 
\le D_j |\delta|^{\alpha_j - m_j}.
\]
with \(D_j\) independent of \(n\). Therefore, \((f_1-f_0)\in \mathcal N_1^d(\boldsymbol\alpha,\,L'')\), where $L'' \coloneqq \max_{1\le j\le d} D_j < \infty$.
We conclude that \(f_0,\,f_1\in\mathcal N_1^d(\boldsymbol{\alpha},\,L)\), with 
$L \coloneqq L' + L'' < \infty$. This completes the verification that both densities belong to the anisotropic Nikol'skii class with a constant independent of \(n\) when \(\mathcal P = \emptyset\).\\[5pt]

\noindent{\em{$\bullet$ $\chi^2$-divergence.}}\\[3pt]
We recall that the \(\chi^2\)-divergence between two probability densities \(p\) and \(q\) on \(\mathbb R^d\) is defined as
\[
\chi^2(p\,\|\, q)\coloneqq \int_{\mathbb R^d} \frac{[p(\mathbf{x}) - q(\mathbf{x})]^2}{q(\mathbf{x})}\,\mathrm{d}\mathbf{x},
\]
whenever \(q>0\) on the support of \(p\); otherwise \(\chi^2(p\,\|\, q)=\infty\).

Because the noise components are independent (condition \ref{eq:E1indep} \hyperref[sec:noise]{of the Main Article}), the convolution factorizes coordinate‑wise. 
Condition \ref{it:E22}, together with condition \ref{eq:E3} \hyperref[sec:noise]{of the Main Article}, implies that 
$\varphi_{\varepsilon_j}$ is continuously differentiable and that its derivative satisfies the bound 
$|\varphi^{(1)}_{\varepsilon_j}(t)| = |\varphi_{\varepsilon_j}(t)|^2\, |r^{(1)}_{\varepsilon_j}(t)|
\le c_{1,j}^2\, d_{1,j}\, (1+|t|)^{-\beta_j}
$, $t\in\mathbb R$,
which is precisely condition \eqref{eq:condder} with $\beta'_j=\beta_j$ and $B_j=c_{1,j}^2d_{1,j}$.
Using Lemma~\ref{lem:chiquad}, which gives a bound for each coordinate, it is shown in Step 2 of the proof of Theorem~3 in \cite{ComteLacour2013}, p.~601, that
the $\chi^2$-divergence between the convolved densities \(f_1\ast f_{\varepsilon}\) and \(f_0\ast f_{\varepsilon}\) is of order $O(n^{-1})$. Specifically, 
\[\chi^2\bigl(f_1\ast f_{\varepsilon}\,\|\,f_0\ast f_{\varepsilon}\bigr)\lesssim 
\frac 1n\prod_{j=1}^d h_j^{-(2\beta_j+1)} h_j^{2\beta_j+1}=O(n^{-1}).\]



\noindent{\em{$\bullet$ Separation rate and lower bound rate}}\\[3pt]
Let \(F_i\) denote the cumulative distribution function associated with the density \(f_i\), for \(i = 0,\,1\). 
The following two facts will be used to establish the separation between \(F_0\) and \(F_1\), which is a key ingredient in the lower bound proof.
\begin{enumerate}[label=\textbf{Fact \arabic*.}, leftmargin=*]
\item The absolute difference between the cumulative distribution functions \(F_1\) and \(F_0\) at a point \(\mathbf{x}\) satisfies
\begin{align}
\bigl|F_1(\mathbf{x})-F_0(\mathbf{x})\bigr|
&= c n^{-1/2}\prod_{j=1}^d 
   \left|{\int_{-\infty}^{x_j} H\!\left(\frac{u-2}{2}\right)\mathrm{d}u}\right| \nonumber\\
&= c n^{-1/2} 2^d \prod_{j=1}^d
    \left|{\int_{-\infty}^{\frac{x_j-2}{2}} H(z)\,\mathrm{d}z}\right| \nonumber\\
&= c n^{-1/2} 2^d
   \prod_{j=1}^d
   \left|{H^{(-1)}\!\left(\frac{x_j-2}{2}\right)}\right|,
\quad \mathbf{x}\in\mathbb R^d,\label{eq:diffcdfs}
\end{align}
where we used the change of variable \(z = (u-2)/2\) in each coordinate.\\

\item For the change of variable \(z = (x_j - 2)/2\), we have
\[
\int_{0}^{2} \left|{H^{(-1)}\!\left(\frac{x_j - 2}{2}\right)}\right|\,\mathrm{d}x_j 
= 2 \int_{-1}^{0} |H^{(-1)}(z)|\,\mathrm{d}z
= 2 \int_{0}^{1} |H^{(-1)}(z)|\,\mathrm{d}z > 0,
\]
where the last equality follows from the oddness of \(H^{(-1)}\), which is a consequence of condition \ref{B1} and the evenness of \(H\). The strict positivity is guaranteed by condition \ref{B2}.
\end{enumerate}

Let \(\hat{F}_n\) be any estimator of the cumulative distribution function based on an i.i.d. sample \(\mathbf{Y}_1,\,\dots,\,\mathbf{Y}_n\) drawn from either \(f_1 \ast f_{\varepsilon}\) or \(f_0 \ast f_{\varepsilon}\) (where \(f_{\varepsilon}\) is the noise density). We have already shown that $
\chi^2(f_1 \ast f_{\varepsilon} \,\|\, f_0 \ast f_{\varepsilon}) = O(n^{-1})$,
with a constant that can be made arbitrarily small by choosing \(c>0\) sufficiently small. 
Since the constant in the \(\chi^2\)-divergence bound is proportional to \(c^2\), choosing \(c>0\) sufficiently small ensures that $
\chi^2(f_1 \ast f_\varepsilon \parallel f_0 \ast f_\varepsilon) \le {\eta}{n}^{-1}$
for any prescribed \(\eta > 0\). Then, by tensorization of the \(\chi^2\)-divergence and Pinsker's inequality,
\[
\| (\mu_1 \ast \mu_\varepsilon)^{\otimes n} - (\mu_0 \ast \mu_\varepsilon)^{\otimes n} \|_{\mathrm{TV}}
\le \sqrt{\frac12 \left[ \left(1 + \frac{\eta}{n}\right)^n - 1 \right]}
\to \sqrt{\frac12 (e^\eta - 1)}.
\]
Since \(\eta > 0\) can be chosen arbitrarily small, the total variation distance can be made arbitrarily close to \(0\), uniformly in \(n\), by taking \(c>0\) sufficiently small. Hence, Le Cam's two-point method applies.
Hence, for \(n\) large enough,
\[
\int_{(\mathbb R^d)^n} \min\bigl\{(f_1 \ast f_\varepsilon)^{\otimes n},\,(f_0 \ast f_\varepsilon)^{\otimes n}\bigr\}(\mathbf{y}_1,\,\ldots,\,\mathbf{y}_n)\,
\mathrm{d}\mathbf{y}_1\,\ldots\,\mathrm{d}\mathbf{y}_n \ge \frac12.
\]
Therefore, by Le Cam's two-point method, the minimax risk is bounded below by
\begin{align*}
&\inf_{\hat{F}_n}\sup_{f_X\in\mathcal{N}_1^d(\boldsymbol{\alpha},L)} 
\mathsf E_{(\mu_X\otimes\mu_\varepsilon)^{\otimes n}}[\|\hat{F}_n-F_X\|_{L^1(\mathbb R^d)}]\\
&\hspace*{1.5cm}\ge \frac12 \int_{\mathbb R^d} |F_1(\mathbf{x})-F_0(\mathbf{x})| \, \mathrm{d}\mathbf{x} \\
&\hspace{3.5cm} \times \int_{(\mathbb R^d)^n} \min\{(f_1 \ast f_\varepsilon)^{\otimes n},\,(f_0 \ast f_\varepsilon)^{\otimes n}\}(\mathbf{y}_1,\,\ldots,\,\mathbf{y}_n)\,
\mathrm{d}\mathbf{y}_1\,\ldots\,\mathrm{d}\mathbf{y}_n \\
&\hspace*{1.5cm}\ge \frac14 \int_{\mathbb R^d} |F_1(\mathbf{x})-F_0(\mathbf{x})| \, \mathrm{d}\mathbf{x}.
\end{align*}
Using Fact 1 and Fact 2, we obtain
\[
\begin{aligned}
\int_{\mathbb R^d} |F_1(\mathbf{x})-F_0(\mathbf{x})| \, \mathrm{d}\mathbf{x}
&= c n^{-1/2} 2^d
   \prod_{j=1}^d \int_{\mathbb R}
   \left|{H^{(-1)}\!\left(\frac{x_j-2}{2}\right)}\right| \mathrm{d}x_j \\
&\geq c n^{-1/2} 2^d
   \prod_{j=1}^d
   \left( 2 \int_{0}^{1} |H^{(-1)}(z)|\,\mathrm{d}z \right) \\
&= c n^{-1/2} 2^{2d}
   \left( \int_{0}^{1} |H^{(-1)}(z)|\,\mathrm{d}z \right)^d\gtrsim n^{-1/2}
\end{aligned}
\]
because \(\int_0^1 |H^{(-1)}(z)|\,\mathrm{d}z>0\) by condition \ref{B2}, and \(c>0\) is fixed.
Combining this with the lower bound on the testing integral, we have that
\[
\inf_{\hat{F}_n} \sup_{f_X \in \mathcal{N}^d_1(\boldsymbol{\alpha},L)} 
\mathsf{E}_{(\mu_X \ast \mu_\varepsilon)^{\otimes n}}[\|\hat{F}_n - F_X\|_{L^1(\mathbb R^d)}] 
\gtrsim n^{-1/2}.
\]
Since when \(\mathcal P = \emptyset\) we have \(\psi_n^{\boldsymbol{\alpha},\boldsymbol{\beta}} = n^{-1/2}\), this yields
\[
\inf_{\hat{F}_n} \sup_{f_X \in \mathcal{N}^d_1(\boldsymbol{\alpha},L)} 
\mathsf{E}_{(\mu_X \ast \mu_\varepsilon)^{\otimes n}}[\|\hat{F}_n - F_X\|_{L^1(\mathbb R^d)}] 
\gtrsim \psi_n^{\boldsymbol{\alpha},\boldsymbol{\beta}},
\]
which completes the proof for the case \(\mathcal P = \emptyset\).\\[5pt]


\noindent\textbf{Case $|\mathcal P| \ge 1$.}\\[-10pt]
\begin{enumerate}[label=\bf{$\bullet$}]
\item \emph{Construction of auxiliary functions $\omega$ and $\xi$.}\\[2pt]
Let $\omega \in C_c^\infty(\mathbb R)$ be any function satisfying the following conditions:\\[-8pt]
\begin{enumerate}[label=(\bf{W\arabic*}), align=left, leftmargin=*, itemsep=0pt]
    \item \label{it:W1} $\int_{\mathbb R} \omega(x) \,\di x = 1$;\\[-8pt]
    \item \label{it:W2} $\operatorname{supp}(\omega) \subset [-1, 1]$;\\[-8pt]
    \item \label{it:W3} $\omega(x) \ge 0$ for all $x \in \mathbb R$;\\[-8pt]
    \item \label{it:W4} there exist constants $0 < \gamma_\omega < 1$ and $c_\omega > 0$ such that $\omega(x) \ge c_\omega$ for $|x| \le \gamma_\omega$.\\[-8pt]
\end{enumerate}
We exhibit an example of such a function $\omega$. Let $\psi \in C_c^\infty(\mathbb R)$ be defined by
\[
\psi(x) \coloneqq
\begin{cases}
\exp\left(-\frac{1}{1-x^2}\right), & |x| < 1, \\[-5pt]
0, & |x| \ge 1.
\end{cases}
\]
Let $\omega(x) \coloneqq \psi(x)/\int_{-1}^1 \psi(s)\,\di s$, $x\in\mathbb R$. Properties \ref{it:W1}--\ref
{it:W3} are readily verified. To verify \ref{it:W4}, for $|x| \le 3/4$,
\[
\omega(x) = \frac{\exp\left(-\frac{1}{1-x^2}\right)}{\int_{-1}^1 \psi(s)\,\di s} \ge \frac{\mathrm{e}^{-16/7}}{\int_{-1}^1 \psi(s)\,\di s} \eqqcolon c_\omega > 0.
\]
Thus, we may take $\gamma_\omega = 3/4$.

\smallskip
\noindent
Let $\xi \in C_c^\infty(\mathbb R)$ be any function satisfying the following conditions:\\[-8pt]
\begin{enumerate}[label=(\bf{U\arabic*}), align=left, leftmargin=*, itemsep=0pt]
    \item \label{it:U1} $\operatorname{supp}(\xi) \subset [-\gamma_\xi,\, \gamma_\xi] \subset (-\gamma_\omega,\, \gamma_\omega)$, with $0 < \gamma_\xi < \gamma_\omega < 1$;\\[-8pt]
    \item \label{it:U2} $\xi(x) \ge 0$ for $x \in [1/2,\, 1/2+\eta]$, for some $\eta > 0$ satisfying $1/2+\eta < \gamma_\xi$;\\[-8pt]
    \item \label{it:U3} $\int_{\mathbb R} \xi(x) \,\di x = 0$;\\[-8pt]
    \item \label{it:U4} $\int_{\mathbb R} x^k \xi(x) \,\di x = 0$ for $k = 1,\, \dots,\, K$, for some integer $K \geq0$;\\[-8pt]
    \item \label{it:U5} $\int_{-\infty}^x \xi(s)\,\di s \ge c_\Xi > 0$ for all $x \in [1/2,\, 1/2+\eta]$, for some constant $c_\Xi > 0$.\\[-8pt]
\end{enumerate}
We show how such a function $\xi$ can be constructed. Choose $\eta > 0$ such that $1/2 + \eta < 3/4$.
Let $\phi_1,\, \phi_2,\,\phi_3 \in C_c^\infty(\mathbb R)$ be non-negative bump functions satisfying
\[
\begin{aligned}
\operatorname{supp}(\phi_1) &\subset [1/4,\, 1/2 - \epsilon], & \|\phi_1\|_{L^1(\mathbb R)} &= 1, \\
\operatorname{supp}(\phi_2) &\subset [1/2,\, 1/2 + \eta], & \|\phi_2\|_{L^1(\mathbb R)} &= 1, \quad \phi_2 \ge 0, \\
\operatorname{supp}(\phi_3) &\subset [1/2+\eta,\, 3/4], &  \|\phi_3\|_{L^1(\mathbb R)} &= 1,
\end{aligned}
\]
where $0 < \epsilon < 1/4$ is small. Define the preliminary function
$\xi_0(x) \coloneqq a\phi_1(x) + b\phi_2(x) - (a+b)\phi_3(x)$,
where $a,\, b > 0$ are constants to be chosen later.
Let $K \geq0$. Choose distinct points $x_0,\,\dots,\, x_K \in [1/4,\,3/4] \setminus [1/2,\, 1/2+\eta]$ and bump functions $\zeta_k \in C_c^\infty(\mathbb R)$ with sufficiently small disjoint supports centered at $x_k$, contained in $[1/4,\,3/4] \setminus [1/2,\, 1/2+\eta]$, such that $\int_{\mathbb R} \zeta_k(x)\,\di x = 1$ and the matrix $
V_{j,k} \coloneqq \int_{\mathbb R} x^j \zeta_k(x)\,\di x$, $j,\,k = 0,\,\dots,\,K$,
is invertible (this holds, for instance, if the supports are sufficiently small and the points $x_k$ are distinct, since $V$ is then close to a Vandermonde matrix).
Define $
\xi(x) = \xi_0(x) + \sum_{k=0}^K \theta_k \zeta_k(x)$,
where the constants $\theta_k$ are chosen to satisfy the moment conditions
$
\int_{\mathbb R} x^j \xi(x)\,\di x = 0$, $j = 0,\, \dots,\, K$.
Since the matrix $V$ is invertible, this linear system has a unique solution $(\theta_0,\,\dots,\,\theta_K)$.
It remains to choose $a,\,b > 0$. Observe that the solution $\theta_k$ is linear in $a$ and $b$, \emph{i.e.}, $\theta_k = a \theta_k^{(1)} + b \theta_k^{(2)}$ for constants $\theta_k^{(1)},\, \theta_k^{(2)}$ independent of $a,\,b$. Since the supports of $\zeta_k$ are contained in $[1/4,\,3/4] \setminus [1/2,\, 1/2+\eta]$, the correction terms do not affect $\xi$ on $[1/2,\, 1/2+\eta]$. Moreover, for $x \in [1/2,\, 1/2+\eta]$, we have
\[
\Xi(x) = \int_{-\infty}^x \xi(s)\,\di s = a + b\int_{1/2}^x \phi_2(s)\,\di s + \sum_{k: x_k < x} \theta_k \int_{-\infty}^x \zeta_k(s)\,\di s.
\]
Choose $b > 0$ first, then choose $a > 0$ sufficiently large (\emph{e.g.}, $a = b^3$) so that the positive contribution $a$ dominates any negative contributions from the correction terms. This is possible because the coefficients $\theta_k^{(1)}$ and $\theta_k^{(2)}$ are fixed, and the supports of $\zeta_k$ can be chosen arbitrarily small; we may first choose the supports so small that the total variation of each correction term is less than, say, $1/(K+1)$. Then, choosing $a = b^3$ ensures that $a$ dominates the correction terms. This guarantees that $\Xi(x) \ge c_\Xi > 0$ for all $x \in [1/2,\, 1/2+\eta]$. Also, since $\xi(x) = b\phi_2(x) \ge 0$ on $[1/2,\, 1/2+\eta]$, positivity is automatic.
Finally, taking $\gamma_\xi \in (1/2+\eta,\, 3/4)$, we have $\operatorname{supp}(\xi) \subset [-\gamma_\xi,\, \gamma_\xi] \subset (-\gamma_\omega,\, \gamma_\omega)$ since $\gamma_\omega = 3/4$. Thus, all conditions \ref{it:U1}--\ref{it:U5} are satisfied.

\smallskip

Define the cumulative distribution functions
\[
\Omega(x) \coloneqq \int_{-\infty}^x \omega(s)\,\di s \quad\text{and}\quad \Xi(x) \coloneqq \int_{-\infty}^x \xi(s)\,\di s.
\]
The function $\Omega$ is continuous and strictly increasing on $(-1,\,1)$, and there exist constants $c_\Omega,\,c_\Xi > 0$ such that 
\[
\Omega(x) \ge c_\Omega > 0 \quad\text{and}\quad \Xi(x) \ge c_\Xi > 0 \quad\text{for all } x \in [1/2,\, 1/2+\eta].
\]


\item \emph{Construction of the density with inactive coordinates $f_0^{\mathrm{inactive}}$.}\\[2pt]
Choose $\gamma_M < 1/2$ and let $\rho \in C_c^\infty(\mathbb R)$ be a bump function such that $\rho \ge 0$, $\operatorname{supp}(\rho) \subset [-\gamma_M/8, \,\gamma_M/8]$, and $\int_{\mathbb R} \rho(x)\,\di x = 1$.
For $\tau > 0$, define $\rho_\tau(x) \coloneqq \tau^{-1}\rho(x/\tau)$. We have $\operatorname{supp}(\rho_\tau) \subset [-\tau, \,\tau]$.
For each coordinate $j = 1,\, \dots,\, d-|\mathcal P|$, define
\[
\rho_j(x)\coloneqq  \left(\mathbf{1}_{\left[-\frac{\gamma_M}{2}M + \tau,\,\frac{\gamma_M}{2}M - \tau\right]} \ast \rho_\tau\right)(x),\quad x\in\mathbb R,
\]
with $\tau < (\gamma_M/8)M$. 
Since $\operatorname{supp}(\rho_\tau) \subset [-\tau, \,\tau]$, the integral
\[
\rho_j(x) = \int_{x - \frac{\gamma_M}{2}M + \tau}^{x + \frac{\gamma_M}{2}M - \tau} \rho_\tau(u) \, \mathrm d u
\]
is nonzero only if the interval of integration intersects $[-\tau,\, \tau]$. This occurs if and only if
$x - ({\gamma_M}/{2})M + \tau \le \tau$ and $ x + ({\gamma_M}/{2})M - \tau \ge -\tau$,
which gives $|x| \le ({\gamma_M}/{2})M$.
Hence, $
\operatorname{supp}(\rho_j) \subset [-({\gamma_M}/{2})M, \, ({\gamma_M}/{2})M]$.
Let
\[
f_0^{\mathrm{inactive}}(\mathbf x_{\setminus\mathcal P}) \coloneqq \frac{1}{(\gamma_M M - 2\tau)^{d-|\mathcal P|}} \prod_{j=1}^{d-|\mathcal P|} \rho_j(x_j).
\]
It is easy to verify that $f_0^{\mathrm{inactive}} \in C_c^\infty(\mathbb R^{d-|\mathcal P|})$ and 
$
\int_{\mathbb R^{d-|\mathcal P|}} f_0^{\mathrm{inactive}}(\mathbf x_{\setminus\mathcal P})\,\di\mathbf x_{\setminus\mathcal P} = 1$. Also, 
$\operatorname{supp}(f_0^{\mathrm{inactive}}) \subset [-({\gamma_M}/{2})M,\, ({\gamma_M}/{2})M]^{d-|\mathcal P|}$. 
Besides, for $x \in [-({\gamma_M}/{4})M, \,({\gamma_M}/{4})M]$, we have
\[
\rho_j(x) = \int_{x - \frac{\gamma_M}{2}M + \tau}^{x + \frac{\gamma_M}{2}M - \tau} \rho_\tau(u) \, \mathrm d u=1.
\]
Consequently,
\[\begin{split}
f_0^{\mathrm{inactive}}(\mathbf x_{\setminus\mathcal P}) &= \frac{1}{(\gamma_M M - 2\tau)^{d-|\mathcal P|}} \prod_{j=1}^{d-|\mathcal P|} \rho_j(x_j)\\
&= \frac{1}{(\gamma_M M - 2\tau)^{d-|\mathcal P|}}\eqqcolon c_0>0, \quad \mathbf x_{\setminus\mathcal P} \in\Bigl[-\frac{\gamma_M}{4} M,\, \frac{\gamma_M}{4} M\Bigr]^{d-|\mathcal P}.
\end{split}
\]
Moreover, $f_0^{\mathrm{inactive}} \in \mathcal N_1^{d-|\mathcal P|}(\boldsymbol\alpha_{\setminus\mathcal P},\, L_0)$ for some $L_0 > 0$. This follows from the following Nikol'skii-type estimate: for any $f \in C_c^\infty(\mathbb R)$, any integer $q \ge 0$, and any $r > q$ such that $0 < r - q \le 1$, there exists $C > 0$ such that
\begin{equation}\label{eq:holderinity}
\int_{\mathbb R} |f^{(q)}(u+h) - f^{(q)}(u)|\,\di u \le C |h|^{r - q}
\quad \forall\,h \in \mathbb R.
\end{equation}
Indeed, let $D := \operatorname{supp}(f)$. If $|h| \le 1$, the Mean Value Theorem gives
\[
|f^{(q)}(u+h) - f^{(q)}(u)| \le \|f^{(q+1)}\|_{\infty} |h|,
\]
and the difference is supported in $D \cup (D-h)$, whose measure is bounded by a constant depending only on $f$, yielding
\[
\int_{\mathbb R} |f^{(q)}(u+h) - f^{(q)}(u)|\,\di u \le C_f \|f^{(q+1)}\|_{\infty} |h| \le C_f \|f^{(q+1)}\|_{\infty} |h|^{r-q},
\]
since $|h| \le 1$ and $r-q \le 1$.
If $|h| > 1$, then
\[
\int_{\mathbb R} |f^{(q)}(u+h) - f^{(q)}(u)|\,\di u \le 2\|f^{(q)}\|_{L^1(\mathbb R)} \le 2\|f^{(q)}\|_{L^1(\mathbb R)} |h|^{r-q},
\]
since $|h|^{r-q} \ge 1$.
Combining the two cases yields the desired estimate. Applying this estimate to each $\rho_j$ with $q_j = \lceil \alpha_j \rceil - 1$ and using the product structure of $f_0^{\mathrm{inactive}}$ gives that $f_0^{\mathrm{inactive}}\in \mathcal N_1^{d-|\mathcal P
|}(\boldsymbol\alpha_{\setminus\mathcal P},\, L_0)$.\\

\item \emph{Definition of the test densities $f_0$ and $f_1$.}\\[2pt]
For each $j \in \mathcal P$, let the scaled functions
\[
\omega_j(x) \coloneqq R_n^{-(\alpha_j+1)} \omega(R_n^{-(\alpha_j+1)}x)
\quad \text{and} \quad
\xi_j(x) \coloneqq R_n^{-(\alpha_j+1)} \xi(R_n^{-(\alpha_j+1)}x),\quad x\in\mathbb R,
\]
and their cumulative distribution functions
\[
\Omega_j(x) \coloneqq \int_{-\infty}^x \omega_j(u)\,\di u = \Omega(R_n^{-(\alpha_j+1)}x)
\,\,\, \text{ and } \,\,\,
\Xi_j(x) \coloneqq \int_{-\infty}^x \xi_j(u)\,\di u = \Xi(R_n^{-(\alpha_j+1)}x),\quad x\in\mathbb R.
\]
We can now define the test densities. For every $\mathbf x\in\mathbb R^d$,
\begin{equation}\label{eq:testdensities}
f_0(\mathbf x) \coloneqq \prod_{j\in\mathcal P} \omega_j(x_j)  f_0^{\mathrm{inactive}}(\mathbf x_{\setminus\mathcal P})
\,\, \text{ and } \,\,
f_1(\mathbf x) \coloneqq \prod_{j\in\mathcal P} [\omega_j(x_j) + \lambda_n \xi_j(x_j)]  f_0^{\mathrm{inactive}}(\mathbf x_{\setminus\mathcal P}),
\end{equation}
for some non-negative sequence $(\lambda_n)_{n\in\mathbb N}$ to be chosen later.
The product structure of the densities is preserved when passing to cumulative distribution functions because integration factorizes over the coordinates. Consequently, for every $\mathbf x\in\mathbb R^d$, the cumulative distribution functions are given by
\[
F_0(\mathbf x) \coloneqq \prod_{j\in\mathcal P} \Omega_j(x_j)  F_0^{\mathrm{inactive}}(\mathbf x_{\setminus\mathcal P})
\quad \text{and} \quad
F_1(\mathbf x) \coloneqq \prod_{j\in\mathcal P} [\Omega_j(x_j) + \lambda_n \Xi_j(x_j)]  F_0^{\mathrm{inactive}}(\mathbf x_{\setminus\mathcal P}),
\]
where $F_0^{\mathrm{inactive}}$ is the cumulative distribution function of $f_0^{\mathrm{inactive}}$.\\

\item \emph{Verification that $f_0$ and $f_1$ are probability densities.}\\[2pt]
We now verify that, for sufficiently large $n$, the functions $f_0$ and $f_1$ are probability densities.
Regarding $f_0$, since $f_0$ is a product of non-negative functions, we have $f_0 \ge 0$. Moreover, by Fubini's theorem,
\[
\int_{\mathbb R^d} f_0(\mathbf x)\,\di \mathbf x 
= \prod_{j\in\mathcal P} \left(\int_{\mathbb R} \omega_j (x_j)\,\di x_j\right) 
\left(\int_{\mathbb R^{d-|\mathcal P|}} f_0^{\mathrm{inactive}} (\mathbf x_{\setminus\mathcal P})\,\di \mathbf x_{\setminus\mathcal P}\right) = 1,
\]
because, by construction, $\int_{\mathbb R} \omega_j (x_j)\,\di x_j = 1$ for each $j \in \mathcal P$, and $f_0^{\mathrm{inactive}}$ is a density, hence it integrates to one. Therefore, $f_0$ is a probability density.

For $f_1$, we first note that, since $\int_{\mathbb R} \xi_j(x_j)\,\di x_j = 0$ by construction,
\[
\int_{\mathbb R} [\omega_j(x_j) + \lambda_n \xi_j(x_j)]\,\di x_j
= \int_{\mathbb R} \omega_j(x_j)\,\di x_j + \lambda_n \int_{\mathbb R} \xi_j(x_j)\,\di x_j
= 1.
\]
Therefore, again by Fubini's theorem,
\[
\int_{\mathbb R^d} f_1(\mathbf x)\,\di \mathbf x
= \prod_{j\in\mathcal P} \left(\int_{\mathbb R} [\omega_j(x_j) + \lambda_n \xi_j(x_j)]\,\di x_j\right)
\left(\int_{\mathbb R^{d-|\mathcal P|}} f_0^{\mathrm{inactive}} (\mathbf x_{\setminus\mathcal P})\,\di \mathbf x_{\setminus\mathcal P}\right)
= 1.
\]
It remains to verify that $f_1 \ge 0$. Since $\operatorname{supp}(\xi) \subset [-\gamma_\xi,\, \gamma_\xi] \subset (-\gamma_\omega,\, \gamma_\omega)$ and $\omega \ge c_\omega > 0$ on $[-\gamma_\omega,\, \gamma_\omega]$, it follows that $\omega_j(x) + \lambda_n \xi_j(x) \ge R_n^{-(\alpha_j+1)}(c_\omega - \lambda_n \|\xi\|_\infty)$ for all $x \in \operatorname{supp}(\xi_j)$.
Because $\lambda_n \to 0$ as $n \to \infty$, for sufficiently large $n$, we have $\lambda_n \|\xi\|_\infty \le c_\omega/2$. Hence, $\omega_j(x) + \lambda_n \xi_j(x) \ge R_n^{-(\alpha_j+1)}c_\omega/2 > 0$ on $\operatorname{supp}(\xi_j)$, while outside $\operatorname{supp}(\xi_j)$ we have $\omega_j(x) + \lambda_n \xi_j(x) = \omega_j(x) \ge 0$. Therefore, for each $j \in \mathcal P$, we have $\omega_j(x_j) + \lambda_n \xi_j(x_j) \ge 0$ for all $x_j \in \mathbb R$. Thus, $f_1(\mathbf x) \ge 0$ for all $\mathbf x \in \mathbb R^d$, since it is a product of non-negative factors. Consequently, $f_1$ is also a probability density for sufficiently large $n$.\\

\item \emph{Verification that $f_0, \,f_1 \in \mathcal N_1^d(\boldsymbol\alpha, L)$.}\\[2pt]
We now show that, for sufficiently large $n$, both densities $f_0,\, f_1$ belong to $\mathcal N_1^d(\bm\alpha,\, L)$ for some constant $L > 0$ independent of $n$. We need to verify conditions \ref{it:diff}, \ref{it:L1}, and \ref{it:smooth} from Definition \ref{def:anisotropy} \hyperref[def:anisotropy]{of the Main Article} for both $f_0$ and $f_1$.

We begin by verifying condition \ref{it:diff} \hyperref[it:diff]{of the Main Article}. Since $\omega,\,\xi \in C_c^\infty(\mathbb R)$, the scaled functions $\omega_j$ and $\xi_j$ are also $C^\infty(\mathbb R)$, so $ (\omega_j + \lambda_n \xi_j) \in C^\infty(\mathbb R)$ for each $j \in \mathcal P$. By assumption, $f_0^{\mathrm{inactive}}\in \mathcal N_1^{d-|\mathcal P|}(\bm\alpha_{\setminus\mathcal P},\, L_0)$, so it admits partial derivatives up to order $m_j \coloneqq \lceil \alpha_j \rceil - 1$ in each inactive coordinate. Since the product structure separates active and inactive coordinates, $f_0$ and $f_1$ admit all partial derivatives up to order $m_j$ in each coordinate direction. Thus, condition \ref{it:diff} \hyperref[it:diff]{of the Main Article} holds.

We now consider the verification of \ref{it:L1} \hyperref[it:L1]{of the Main Article}. For each $j \in \mathcal P$ and each $k = 0,\,\dots,\, m_j$, we consider the partial derivative
\[
\partial_j^k f_1(\mathbf x) = g_j^{(k)}(x_j) \prod_{\ell \in \mathcal P \setminus \{j\}} g_\ell(x_\ell) f_0^{\mathrm{inactive}}(\mathbf x_{\setminus\mathcal P}),\quad\mathbf x\in\mathbb R^d, 
\]
where $g_j^{(k)}$ denotes the $k$-th derivative of $g_j\coloneqq \omega_j+\lambda_n\xi_j$.
Now, we need a uniform bound on $\|\partial_j^k f_1\|_{L^1(\mathbb R^d)}$. By Fubini's theorem, and since $f_0^{\mathrm{inactive}}$ is a density, for sufficiently large $n$ such that $\lambda_n \le 1$, we have
\[\begin{aligned}
\|\partial_j^k f_1\|_{L^1(\mathbb R^d)}
&= \left(\int_{\mathbb R} |g_j^{(k)}(x_j)|\,\di x_j\right)
\prod_{\ell \in \mathcal P \setminus \{j\}} \left(\int_{\mathbb R} |g_\ell(x_\ell)|\,\di x_\ell\right)
\left(\int_{\mathbb R^{d-|\mathcal P|}} |f_0^{\mathrm{inactive}}(\mathbf x_{\setminus\mathcal P})|\,\di\mathbf x_{\setminus\mathcal P}\right)\\
&\leq \left(R_n^{-k(\alpha_j+1)} 
\int_{\mathbb R} |\omega^{(k)}(u) + \lambda_n \xi^{(k)}(u)|\,\di u\right)(1 + \lambda_n \|\xi\|_{L^1(\mathbb R)})^{|\mathcal P|-1}\\
&\leq R_n^{-k(\alpha_j+1)}\bigl(\|\omega^{(k)}\|_{L^1(\mathbb R)} + \|\xi^{(k)}\|_{L^1(\mathbb R)}\bigr)\bigl(1 + \|\xi\|_{L^1(\mathbb R)}\bigr)^{|\mathcal P|-1}\\
&\leq \bigl(\|\omega^{(k)}\|_{L^1(\mathbb R)} + \|\xi^{(k)}\|_{L^1(\mathbb R)}\bigr)\bigl(1 + \|\xi\|_{L^1(\mathbb R)}\bigr)^{|\mathcal P|-1}<\infty,
\end{aligned}
\]
where in the last inequality we used the fact that $R_n\geq1$, so $R_n^{-k(\alpha_j+1)}\leq 1$.
Thus, the bound is independent of $n$. A similar reasoning leads to an analogous bound for $f_0$ with $\lambda_n=0$. For inactive coordinates $j \notin \mathcal P$, since $f_0^{\mathrm{inactive}} \in \mathcal N_1^{d-|\mathcal P|}(\bm\alpha_{\setminus\mathcal P},\, L_0)$, we already have $
\|\partial_j^k f_0^{\mathrm{inactive}}\|_{L^1(\mathbb R^{d-|\mathcal P|})} \le L_0$,
and since the active factors integrate to a constant independent of $n$, the same uniform bound holds for the full density.
Thus, condition \ref{it:L1} \hyperref[it:L1]{of the Main Article} holds with some constant $L$ independent of $n$. 

We now verify condition \ref{it:smooth} \hyperref[it:smooth]{of the Main Article}. For each $j \in \mathcal P$ and each $\delta \in \mathbb R$, we need to bound
$\|\Delta_{\delta\mathbf{e}_j}\partial_j^{m_j} f_r\|_{L^1(\mathbb R^d)}$ for $r=0,\,1$. We prove the condition for $f_1$; the proof for $f_0$ is analogous (with $\lambda_n = 0$). Reasoning as before, we get 
\[\begin{aligned}
\|\Delta_{\delta\mathbf{e}_j}\partial_j^{m_j} f_1\|_{L^1(\mathbb R^d)}
&= \left(\int_{\mathbb R} \big|g_j^{(m_j)}(x + \delta) - g_j^{(m_j)}(x)\big|\,\di x\right)\\
&\qquad\qquad\qquad\times 
\prod_{\ell \in \mathcal P \setminus \{j\}} \left(\int_{\mathbb R} |g_\ell(x_\ell)|\,\di x_\ell\right)
\left(\int_{\mathbb R^{d-|\mathcal P|}} |f_0^{\mathrm{inactive}}|\,\di \mathbf x_{\setminus\mathcal P}\right)\\
&\leq \left(\int_{\mathbb R} \big|g_j^{(m_j)}(x + \delta) - g_j^{(m_j)}(x)\big|\,\di x\right)\bigl(1 + \|\xi\|_{L^1(\mathbb R)}\bigr)^{|\mathcal P|-1}.
\end{aligned}
\]
It remains to bound the first factor. We have
\[
\begin{aligned}
\int_{\mathbb R}\big|g_j^{(m_j)}(x+\delta) - g_j^{(m_j)}(x)\big|\,\di x &= R_n^{-(\alpha_j+1)(m_j+1)}\\&
 \qquad \times \int_{\mathbb R} \Big|
\omega^{(m_j)}(R_n^{-(\alpha_j+1)}(x+\delta)) - \omega^{(m_j)}(R_n^{-(\alpha_j+1)}x) \\
&\qquad\quad + \lambda_n \bigl[ \xi^{(m_j)}(R_n^{-(\alpha_j+1)}(x+\delta)) - \xi^{(m_j)}(R_n^{-(\alpha_j+1)}x) \bigr] \Big|\,\di x \\
&\le R_n^{-(\alpha_j+1)(m_j+1)}\\
&\qquad\times  \left[
\int_{\mathbb R} \big|\omega^{(m_j)}(R_n^{-(\alpha_j+1)}(x+\delta)) - \omega^{(m_j)}(R_n^{-(\alpha_j+1)}x)\big |\,\di x \right. \\
&\qquad\quad \left. + \lambda_n \int_{\mathbb R} \big|\xi^{(m_j)}(R_n^{-(\alpha_j+1)}(x+\delta)) - \xi^{(m_j)}(R_n^{-(\alpha_j+1)}x)\big|\,\di x
\right]\\
&=R_n^{-(\alpha_j+1)m_j} \left[
\int_{\mathbb R} \big|\omega^{(m_j)}(u + R_n^{-(\alpha_j+1)}\delta) - \omega^{(m_j)}(u)\big|\,\di u \right. \\
&\qquad\qquad\qquad\quad \left. + \lambda_n \int_{\mathbb R}\big|\xi^{(m_j)}(u + R_n^{-(\alpha_j+1)}\delta) - \xi^{(m_j)}(u)\big|\,\di u
\right],
\end{aligned}
\]
where we used the change of variable $u = R_n^{-(\alpha_j+1)}x$.
Since $\omega,\, \xi \in C_c^\infty(\mathbb R)$, by the Nikol'skii property \eqref{eq:holderinity} there exist constants $C_\omega,\, C_\xi > 0$ such that, for any $\eta \in \mathbb R$,
\[
\int_{\mathbb R} |\omega^{(m_j)}(u+\eta) - \omega^{(m_j)}(u)|\,\di u \le C_\omega |\eta|^{\alpha_j - m_j}\]
and 
\[
\int_{\mathbb R} |\xi^{(m_j)}(u+\eta) - \xi^{(m_j)}(u)|\,\di u \le C_\xi |\eta|^{\alpha_j - m_j}.
\]
Applying these estimates with $\eta \coloneqq R_n^{-(\alpha_j+1)}\delta$, we obtain
\[
\begin{aligned}\int_{\mathbb R}\big|g_j^{(m_j)}(x+\delta) - g_j^{(m_j)}(x)\big|\,\di x
&\le R_n^{-(\alpha_j+1)m_j} \\&\qquad\times\bigl[
C_\omega (R_n^{-(\alpha_j+1)}|\delta|)^{\alpha_j - m_j}
+ \lambda_n C_\xi (R_n^{-(\alpha_j+1)}|\delta|)^{\alpha_j - m_j}
\bigr] \\
&= (C_\omega + \lambda_n C_\xi) R_n^{-(\alpha_j+1)\alpha_j} |\delta|^{\alpha_j - m_j}\\
& \le (C_\omega + C_\xi) |\delta|^{\alpha_j - m_j},
\end{aligned}
\]
since $R_n \ge 1$ and $\lambda_n \le 1$ for sufficiently large $n$.
Combining this with the bound for the product of the other factors yields
\[
\|\Delta_{\delta\mathbf{e}_j}\partial_j^{m_j} f_1\|_{L^1(\mathbb R^d)}
\le (C_\omega + C_\xi) (1 + \|\xi\|_{L^1(\mathbb R)})^{|\mathcal P|-1} |\delta|^{\alpha_j - m_j}.
\]
Thus, condition \ref{it:smooth} \hyperref[it:smooth]{of the Main Article} holds for $f_1$ with constant $
L \coloneqq (C_\omega + C_\xi) (1 + \|\xi\|_{L^1(\mathbb R)})^{|\mathcal P|-1}$
which is independent of $n$. For inactive coordinates $j \notin \mathcal P$, the bound follows directly from $f_0^{\mathrm{inactive}} \in \mathcal N_1^{d-|\mathcal P|}(\bm\alpha_{\setminus\mathcal P},\, L_0)$, with the active factors contributing a constant independent of $n$.
Thus, condition \ref{it:smooth} \hyperref[it:smooth]{of the Main Article} holds with some constant $L$ independent of $n$.

We have verified that both $f_0$ and $f_1$ satisfy conditions \ref{it:diff}, \ref{it:L1}, and \ref{it:smooth} \hyperref[it:diff]{of the Main Article} with constants that are independent of $n$. Therefore, $f_0,\,f_1 \in \mathcal N_1^d(\bm\alpha,\, L)$ for some $L > 0$ independent of $n$.\\


\item \emph{Separation rate of the cumulative distribution functions.}\\[2pt]
We now derive the separation rate of the cumulative distribution functions.  
Let $F_0$ and $F_1$ be the cumulative distribution functions corresponding to $f_0$ and $f_1$.  
We show that there exists a constant $c_1 > 0$ such that, for all sufficiently large $n$,
\begin{equation}\label{eq:separate}
\|F_1 - F_0\|_{L^1(\mathbb R^d)} \ge c_1 n^{-E/2}.
\end{equation}
Using Lemma \ref{lem:foexp}, we obtain
\[
\begin{aligned}
F_1(\mathbf{x}) - F_0(\mathbf{x}) 
&= \int_{(-\infty, x_1] \times \cdots \times (-\infty, x_d]} [f_1(\mathbf{u}) - f_0(\mathbf{u})] \, \di\mathbf{u} \\
&=\Bigg[\prod_{j\in\mathcal P} [\Omega_j(x_j) + \lambda_n \Xi_j(x_j)] - \prod_{j\in\mathcal P} \Omega_j(x_j)\Bigg] F_0^{\mathrm{inactive}}(\mathbf{x}_{\setminus\mathcal P})\\
&= \lambda_n \underbrace{\Bigg[\sum_{j\in\mathcal P} \Xi_j(x_j) \prod_{\substack{\ell\in\mathcal P \\ \ell \ne j}} \Omega_\ell(x_\ell
)\Bigg] F_0^{\mathrm{inactive}}(\mathbf{x}_{\setminus\mathcal P})}_{\eqqcolon S(\mathbf{x})} + O(\lambda_n^2),
\end{aligned}
\]
where $F_0^{\mathrm{inactive}}$ is the cumulative distribution function of $f_0^{\mathrm{inactive}}$, 
\[
\Omega_j(x) \coloneqq \int_{-\infty}^x \omega_j(u)\,\di u = \Omega(R_n^{-(\alpha_j+1)}x)
\,\,\,\text{ and } \,\,\,
\Xi_j(x) \coloneqq \int_{-\infty}^x \xi_j(u)\,\di u = \Xi(R_n^{-(\alpha_j+1)}x),\quad x\in\mathbb R,
\]
with $\Omega(t) \coloneqq \int_{-\infty}^t \omega(s)\,\di s$ and $\Xi(t) \coloneqq \int_{-\infty}^t \xi(s)\,\di s$. Choose $\eta > 0$ such that $1/2 + \eta < 1$. Define the region
\[
\mathcal S_n \coloneqq \prod_{j\in\mathcal P} \left[\frac12 R_n^{\alpha_j+1},\, \left(\frac12+\eta\right) R_n^{\alpha_j+1}\right] \times \left[-\frac M2,\, \frac M2\right]^{d-|\mathcal P|}.
\]
Since $\Omega$ is continuous and strictly positive on $(-1,\,1)$, and the interval $[1/2,\,1/2 + \eta]$ is a compact subset of $(-1,\,1)$, there exists a constant $c_\Omega > 0$ such that $
\Omega(x) \ge c_\Omega$ for all  $x \in [1/2,\,1/2 + \eta]$.
Therefore, for $x_j \in [R_n^{\alpha_j+1}/2,\,(1/2+\eta) R_n^{\alpha_j+1}]$, we have
$
(R_n^{-(\alpha_j+1)}x_j) \in [1/2,\,1/2 + \eta],
$
so $
\Omega_j(x_j) = \Omega(R_n^{-(\alpha_j+1)}x_j) \ge c_\Omega > 0$.
Similarly, since $\Xi$ is continuous and $\int_{-\infty}^1 \xi(u)\,\di u > 0$ by assumption, there exists a constant $c_\Xi> 0$ such that $\Xi(x) \ge c_\Xi$ for all $x \in [1/2,\,1/2 + \eta]$. Therefore, for $x_j$ in the same interval, $
\Xi_j(x_j) = \Xi(R_n^{-(\alpha_j+1)}x_j) \ge c_\Xi > 0$.
Also, for $\mathbf{x}_{\setminus\mathcal P} \in [-M/2,\, M/2]^{d-|\mathcal P|}$, we have 
$F_0^{\mathrm{inactive}}(\mathbf{x}_{\setminus\mathcal P}) \ge \kappa_0 > 0$ for some constant 
$\kappa_0 > 0$. Indeed, since $f_0^{\mathrm{inactive}}$ is continuous and strictly positive on the 
compact set $[-M/2,\, M/2]^{d-|\mathcal P|}$, its cumulative distribution function 
$F_0^{\mathrm{inactive}}$ is also continuous and strictly positive there; hence, it attains a positive minimum $\kappa_0 \coloneqq \min_{\mathbf{x} \in [-M/2,\, M/2]^{d-|\mathcal P|}} F_0^{\mathrm{inactive}}(\mathbf{x}) > 0$ on this compact set. Thus, for $\mathbf{x} \in \mathcal S_n$, each term in the sum is nonnegative. 
Using the reverse triangle inequality, $|A+B| \ge |A| - |B|$, with 
$A = \lambda_n S(\mathbf{x}) \ge 0$ and $B = O(\lambda_n^2)$, we obtain
\[
|F_1(\mathbf{x}) - F_0(\mathbf{x})| 
= |\lambda_n S(\mathbf{x}) + O(\lambda_n^2)| \ge \lambda_n S(\mathbf{x}) - |O(\lambda_n^2)|.
\]
Since $|O(\lambda_n^2)| \le C\lambda_n^2$ for some constant $C > 0$, it follows that
$|F_1(\mathbf{x}) - F_0(\mathbf{x})| \ge \lambda_n S(\mathbf{x}) - C\lambda_n^2$.
On the region $\mathcal S_n$, we have $S(\mathbf{x}) \ge |\mathcal P| c_\Xi c_\Omega^{|\mathcal P|-1} \kappa_0$. Therefore, since $|\mathcal P| \ge 1$,
\[
|F_1(\mathbf{x}) - F_0(\mathbf{x})| \ge \lambda_n |\mathcal P| c_\Xi c_\Omega^{|\mathcal P|-1} \kappa_0 - C\lambda_n^2 \ge \lambda_n c_\Xi c_\Omega^{|\mathcal P|-1} \kappa_0 - C\lambda_n^2,
\]
where $C > 0$ is a constant. For sufficiently large $n$, $\lambda_n$ is small enough so that the linear term dominates, hence there exists a constant $c > 0$ such that $
|F_1(\mathbf{x}) - F_0(\mathbf{x})| \ge c \lambda_n$.
The volume of $\mathcal S_n$ equals
$
\prod_{j\in\mathcal P} (\eta R_n^{\alpha_j+1}) M^{d-|\mathcal P|} = \eta^{|\mathcal P|} M^{d-|\mathcal P|} \prod_{j\in\mathcal P} R_n^{\alpha_j+1}$.
Therefore, since $\lambda_n\coloneqq n^{-E/2}\prod_{j\in\mathcal P} R_n^{-(\alpha_j+1)}$, we get
\[
\|F_1 - F_0\|_{L^1(\mathbb R^d)} 
\ge \int_{\mathcal S_n} |F_1(\mathbf{x}) - F_0(\mathbf{x})|\,\di \mathbf{x} 
\ge c \lambda_n \eta^{|\mathcal P|} M^{d-|\mathcal P|} \prod_{j\in\mathcal P} R_n^{\alpha_j+1} = c \eta^{|\mathcal P|}  M^{d-|\mathcal P|} n^{-E/2}.
\]
Since $\eta^{|\mathcal P|} M^{d-|\mathcal P|} > 0$ is a constant independent of $n$, inequality \eqref{eq:separate} follows. \\

\item \emph{Bound on the $\chi^2$-divergence.}\\[2pt]
We now prove that $\chi^2(f_1 \ast f_\varepsilon \parallel f_0 \ast f_\varepsilon)=o( n^{-1})$.
To this end we refine the construction of $f_0^{\mathrm{inactive}}$ 
to ensure a uniform lower bound on its support. The previous construction gave 
$f_0^{\mathrm{inactive}}$ as a product of smoothed indicator functions, which vanish at the 
boundary of their support. To control the $\chi^2$-divergence, we need a stronger property: 
$f_0^{\mathrm{inactive}}$ must be bounded below by a positive constant times a power of the 
distance to the boundary of its support. This is achieved by replacing the standard bump 
function $\rho$ with a refined bump function $\varrho$ having polynomial decay near the 
boundary.

Let $\zeta \in C_c^\infty(0,\,1)$ be defined by
\[
\zeta(s) \coloneqq
\begin{cases}
[s(1-s)]^m, & s \in (0,\,1), \\
0, & \text{otherwise},
\end{cases}
\]
so that $\zeta > 0$ on $(0,\,1)$ and there exists $c_\zeta > 0$ such that
$
\zeta(s) \ge c_\zeta (1-s)^m$ for all $s \in [1/2,\, 1)$. 
Define
\[
T(y) \coloneqq \frac{\int_{y}^{1} \zeta(s)\,\di s}{\int_{0}^{1} \zeta(t)\,\di t}, \quad y \in [0,\,1].
\]
Then, $T \in C^\infty([0,\,1])$, $T(0)=1$, $T(1)=0$, all derivatives vanish at the endpoints, and there exists a constant $c_T > 0$ such that $
T(y) \ge c_T (1-y)^{m+1}$ for all $y \in [0,\,1]$.
Fix $0 < \delta < 1$ (for instance, $\delta = 1/2$). Define $\varrho_0 \in C_c^\infty(\mathbb R)$ as
\[
\varrho_0(x)\coloneqq
\begin{cases}
1, & |x| \le \delta, \\[4pt]
T\!\left(\dfrac{|x| - \delta}{1 - \delta}\right), & \delta < |x| < 1, \\[4pt]
0, & |x| \ge 1.
\end{cases}
\]
For $x \in [\delta,\, 1]$, we have
\[
\varrho_0(x) = T\!\left(\frac{x - \delta}{1 - \delta}\right)
\ge c_T \left(1 - \frac{x - \delta}{1 - \delta}\right)^{m+1}
= \frac{c_T}{(1-\delta)^{m+1}} (1-x)^{m+1}
\eqqcolon c_\varrho (1-x)^{m+1}.
\]
The same estimate holds symmetrically for $x \in [-1,\, -\delta]$. For $x \in [-\delta,\, \delta]$, the bound follows from the flat top.
Finally, since $\int_{-1}^{1} \varrho_0(x)\,\di x > 0$, we normalize to obtain 
\[
\varrho(x)\coloneqq \frac{\varrho_0(x)}{\int_{-1}^{1} \varrho_0(s)\,\di s}.
\]
Then, $\varrho \in C_c^\infty(\mathbb R)$, $\operatorname{supp}(\varrho) = [-1,\,1]$, $\int_{\mathbb R} \varrho(x)\,\di x = 1$, and there exists a constant $c_\varrho' > 0$ such that
\[
\varrho(x) \ge c_\varrho' \operatorname{dist}(x,\, \partial \operatorname{supp}(\varrho))^{m+1}
\quad \text{for all } x \in [-1,\,1],
\]
where $\partial \operatorname{supp}(\varrho)=\{-1,\,1\}$.
The polynomial decay of $\varrho$ transfers to the density $f_0^{\mathrm{inactive}}$. Indeed, $f_0^{\mathrm{inactive}}$ is constructed as a product of smoothed indicator functions, each built from $\varrho$ via scaling and convolution. By the scaling properties of the distance function and the product structure of $f_0^{\mathrm{inactive}}$, the decay estimate above implies the existence of a constant $c > 0$ such that
\begin{equation}\label{eq:lowerboundonfinactive}
f_0^{\mathrm{inactive}}(\mathbf x_{\setminus\mathcal P})
\ge c \operatorname{dist}(\mathbf x_{\setminus\mathcal P},\, \partial \operatorname{supp}(f_0^{\mathrm{inactive}}))^{m+1},
\end{equation}
where recall that $
\operatorname{supp}(f_0^{\mathrm{inactive}}) = \left[-({\gamma_M}/{2})M,\, ({\gamma_M}/{2})M\right]^{d-|\mathcal P|}$.

We now introduce the regions that are used to localize the $\chi^2$-integral and to lower bound $f_0 \ast f_\varepsilon$.
Choose $\gamma$ such that $\gamma_\xi < \gamma < \gamma_\omega = 3/4$.
We define the full support region
\[
\mathcal K_n^{\mathrm{full}} \coloneqq 
\prod_{j\in\mathcal P} \big[-\gamma R_n^{\alpha_j+1},\, \gamma R_n^{\alpha_j+1}\big]
\times \left[-\frac{\gamma_M}{2}M,\, \frac{\gamma_M}{2}M\right]^{d-|\mathcal P|}.
\]
Since $\gamma_\xi < \gamma$, it follows from the support properties of $\xi$ and $\omega$ that $
\operatorname{supp}(f_1 - f_0) \subset \mathcal K_n^{\mathrm{full}}$.
We also define the region where $f_0^{\mathrm{inactive}}$ is constant as 
\[
\mathcal K_n \coloneqq 
\prod_{j\in\mathcal P} \big[-\gamma R_n^{\alpha_j+1},\, \gamma R_n^{\alpha_j+1}\big]
\times \left[-\frac{\gamma_M}{8}M,\, \frac{\gamma_M}{8}M\right]^{d-|\mathcal P|}.
\]
Define the boundary region
\[
\mathcal B_n \coloneqq 
\bigg(\prod_{j\in\mathcal P} \big[-\gamma R_n^{\alpha_j+1},\, \gamma R_n^{\alpha_j+1}\big]\bigg)
\times \bigg( \left[-\frac{\gamma_M}{2}M,\, \frac{\gamma_M}{2}M\right]^{d-|\mathcal P|}
\setminus \left[-\frac{\gamma_M}{8}M,\, \frac{\gamma_M}{8}M\right]^{d-|\mathcal P|} \bigg).
\]
We have the partition \(\mathcal K_n^{\mathrm{full}} = \mathcal K_n \cup \mathcal B_n\).
On $\mathcal K_n$, we have $f_0^{\mathrm{inactive}}(\mathbf x_{\setminus\mathcal P}) = c_0 > 0$. 
For the active coordinates, since $\gamma < \gamma_\omega$ and $\omega(x) \ge c_\omega > 0$ for $|x| \le \gamma_\omega$, we have $
\omega_j(x_j) \ge c_\omega R_n^{-(\alpha_j+1)}$ for $|x_j| \le \gamma R_n^{\alpha_j+1}$.
Therefore, on $\mathcal K_n$,
\[
f_0(\mathbf x) = \prod_{j\in\mathcal P} \omega_j(x_j) f_0^{\mathrm{inactive}}(\mathbf x_{\setminus\mathcal P})
\ge c_\ast^{\mathrm{core}} \prod_{j\in\mathcal P} R_n^{-(\alpha_j+1)},
\]
where $c_\ast^{\mathrm{core}} \coloneqq c_0 \, c_\omega^{|\mathcal P|} > 0$.
On $\mathcal B_n$, the distance from $\mathbf x_{\setminus\mathcal P}$ to the boundary of $\operatorname{supp}(f_0^{\mathrm{inactive}})$ is at least $(\gamma_M/4) M$ (since the core set extends to $(\gamma_M/8)M$ and the full support to $(\gamma_M/2)M$). Applying the decay estimate \eqref{eq:lowerboundonfinactive}, we obtain a positive lower bound for $f_0^{\mathrm{inactive}}$ on $\mathcal B_n$,
\[
f_0^{\mathrm{inactive}}(\mathbf x_{\setminus\mathcal P}) \ge c \left(\frac{\gamma_M}{4} M\right)^{m+1} \eqqcolon c_0'' > 0.
\]
Consequently, on $\mathcal B_n$, for $c_\ast^{\mathrm{bound}} \coloneqq c_0'' \, c_\omega^{|\mathcal P|} > 0,$
\[
f_0(\mathbf x) \ge c_\ast^{\mathrm{bound}} \prod_{j\in\mathcal P} R_n^{-(\alpha_j+1)}.
\]
Defining $\underline c_\ast \coloneqq \min\{c_\ast^{\mathrm{core}}, \, c_\ast^{\mathrm{bound}}\} > 0$,
we obtain the uniform lower bound on $\mathcal K_n^{\mathrm{full}}$,
\begin{equation}\label{eq:lb}
f_0(\mathbf x) \ge \underline c_\ast \prod_{j\in\mathcal P} R_n^{-(\alpha_j+1)}
\quad \text{for all } \mathbf x \in \mathcal K_n^{\mathrm{full}}.
\end{equation}

We now show that the lower bound \eqref{eq:lb} is preserved under convolution with the noise density $f_\varepsilon$, up to a multiplicative constant. Since $f_\varepsilon$ is a probability density, for any $0 < \varpi < 1$ there exists a compact set $K_\varpi \subset \mathbb R^d$ such that 
$\int_{K_\varpi} f_\varepsilon(\mathbf z) \, \di\mathbf z \ge 1 - \varpi$.
Let 
$
R_{K_\varpi}\coloneqq \max_{\mathbf z \in K_\varpi} \max_{1 \le j \le d} |z_j| < \infty.
$
We choose $K_\varpi$ sufficiently small so that
\begin{equation}\label{eq:boundonR_K}
R_{K_\varpi} \le \min\left\{(\gamma_\omega - \gamma) R_n^{\alpha_j+1},\; \frac{\gamma_M}{16}M \right\}.
\end{equation}
Define the enlarged region for the core set $\mathcal K_n$ as
\[
\mathcal K_n^\ast \coloneqq
\prod_{j\in\mathcal P} \big[-\gamma_\omega R_n^{\alpha_j+1},\, \gamma_\omega R_n^{\alpha_j+1}\big]
\times \left[-\frac{\gamma_M}{4}M,\, \frac{\gamma_M}{4}M\right]^{d-|\mathcal P|}.
\]
We claim that $f_0$ is bounded below on $\mathcal K_n^\ast$. Indeed, on the active coordinates we have the same lower bound $c_\omega R_n^{-(\alpha_j+1)}$. On the inactive coordinates, since
\[
\left[-\frac{\gamma_M}{4}M,\, \frac{\gamma_M}{4}M\right]^{d-|\mathcal P|}
\subset \left[-\frac{\gamma_M}{2}M,\, \frac{\gamma_M}{2}M\right]^{d-|\mathcal P|},
\]
and $f_0^{\mathrm{inactive}}$ is continuous and positive on this compact set 
(it equals $c_0 > 0$ on the set $[-(\gamma_M/4)M, \,(\gamma_M
/4)M]^{d-|\mathcal P|}$), there exists $\tilde c_0 > 0$ such that $
f_0^{\mathrm{inactive}}(\mathbf x_{\setminus\mathcal P}) \ge \tilde c_0 > 0$.
Therefore, on $\mathcal K_n^\ast$,
\begin{equation}\label{eq:lowerboundonKast}
f_0(\mathbf x) \ge c_\ast^{(1)} \prod_{j\in\mathcal P} R_n^{-(\alpha_j+1)},
\end{equation}
where $c_\ast^{(1)} \coloneqq \tilde c_0 \, c_\omega^{|\mathcal P|} > 0$.
For $\mathbf x \in \mathcal K_n$ and $\mathbf z \in K_\varpi$, using \eqref{eq:boundonR_K} for the active coordinates we have
\[
|x_j - z_j| \le |x_j| + |z_j|
\le \gamma R_n^{\alpha_j+1} + R_{K_\varpi}
\le \gamma R_n^{\alpha_j+1} + (\gamma_\omega - \gamma) R_n^{\alpha_j+1}
= \gamma_\omega R_n^{\alpha_j+1},
\]
and for the inactive coordinates
\[
|x_k - z_k| \le |x_k| + |z_k|
\le \frac{\gamma_M}{8}M + \frac{\gamma_M}{16}M
= \frac{3\gamma_M}{16}M < \frac{\gamma_M}{4}M.
\]
Thus, $\mathbf x - \mathbf z \in \mathcal K_n^\ast$. Equivalently, $\mathbf x - K_\varpi \coloneqq \{\mathbf x - \mathbf z:\, \mathbf z\in K_\varpi\}\subset \mathcal K_n^\ast$.
Now, for every $\mathbf x \in \mathcal K_n$, using \eqref{eq:lowerboundonKast} and the fact that $\mathbf x - K_\varpi \subset \mathcal K_n^\ast$, we have 
\begin{align}\label{eq:lbeta}
(f_0 \ast f_\varepsilon)(\mathbf x)
= \int_{\mathbb R^d} f_0(\mathbf u) f_\varepsilon(\mathbf x - \mathbf u) \, \di\mathbf u
&\ge \int_{\mathcal K_n^\ast} f_0(\mathbf u) f_\varepsilon(\mathbf x - \mathbf u) \, \di \mathbf u\nonumber\\
&\geq 
c_\ast^{(1)} \prod_{j\in\mathcal P} R_n^{-(\alpha_j+1)}
\int_{\mathcal K_n^\ast} f_\varepsilon(\mathbf x - \mathbf u) \, \di\mathbf u \nonumber\\
&= c_\ast^{(1)} \prod_{j\in\mathcal P} R_n^{-(\alpha_j+1)}
\int_{K_\varpi} f_\varepsilon(\mathbf z) \, \di \mathbf z \nonumber\\
&\ge c_\ast^{(1)} (1 - \varpi) \prod_{j\in\mathcal P} R_n^{-(\alpha_j+1)}.
\end{align}
For $\mathbf x \in \mathcal B_n$, we proceed similarly. Since $\mathbf x$ is at distance at least $(\gamma_M/4)M$ from the boundary of $\operatorname{supp}(f_0^{\mathrm{inactive}})$, and $R_{K_\varpi} \le (\gamma_M/16)M$, the set $\mathbf x - K_\varpi$ remains within the compact subset of $\operatorname{supp}(f_0^{\mathrm{inactive}})$ where $f_0$ is bounded below by $c_\ast^{(2)} \prod_{j\in\mathcal P} R_n^{-(\alpha_j+1)}$. The same convolution argument yields that for some constant $c_\ast^{(2)} > 0$
\[
(f_0 \ast f_\varepsilon)(\mathbf x)
\ge c_\ast^{(2)} (1 - \varpi) \prod_{j\in\mathcal P} R_n^{-(\alpha_j+1)}
\quad \text{for all } \mathbf x \in \mathcal B_n.
\]
Taking $c_\ast\coloneqq \min\{c_\ast^{(1)},\, c_\ast^{(2)}\} > 0$, we obtain the uniform lower bound 
\[
(f_0 \ast f_\varepsilon)(\mathbf x)
\ge c_\ast (1 - \varpi) \prod_{j\in\mathcal P} R_n^{-(\alpha_j+1)}
\quad \text{for all } \mathbf x \in \mathcal K_n^{\mathrm{full}}.
\]
Since $\operatorname{supp}(f_1 - f_0) \subset \mathcal K_n^{\mathrm{full}}=\mathcal K_n \cup \mathcal B_n$,
the $\chi^2$-divergence between the convolved densities can be written as an integral over $\mathcal K_n^{\mathrm{full}}$,
\[\begin{split}
\chi^2(f_1 \ast f_\varepsilon \parallel f_0 \ast f_\varepsilon)
&= \int_{\mathcal K_n^{\mathrm{full}}}
\frac{[(f_1\ast f_\varepsilon)(\mathbf x) - (f_0\ast f_\varepsilon)(\mathbf x)]^2}{(f_0 \ast f_\varepsilon)(\mathbf x)}
\, \di\mathbf x\\
&= \int_{\mathcal K_n \cup \mathcal B_n}
\frac{[(f_1\ast f_\varepsilon)(\mathbf x) - (f_0\ast f_\varepsilon)(\mathbf x)]^2}{(f_0 \ast f_\varepsilon)(\mathbf x)}
\, \di\mathbf x.
\end{split}
\]
By the lower bound on $f_0 \ast f_\varepsilon$ established above, there exists $c_0\coloneqq c_\ast (1 - \varpi) > 0$ independent of $n$ such that uniformly on $\mathcal K_n \cup \mathcal B_n$,
\[
(f_0 \ast f_\varepsilon)(\mathbf x) \ge c_0 \prod_{j\in\mathcal P} R_n^{-(\alpha_j+1)}.
\]
Consequently,
\[
\chi^2(f_1 \ast f_\varepsilon \parallel f_0 \ast f_\varepsilon) \le C \prod_{j\in\mathcal P} R_n^{\alpha_j+1} \int_{\mathcal K_n \cup \mathcal B_n} [(f_1\ast f_\varepsilon)(\mathbf x) - (f_0\ast f_\varepsilon)(\mathbf x)]^2 \, \di\mathbf x.
\]
Extending the integral to all of $\mathbb R^d$, applying Young's inequality with $\|f_\varepsilon\|_{L^1(\mathbb R^d)} = 1$, and using Plancherel's theorem, we obtain
\begin{align}
\chi^2(f_1 \ast f_\varepsilon \parallel f_0 \ast f_\varepsilon)
&\le C \prod_{j\in\mathcal P} R_n^{\alpha_j+1} \|(f_1 - f_0) \ast f_\varepsilon\|_{L^2(\mathbb R^d)}^2 \notag \\
&\le C \prod_{j\in\mathcal P} R_n^{\alpha_j+1} \|f_1 - f_0\|_{L^2(\mathbb R^d)}^2\|f_\varepsilon\|_{L^1(\mathbb R^d)}  \notag \\
&= C \prod_{j\in\mathcal P} R_n^{\alpha_j+1} \frac{1}{(2\pi)^d}\|\mathcal F[f_1-f_0]\|_{L^2(\mathbb R^d)}^2. \notag
\end{align}
By Lemma \ref{lem:foexp} applied with $m = |\mathcal P|$, $A_\ell = \omega_\ell(x_\ell)$, $B_\ell = \xi_\ell(x_\ell)$, and $\delta = \lambda_n$, we obtain pointwise for each $\mathbf x \in \mathbb R^d$
\[
\begin{aligned}
f_1(\mathbf x) - f_0(\mathbf x)
&= \Bigg[ \prod_{j\in\mathcal P} [\omega_j(x_j) + \lambda_n \xi_j(x_j)] - \prod_{j\in\mathcal P} \omega_j(x_j) \Bigg] f_0^{\mathrm{inactive}}(\mathbf x_{\setminus\mathcal P}) \\
&
= \lambda_n \sum_{j\in\mathcal P} \xi_j(x_j) \prod_{\substack{\ell\in\mathcal P \\ \ell \neq j}}  \omega_\ell(x_\ell) \, f_0^{\mathrm{inactive}}(\mathbf x_{\setminus\mathcal P})
+ R(\lambda_n;\, \mathbf x),
\end{aligned}
\]
where the remainder $R(\lambda_n;\, \mathbf x)$ from \eqref{eq:remainder} in Lemma \ref{lem:foexp} has the explicit form
\[
R(\lambda_n;\, \mathbf x)
= \lambda_n^2 \sum_{k=2}^{|\mathcal P|}\lambda_n^{k-2} \sum_{\substack{S \subset \mathcal P \\ |S| = k}}
\Bigg( \prod_{j \in S} \xi_j(x_j) \Bigg)
\Bigg( \prod_{\ell \notin S} \omega_\ell(x_\ell) \Bigg)
 f_0^{\mathrm{inactive}}(\mathbf x_{\setminus\mathcal P}).
\]
Taking the Fourier transform, which factorizes over independent coordinates, we get
\[
\mathcal F[R(\lambda_n;\, \cdot)](\mathbf t)
= \lambda_n^2 \sum_{k=2}^{|\mathcal P|} \lambda_n^{k-2} \sum_{\substack{S \subset \mathcal P \\ |S| = k}}
\Bigg( \prod_{j \in S} \mathcal F[\xi_j](t_j) \Bigg)
\Bigg( \prod_{\ell \notin S} \mathcal F[\omega_\ell](t_\ell) \Bigg)
\mathcal F[f_0^{\mathrm{inactive}}](\mathbf t_{\setminus\mathcal P}).
\]
Since $\omega \in C_c^\infty(\mathbb R)$ and $f_0^{\mathrm{inactive}}\in C_c^\infty(\mathbb R^{d-|\mathcal P|})$ are smooth and compactly supported, we have $\|\mathcal F[\omega]\|_\infty \le C_\omega$ and 
$\|\mathcal F[f_0^{\mathrm{inactive}}]\|_\infty \le C_0$ for suitable constants $C_0$ and $C_\omega$. 
Using $\| \mathcal F[\xi_j]\|_{L^2(\mathbb R)}^2=R_n ^{-(\alpha_j+1)}\|\mathcal F[\xi]\|_{L^2(\mathbb R)}^2<\infty$, for a suitable constant $C_2>0$ we obtain
\[
\|\mathcal F[R(\lambda_n;\, \cdot)]\|_{L^2(\mathbb R^d)}^2
\le C_2 \lambda_n^4 \prod_{j\in\mathcal P} R_n^{-2(\alpha_j+1)}.
\]
Indeed, the dominant term in the sum over $S$ is $|S|=2$; all higher-order terms are smaller since $R_n \ge 1$ and $\lambda_n \le 1$ for sufficiently large $n$. 

Taking the Fourier transform of $f_1-f_0$, and using the same arguments as those used for bounding the remainder term, we obtain
\[\begin{aligned}
\mathcal F[f_1-f_0](\mathbf t)
&= \lambda_n \sum_{j\in\mathcal P} \mathcal F[\xi_j](t_j)  \prod_{\substack{\ell\in\mathcal P \\ \ell \neq j}} \mathcal F[\omega_\ell](t_\ell)\mathcal F[f_0^{\mathrm{inactive}}](\mathbf t_{\setminus\mathcal P})+\mathcal F[R(\lambda_n;\,\cdot)](\mathbf t)\\
&=\lambda_n \sum_{j\in\mathcal P} \mathcal F[\xi](R_n ^{\alpha_j+1}t_j)  \prod_{\substack{\ell\in\mathcal P \\ \ell \neq j}} \mathcal F[\omega](R_n ^{\alpha_\ell+1}t_\ell) \mathcal F[f_0^{\mathrm{inactive}}](\mathbf t_{\setminus\mathcal P})+\mathcal F[R(\lambda_n;\,\cdot)](\mathbf t).
\end{aligned}
\]
Using the same arguments as those used for bounding the remainder term, we obtain
\[\begin{aligned}
\|\mathcal F[f_1-f_0]\|_{L^2(\mathbb R^d)}^2&\leq 2\lambda_n^2 C_0^2C_\omega^{2(|\mathcal P|-1)}\sum_{j\in\mathcal P} \|\mathcal F[\xi](R_n ^{\alpha_j+1}\cdot)\|_{L^2(\mathbb R)}^2 +2\|\mathcal F[R(\lambda_n;\,\cdot)]\|_{L^2(\mathbb R^d)}^2\\
& =2
\lambda_n^2 C_0^2C_\omega^{2(|\mathcal P|-1)}\|\mathcal F[\xi]\|_{L^2(\mathbb R)}^2\sum_{j\in\mathcal P} R_n ^{-(\alpha_j+1)} +O\!\left(\lambda_n^4 \prod_{j\in\mathcal P} R_n^{-2(\alpha_j+1)}\right).
\end{aligned}\]
Since $R_n\geq1$ and $\lambda_n \to 0$, the remainder term contributes only lower-order terms. Thus, for a suitable constant \(
C_1 \coloneqq 2(2\pi)^{-d} C C_0^2 C_\omega^{2(|\mathcal P|-1)} |\mathcal P| \|\mathcal F[\xi]\|_{L^2(\mathbb R)}^2\), 
we have
\[
\chi^2(f_1 \ast f_\varepsilon \parallel f_0 \ast f_\varepsilon)
\le C_1 \lambda_n^2 \prod_{j\in\mathcal P} R_n^{\alpha_j+1}.
\]
Taking into account that $\lambda_n = n^{-E/2} \prod_{j\in\mathcal P} R_n^{-(\alpha_j+1)}$, it follows that 
$$\lambda_n^2 \prod_{j\in\mathcal P} R_n^{\alpha_j+1}=n^{-E} \prod_{j\in\mathcal P} R_n^{-(\alpha_j+1)}.$$ Choosing $R_n\coloneqq n^{1/E}$, we obtain
\[
\begin{aligned}
\chi^2(f_1 \ast f_\varepsilon \parallel f_0 \ast f_\varepsilon)
&\le C_1 n^{-E} \prod_{j\in\mathcal P} R_n^{-(\alpha_j+1)}
= C_1 n^{-E-\frac{1}{E}\sum_{j\in\mathcal P}(\alpha_j+1)}.
\end{aligned}
\]
Since $\sum_{j\in\mathcal P} (\alpha_j+1) > |\mathcal P| \ge 1$ and $E > 0$, we have
\[
E + \frac{1}{E}\sum_{j\in\mathcal P} (\alpha_j+1)
> E + \frac{1}{E} 
> 1.
\]
Consequently,
$$\chi^2(f_1 \ast f_\varepsilon \parallel f_0 \ast f_\varepsilon)=o( n^{-1}).$$

\item \emph{Application of Le Cam's two-point method.}\\[2pt]
We can now prove the minimax lower bound.
By Le Cam's two-point method, see, \emph{e.g.}, Theorem 2.6 in \citet{tsybakov2009introduction}, p. 100,
\begin{multline*}
\inf_{\hat F_n} \sup_{f_X \in \mathcal N_1^d(\boldsymbol\alpha,\, L)}
\mathsf{E}_{(\mu_X \ast \mu_\varepsilon)^{\otimes n}}
\bigl[ \| \hat F_n - F_X \|_{L^1(\mathbb{R}^d)} \bigr]
\ge \frac{1}{2} \| F_1 - F_0 \|_{L^1(\mathbb{R}^d)} \\
\times [
1 - \| (\mu_1 \ast \mu_\varepsilon)^{\otimes n}
    - (\mu_0 \ast \mu_\varepsilon)^{\otimes n}
\|_{\mathrm{TV}} ],
\end{multline*}
where \(\mu_i\) is the distribution corresponding to the density \(f_i\) for \(i = 0,\, 1\).
By the $\chi^2$ bound established above, we have $\chi^2(f_1 \ast f_\varepsilon \parallel f_0 \ast f_\varepsilon)=o(n^{-1})$.
By tensorization of the $\chi^2$ divergence,
\[
\chi^2\bigl((f_1 \ast f_\varepsilon)^{\otimes n} \parallel (f_0 \ast f_\varepsilon)^{\otimes n}\bigr)
= [1 + \chi^2(f_1 \ast f_\varepsilon \parallel f_0 \ast f_\varepsilon)]^n - 1
= (1 + o(n^{-1}))^n - 1 \to 0.
\]
Then, by Pinsker's inequality,
\[
\| (\mu_1 \ast \mu_\varepsilon)^{\otimes n}
    - (\mu_0 \ast \mu_\varepsilon)^{\otimes n}
\|_{\mathrm{TV}}
\le
\sqrt{\frac12 \chi^2\bigl((f_1 \ast f_\varepsilon)^{\otimes n} \parallel (f_0 \ast f_\varepsilon)^{\otimes n}\bigr)}
\to 0.
\]
Therefore,
$1 - \| (\mu_1 \ast \mu_\varepsilon)^{\otimes n}
    - (\mu_0 \ast \mu_\varepsilon)^{\otimes n}
\|_{\mathrm{TV}} \to 1$.
Hence, from Le Cam's lower bound, since \(\|F_1 - F_0\|_{L^1(\mathbb R^d)}\geq c_1 n^{-E/2}\) from \eqref{eq:separate}, for $n$ sufficiently large we obtain
\[
\inf_{\hat F_n} \sup_{f_X \in \mathcal N_1^d(\boldsymbol\alpha,\, L)}
\mathsf{E}_{(\mu_X\ast\mu_\varepsilon)^{\otimes n}} [\|\hat F_n - F_X\|_{L^1(\mathbb R^d)}]
\ge \frac12 \|F_1 - F_0\|_{L^1(\mathbb R^d)} [1 - o(1)]\geq\frac{c_1}{4}n^{-\frac E2}.
\]
\end{enumerate}
This completes the proof.
\end{proof}

\subsection{Linear transformations of errors with independent components}\label{sec:lineartransf}
In this section, we consider error vectors $\boldsymbol{\eta}$ whose components may be correlated. Our goal is to derive rates of convergence for deconvolution estimators under such dependence. We argue that, if the dependence structure is known, we can without loss of generality assume that the error components are independent, provided that $\boldsymbol{\eta}$ can be expressed as a linear transformation of a vector with independent components, that is, we assume there exists an invertible matrix $\mathbf{A}$ (which is known) such that $\boldsymbol{\eta} = \mathbf{A}\boldsymbol{\varepsilon}$, where the components of $\boldsymbol{\varepsilon}$ are independent, though not necessarily identically distributed. Because $\mathbf{A}$ is invertible, we can transform the model and transfer convergence rates, rendering the independence assumption non-restrictive in this setting.

Consider the observation model 
$\mathbf{Y}_i=\mathbf{X}_i+\boldsymbol{\eta}_i$, $i=1,\,\ldots,\,n$.
Premultiplying both sides by $\mathbf{A}^{-1}$ gives $
\mathbf{A}^{-1}\mathbf{Y}_i=\mathbf{A}^{-1}\mathbf{X}_i+\mathbf{A}^{-1}\boldsymbol{\eta}_i=\mathbf{A}^{-1}\mathbf{X}_i+\boldsymbol{\varepsilon}_i$, $i=1,\,\ldots,\,n$.
Hence, by applying the linear transformation $\mathbf{A}^{-1}$ to the observed data, the original model with correlated errors $\boldsymbol{\eta}_i$ is reduced to a convolution model with independent errors $\boldsymbol{\varepsilon}_i$.
Because the transformation is invertible, any convergence-rate results obtained for the transformed model directly transfer back to the original model. This justifies the claim that the independence assumption is not restrictive.

To estimate the probability measure of the random vector $\mathbf{A}^{-1}\mathbf{X}_1$, we use the density estimator 
\[
\hat f_{n,\mathbf{h}}^{\mathbf{A}^{-1}}(\mathbf{z})
=\frac{1}{n}\sum_{i=1}^n\prod_{j=1}^d\frac{1}{h_j}\tilde K_{j,h_j}\biggl(\frac{z_j-(\mathbf{A}^{-1}\mathbf{Y}_i)_j}{h_j}\biggr),\quad \mathbf{z}\in\mathbb R^d,
\]
and the corresponding distribution function estimator 
\[
\widehat F_{n,\mathbf{h}}^{\mathbf{A}^{-1}}(\mathbf{z})=\int_{-\infty}^{z_1}\!\!\cdots\int_{-\infty}^{z_d}
\hat f_{n,\mathbf{h}}^{\mathbf{A}^{-1}}(\mathbf{v})\,\mathrm{d}\mathbf{v},\quad \mathbf{z}\in\mathbb R^d.
\]
Since \(F_X(\mathbf{x}) = F_X^{\mathbf{A}^{-1}}(\mathbf{A}^{-1}\mathbf{x})\), where $F_X^{\mathbf{A}^{-1}}$ is the cumulative distribution function of $\mathbf{A}^{-1}\mathbf{X}_1$, we define the estimator for \(F_X\) by
$\widehat F_{n,\mathbf{h}}(\mathbf{x}) \coloneqq \widehat F_{n,\mathbf{h}}^{\mathbf{A}^{-1}}(\mathbf{A}^{-1}\mathbf{x})$, $\mathbf{x}\in\mathbb R^d$. 

Let $|\mathbf{A}|\coloneqq|\det(\mathbf{A})|$. Then,
\[
\|\widehat F_{n,\mathbf{h}}-F_X\|_{L^1(\mathbb{R}^d)}
=|\mathbf{A}|\;\|\widehat F_{n,\mathbf{h}}^{\mathbf{A}^{-1}}-F_X^{\mathbf{A}^{-1}}\|_{L^1(\mathbb{R}^d)}.
\]
This follows from the change of variables $\mathbf{w}=\mathbf{A}^{-1}\mathbf{x}$ since
\[
\int_{\mathbb{R}^d}|\widehat F_{n,\mathbf{h}}(\mathbf{x})-F_X(\mathbf{x})|\,\mathrm{d}\mathbf{x}
=\int_{\mathbb{R}^d}|\widehat F_{n,\mathbf{h}}^{\mathbf{A}^{-1}}(\mathbf{w})-F_X^{\mathbf{A}^{-1}}(\mathbf{w})|\;|\mathbf{A}|\,\mathrm{d}\mathbf{w}.
\]
Thus, rates of convergence for the transformed problem transfer to the original problem, up to the constant factor 
$|\mathbf{A}|$.

In practice, however, the correlation structure of the errors, the covariance matrix $\boldsymbol{\Sigma} = \mathsf{E}[\boldsymbol{\eta}\boldsymbol{\eta}^\top]$, is unknown. Because $\boldsymbol{\Sigma}$ determines $\mathbf{A}$ up to right multiplication by an orthogonal matrix (equivalently, up to a diagonal scaling if the variances are fixed, since $\boldsymbol{\Sigma} = \mathbf{A} \,\mathrm{diag}(\operatorname{Var}(\varepsilon_j))\, \mathbf{A}^\top$), the matrix $\mathbf{A}$ is also unknown, and the above transformation cannot be applied. Several approaches may circumvent this limitation, though they are beyond the scope of the present work. If independent replicates of the noise are available from a separate pure-noise sample, then $\boldsymbol{\Sigma}$ can be consistently estimated and plugging this estimate into the transformation introduces additional estimation error, possibly leading to slower convergence rates. If $\mathbf{A}$ is completely unknown, the problem becomes a \emph{blind deconvolution}, which is not addressed here, as this work assumes the error distribution is fully known. Extending the theoretical results to unknown correlation remains an open problem.

\section{Proofs for adaptive deconvolution}\phantomsection\label{sec:S4}

In what follows, we provide the proof of Theorem~\ref{thm:oracle_ineq}, which establishes an oracle inequality for the adaptive $\varrho_n$-approximate minimum-$L^2$-distance estimator $\widehat F_{n,\hat{\mathbf{h}}}^\ast$. The proof is based on the fixed-bandwidth concentration inequality provided in Proposition~\ref{prop:oracle}, combined with an analysis of the Goldenshluger--Lepski-type bandwidth selector defined in \eqref{eq:hopt} \hyperref[eq:hopt]{of the Main Article}. 

For the reader's convenience, we recall below a version of Talagrand's inequality as derived in~\citet{ComteLacour2013}, p.~587.
\begin{lemma}[Talagrand's inequality from~\citet{ComteLacour2013}]\label{lemma:talagrand}
Let \(\mathbf{Y}_1,\,\dots,\,\mathbf{Y}_n\) be i.i.d. random vectors in \(\mathbb R^d\), and let \(\bar{B}\) be a countable subset of a function space. For each \(f\in \bar{B}\), let 
$
\psi_f:\mathbb R^d \to \mathbb R
$
be a real-valued measurable function indexed by \(f\). Define the centered empirical process
\[
\nu_n(f) \coloneqq \frac{1}{n} \sum_{i=1}^n 
[ \psi_f(\mathbf{Y}_i) - \mathsf{E}[\psi_f(\mathbf{Y}_1)]],
\quad f\in \bar{B}.
\]
Assume that there exist constants \(0<\Gamma,\,M,\,\upsilon<\infty\) such that
\[
\mathsf{E}\bigg[\sup_{f\in \bar{B}} |\nu_n(f)|\bigg]\le \Gamma,\qquad
\sup_{f\in \bar{B}} \| \psi_f\|_{\infty} \le M, \qquad
\sup_{f\in \bar{B}} \operatorname{Var}
[\psi_f(\mathbf{Y}_1)]\le \upsilon,
\]
where \(\| \psi_f\|_{\infty}\coloneqq\sup_{\mathbf{y}\in\mathbb R^d}|\psi_f(\mathbf{y})|\) denotes the supremum norm. Then, for any \( \eta > 0 \),
\begin{equation}\label{eq:Talineq}
\begin{split}
\mathsf{E}\bigg[\bigg( \sup_{f\in \bar{B}} |\nu_n(f)| - (1+2\eta)\Gamma \bigg)_+ \bigg]
&\le \sqrt{\frac{3\pi}{2}} \sqrt{\frac{\upsilon}{n}} \,
\exp\left(- \frac{\eta^{2}}{6} \cdot \frac{n \Gamma^{2}}{\upsilon}\right)\\
&\hspace*{3cm} + \frac{21}{\eta\wedge 1} \cdot\frac{M}{n}
\exp\left(- \frac{\eta (\eta\wedge 1)}{21} \cdot \frac{n \Gamma}{M}\right),
\end{split}
\end{equation}
with \((x)_+ \coloneqq \max\{x,\,0\}\) the positive part of $x$. 
\end{lemma}

While \cite{ComteLacour2013} established adaptivity of anisotropic kernel estimators for the \emph{density} \(f_X\), our result provides a concentration bound for the corresponding \emph{cumulative distribution function} estimator \(\widehat F_{n,\mathbf{h}}\). This distinction is substantial, as it requires a fine control of the integrated variance, yielding the leading variance term \(V^2(\mathbf{h}) = n^{-1}\prod_{j=1}^d h_j^{-2(\beta_j - 1/2)_+}\).
The bound improves upon the analogous result in Proposition~7 of \cite{ComteLacour2013}, pp.~604--606, by explicitly accommodating the critical smoothness exponents \(\beta_j = 1/2\), from which the log-factor 
\(\prod_{j=1}^d |\log h_j|^{\mathbf{1}_{\{\beta_j=1/2\}}}\)
emerges.
The proof relies on an application of Talagrand's concentration inequality for empirical processes over \(L^2\)-unit balls, leading to exponential deviation bounds with explicit constants.

\begin{proposition}[Concentration inequality for the $L^2$-error of $\widehat F_{n,\mathbf{h}}$]\label{prop:oracle}
Assume that the noise conditions \ref{eq:E1indep}, \ref{it:E21} and \ref{eq:E3} \hyperref[eq:E1indep]{of the Main Article} hold with \(\beta_j>0\) for all \(j=1,\,\ldots,\,d\). Consider the kernel-based estimator \(\widehat F_{n,\mathbf{h}}\) constructed with a product kernel \(K(\mathbf{x})=\prod_{j=1}^d K_j(x_j)\), \(\mathbf{x}\in\mathbb R^d\), where each univariate kernel \(K_j\) satisfies the conditions (i)--(iii) listed in Proposition~\ref{prop:variance}.
Then, there exist constants \(C_1,\,C_2,\,C_3>0\), depending only on the kernels \(K_j\) and the noise distribution, such that, for every bandwidth vector \(\mathbf{h}\in(0,\,1)^d\), for any sequence $C_0(\mathbf{h})\geq1$ (which may depend on $\mathbf{h}$), and for every \(\eta>0\), the following inequality holds:
\begin{align*}
&\mathsf E_{(\mu_X\ast\mu_\varepsilon)^{\otimes n}}\Bigl[\Bigl(
\|\widehat F_{n,\mathbf{h}}-(K_{\mathbf{h}}\ast F_X)\|_{L^2(\mathbb R^d)} 
- (1+2\eta) C_0(\mathbf{h})C_1
\mathcal V(\mathbf{h})
\Bigr)_+\Bigr] \\
&\qquad\qquad\qquad\qquad\le \sqrt{\frac{3\pi}{2}} \sqrt{\frac{\upsilon}{n}} \,
\exp\biggl(- \frac{\eta^{2}C_1^2 }{6C_3^2}C_0^2(\mathbf{h})\biggr)
+ \frac{21}{\eta\wedge 1} C_2
\frac{\mathcal V(\mathbf{h})}{\sqrt{n}}
\exp\biggl(- \frac{\eta (\eta\wedge 1)C_1}{21C_2} \sqrt{n}\biggr),
\end{align*}
where $\mathcal V(\mathbf{h})$ is defined in \eqref{eq:varprox} \hyperref[eq:varprox]{of the Main Article} and \(\upsilon \coloneqq C_3^2 n \mathcal V^2(\mathbf{h})\).
\end{proposition}

\begin{proof}
The proof proceeds by applying Talagrand's inequality~\eqref{eq:Talineq}, stated in Lemma~\ref{lemma:talagrand}. Let 
\[
B(0,\,1)\coloneqq \bigl\{ f\in L^1(\mathbb R^d)\cap L^2(\mathbb R^d):\, \|f\|_{L^2(\mathbb R^d)} = 1 \bigr\}
\]
and let \(\bar{B}(0,\,1)\) be a countable dense subset of \(B(0,\,1)\), whose existence follows from the separability of \(L^2(\mathbb R^d)\). For \(f\in \bar{B}(0,\,1)\), define
\begin{align*}
\psi_f(\mathbf{y})
&\coloneqq \int_{\mathbb R^d} \overline{f(\mathbf{v})}
\left(
\int_{-\infty}^{v_1} \cdots \int_{-\infty}^{v_d}
\mathcal F^{-1}\!\left[
\frac{\varphi_K(\mathbf{h}\odot \mathbf{t})\, \mathrm{e}^{\mathrm{i} \langle \mathbf{t},\, \mathbf{y} \rangle}}
{\varphi_\varepsilon(\mathbf{t})}
\right](\mathbf{u})
\,\mathrm{d}\mathbf{u}
\right)
\,\mathrm{d}\mathbf{v},
\quad \mathbf{y}\in\mathbb R^d,
\end{align*}
where \(\mathbf{h}\odot\mathbf{t} \coloneqq (h_1 t_1,\, \ldots,\, h_d t_d)^\top\) denotes the componentwise (Hadamard) product of the two vectors $\mathbf{h}$ and $\mathbf{t}$\footnote{Recall that the componentwise product corresponds to the action of the diagonal bandwidth matrix \(\mathbf{H}=\operatorname{diag}(h_1,\,\ldots,\,h_d)\) on \(\mathbf{t}\), \emph{i.e.}, \(\mathbf{H}\mathbf{t} = \mathbf{h}\odot\mathbf{t}\).}, and \(\overline{z}\) denotes the complex conjugate of \(z\in\mathbb C\).

Recalling the definition of \(G_{j,h_j}\) in \eqref{eq:gfun}, we note that \(\psi_f(\mathbf{y})\) can be equivalently written as
\begin{equation}\label{eq:termproc}
\psi_f(\mathbf{y})
= \int_{\mathbb R^d} \overline{f(\mathbf{v})}
\prod_{j=1}^d G_{j,h_j}\!\left(\frac{v_j-y_j}{h_j}\right)
\,\mathrm{d}\mathbf{v},
\quad \mathbf{y}\in\mathbb R^d.
\end{equation}
We define the empirical process indexed by \(f\in \bar{B}(0,\,1)\) as follows:
\[
\nu_n(f) \coloneqq \langle \widehat F_{n,\mathbf{h}} - (K_{\mathbf{h}}\ast F_X),\, f \rangle_2
= \frac{1}{n}\sum_{i=1}^n [ \psi_f(\mathbf{Y}_i) - \mathsf E_{(\mu_X\ast\mu_\varepsilon)}[\psi_f(\mathbf{Y}_1)] ].
\]
The continuity of the mapping \(f \mapsto \nu_n(f)\) is required to justify the equality of suprema
\[
\sup_{f \in B(0,1)} |\nu_n(f)| = \sup_{f \in \bar{B}(0,1)} |\nu_n(f)|,
\]
which is used below in \eqref{eq:taldiseq} to write
\[
\|\widehat F_{n,\mathbf{h}}-(K_{\mathbf{h}}\ast F_X)\|_{L^2(\mathbb R^d)} = \sup_{f \in \bar{B}(0,1)} |\nu_n(f)|.
\]
This continuity is essential for the application of Talagrand's inequality. Indeed, Talagrand's concentration inequality for empirical processes is typically stated for separable function classes; the supremum over a non-separable index set may fail to be measurable, and the constants in the inequality may depend on separability assumptions. By constructing the countable dense subset \(\bar{B}(0,\,1)\) and establishing continuity of the empirical process, we ensure that the supremum over the countable subclass equals that over the entire unit ball. This allows us to apply Talagrand's inequality to the countable class \(\bar{B}(0,\,1)\) without measurability issues, while still controlling the supremum over the whole ball \(B(0,\,1)\).

To prove that the mapping \(f \mapsto \nu_n(f)\) is continuous on \(\bar{B}(0,\,1)\subset L^2(\mathbb R^d)\) equipped with the \(L^2\)-norm, we proceed as follows.
First, note that \(\nu_n(f)\) is linear in \(f\), since \(\psi_f\) is linear in \(f\) as seen from its integral definition. Thus, it suffices to show that \(\nu_n\) is bounded on the unit ball of \(L^2(\mathbb R^d)\), \emph{i.e.}, there exists a constant \(B<\infty\) such that for all \(f\in B(0,\,1)\), we have \(|\nu_n(f)|\le B\). Then, linearity combined with boundedness implies continuity.
Using the representation \eqref{eq:termproc} and the Cauchy--Schwarz inequality, for any \(\mathbf{y}\in\mathbb R^d\),
\[
|\psi_f(\mathbf{y})|
\le \|f\|_{L^2(\mathbb R^d)}
\bigg(
\int_{\mathbb R^d}
\prod_{j=1}^d
\left| G_{j,h_j}\!\left(\frac{v_j-y_j}{h_j}\right) \right|^2
\,\mathrm{d}\mathbf{v}
\bigg)^{1/2}
=
\|f\|_{L^2(\mathbb R^d)}
\prod_{j=1}^d
\bigg(
\int_{\mathbb R}
\left| G_{j,h_j}\!\left(\frac{u_j}{h_j}\right) \right|^2
\,\mathrm{d}u_j
\bigg)^{1/2},
\]
where we used the change of variables \(u_j=v_j-y_j\) and the product structure. The last expression is independent of \(\mathbf{y}\). Therefore,
\[
\sup_{\mathbf{y}\in\mathbb R^d} |\psi_f(\mathbf{y})|
\le \|f\|_{L^2(\mathbb R^d)} \prod_{j=1}^d \|G_{j,h_j}\|_{L^2(\mathbb R)},
\]
where each \(\|G_{j,h_j}\|_{L^2(\mathbb R)}<\infty\) under the assumptions on \(K_j\) as in Proposition~\ref{prop:variance}. Consequently, since \(\|f\|_{L^2(\mathbb R^d)}=1\),
\[\begin{split}
|\nu_n(f)|
&\le \frac{1}{n}\sum_{i=1}^n
[|\psi_f(\mathbf{Y}_i)| + \mathsf E[|\psi_f(\mathbf{Y}_1)|]]\\
&\le 2\|f\|_{L^2(\mathbb R^d)} \prod_{j=1}^d \|G_{j,h_j}\|_{L^2(\mathbb R)}\\
&=2\prod_{j=1}^d \|G_{j,h_j}\|_{L^2(\mathbb R)}\eqqcolon B<\infty.
\end{split}
\]
Thus, \(\nu_n\) is bounded on \(B(0,\,1)\) and, by linearity, it is continuous. Since \(\bar{B}(0,\,1)\) is dense in \(B(0,\,1)\), the continuity extends to the closure, and the supremum over the dense subset equals the supremum over the whole ball.

We now compute the constants $\Gamma$, $M$ and $\upsilon$ required in Lemma~\ref{lemma:talagrand}.
\begin{enumerate}[label=\bf{$\bullet$}]
    \item \emph{Evaluation of the constant $\Gamma$.} Note that 
\begin{eqnarray}\label{eq:taldiseq}
\|\widehat F_{n,\mathbf{h}} - (K_{\mathbf{h}} \ast F_X)\|_{L^2(\mathbb R^d)}
&=& \sup_{f \in B(0,1)} |\langle \widehat F_{n,\mathbf{h}} - (K_{\mathbf{h}} \ast F_X),\, f \rangle_2|
\nonumber\\&=& \sup_{f \in \bar{B}(0,1)} |\langle \widehat F_{n,\mathbf{h}} - (K_{\mathbf{h}} \ast F_X),\, f \rangle_2|=\sup_{f \in \bar{B}(0,1)} |\nu_n(f)|.
\end{eqnarray}
Then, by Jensen's inequality and using the identity $ \mathsf{E}_{(\mu_X\ast\mu_\varepsilon)^{\otimes n}}[\widehat F_{n,\mathbf{h}}(\mathbf{x})]=(K_{\mathbf{h}} \ast F_X)(\mathbf{x})$, which holds pointwise for every $\mathbf{x} \in \mathbb{R}^d$ as established in \eqref{eq:bias} \hyperref[eq:bias]{of the Main Article}, we proceed analogously to Proposition~\ref{prop:variance} to obtain 
\begin{align}\label{eq:ineqproc}
&\Bigg(\mathsf E_{(\mu_X\ast\mu_\varepsilon)^{\otimes n}}\Bigg[\sup_{f\in\bar{B}(0,1)}|\nu_n(f)|\Bigg]\Bigg)^2\\
&\hspace*{4cm}\leq 
\mathsf E_{(\mu_X\ast\mu_\varepsilon)^{\otimes n}}\Bigg[\sup_{f\in\bar{B}(0,1)}|\nu_n(f)|^2\Bigg]\nonumber\\
&\hspace*{4cm}=
\mathsf E_{(\mu_X\ast\mu_\varepsilon)^{\otimes n}}\Bigg[\sup_{f\in\bar{B}(0,1)}|\langle \widehat F_{n,\mathbf{h}}-(K_{\mathbf{h}}\ast F_X),\, f \rangle_2|^2\Bigg]\nonumber\\
&\hspace*{4cm}\leq\mathsf E_{(\mu_X\ast\mu_\varepsilon)^{\otimes n}}\Bigg[\sup_{f\in\bar{B}(0,1)}
\Big(\|\widehat F_{n,\mathbf{h}}-(K_{\mathbf{h}}\ast F_X)\|_{L^2(\mathbb R^d)}^2 \|f\|_{L^2(\mathbb R^d)}^2\Big)\Bigg]\nonumber\\
&\hspace*{4cm}=\mathsf E_{(\mu_X\ast\mu_\varepsilon)^{\otimes n}}\Bigl[
\|\widehat F_{n,\mathbf{h}}-(K_{\mathbf{h}}\ast F_X)\|_{L^2(\mathbb R^d)}^2\Bigr]\nonumber\\
&\hspace*{4cm}=\mathsf E_{(\mu_X\ast\mu_\varepsilon)^{\otimes n}}\Bigl[
\|\widehat F_{n,\mathbf{h}}-\mathsf E_{(\mu_X\ast\mu_\varepsilon)^{\otimes n}}[\widehat F_{n,\mathbf{h}}]\|_{L^2(\mathbb R^d)}^2\Bigr]\nonumber\\
&\hspace*{4cm}=\int_{\mathbb R^d}\mathsf E_{(\mu_X\ast\mu_\varepsilon)^{\otimes n}}\bigl[
|\widehat F_{n,\mathbf{h}}(\mathbf{x})-\mathsf E_{(\mu_X\ast\mu_\varepsilon)^{\otimes n}}[\widehat F_{n,\mathbf{h}}(\mathbf{x})]|^2\bigr]\,\di \mathbf{x}\nonumber\\
&\hspace*{4cm}=\int_{\mathbb R^d}
\operatorname{Var}[\widehat F_{n,\mathbf{h}}(\mathbf{x})]
\,\di \mathbf{x}\nonumber\\
&\hspace*{4cm}\lesssim \mathcal V^2(\mathbf{h}).
\end{align}
We set $C_1(\mathbf{h})\coloneqq C_1 \prod_{j=1}^d |\log h_j|^{\mathbf{1}_{\{\beta_j=1/2\}}/2}$,
where \(C_1^2>0\) is the multiplicative constant omitted in \eqref{eq:ineqproc}. We obtain
\[
\Bigg(\mathsf{E}_{(\mu_X\ast\mu_\varepsilon)^{\otimes n}}\Bigg[\sup_{f\in\bar{B}(0,1)} |\nu_n(f)|\Bigg]\Bigg)^2 \leq [C_1(\mathbf{h}) V(\mathbf{h})]^2,
\]
and, for any function $C_0(\mathbf{h})\geq1$, 
\[
\mathsf{E}_{(\mu_X\ast\mu_\varepsilon)^{\otimes n}}\Bigg[\sup_{f\in\bar{B}(0,1)} |\nu_n(f)|\Bigg] \leq C_1(\mathbf{h}) V(\mathbf{h})\leq C_0(\mathbf{h})C_1(\mathbf{h}) V(\mathbf{h}).
\]
Define 
\begin{equation*}\label{eq:defH}
\Gamma \coloneqq C_0(\mathbf{h})C_1(\mathbf{h}) V(\mathbf{h}).
\end{equation*}
Conclude that $$\mathsf{E}_{(\mu_X\ast\mu_\varepsilon)^{\otimes n}}\Bigg[\sup_{f\in\bar{B}(0,1)} |\nu_n(f)|\Bigg]
\leq \Gamma.$$
\item \emph{Evaluation of the constant $M$.} Using the expression for $\psi_f$ given in \eqref{eq:termproc} and applying the Cauchy--Schwarz inequality, we reason as in Proposition~\ref{prop:variance} to obtain 
\begin{align}\label{eq:constsupnorm}
\sup_{f\in\bar{B}(0,1)}\|\psi_f\|_\infty&=
\sup_{f\in\bar{B}(0,1)}
\sup_{\mathbf{y}\in\mathbb R^d}
|\psi_f(\mathbf{y})|\nonumber\\
&=\sup_{f\in\bar{B}(0,1)}
\sup_{\mathbf{y}\in\mathbb R^d}
\biggl|\int_{\mathbb R^d}\overline{f(\mathbf{v})}\prod_{j=1}^dG_{j,h_j}\bigg(\frac{v_j-y_j}{h_j}\bigg)\,\di \mathbf{v}\biggr|\nonumber\\
&\leq\biggl(\sup_{f\in\bar{B}(0,1)}\|f\|_{L^2(\mathbb R^d)}\biggr)
\sup_{\mathbf{y}\in\mathbb R^d}
\Biggl(\int_{\mathbb R^d}\prod_{j=1}^d\biggl|G_{j,h_j}\bigg(\frac{v_j-y_j}{h_j}\bigg)\biggr|^2\,\di \mathbf{v}\Biggr)^{1/2}\nonumber\\
&=\prod_{j=1}^d\Bigg(\sup_{y_j\in\mathbb R}\int_{\mathbb R}\bigg|G_{j,h_j}\bigg(\frac{v_j-y_j}{h_j}\bigg)\bigg|^2\di v_j\Bigg)^{1/2}\nonumber\\
&\lesssim 
\sqrt{n}\,\mathcal V(\mathbf{h}).
\end{align}
Let $C_2(\mathbf{h})\coloneqq C_2 \prod_{j=1}^d |\log h_j|^{\mathbf{1}_{\{\beta_j=1/2\}}/2}$, where $C_2$ is the multiplicative constant omitted in \eqref{eq:constsupnorm}. We obtain
\[
\sup_{f\in\bar{B}(0,1)}\|\psi_f\|_\infty \leq C_2(\mathbf{h}) \sqrt{n}V(\mathbf{h}).
\]
We can therefore define $M\coloneqq C_2(\mathbf{h})\sqrt{n}V(\mathbf{h})$ and conclude that 
\[
\sup_{f\in\bar{B}(0,1)}\|\psi_f\|_\infty \leq M.
\]
\item \emph{Evaluation of the constant $\upsilon$.} Following a similar approach to the one used to identify the constant $M$, we obtain 
\begin{align}\label{eq:disvar}
\sup_{f\in\bar{B}(0,1)}\operatorname{Var}
[\psi_f(\mathbf{Y}_1)]&\leq
\sup_{f\in\bar{B}(0,1)}\mathsf E_{(\mu_X\ast\mu_\varepsilon)}[|\psi_f(\mathbf{Y}_1)|^2]\nonumber\\
&\leq\biggl(\sup_{f\in\bar{B}(0,1)}\|f\|_{L^2(\mathbb R^d)}^2\biggr)
\mathsf E_{(\mu_X\ast\mu_\varepsilon)}\Biggl[\prod_{j=1}^d\int_{\mathbb R}\biggl|G_{j,h_j}\biggl(\frac{v_j-Y_{1,j}}{h_j}\biggr)\biggr|^2\di v_j\Biggr]\nonumber\\
&=\mathsf E_{(\mu_X\ast\mu_\varepsilon)}\Biggl[\prod_{j=1}^d\int_{\mathbb R}\biggl|G_{j,h_j}\biggl(\frac{v_j-Y_{1,j}}{h_j}\biggr)\biggr|^2\di v_j\Biggr]\nonumber\\
&\leq\prod_{j=1}^d\biggl(\sup_{y_j\in\mathbb R}\int_{\mathbb R}\biggl|G_{j,h_j}\biggl(\frac{v_j-y_j}{h_j}\biggr)\biggr|^2\di v_j\biggr)\nonumber\\
&\lesssim n\, \mathcal V^2(\mathbf{h}).
\end{align}
Set $C_3(\mathbf{h}) \coloneqq C_3\prod_{j=1}^d |\log h_j|^{\mathbf{1}_{\{\beta_j=1/2\}}/2}$,
where $C_3^2>0$ is the multiplicative constant omitted in \eqref{eq:disvar}. We obtain
\[
\sup_{f\in\bar{B}(0,1)}\operatorname{Var}
[\psi_f(\mathbf{Y}_1)] \leq C^2_3(\mathbf{h}) nV^2(\mathbf{h}).
\]
Hence, we define 
\begin{equation}\label{eq:defupsilon}
\upsilon\coloneqq C^2_3(\mathbf{h})nV^2(\mathbf{h})=C_3^2n \,\mathcal V^2(\mathbf{h}).
\end{equation}
Before proceeding, we note that the quantities $\Gamma$, $M$, and $\upsilon$ defined above all depend on the bandwidth vector $\mathbf{h}$ through $V(\mathbf{h})$, $C_0(\mathbf{h})$, $C_1(\mathbf{h})$, $C_2(\mathbf{h})$, and $C_3(\mathbf{h})$. To ease the notation, we suppress this dependence and write $\Gamma$, $M$, and $\upsilon$ instead of $\Gamma(\mathbf{h})$, $M(\mathbf{h})$, and $\upsilon(\mathbf{h})$ in what follows. However, it should be kept in mind that these quantities do depend on $\mathbf{h}$; this dependence will always be clear from the context and is essential for the subsequent application of Talagrand's inequality.
\item \emph{Application of Talagrand's inequality.}
Having evaluated all the constants $\Gamma,\,M,\,\upsilon$, and recalling from \eqref{eq:taldiseq} that
$\sup_{f\in\bar{B}(0,1)}|\nu_n(f)|=\|\widehat F_{n,\mathbf{h}}-(K_{\mathbf{h}}\ast F_X)\|_{L^2(\mathbb R^d)}$, Talagrand's inequality \eqref{eq:Talineq} yields, for any $\eta>0$,
\begin{align*}
&\mathsf{E}\Big[\Big(\|\widehat F_{n,\mathbf{h}}-(K_{\mathbf{h}}\ast F_X)\|_{L^2(\mathbb R^d)} -(1+2\eta) C_0(\mathbf{h})C_1(\mathbf{h})V(\mathbf{h})\Big)_+\Big]\\
&\,\,\,\,\,\,\,\,\,\,\,\,\le \sqrt{\frac{3\pi}{2}} \sqrt{\frac{\upsilon}{n}} \,
\exp\bigg(- \frac{\eta^{2}C_1^2}{6C_3^2}C^2_0(\mathbf{h})\bigg)
+ \frac{21}{\eta\wedge 1} C_2(\mathbf{h})\frac{V(\mathbf{h})}{\sqrt{n}}
\exp\bigg(- \frac{\eta (\eta\wedge 1)}{21} \frac{C_0(\mathbf{h})C_1(\mathbf{h})}{C_2(\mathbf{h})}\sqrt{n}\bigg)\\
&\,\,\,\,\,\,\,\,\,\,\,\,\leq\sqrt{\frac{3\pi}{2}} \sqrt{\frac{\upsilon}{n}} \,
\exp\bigg(- \frac{\eta^{2}C_1^2}{6C_3^2}C^2_0(\mathbf{h})\bigg)+ \frac{21}{\eta\wedge 1} C_2\frac{\mathcal V(\mathbf{h})}{\sqrt{n}}\exp\bigg(- \frac{\eta (\eta\wedge 1)C_1}{21C_2} \sqrt{n}\bigg),
\end{align*}
where in the last inequality we used \(C_0(\mathbf{h})\ge 1\) in the exponent of the second addendum.
\end{enumerate}
The proof is complete.
\end{proof}

Proposition~\ref{prop:oracle} establishes a sharp fixed-bandwidth concentration inequality for the kernel-based estimator $\widehat F_{n,\mathbf{h}}$. It controls the expected positive part of the discrepancy between 
\(\|\widehat F_{n,\mathbf{h}}-(K_{\mathbf{h}}\ast F_X)\|_{L^2(\mathbb R^d)}\) and the deterministic proxy
$(1+2\eta)C_0(\mathbf{h})C_1 \mathcal V(\mathbf{h})$, which turns out to be equal to $(1+2\eta) C_0(\mathbf{h}) C_1V(\mathbf{h})
\prod_{j=1}^d |\log h_j|^{\mathbf{1}_{\{\beta_j = 1/2\}}/2}$,
where \(V(\mathbf{h})\) is the leading variance term, while the logarithmic factors arise from the boundary case \(\beta_j=1/2\). The right-hand side of the concentration inequality consists of two exponential terms: the first decays like \(\exp(-c_1 \eta^2 C_0(\mathbf{h}))\), reflecting the Gaussian-like concentration of the empirical process around its expectation; the second decays like \(\exp(-c_2 \eta \sqrt{n})\), which corresponds to the Bernstein-type tail of the supremum of the empirical process. The inclusion of the arbitrary sequence \(C_0(\mathbf{h})\ge 1\) is a flexibility tool, allowing us to later tune its choice to satisfy summability constraints on the grid $\mathcal H_{M_n}$ when proving the final oracle inequality. 
Crucially, increasing \(C_0(\mathbf{h})\) inflates the variance proxy on the left-hand side but sharpens the exponential rate of the first tail term on the right. It is this trade-off that underpins the bandwidth selection procedure and enables it to achieve adaptive minimax rates. Indeed, by minimizing the penalized criterion
\(
D(\mathbf{h}) + \tilde V(\mathbf{h})
\) over $\mathbf{h}$, one balances the empirical bias proxy \(D(\mathbf{h})\) against the penalty term $\tilde V(\mathbf{h})$, which is proportional to the standard deviation proxy \(\mathcal V(\mathbf{h})\). The concentration bound above then guarantees that the selected bandwidth performs as well as the best fixed bandwidth, up to universal constants.

The following theorem establishes an oracle inequality for the adaptive $\varrho_n$-approximate minimum-$L^2$-distance estimator $\widehat F_{n,\hat{\mathbf{h}}}^\ast$. The key steps of the proof are: 
reduction to the adaptive kernel-based estimator $\check{F}_n = \widehat F_{n,\hat{\mathbf{h}}}$,
an error decomposition using the convolution symmetry of the kernel,
bounding the empirical bias proxy $D(\mathbf{h})$ via the concentration inequality,
proving summability over the anisotropic grid $\mathcal H_{M_n}$, and
transferring the resulting bound to the $\varrho_n$-approximate minimum-$L^2$-distance estimator via a projection argument.
The oracle inequality shows that the adaptive estimator attains the optimal bias-variance trade-off up to universal constants and a negligible remainder term.

\begin{proof}[Proof of Theorem~\ref{thm:oracle_ineq} (Oracle inequality)]
The proof is split into steps.
\begin{steps}
\item \emph{Reduction to the adaptive kernel-based estimator.}\label{st:reduc}\\[2pt]
Recall the definition of the adaptive kernel-based estimator
\[
\check{F}_n \coloneqq \widehat{F}_{n,\hat{\mathbf{h}}}, \quad \text{with }\,
\hat{\mathbf{h}} := \argmin_{\mathbf{h} \in \mathcal H_{M_n}} \big\{D(\mathbf{h}) + \tilde V(\mathbf{h}) \big\},
\]
where \(D(\mathbf{h})\) and \(\tilde V(\mathbf{h})\) are defined in \eqref{eq:Ah} and \eqref{eq:proxyvar} \hyperref[eq:Ah]{of the Main Article} (see also \eqref{eq:defvtilde}), respectively.

For an arbitrary bandwidth vector \(\mathbf{h} \in \mathcal H_{M_n}\), we will show that the \(L^2\)-error of \(\check{F}_n\) is bounded, up to universal constants, by the oracle risk associated with this fixed \(\mathbf{h}\), plus a negligible remainder. Since \(\mathbf{h}\in\mathcal H_{M_n}\) is arbitrary, this will allow us to take the infimum over all \(\mathbf{h}\in\mathcal H_{M_n}\) at the end of the proof.

By the same arguments used to obtain~\eqref{eq:stimdist} in Theorem~\ref{thm:ubrate}, for every \(\mathbf{h}\in\mathcal H_{M_n}\), the minimum-$L^2$-distance estimator \(\widehat F_{n,\mathbf{h}}^\ast\) satisfies
\[
\|\widehat F_{n,\mathbf{h}}^\ast - F_X\|_{L^2(\mathbb R^d)}
\leq 2\|\widehat F_{n,\mathbf{h}} - F_X\|_{L^2(\mathbb R^d)} + \varrho_n.
\]
In particular, choosing \(\mathbf{h} = \hat{\mathbf{h}}\) yields (see also \eqref{eq:reductionstep} \hyperref[eq:reductionstep]{of the Main Article})
\[
\|\widehat F_{n,\hat{\mathbf{h}}}^\ast - F_X\|_{L^2(\mathbb R^d)}
\leq 2\|\check{F}_n - F_X\|_{L^2(\mathbb R^d)} + \varrho_n.
\]
Therefore, it suffices to prove an oracle inequality for the adaptive kernel-based estimator \(\check{F}_n\). This is the goal of the following steps.
\item \emph{Error decomposition via convolution symmetry.}\label{st:decomp}\\[2pt]
For every pair of bandwidth vectors $\mathbf{h},\,\mathbf{h}'\in\mathbb R_+^d$, the following convolution identity holds.\footnote{Recall that for $\mathbf{a}=(a_1,\,\dots,\,a_d)^{\!\top}$ and $\mathbf{t}=(t_1,\,\dots,\,t_d)^{\!\top}$, 
we denote by $\mathbf{a}\odot\mathbf{t}$ the componentwise (Hadamard) product 
$
\mathbf{a}\odot\mathbf{t}=(a_1t_1,\,\dots,\,a_dt_d)^{\!\top}$.} For each $\mathbf{x}\in\mathbb R^d$,
\begin{align}\label{eq:identityhh'}
(K_{\mathbf{h}'}\ast \widehat F_{n,\mathbf{h}})(\mathbf{x})
&= \biggl(K_{\mathbf{h}'}\ast \mathcal F^{-1}\biggl[\frac{\varphi_n(\mathbf{t})\varphi_K(\mathbf{h}\odot \mathbf{t})}{\varphi_\varepsilon(\mathbf{t})}\biggr]\biggr)(\mathbf{x}) \nonumber\\
&= \mathcal F^{-1}\bigg[\frac{\varphi_K(\mathbf{h}'\odot \mathbf{t})\varphi_n(\mathbf{t})\varphi_K(\mathbf{h}\odot \mathbf{t})}{\varphi_\varepsilon(\mathbf{t})}\bigg](\mathbf{x}) \nonumber\\
&= \bigg(K_{\mathbf{h}}\ast \mathcal F^{-1}\bigg[\frac{\varphi_n(\mathbf{t})\varphi_K(\mathbf{h}'\odot \mathbf{t})}{\varphi_\varepsilon(\mathbf{t})}\bigg]\bigg)(\mathbf{x}) = (K_{\mathbf{h}}\ast \widehat F_{n,\mathbf{h}'})(\mathbf{x}).
\end{align}
For any fixed $\mathbf{h}\in\mathcal H_{M_n}$, the triangle inequality gives
\begin{align*}
\|\check{F}_n-F_X\|_{L^2(\mathbb R^d)} &\leq \underbrace{\|\check{F}_n - (K_{\hat{\mathbf{h}}}\ast \widehat F_{n,\mathbf{h}})\|_{L^2(\mathbb R^d)}}_{=:D_3}+\underbrace{\|(K_{\hat{\mathbf{h}}}\ast \widehat F_{n,\mathbf{h}})-\widehat F_{n,\mathbf{h}}\|_{L^2(\mathbb R^d)}}_{=:D_4}+\|\widehat F_{n,\mathbf{h}} - F_X\|_{L^2(\mathbb R^d)}\\
&=D_3+D_4+\|\widehat F_{n,\mathbf{h}} - F_X\|_{L^2(\mathbb R^d)}.
\end{align*}
\item \emph{Bounding $D_3$ and $D_4$ by the oracle discrepancy.}\label{st:oracle}\\[2pt]
By adding and subtracting $\tilde V({\hat{\mathbf{h}}})$ and by definition of $D(\mathbf{h})$ in \eqref{eq:Ah} \hyperref[eq:Ah]{of the Main Article}), one obtains
 \begin{align*}
D_3\coloneqq\|\check{F}_n-(K_{\hat{\mathbf{h}}}\ast \widehat F_{n,\mathbf{h}})\|_{L^2(\mathbb R^d)}&=\bigl[\|\check{F}_n-(K_{\hat{\mathbf{h}}}\ast \widehat F_{n,\mathbf{h}})\|_{L^2(\mathbb R^d)}-\tilde V(\hat{\mathbf{h}})\bigr]+\tilde V(\hat{\mathbf{h}})\leq D(\mathbf{h})+\tilde V({\hat{\mathbf{h}}}).
\end{align*}
Analogously, using the symmetry in \eqref{eq:identityhh'} 
and the definition of $D(\hat{\mathbf{h}})$,
\begin{align*}
D_4\coloneqq\|(K_{\hat{\mathbf{h}}}\ast \widehat{F}_{n,\mathbf{h}})-\widehat F_{n,\mathbf{h}}\|_{L^2(\mathbb R^d)}
&=\bigl[\|(K_{\hat{\mathbf{h}}}\ast \widehat{F}_{n,\mathbf{h}})-\widehat F_{n,\mathbf{h}}\|_{L^2(\mathbb R^d)}-\tilde  V(\mathbf{h})\bigr]+\tilde V(\mathbf{h})\\
&=\bigl[\|(K_{{\mathbf{h}}}\ast \widehat{F}_{n,\hat{\mathbf{h}}})-\widehat F_{n,\mathbf{h}}\|_{L^2(\mathbb R^d)}-\tilde V(\mathbf{h})\bigr]+\tilde V(\mathbf{h})\\
&\leq D(\hat{\mathbf{h}})+\tilde V(\mathbf{h}).
\end{align*}
Thus,
\[D_4\leq D(\hat{\mathbf{h}})+\tilde V(\mathbf{h}).\]
Combining the upper bounds on $D_3$ and $D_4$, and recalling the definition of $\hat{\mathbf{h}}$ in \eqref{eq:hopt} \hyperref[eq:hopt]{of the Main Article}, we have 
\begin{align*}
D_3+D_4&\leq D(\mathbf{h})+\tilde V(\hat{\mathbf{h}})+D(\hat{\mathbf{h}})+\tilde V(\mathbf{h})\\
&= \bigl[D(\mathbf{h})+\tilde V(\mathbf{h})\bigr]+ \bigl[D(\hat{\mathbf{h}})+\tilde V(\hat{\mathbf{h}})\bigr]\\
&\leq \bigl[D(\mathbf{h})+\tilde V(\mathbf{h})\bigr]+ \bigl[D(\mathbf{h})+\tilde V(\mathbf{h})\bigr]=2\big[D(\mathbf{h})+\tilde V(\mathbf{h})\big].
\end{align*}
\item \emph{Bias-variance decomposition for the kernel estimator.}\label{st:bias-variance}\\[2pt]
Recalling that $B_{\mathbf{h}}\coloneqq F_X-(K_{\mathbf{h}}\ast F_X)$, the $L^2$-error of the estimator $\widehat F_{n,\mathbf{h}}$ is decomposed via the triangle inequality as
\begin{align*}
\|\widehat F_{n,\mathbf{h}}-F_X\|_{L^2(\mathbb R^d)}
&\leq \|B_{\mathbf{h}}\|_{L^2(\mathbb R^d)}+
\|\widehat F_{n,\mathbf{h}}-\mathsf E_{(\mu_X\ast\mu_\varepsilon)^{\otimes n}}[\widehat F_{n,\mathbf{h}}]\|_{L^2(\mathbb R^d)}.
\end{align*}
Using the reverse Jensen inequality $\mathsf E[\sqrt{X}]\leq \sqrt{\mathsf E[X]}$ for a non-negative random variable $X$ because \(x \mapsto \sqrt{x}\) is concave, and Proposition \ref{prop:variance}, we obtain 
\begin{align*}
\mathsf E_{(\mu_X\ast\mu_\varepsilon)^{\otimes n}}[\|\widehat F_{n,\mathbf{h}}-\mathsf E_{(\mu_X\ast\mu_\varepsilon)^{\otimes n}}[\widehat F_{n,\mathbf{h}}]\|_{L^2(\mathbb R^d)}]\leq \int_{\mathbb R^d}\sqrt{\operatorname{Var}[\widehat F_{n,\mathbf{h}}(\mathbf{x})]}\,\di \mathbf{x}
\lesssim \mathcal V(\mathbf{h}).
\end{align*}
Consequently, from the error decomposition via convolution symmetry and the oracle discrepancy bounds on $D_3$ and $D_4$, for a suitable constant $C>0$,
\begin{align}\label{eq:return}
\mathsf E_{(\mu_X\ast\mu_\varepsilon)^{\otimes n}}[\|\check{F}_n-F_X\|_{L^2(\mathbb R^d)}]
&\leq 2\mathsf E_{(\mu_X\ast\mu_\varepsilon)^{\otimes n}}[D(\mathbf{h})]\nonumber\\&\hspace*{2cm}+
\|B_{\mathbf{h}}\|_{L^2(\mathbb R^d)}+2\tilde V(\mathbf{h})+C
\mathcal V(\mathbf{h}).
\end{align}
\item \emph{Study of the oracle discrepancy $D(\mathbf{h})$.}\label{st:ordisc} \\[2pt]
We first derive an upper bound on $\mathsf E_{(\mu_X\ast\mu_\varepsilon)^{\otimes n}}{[D(\mathbf{h})]}$ 
using Proposition~\ref{prop:oracle}. To this end, observe that
\begin{align*}
\|\widehat F_{n,\mathbf{h}'} - (K_{\mathbf{h}'}\ast \widehat F_{n,\mathbf{h}})\|_{L^2(\mathbb R^d)}
&\leq \|\widehat F_{n,\mathbf{h}'} - (K_{\mathbf{h}'}\ast F_X)\|_{L^2(\mathbb R^d)}\\
&\hspace*{3cm}+ \underbrace{\|K_{\mathbf{h}'}\ast \widehat F_{n,\mathbf{h}} - (K_{\mathbf{h}'}\ast F_X\ast K_{\mathbf{h}})\|_{L^2(\mathbb R^d)}}_{\eqqcolon T_1}\\ &\hspace*{3cm}+ \underbrace{\|(K_{\mathbf{h}'}\ast F_X) - (K_{\mathbf{h}'}\ast F_X\ast K_{\mathbf{h}})\|_{L^2(\mathbb R^d)}}_{\eqqcolon T_2}\\
&=\|\widehat F_{n,\mathbf{h}'} - \mathsf E_{(\mu_X\ast\mu_\varepsilon)^{\otimes n}}[\widehat F_{n,\mathbf{h}'}]\|_{L^2(\mathbb R^d)}+T_1+T_2.
\end{align*}
For the term $T_1$, using the symmetry \eqref{eq:identityhh'}, we obtain
\begin{align}\label{eq:term1}
T_1\coloneqq\|K_{\mathbf{h}'}\ast \widehat F_{n,\mathbf{h}} - (K_{\mathbf{h}'}\ast F_X\ast K_{\mathbf{h}})\|_{L^2(\mathbb R^d)}
&= \|K_{\mathbf{h}}\ast \widehat F_{n,\mathbf{h}'} - K_{\mathbf{h}}\ast (K_{\mathbf{h}'}\ast F_X)\|_{L^2(\mathbb R^d)}\nonumber\\
&= \|K_{\mathbf{h}}\ast [\widehat F_{n,\mathbf{h}'} - (K_{\mathbf{h}'}\ast F_X)]\|_{L^2(\mathbb R^d)} \nonumber\\
&\leq \|K_{\mathbf{h}}\|_{L^1(\mathbb R^d)} \|\widehat F_{n,\mathbf{h}'} - (K_{\mathbf{h}'}\ast F_X)\|_{L^2(\mathbb R^d)} \nonumber\\
&= C_K \|\widehat F_{n,\mathbf{h}'} -\mathsf E_{(\mu_X\ast\mu_\varepsilon)^{\otimes n}}[\widehat F_{n,\mathbf{h}'}]\|_{L^2(\mathbb R^d)},
\end{align}
where $$C_K\coloneqq \prod_{j=1}^d\|K_j\|_{L^1(\mathbb R)}<\infty\quad\text{(recall assumption~(i) \hyperref[prop:variance]{of the main Article} on $K_j$).}$$ 
For the term $T_2$, we have
\begin{align}\label{eq:term2}
T_2\coloneqq\|(K_{\mathbf{h}'}\ast F_X) - (K_{\mathbf{h}'}\ast F_X\ast K_{\mathbf{h}})\|_{L^2(\mathbb R^d)}
&= \|K_{\mathbf{h}'}\ast [F_X - (K_{\mathbf{h}}\ast F_X)]\|_{L^2(\mathbb R^d)} \nonumber\\
&\leq \|K_{\mathbf{h}'}\|_{L^1(\mathbb R^d)} \|B_{\mathbf{h}}\|_{L^2(\mathbb R^d)}.
\end{align}
Combining the bounds \eqref{eq:term1} and \eqref{eq:term2} yields
\begin{align*}
\|\widehat F_{n,\mathbf{h}'} - (K_{\mathbf{h}'}\ast \widehat F_{n,\mathbf{h}})\|_{L^2(\mathbb R^d)}
&\leq (1 +C_K) \Big[\|\widehat F_{n,\mathbf{h}'} - \mathsf E_{(\mu_X\ast\mu_\varepsilon)^{\otimes n}}[\widehat F_{n,\mathbf{h}'}]\|_{L^2(\mathbb R^d)}\\
&\hspace*{5cm}+ \frac{C_K}{1+C_K}\|B_{\mathbf{h}}\|_{L^2(\mathbb R^d)}\Big].
\end{align*}
Since the function $f(x)=\max\{0,\,x\}$ is non-decreasing, substituting this bound into the definition of $D(\mathbf{h})$ in \eqref{eq:Ah} \hyperref[eq:Ah]{of the Main Article}, we obtain
\begin{align*}
D(\mathbf{h})&\coloneqq\sup_{\mathbf{h}'\in\mathcal H_{M_n}}
\Bigl(\|\widehat F_{n,\mathbf{h}'}-(K_{\mathbf{h}'}\ast\widehat F_{n,\mathbf{h}})\|_{L^2(\mathbb R^d)}- \tilde V(\mathbf{h}')\Bigr)_+\\
&\leq \sup_{\mathbf{h}'\in\mathcal H_{M_n}}\Bigl((1+C_K) \|\widehat F_{n,\mathbf{h}'}-\mathsf E_{(\mu_X\ast\mu_\varepsilon)^{\otimes n}}[\widehat F_{n,\mathbf{h}'}]\|_{L^2(\mathbb R^d)}-\tilde V(\mathbf{h}')\Bigr)_+
+C_K\|B_{\mathbf{h}}\|_{L^2(\mathbb R^d)}\\
&=(1 + C_K)
\bigg[\sup_{\mathbf{h}'\in\mathcal H_{M_n}} \bigg(\|\widehat F_{n,\mathbf{h}'}-\mathsf E_{(\mu_X\ast\mu_\varepsilon)^{\otimes n}}[\widehat F_{n,\mathbf{h}'}]\|_{L^2(\mathbb R^d)}
-\frac{\tilde V(\mathbf{h}')}{1+C_K}\bigg)_+\\
&\hspace*{8.5cm}+\frac{C_K}{1+C_K}\|B_{\mathbf{h}}\|_{L^2(\mathbb R^d)}\bigg].
\end{align*}
Thus,
\begin{equation*}
D(\mathbf{h})\leq(1 + C_K)
\bigg[\sup_{\mathbf{h}'\in\mathcal H_{M_n}} \bigg(\|\widehat F_{n,\mathbf{h}'}-\mathsf E_{(\mu_X\ast\mu_\varepsilon)^{\otimes n}}[\widehat F_{n,\mathbf{h}'}]\|_{L^2(\mathbb R^d)}
-\frac{\tilde V(\mathbf{h}')}{1+C_K}\bigg)_+
+\frac{C_K}{1+C_K}\|B_{\mathbf{h}}\|_{L^2(\mathbb R^d)}\bigg]
\end{equation*}
and 
\begin{align}\label{eq:abound}
\mathsf E_{(\mu_X\ast\mu_\varepsilon)^{\otimes n}}[D(\mathbf{h})]
&\leq (1 + C_K)\nonumber\\
&\,\,\,\,\,
\times\bigg\{
\mathsf E_{(\mu_X\ast\mu_\varepsilon)^{\otimes n}}
\bigg[
\sup_{\mathbf{h}'\in\mathcal H_{M_n}}
\bigg(
\|\widehat F_{n,\mathbf{h}'}-\mathsf E_{(\mu_X\ast\mu_\varepsilon)^{\otimes n}}[\widehat F_{n,\mathbf{h}'}]\|_{L^2(\mathbb R^d)}
- \frac{\tilde V(\mathbf{h}')}{1+C_K}
\bigg)_+
\bigg] \nonumber \\
&\hspace*{5.5cm}
 + \frac{C_K}{1+C_K}\|B_{\mathbf{h}}\|_{L^2(\mathbb R^d)}
\bigg\}.
\end{align}
\item \emph{Application of the concentration inequality.}\label{st:concentration}\\[2pt]
Let $C_V > 1$ be a universal constant to be specified below, and set
\[
\eta \coloneqq \frac{C_V - 1}{2} > 0.
\]
Note that $\eta$ does not depend on $\mathbf{h}'$.
For each $\mathbf{h}' \in \mathcal H_{M_n}$, recall from the proof of Proposition~\ref{prop:oracle} that
$
\Gamma \coloneqq C_0(\mathbf{h}') C_1(\mathbf{h}') V(\mathbf{h}')$,
where $C_0(\mathbf{h}') \ge 1$ will be chosen to ensure the summability over the grid, and with $C_1(\mathbf{h}') = C_1 \prod_{j=1}^d |\log h_j'|^{\mathbf{1}_{\{\beta_j=1/2\}}/2}$
from the proof of Proposition~\ref{prop:oracle}.
Define the variance penalty term 
\[
\tilde V^2(\mathbf{h}') \coloneqq C_V^2 (1+C_K)^2 \Gamma^2.
\]
Then, for every $\mathbf{h}' \in \mathcal H_{M_n}$,
\begin{equation}\label{eq:defvtilde}
\frac{\tilde V(\mathbf{h}')}{1+C_K}
= C_V \Gamma
= (1+2\eta)\Gamma.
\end{equation}
The purpose of this definition is to ensure that the penalty term 
$\tilde V(\mathbf{h}')/(1+C_K)$ in the definition of $D(\mathbf{h})$ 
matches the shifted expectation term $(1+2\eta)\Gamma$ from 
Talagrand's inequality (Lemma~\ref{lemma:talagrand}). This allows us to apply 
the concentration inequality uniformly over the grid $\mathcal H_{M_n}$. 
The constant $C_V > 1$ is chosen so that $\eta > 0$, which is required for 
the exponential decay in Talagrand's inequality.


Substituting the bound \eqref{eq:abound} into \eqref{eq:return} yields
\begin{align}\label{eq:finalbound}
&\mathsf E_{(\mu_X\ast\mu_\varepsilon)^{\otimes n}}[\|\check{F}_n-F_X\|_{L^2(\mathbb R^d)}]\nonumber\\
&\qquad\quad\leq 2(1+C_K)
\mathsf E_{(\mu_X\ast\mu_\varepsilon)^{\otimes n}}\bigg[\sup_{\mathbf{h}'\in\mathcal H_{M_n}} \bigg(\|\widehat F_{n,\mathbf{h}'}-\mathsf E_{(\mu_X\ast\mu_\varepsilon)^{\otimes n}}[\widehat F_{n,\mathbf{h}'}]\|_{L^2(\mathbb R^d)}
-\frac{\tilde V(\mathbf{h}')}{1+C_K}\bigg)_+\bigg]\nonumber\\
&\hspace*{5cm}+(1+2C_K)\|B_{\mathbf{h}}\|_{L^2(\mathbb R^d)}+2\tilde V(\mathbf{h})+C\mathcal V(\mathbf{h}).
\end{align}
Applying Proposition~\ref{prop:oracle}, we therefore obtain
\begin{align}\label{eq:meanA}
&\mathsf E_{(\mu_X\ast\mu_\varepsilon)^{\otimes n}}\bigg[\sup_{\mathbf{h}'\in\mathcal H_{M_n}}\bigg(\|\widehat F_{n,\mathbf{h}'}-\mathsf E_{(\mu_X\ast\mu_\varepsilon)^{\otimes n}}[\widehat F_{n,\mathbf{h}'}]\|_{L^2(\mathbb R^d)}
-\frac{\tilde V(\mathbf{h}')}{1+C_K}\bigg)_+\bigg]\nonumber\\
&\hspace{0.8cm}=
\mathsf E_{(\mu_X\ast\mu_\varepsilon)^{\otimes n}}\bigg[\sup_{\mathbf{h}'\in\mathcal H_{M_n}}\bigg(\|\widehat F_{n,\mathbf{h}'}-\mathsf E_{(\mu_X\ast\mu_\varepsilon)^{\otimes n}}[\widehat F_{n,\mathbf{h}'}]\|_{L^2(\mathbb R^d)}
-(1+2\eta) \Gamma\bigg)_+\bigg]\nonumber\\
&\hspace{0.8cm}\leq
\sum_{\mathbf{h}'\in\mathcal H_{M_n}}\mathsf E_{(\mu_X\ast\mu_\varepsilon)^{\otimes n}}\bigg[\bigg(\|\widehat F_{n,\mathbf{h}'}-\mathsf E_{(\mu_X\ast\mu_\varepsilon)^{\otimes n}}[\widehat F_{n,\mathbf{h}'}]\|_{L^2(\mathbb R^d)}
-(1+2\eta) \Gamma\bigg)_+\bigg]\nonumber\\
&\hspace{0.8cm}\leq\sqrt{\frac{3\pi}{2}}
\frac{1}{\sqrt{n}}\sum_{\mathbf{h}'\in\mathcal H_{M_n}}\sqrt{\upsilon}\exp\bigg(- \frac{\eta^{2}C_1^2}{6C_3^2}C^2_0(\mathbf{h}')\bigg)\nonumber\\&\hspace{3cm}+ \frac{21}{\eta\wedge 1} 
\frac{C_2}{\sqrt{n}}
\exp\bigg(- \frac{\eta (\eta\wedge 1)C_1}{21C_2} \sqrt{n}\bigg)
\sum_{\mathbf{h}'\in\mathcal H_{M_n}}\mathcal V(\mathbf{h}'),
\end{align}
where $\upsilon$ is defined in \eqref{eq:defupsilon}
and $C_1,\,C_2,\,C_3>0$ denote the constants introduced in the proof of Proposition~\ref{prop:oracle}. 
\item \emph{Summability over the discrete grid $\mathcal H_{M_n}$.}\label{st:summability}\\[2pt]
Proposition~\ref{prop:oracle} leaves the choice of $C_0(\mathbf{h})$ unspecified. We now prove that, for a suitable choice of $C_0(\mathbf{h}')$, the following summability bounds hold:
\begin{equation}\label{eq:bound1}
\sum_{\mathbf{h}'\in\mathcal H_{M_n}}\sqrt{\upsilon}\exp\bigg(- \frac{\eta^{2}C_1^2}{6C_3^2}C^2_0(\mathbf{h}')\bigg)=O(1)
\end{equation}
and
\begin{equation}\label{eq:bound2}
\exp\bigg(- \frac{\eta (\eta\wedge 1)C_1}{21C_2} \sqrt{n}\bigg)
\sum_{\mathbf{h}'\in\mathcal H_{M_n}}\mathcal V(\mathbf{h}')
= O(1),
\end{equation}
where we recall from \eqref{eq:defupsilon} that
$
\upsilon\coloneqq C_3^2(\mathbf{h}')nV^2(\mathbf{h}')
= C_3^2\mathcal V^2(\mathbf{h}')$.
Let \(C_0(\mathbf{h}')\coloneqq\sum_{j=1}^d C_{0j}(h_j')\), where we choose, for each \(j=1,\,\ldots,\,d\),
\[
C_{0j}(h_j') \coloneqq \frac{\sqrt{6}C_3}{\eta C_1}
\sqrt{\bigl[ \bigl(\beta_j-\tfrac12\bigr)_+ + (1 + \delta) \bigr] |\log h_j'|},
\quad \delta>0.
\]
Then,
\[
C_0^2(\mathbf h') = \sum_{j=1}^d C_{0j}^2(h_j')
= \frac{6C_3^2}{\eta C_1^2}
\sum_{j=1}^d \bigl[ \bigl(\beta_j-\tfrac12\bigr)_+ + (1 + \delta) \bigr] |\log h_j'|.
\]
On the grid \(\mathcal H_{M_n}\), we have \(h_j' = k_j^{-1}\), hence \(|\log h_j'| = \log k_j\). Thus,
\[\begin{split}
\sum_{\mathbf h'\in\mathcal H_{M_n}}
\sqrt{\upsilon}
\exp\!\left(-\frac{\eta^2 C_1^2}{6C_3^2} C^2_0(\mathbf h')\right)
&=
\frac{C_3}{n}
\sum_{k_1=1}^{M_n}\cdots\sum_{k_d=1}^{M_n} \prod_{j=1}^dk_j^{-(1+\delta)}
(\log k_j)^{\mathbf 1_{\{\beta_j=1/2\}}/2}\\
&=\frac{C_3}{n}\prod_{j=1}^d\Biggl(\sum_{k=1}^{M_n} k^{-(1+\delta)}
(\log k)^{\mathbf 1_{\{\beta_j=1/2\}}/2}\Biggr)=O(1),
\end{split}\]
since \(\delta>0\) and $\sum_{k=1}^\infty k^{-(1+\delta)}(\log k)^{\gamma_j}<\infty$ for every $\gamma_j\in\{0,\,1/2\}$. 

We now show that also the second summability bound in \eqref{eq:bound2} holds. 
Since on the anisotropic grid \(\mathcal H_{M_n}\) we have \(h_j' = k_j^{-1}\) for \(j=1,\,\dots,\,d\), so that \(|\log h_j'| = \log k_j\), and since 
\(V(\mathbf h') \le 1\) by definition of the grid, it follows that
\[
\begin{aligned}
&\exp\!\left(- \frac{\eta (\eta\wedge 1)C_1}{21C_2} \sqrt{n}\right)
\sum_{\mathbf h'\in\mathcal H_{M_n}}
V(\mathbf h')
\prod_{j=1}^d |\log h_j'|^{\mathbf 1_{\{\beta_j=1/2\}}/2}\\
&\qquad\qquad\qquad\qquad\qquad\qquad
\le
\exp\!\left(- \frac{\eta (\eta\wedge 1)C_1}{21C_2} \sqrt{n}\right)
\sum_{k_1=1}^{M_n} \cdots \sum_{k_d=1}^{M_n}
\prod_{j=1}^d (\log k_j)^{\mathbf 1_{\{\beta_j=1/2\}}/2} \\
&\qquad\qquad\qquad\qquad\qquad\qquad
=
\exp\!\left(- \frac{\eta (\eta\wedge 1)C_1}{21C_2} \sqrt{n}\right)
\prod_{j=1}^d \biggl( \sum_{k=1}^{M_n} (\log k)^{\mathbf 1_{\{\beta_j=1/2\}}/2} \biggr) \\
&\qquad\qquad\qquad\qquad\qquad\qquad
\le
\exp\!\left(- \frac{\eta (\eta\wedge 1)C_1}{21C_2} \sqrt{n}\right)
M_n^d (\log M_n)^{\frac12\sum_{j=1}^d \mathbf 1_{\{\beta_j=1/2\}}} \\
&\qquad\qquad\qquad\qquad\qquad\qquad
\le
\exp\!\left(- \frac{\eta (\eta\wedge 1)C_1}{21C_2} \sqrt{n}\right)
n^{d/2} (\log n)^{\frac12\sum_{j=1}^d \mathbf 1_{\{\beta_j=1/2\}}}
= O(1),
\end{aligned}
\]
where in the penultimate inequality we used the choice \(M_n = \lfloor n^{1/2} \rfloor\), and in the last step we used that the exponential decay in \(\sqrt{n}\) dominates any polynomial growth. Hence, \eqref{eq:bound2} follows.

Combining \eqref{eq:bound1} and \eqref{eq:bound2} with \eqref{eq:meanA}, we obtain
\[
\mathsf E_{(\mu_X\ast\mu_\varepsilon)^{\otimes n}}\bigg[\sup_{\mathbf{h}'\in\mathcal H_{M_n}}\bigg(\|\widehat F_{n,\mathbf{h}'}-\mathsf E_{(\mu_X\ast\mu_\varepsilon)^{\otimes n}}[\widehat F_{n,\mathbf{h}'}]\|_{L^2(\mathbb R^d)}
-\frac{\tilde V(\mathbf{h}')}{1+C_K}\bigg)_+\bigg]\lesssim\frac{1}{\sqrt{n}}.
\]
Substituting this into \eqref{eq:finalbound} yields, for the arbitrary fixed \(\mathbf{h}\in\mathcal H_{M_n}\),
\[
\begin{aligned}
\mathsf E_{(\mu_X\ast\mu_\varepsilon)^{\otimes n}}[\|\check{F}_n-F_X\|_{L^2(\mathbb R^d)}]
&\lesssim
\|B_{\mathbf{h}}\|_{L^2(\mathbb R^d)}
+\tilde V(\mathbf{h})
+\mathcal V(\mathbf{h})
+\frac{1}{\sqrt{n}} \\
&\lesssim
\|B_{\mathbf{h}}\|_{L^2(\mathbb R^d)}
+
\mathcal V(\mathbf{h})
+\frac{1}{\sqrt{n}}.
\end{aligned}
\]
Since \(\tilde V(\mathbf{h}) = C_V(1+C_K)\Gamma\) and, by construction, \(\Gamma\) is of the same order as \(\mathcal V(\mathbf{h})\) up to universal constants (see the definition of \(\Gamma\) in the proof of Proposition~\ref{prop:oracle}), the term \(2\tilde V(\mathbf{h})\) in \eqref{eq:finalbound} is absorbed into the \(\mathcal V(\mathbf{h})\)-term up to a multiplicative constant.
Because the above inequality holds for every fixed \(\mathbf{h}\in\mathcal H_{M_n}\), we may take the infimum over \(\mathbf{h}\in\mathcal H_{M_n}\) on the right-hand side. This gives the oracle inequality for the adaptive kernel estimator:
\[
\mathsf E_{(\mu_X\ast\mu_\varepsilon)^{\otimes n}}[\|\check{F}_n-F_X\|_{L^2(\mathbb R^d)}]
\lesssim\inf_{\mathbf{h}\in\mathcal H_{M_n}}\big\{\|B_{\mathbf{h}}\|_{L^2(\mathbb R^d)}+
\mathcal V(\mathbf{h})\big\}+\frac{1}{\sqrt{n}}.
\]
Finally, recalling the reduction step, namely$
\|\widehat F_{n,\hat{\mathbf{h}}}^\ast - F_X\|_{L^2(\mathbb R^d)}
\leq 2\|\check{F}_n - F_X\|_{L^2(\mathbb R^d)} + \varrho_n$,
and taking expectations, we obtain the oracle inequality for the original $\varrho_n$-approximation minimum-$L^2$-distance estimator
\[
\mathsf E_{(\mu_X\ast\mu_\varepsilon)^{\otimes n}}\Bigl[\|\widehat F_{n,\hat{\mathbf{h}}}
^\ast-F_X\|_{L^2(\mathbb R^d)}\Bigr]
\lesssim
\inf_{\mathbf{h}\in\mathcal H_{M_n}}
\big\{
\|B_{\mathbf{h}}\|_{L^2(\mathbb R^d)}
+
\mathcal V(\mathbf{h})
\big\}
+\frac{1}{\sqrt{n}} + \varrho_n.
\]
\item \emph{Grid approximation.}\\[2pt]
It remains to justify that the infimum over the discrete grid $\mathcal H_{M_n}$ is equivalent, up to universal constants, to the infimum over $(0,\,1)^d$. Let $\mathbf{h}_{\mathrm{opt}} \in (0,\,1)^d$ be a minimizer of
$\|B_{\mathbf{h}}\|_{L^2(\mathbb R^d)}
+ \mathcal V(\mathbf{h})$. 
From the bias-variance trade-off, the optimal bandwidths satisfy
\[
h_{\mathrm{opt},j} \asymp n^{-\frac{E}{2(\alpha_j+1)}} \times (\log n)^{\omega_j},
\]
with $\omega_j \ge 0$ because the logarithmic factors in the variance appear with non-negative exponents (they are zero when \(\beta_j \ne 1/2\) and positive when \(\beta_j = 1/2\)); hence they can only increase the bandwidth relative to the polynomial rate. Consequently,
\[
h_{\mathrm{opt},j} \ge c \, n^{-\frac{E}{2(\alpha_j+1)}} \ge c \, n^{-1/2}
\]
for some universal constant $c > 0$. Since $\mathcal H_{M_n}$ contains all bandwidths of the form $k_j^{-1}$ with $1 \le k_j \le M_n = \lfloor n^{1/2} \rfloor$, for each coordinate $j$ there exists an integer $k_j \le M_n$ such that $k_j^{-1} \asymp h_{\mathrm{opt},j}$.
Set $\mathbf{h}^{(\mathbf{k})} \coloneqq (k_1^{-1},\, \ldots,\, k_d^{-1}) \in \mathcal H_{M_n}$. Then, $h_j^{(\mathbf{k})} \asymp h_{\mathrm{opt},j}$ for all $j$. Consequently,
\[
\frac{V(\mathbf{h}^{(\mathbf{k})})}{V(\mathbf{h}_{\mathrm{opt}})}
=
\prod_{j=1}^d
\biggl(
\frac{h_{\mathrm{opt},j}}{h_j^{(\mathbf{k})}}
\biggr)^{(\beta_j - \frac12)_+}
= O(1),
\]
so $V(\mathbf{h}^{(\mathbf{k})}) \le D_1 V(\mathbf{h}_{\mathrm{opt}})$ for some universal constant $D_1 > 0$. Similarly, since $h_j^{(\mathbf{k})} \asymp h_{\mathrm{opt},j}$, we have $|\log h_j^{(\mathbf{k})}| \asymp |\log h_{\mathrm{opt},j}|$, hence the product of the logarithmic factors differ by at most a multiplicative constant $D_2$. For the bias, the asymptotic equivalence $h_j^{(\mathbf{k})} \asymp h_{\mathrm{opt},j}$, together with the bound $\|B_{\mathbf{h}}\|_{L^2(\mathbb R^d)} \le L\sum_{j=1}^d C_j h_j^{\alpha_j+1}$ valid for every $\mathbf{h} \in (0,1)^d$, gives
$
\|B_{\mathbf{h}^{(\mathbf{k})}}\|_{L^2(\mathbb R^d)}
\le D_3 \|B_{\mathbf{h}_{\mathrm{opt}}}\|_{L^2(\mathbb R^d)}
$
for some universal constant $D_3 > 0$. Hence, taking $D := \max\{D_1 D_2,\, D_3\}$, we obtain
\[
\begin{aligned}
&\inf_{\mathbf{h} \in \mathcal H_{M_n}}
\bigl\{
\|B_{\mathbf{h}}\|_{L^2(\mathbb R^d)}
+
\mathcal V(\mathbf{h})
\bigr\}\le
D
\inf_{\mathbf{h} \in (0,1)^d}
\bigl\{
\|B_{\mathbf{h}}\|_{L^2(\mathbb R^d)}
+
\mathcal V(\mathbf{h})
\bigr\}.
\end{aligned}
\]
The reverse inequality is trivial since $\mathcal H_{M_n} \subset (0,\,1)^d$. This proves the equivalence up to universal constants.
\end{steps}
The proof is thus concluded.
\end{proof}

\smallskip

We now establish the unified oracle inequalities stated in Theorem~\ref{thm:unified_oracle}. The proof proceeds in three steps: first, we recall the \(L^2\)-oracle for the unprojected estimator; second, we derive the corresponding \(L^1\)-oracle using the bias comparability assumption; and third, we transfer these results to the approximate \(L^1\)- and \(L^2\)-projection estimators.

\begin{proof}[Proof of Theorem~\ref{thm:unified_oracle} (Unified oracle inequalities for \(L^1\)- and \(L^2\)-norms)]
We establish the oracle inequalities in three steps.\\[-5pt]
\begin{steps}
\item\emph{$L^2$-oracle for \(\check{F}_n\).}\label{st:oraclekern}\\[2pt]
This is Theorem~\ref{thm:oracle_ineq}, which gives
\begin{equation}
\mathsf{E}_{(\mu_X\ast\mu_\varepsilon)^{\otimes n}}[\|\check{F}_n - F_X\|_{L^2(\mathbb R^d)}]
\lesssim
\mathcal R_n.
\label{eq:L2oracle}
\end{equation}

\item\emph{$L^1$-oracle for \(\check{F}_n\).}\label{st:l1oracle}\\[2pt]
Decompose the $L^1$-risk to obtain
\begin{multline}\label{eq:L1decomp}
\mathsf{E}_{(\mu_X\ast\mu_\varepsilon)^{\otimes n}}[\|\check{F}_n - F_X\|_{L^1(\mathbb R^d)}]
\le
\mathsf{E}_{(\mu_X\ast\mu_\varepsilon)^{\otimes n}}[\|\check{F}_n - \mathsf{E}_{(\mu_X\ast\mu_\varepsilon)^{\otimes n}}[\check{F}_n]\|_{L^1(\mathbb R^d)}] \\
+ \mathsf{E}_{(\mu_X\ast\mu_\varepsilon)^{\otimes n}}[\|B_{\hat{\mathbf{h}}}\|_{L^1(\mathbb R^d)}].
\end{multline}
The first term on the right-hand side of \eqref{eq:L1decomp} is controlled by the stochastic error bound from Proposition~\ref{prop:variance} and the property \(\mathsf{E}_{(\mu_X\ast\mu_\varepsilon)^{\otimes n}}[\mathcal V(\hat{\mathbf{h}})] \lesssim \mathcal R_n\) established in Theorem~\ref{thm:oracle_ineq}, which yield
\[\mathsf{E}_{(\mu_X\ast\mu_\varepsilon)^{\otimes n}}[\|\check{F}_n - \mathsf{E}_{(\mu_X\ast\mu_\varepsilon)^{\otimes n}}[\check{F}_n]\|_{L^1(\mathbb R^d)}] \lesssim \mathsf{E}_{(\mu_X\ast\mu_\varepsilon)^{\otimes n}}[\mathcal V(\hat{\mathbf{h}})] \lesssim \mathcal R_n.\]
For the bias term, the comparability assumption \(\|B_{\mathbf{h}}\|_{L^1(\mathbb R^d)} \le C_1\|B_{\mathbf{h}}\|_{L^2(\mathbb R^d)}\) gives
$$
\mathsf{E}_{(\mu_X\ast\mu_\varepsilon)^{\otimes n}}[\|B_{\hat{\mathbf{h}}}\|_{L^1(\mathbb R^d)}]
\le
C_1 \mathsf{E}_{(\mu_X\ast\mu_\varepsilon)^{\otimes n}}[\|B_{\hat{\mathbf{h}}}\|_{L^2(\mathbb R^d)}].$$
To bound the latter, we apply the triangle inequality to obtain
\[
\|B_{\hat{\mathbf{h}}}\|_{L^2(\mathbb R^d)}
\le
\|\check{F}_n - F_X\|_{L^2(\mathbb R^d)}
+
\|\check{F}_n - \mathsf{E}_{(\mu_X\ast\mu_\varepsilon)^{\otimes n}}[\check{F}_n]\|_{L^2(\mathbb R^d)}.
\]
Taking expectations and applying \eqref{eq:L2oracle} together with the \(L^2\)-bound from Proposition~\ref{prop:variance} gives 
$$\mathsf{E}_{(\mu_X\ast\mu_\varepsilon)^{\otimes n}}[\|B_{\hat{\mathbf{h}}}\|_{L^2(\mathbb R^d)}] \lesssim \mathcal R_n.$$
Hence,
\(\mathsf{E}_{(\mu_X\ast\mu_\varepsilon)^{\otimes n}}[\|B_{\hat{\mathbf{h}}}\|_{L^1(\mathbb R^d)}] \lesssim \mathcal R_n\).
Combining \eqref{eq:L1decomp} with the preceding estimates yields the \(L^1\)-oracle inequality for the unprojected estimator,
\begin{equation}
\mathsf{E}_{(\mu_X\ast\mu_\varepsilon)^{\otimes n}}[\|\check{F}_n - F_X\|_{L^1(\mathbb R^d)}]
\lesssim
\mathcal R_n.
\label{eq:L1oracle_unproj}
\end{equation}

\item \emph{Transfer to approximate projection estimators.}\label{st:transfer}\\[2pt]
By definition, \(\widehat F_{n,\hat{\mathbf{h}}}^{\ast,1}=\widehat F_{n,\hat{\mathbf{h}}}^{\ast}\) and satisfies
$
\|\widehat F_{n,\hat{\mathbf{h}}}^{\ast,1} - F_X\|_{L^1(\mathbb R^d)}\leq 2\|\check F_n - F_X\|_{L^1(\mathbb R^d)}+\varrho_{n,1}$.
Taking expectations and applying \eqref{eq:L1oracle_unproj} gives
\begin{equation}
\mathsf{E}_{(\mu_X\ast\mu_\varepsilon)^{\otimes n}}\Bigl[\|\widehat F_{n,\hat{\mathbf{h}}}^{\ast,1} - F_X\|_{L^1(\mathbb R^d)}\Bigr]
\lesssim
\mathcal R_n.
\label{eq:L1oracle_proj}
\end{equation}
Analogously, for the estimator $\widehat F_{n,\hat{\mathbf{h}}}^{\ast,2}$, 
Theorem~\ref{thm:oracle_ineq} gives exactly
\begin{equation}
\mathsf{E}_{(\mu_X\ast\mu_\varepsilon)^{\otimes n}}\Bigl[\|\widehat F_{n,\hat{\mathbf{h}}}^{\ast,2} - F_X\|_{L^2(\mathbb R^d)}\Bigr]
\lesssim
\mathcal R_n.
\label{eq:L2oracle_proj}
\end{equation}
\end{steps}
Inequalities \eqref{eq:L2oracle}, \eqref{eq:L1oracle_unproj}, \eqref{eq:L1oracle_proj}, and \eqref{eq:L2oracle_proj} are the desired oracle inequalities and the proof is complete.
\end{proof}


\section{Proofs related to the $1$-Wasserstein minimax result}\phantomsection\label{sec:S5}
In this section, we prove that the restricted \(L^1\)-projection estimator has finite first moment and achieves the optimal minimax rate \(\psi_n^{\boldsymbol{\alpha},\boldsymbol{\beta}} = n^{-E/2}\) for the \(1\)-Wasserstein distance over the class of product distributions \(\mathcal P_{\text{prod}}\), with matching lower and upper bounds.

We first establish the finite first moment property for both the integrated estimator, which allows us to construct a restricted estimator with compact support that preserves the optimal minimax rate. The final estimator is then obtained as the product of the univariate restricted estimators, and its optimality follows from the factorization property of the \(1\)-Wasserstein distance.

\begin{lemma}[Finite first moment of the integrated estimator]\label{lem:finiteMomentKernel}
Let 
\[
\hat{f}_{j,n,\mathbf{h}}(x) \coloneqq \frac{1}{n} \sum_{i=1}^n \frac{1}{h_j} \tilde{K}_{j,h_j}\!\left(\frac{x - Y_{i,j}}{h_j}\right),\quad x\in\mathbb R,
\]
where \(\tilde{K}_{j,h_j}\) is a kernel function satisfying
$\|\tilde{K}_{j,h_j}\|_{L^1(\mathbb R)} < \infty$ and $\int_{\mathbb R} |u| \, |\tilde{K}_{j,h_j}(u)| \, \di u < \infty$. 
Then, there exists a constant \(C_{j,h_j} > 0\) such that, almost surely,
\begin{equation}\label{eq:ineqg}
\int_{\mathbb R} |x| |\hat{f}_{j,n,\mathbf{h}}(x)|\,\di x \le C_{j,h_j}\biggl(h_j + \frac{1}{n}\sum_{i=1}^n |Y_{i,j}|\biggr) < \infty.
\end{equation}
\end{lemma}

\begin{proof}
Using \(|x| \le |x - Y_{i,j}| + |Y_{i,j}|\) and the change of variable \(u = (x - Y_{i,j})/h_j\), we have
\[
\begin{aligned}
\int_{\mathbb R} |x| \, |\hat{f}_{j,n,\mathbf{h}}(x)| \,\di x
&= \frac{1}{n} \sum_{i=1}^n \frac{1}{h_j} \int_{\mathbb R} |x| \, \left|\tilde{K}_{j,h_j}\!\left(\frac{x - Y_{i,j}}{h_j}\right)\right|\,\di x\\
&\le h_j \int_{\mathbb R} |u| \, |\tilde{K}_{j,h_j}(u)| \, \di u
+ \biggl(\frac{1}{n} \sum_{i=1}^n |Y_{i,j}|\biggr) \int_{\mathbb R} |\tilde{K}_{j,h_j}(u)| \, \di u\\
&= h_j \int_{\mathbb R} |u| \, |\tilde{K}_{j,h_j}(u)| \, \di u
+ \biggl(\frac{1}{n} \sum_{i=1}^n |Y_{i,j}|\biggr) \|\tilde{K}_{j,h_j}\|_{L^1(\mathbb R)}.
\end{aligned}
\]
Let
\[
C_{j,h_j}\coloneqq \max\left\{\int_{\mathbb R} |u| \, |\tilde{K}_{j,h_j}(u)| \, \di u,\; \|\tilde{K}_{j,h_j}\|_{L^1(\mathbb R)}\right\}.
\]
Since the kernel integrals are finite by assumption, \(C_{j,h_j} < \infty\). Then, \eqref{eq:ineqg} follows.
\end{proof}

We can now construct a restricted estimator with compact support that achieves the same rate.

\begin{corollary}[Rate preservation under compact support restriction]\label{cor:finiteMomentRestricted}
Let 
\(g(x) = \widehat{F}_{j,n,\mathbf{h}}(x)=\int_{-\infty}^x\hat{f}_{j,n,\mathbf{h}}(u) \,\di u\), $x\in\mathbb R$, be the univariate integrated estimator.
Assume that the minimizer 
$$F^\ast \coloneqq \arg\min_{F \in \mathfrak F_1^{\textnormal{fin}}} \|F - g\|_{L^1(\mathbb R)}$$
exists, where \(\mathfrak F_1^{\textnormal{fin}} := \{ F \in \mathfrak F_1 :\, F \text{ has finite first moment} \}\). Here \(\mathfrak F_1\) denotes the class of univariate distribution functions on \(\mathbb R\).
For \((M_n)_{n\in\mathbb N}\), let $\mathfrak F_{M_n} := \{ F \in \mathfrak F_1 :\, F(x) = 0 \text{ for } x < -M_n, \;\, F(x) = 1 \text{ for } x > M_n \}$. Let \(\widehat{F}_{j,n,\mathbf{h}}^{\ast}\) be a \(\varrho_n\)-approximate \(L^1\)-minimizer on \(\mathfrak F_{M_n}\) defined as
\[
\|\widehat{F}_{j,n,\mathbf{h}}^{\ast} - g\|_{L^1(\mathbb R)}
\le \inf_{F \in \mathfrak F_{M_n}} \|F - g\|_{L^1(\mathbb R)} + \varrho_n,
\]
where \(\varrho_n = O(\psi_n^{\boldsymbol{\alpha},\boldsymbol{\beta}})\) and \(\psi_n^{\boldsymbol{\alpha},\boldsymbol{\beta}}\) denotes the convergence rate of the unrestricted estimator defined in \eqref{eq:rate} \hyperref[eq:rate]{of the Main Article}, with \(\psi_n^{\boldsymbol{\alpha},\boldsymbol{\beta}} \to 0\).
Then, \(\widehat{F}_{j,n,\mathbf{h}}^{\ast}\) has finite first moment and, provided \(M_n\) is chosen such that
$
\inf_{F \in \mathfrak F_{M_n}} \|F - g\|_{L^1(\mathbb R)} = \|F^\ast - g\|_{L^1(\mathbb R)} + o(\psi_n^{\boldsymbol{\alpha},\boldsymbol{\beta}})$,
it achieves the same convergence rate as the unrestricted estimator $\|\widehat{F}_{j,n,\mathbf{h}}^{\ast} - g\|_{L^1(\mathbb R)}
= \|F^\ast - g\|_{L^1(\mathbb R)} + O(\psi_n^{\boldsymbol{\alpha},\boldsymbol{\beta}})$. 
\end{corollary}

\begin{proof}
By Lemma~\ref{lem:finiteMomentKernel}, \(\hat{f}_{j,n,\mathbf{h}}\) has finite first moment. 
Since \(F^\ast \in \mathfrak F_1^{\text{fin}}\) by assumption, it has finite first moment.
By construction, \(\widehat{F}_{j,n,\mathbf{h}}^{\ast} \in \mathfrak F_{M_n}\), so its associated measure has support contained in \([-M_n,\, M_n]\) and therefore has finite first moment
\[
\int_{\mathbb R} |x| \, \di\widehat{F}_{j,n,\mathbf{h}}^{\ast}(x) \le M_n < \infty.
\]
It remains to show that the convergence rate is not compromised. Since \(F^\ast\) has finite first moment, for each \(M > 0\) define the truncated cumulative distribution function
\[
F^\ast_M(x) \coloneqq
\begin{cases}
0, & x < -M, \\
F^\ast(x), & |x| \le M, \\
1, & x > M.
\end{cases}
\]
Then, \(F^\ast_M \in \mathfrak F_M\) and, by the dominated convergence theorem, $\|F^\ast_M - F^\ast\|_{L^1(\mathbb R)} \to 0$ as $M \to \infty$.
Hence, for the given rate \(\psi_n^{\boldsymbol{\alpha},\boldsymbol{\beta}}\), we can choose \(M_n\) such that $
\|F^\ast_{M_n} - F^\ast\|_{L^1(\mathbb R)} = o(\psi_n^{\boldsymbol{\alpha},\boldsymbol{\beta}})$.
Therefore, the required condition on \(M_n\) can be satisfied. By the choice of \(M_n\) and since \(\varrho_n = O(\psi_n^{\boldsymbol{\alpha},\boldsymbol{\beta}})\),
\[\begin{split}
\|\widehat{F}_{j,n,\mathbf{h}}^{\ast} - g\|_{L^1(\mathbb R)}
&\le \inf_{F \in \mathfrak F_{M_n}} \|F - g\|_{L^1(\mathbb R)} + \varrho_n \\
&= \|F^\ast - g\|_{L^1(\mathbb R)} + o(\psi_n^{\boldsymbol{\alpha},\boldsymbol{\beta}}) + O(\psi_n^{\boldsymbol{\alpha},\boldsymbol{\beta}}) \\
&= \|F^\ast - g\|_{L^1(\mathbb R)} + O(\psi_n^{\boldsymbol{\alpha},\boldsymbol{\beta}}).
\end{split}
\]
Therefore, the restricted estimator achieves the same rate as the unrestricted estimator.
\end{proof}

Corollary~\ref{cor:finiteMomentRestricted} shows that the restriction to a compact support does not prevent optimal convergence, provided the minimizer \(F^\ast\) exists. The sequence \(M_n\) can be chosen (depending on the unknown \(F^\ast\)) such that the restricted estimator achieves the same rate as the unrestricted one.

\begin{proof}[Proof of Theorem~\ref{th:W1minimax} (Minimax optimality of the projection estimator for \(\We\) on \(\mathcal P_{\text{prod}}\))]
\mbox{}\\
The proof adapts the construction from the \(L^1\)-cumulative distribution function lower bound to the product setting.
Let \(j_0 \in \{1,\,\dots,\,d\}\) be a fixed coordinate index (the construction below works for any choice of \(j_0\) and the resulting rate does not depend on this choice). The role of \(j_0\) is that of a localization device: the two test densities are built to differ only on the \(j_0\)-th coordinate, while on all other coordinates they coincide. This has two key consequences. First,
the separation in \(W_1\) between the two hypotheses is proportional to \(\psi_n^{\boldsymbol{\alpha},\boldsymbol{\beta}}\). Second, the \(\chi^2\)-divergence between the convolved densities remains sufficiently small, specifically \(o(n^{-1})\), so that Le Cam's two-point method applies.
The resulting lower bound rate \(\psi_n^{\boldsymbol{\alpha},\boldsymbol{\beta}} = n^{-E/2}\), with $E$ as defined  in~\eqref{eq:rate} \hyperref[eq:rate]{of the Main Article}, is governed by the aggregate anisotropic smoothness and noise exponents across all coordinates, not by a single worst-behaving coordinate. The positive-part censoring \((\beta_j-1/2)_+\) means that coordinates with \(\beta_j \le 1/2\) do not contribute to the exponent, effectively reducing the dimension of the problem. This is an anisotropic aggregation over the coordinates that survive the censoring, rather than a worst-coordinate maximum.

\smallskip

We first prove the lower bound. From the \(L^1\)-lower bound proof, there exist functions \(\omega,\, \xi \in C_c^\infty(\mathbb R)\) with properties \ref{it:W1}--\ref{it:W4} and \ref{it:U1}--\ref{it:U5}. These will be used to construct the test densities below. In particular,
\begin{itemize}
    \item $\omega \ge 0$, $\operatorname{supp}(\omega) \subset [-\gamma_\omega,\,\gamma_\omega]$, $\int_{\mathbb R} \omega(x)\,\,\di x = 1$ and $\omega \ge c_\omega > 0$ on $[-\gamma_\omega,\, \gamma_\omega]$;
    \item $\operatorname{supp}(\xi) \subset [-\gamma_\xi,\, \gamma_\xi] \subset (-\gamma_\omega,\, \gamma_\omega)$, $\int_{\mathbb R} \xi(x)\,\di x = 0$ and $\Xi(x) \coloneqq \int_{-\infty}^x \xi(s)\,\di s \ge c_\Xi > 0$ for $x \in [1/2,\, 1/2+\eta]$.
\end{itemize}
Set $R_n \coloneqq n^{1/E}$ and
\(
\lambda_{j_0,n} := n^{-E/2} R_n^{-(\alpha_{j_0}+1)}
\).
Define for each $j$,
\[
\omega_{j,n}(x) \coloneqq 
\begin{cases}
R_n^{-(\alpha_{j_0}+1)} \omega(R_n^{-(\alpha_{j_0}+1)} x), & j = j_0, \\
\omega(x), & j \neq j_0,
\end{cases}
\]
and
\[
\xi_{j,n}(x) \coloneqq 
\begin{cases}
R_n^{-(\alpha_{j_0}+1)} \xi(R_n^{-(\alpha_{j_0}+1)} x), & j = j_0, \\
0, & j \neq j_0.
\end{cases}
\]
The test densities are
\begin{equation*}\label{eq:testdensitiesW1}
f_0(\mathbf x) \coloneqq \prod_{j=1}^d \omega_{j,n}(x_j), \qquad
f_1(\mathbf x) \coloneqq  [\omega_{j_0,n}(x_{j_0}) + \lambda_{j_0,n} \xi_{j_0,n}(x_{j_0})]
\prod_{\substack{j=1 \\ j \neq j_0}}^d \omega_{j,n}(x_j),\quad\mathbf{x}\in\mathbb R^d.
\end{equation*}
The verification that $f_0,\, f_1 \in \mathcal P_{\text{prod}}$ follows the same arguments as in the $L^1$-lower bound proof. The key points are the following:
\begin{itemize}
    \item since $\omega,\, \xi \in C_c^\infty(\mathbb R)$, we have $\omega_{j,n},\, \xi_{j,n} \in \mathcal N_1^1(\alpha_j, L_j)$ with constants independent of $n$;
    \item for sufficiently large $n$, $\lambda_{j_0,n} \|\xi\|_\infty \le c_\omega/2$, so $\omega_{j_0,n} + \lambda_{j_0,n} \xi_{j_0,n} \ge 0$;
    \item the product structure is preserved by construction;
    \item the finite first moment condition holds because $\omega$ and $\xi$ have compact support.
\end{itemize}
Then, the product densities $f_0$ and $f_1$ belong to $\mathcal N_1^d(\boldsymbol\alpha,\, L)$ for some $L > 0$ independent of $n$. In fact, if $f_j \in \mathcal N_1^1(\alpha_j,\, L_j)$ for $j = 1,\,\dots,\,d$, then the product density
\[
f(\mathbf x) = \prod_{j=1}^d f_j(x_j), \quad \mathbf x \in \mathbb R^d,
\]
belongs to $\mathcal N_1^d(\boldsymbol\alpha,\, L)$, with $\boldsymbol\alpha = (\alpha_1,\,\dots,\,\alpha_d)^\top$ and
$
L \coloneqq \max_{1 \le \ell \le d} L'_\ell
$,
where $L_\ell' \coloneqq \max\{L_\ell, \,1\}$.
For each $j = 1,\,\dots,\,d$, let $m_j \coloneqq \lceil \alpha_j \rceil - 1$. The partial derivative of order $m_j$ with respect to the $j$-th coordinate is
\[
\partial_j^{m_j} f(\mathbf x) = f_j^{(m_j)}(x_j) \prod_{\substack{\ell=1 \\ \ell \neq j}}^d f_\ell(x_\ell).
\]
For the $L^1$-integrability condition, since $\|f_\ell\|_{L^1(\mathbb R)} = 1$ for each $\ell$,
\[
\|\partial_j^{m_j} f\|_{L^1(\mathbb R^d)}
= \|f_j^{(m_j)}\|_{L^1(\mathbb R)} \prod_{\substack{\ell=1 \\ \ell \neq j}}^d \|f_\ell\|_{L^1(\mathbb R)}
\le L_j < \infty.
\]
For the Hölder condition, for any $\delta \in \mathbb R$,
\[
\begin{aligned}
\|\Delta_{\delta \mathbf e_j} \partial_j^{m_j} f\|_{L^1(\mathbb R^d)}= \|f_j^{(m_j)}(\cdot + \delta) - f_j^{(m_j)}(\cdot)\|_{L^1(\mathbb R)}
\prod_{\substack{\ell=1 \\ \ell \neq j}}^d \|f_\ell\|_{L^1(\mathbb R)}\le L_j |\delta|^{\alpha_j - m_j}.
\end{aligned}
\]
This proves that $f \in \mathcal N_1^d(\boldsymbol\alpha,\, L)$.
Thus, $f_0,\, f_1 \in \mathcal P_{\text{prod}}$ for all sufficiently large $n$.

\smallskip

We now compute the separation rate in $W_1$. 
Let $\mu_0,\,\mu_1$ be the probability measures corresponding to $f_0,\, f_1$. Since both are product measures and $W_1$ factorizes over independent coordinates, we have
$
W_1(\mu_0, \,\mu_1) = W_1(\mu_{j_0}^0,\, \mu_{j_0}^1)$,
where $\mu_{j_0}^0 \sim \omega_{j_0,n}$ and $\mu_{j_0}^1 \sim \omega_{j_0,n} + \lambda_{j_0,n} \xi_{j_0,n}$.
In one dimension, $W_1$ equals the $L^1$-distance between the cumulative distribution functions so
$W_1(\mu_{j_0}^0,\, \mu_{j_0}^1) = \lambda_{j_0,n} \int_{\mathbb R} |\Xi_{j_0,n}(x)| \, \di x$,
where $\Xi_{j_0,n}(x) \coloneqq \Xi(R_n^{-(\alpha_{j_0}+1)}x)$ and $\Xi$ is the primitive of $\xi$.
Define $
\kappa_{j_0} \coloneqq \int_{\mathbb R} |\Xi(u)| \, \di u > 0$. By the change of variable $u = R_n^{-(\alpha_{j_0}+1)}x$,
\begin{equation}\label{eq:W1separation}
W_1(\mu_0, \,\mu_1) = \lambda_{j_0,n} \int_{\mathbb R} |\Xi_{j_0,n}(x)| \, \di x
= \lambda_{j_0,n} R_n^{\alpha_{j_0}+1} \int_{\mathbb R} |\Xi(u)| \, \di u
= \lambda_{j_0,n} R_n^{\alpha_{j_0}+1} \kappa_{j_0}
= \kappa_{j_0} n^{-E/2}.
\end{equation}

We now bound the $\chi^2$-divergence. 
Set \(\epsilon \coloneqq (\gamma_\omega-\gamma_\xi)/2 > 0\). Since \(\gamma_\xi < \gamma_\omega\), we have \(\gamma_\xi + \epsilon < \gamma_\omega\). The support of \(f_1 - f_0\) is contained in
\[
\mathcal K_n\coloneqq \Big\{ \mathbf x \in \mathbb R^d :\, |x_{j_0}| \le \gamma_\xi R_n^{\alpha_{j_0}+1} + \epsilon, \; |x_j| \le \gamma_\xi + \epsilon \text{ for } j \neq j_0 \Big\}.
\]
We derive a lower bound on $f_0 \ast f_\varepsilon$ on $\mathcal K_n$.
We show that there exists a constant $c_0 > 0$ such that
\[
(f_0 \ast f_\varepsilon)(\mathbf x) \ge c_0 R_n^{-(\alpha_{j_0}+1)}, \quad \forall\, \mathbf x \in \mathcal K_n.
\]
Indeed, on the support of $f_1 - f_0$,
\begin{itemize}
    \item for $j = j_0$, since $\gamma_\xi + \epsilon < \gamma_\omega$ and $\omega(x) \ge c_\omega > 0$ for $|x| \le \gamma_\omega$, for sufficiently large $n$ we have $\gamma_\xi R_n^{\alpha_{j_0}+1} + \epsilon \le \gamma_\omega R_n^{\alpha_{j_0}+1}$, and hence
    \[
    \omega_{j_0,n}(x_{j_0}) \ge c_\omega R_n^{-(\alpha_{j_0}+1)}, \quad |x_{j_0}| \le \gamma_\xi R_n^{\alpha_{j_0}+1} + \epsilon;
    \]
   \item for $j \neq j_0$, since $\gamma_\xi + \epsilon < \gamma_\omega$, we have $\omega_{j,n}(x_j) = \omega(x_j) \ge c_\omega > 0$ for $|x_j| \le \gamma_\xi + \epsilon$.
\end{itemize}
Thus, on $\mathcal K_n$,
\[
f_0(\mathbf x) = \prod_{j=1}^d \omega_{j,n}(x_j) \ge c_* R_n^{-(\alpha_{j_0}+1)},
\]
where $c_* \coloneqq c_\omega^d > 0$.
Now, for any $0 < \varpi < 1$, there exists a compact set $K_\varpi \subset \mathbb R^d$ such that
$
\int_{K_\varpi} f_\varepsilon(\mathbf z) \, \di \mathbf z \ge 1 - \varpi$. 
Choose $K_\varpi$ sufficiently small so that $\mathbf x - K_\varpi \subset \mathcal K_n^*$ for all $\mathbf x \in \mathcal K_n$, where $\mathcal K_n^*$ is a slightly enlarged region where the same lower bound on $f_0$ holds.
Then, for $\mathbf x \in \mathcal K_n$,
\[
\begin{aligned}
(f_0 \ast f_\varepsilon)(\mathbf x)
\ge \int_{\mathbf x - K_\varpi} f_0(\mathbf u) f_\varepsilon(\mathbf x - \mathbf u) \, \di\mathbf u 
\ge c_* R_n^{-(\alpha_{j_0}+1)} \int_{K_\varpi} f_\varepsilon(\mathbf z) \, \di\mathbf z 
\ge c_* (1 - \varpi) R_n^{-(\alpha_{j_0}+1)}.
\end{aligned}
\]
Therefore, for $c_0\coloneqq c_* (1 - \varpi) > 0$, we have $
(f_0 \ast f_\varepsilon)(\mathbf x) \ge c_0 R_n^{-(\alpha_{j_0}+1)}$ for all $\mathbf x \in \mathcal K_n$.
By possibly reducing the constant, the same lower bound holds on the slightly enlarged region $\mathcal K_n^*$: there exists $c_0' \in (0,\,c_0]$ such that
$(f_0 \ast f_\varepsilon)(\mathbf x) \ge c_0' R_n^{-(\alpha_{j_0}+1)}$ for all $\mathbf x \in \mathcal K_n^*$.
Since, for \(n\) sufficiently large, $\operatorname{supp}\big((f_1 - f_0) \ast f_\varepsilon\big) \subset \mathcal K_n^*$,
\[
\begin{aligned}
\chi^2(f_1 \ast f_\varepsilon \parallel f_0 \ast f_\varepsilon)
&= \int_{\mathbb R^d} \frac{[(f_1 - f_0) \ast f_\varepsilon(\mathbf x)]^2}{(f_0 \ast f_\varepsilon)(\mathbf x)} \, \di\mathbf x\\
&\leq \int_{\mathcal K_n^*} \frac{[(f_1 - f_0) \ast f_\varepsilon(\mathbf x)]^2}{(f_0 \ast f_\varepsilon)(\mathbf x)} \, \di\mathbf x \\
&\le \frac{1}{c_0'} R_n^{\alpha_{j_0}+1} \int_{\mathcal K_n^*} [(f_1 - f_0) \ast f_\varepsilon(\mathbf x)]^2 \, \di\mathbf x\\
&\le \frac{1}{c_0'} R_n^{\alpha_{j_0}+1}\|(f_1 - f_0) \ast f_\varepsilon\|_{L^2(\mathbb R^d)}^2 \\
&\le \frac{1}{c_0'} R_n^{\alpha_{j_0}+1}\|f_1 - f_0\|_{L^2(\mathbb R^d)}^2 \|f_\varepsilon\|_{L^1(\mathbb R^d)}^2 \\
&= \frac{1}{c_0'} R_n^{\alpha_{j_0}+1}\|f_1 - f_0\|_{L^2(\mathbb R^d)}^2.
\end{aligned}
\]
We compute $\|f_1 - f_0\|_{L^2(\mathbb R^d)}^2$. Since $f_1 - f_0$ factorizes, we have
\[
\|f_1 - f_0\|_{L^2(\mathbb R^d)}^2
= \lambda_{j_0,n}^2 \|\xi_{j_0,n}\|_{L^2(\mathbb R)}^2 \prod_{\substack{j=1 \\ j \neq j_0}}^d \|\omega_{j,n}\|_{L^2(\mathbb R)}^2
= \lambda_{j_0,n}^2 R_n^{-(\alpha_{j_0}+1)} \|\xi\|_{L^2(\mathbb R)}^2 \prod_{\substack{j=1 \\ j \neq j_0}}^d \|\omega\|_{L^2(\mathbb R)}^2.
\]
Thus,
\begin{equation}\label{eq:chi2boundW1}
\chi^2(f_1 \ast f_\varepsilon \parallel f_0 \ast f_\varepsilon)
\le \frac{C}{c_0'} \lambda_{j_0,n}^2 R_n^{\alpha_{j_0}+1}
= \frac{C}{c_0'} n^{-E} R_n^{-(\alpha_{j_0}+1)}
\lesssim n^{-E - (\alpha_{j_0}+1)/E}
\lesssim n^{-1}
\end{equation}
because $\alpha_{j_0} > 0$ implies $E + (\alpha_{j_0}+1)/E \ge 2\sqrt{\alpha_{j_0}+1} > 1$ for every $E > 0$. Therefore,
$\chi^2(f_1 \ast f_\varepsilon \parallel f_0 \ast f_\varepsilon) = o(n^{-1})$.
By Le Cam's two-point method,
\[
\inf_{\hat{\mu}_n} \sup_{\mu_X \in \mathcal P_{\text{prod}}} \mathsf{E}_{(\mu_X\ast\mu_\varepsilon)^{\otimes n}}[W_1(\hat{\mu}_n,\, \mu_X)]
\ge \frac{1}{2} W_1(\mu_0,\, \mu_1)
\left[ 1 - \| (\mu_1 \ast \mu_\varepsilon)^{\otimes n} - (\mu_0 \ast \mu_\varepsilon)^{\otimes n} \|_{\mathrm{TV}} \right].
\]
From \eqref{eq:chi2boundW1}, Pinsker's inequality and tensorization of the $\chi^2$-divergence yield
$
\| (\mu_1 \ast \mu_\varepsilon)^{\otimes n} - (\mu_0 \ast \mu_\varepsilon)^{\otimes n} \|_{\mathrm{TV}} \to 0$.
Thus, for $n$ sufficiently large, using \eqref{eq:W1separation}:
\[
\inf_{\hat{\mu}_n} \sup_{\mu_X \in \mathcal P_{\text{prod}}} \mathsf{E}_{(\mu_X\ast\mu_\varepsilon)^{\otimes n}}[W_1(\hat{\mu}_n,\, \mu_X)]
\ge \frac{1}{4} W_1(\mu_0,\, \mu_1)\ge \frac{\kappa_{j_0}}{4} \psi_n^{\boldsymbol{\alpha},\boldsymbol{\beta}}.
\]
Therefore,
\[
\inf_{\hat{\mu}_n} \sup_{\mu_X \in \mathcal P_{\text{prod}}} \mathsf{E}_{(\mu_X\ast\mu_\varepsilon)^{\otimes n}}[W_1(\hat{\mu}_n,\, \mu_X)] \gtrsim \psi_n^{\boldsymbol{\alpha},\boldsymbol{\beta}},
\]
which completes the proof of the lower bound.

\smallskip

We now prove the upper bound. For each \(j = 1,\,\dots,\,d\), let \(\widehat{F}_{j,n,\mathbf{h}}\) be the univariate integrated estimator, and let \(\widehat{F}_{j,n,\mathbf{h}}^{\ast}\) be the restricted \(L^1\)-projection estimator defined in Corollary \ref{cor:finiteMomentRestricted}, with \(\varrho_{j,n} = O({n}^{-1/2})\) and \(M_n\) chosen as in Corollary~\ref{cor:finiteMomentRestricted}. Define the product estimator
\[
\widehat{F}_{n,\mathbf{h}}^*(\mathbf x) \coloneqq \prod_{j=1}^d \widehat{F}_{j,n,\mathbf{h}}^\ast(x_j), \quad \mathbf x = (x_1,\,\dots,\,x_d)^\top \in \mathbb R^d,
\]
and let \(\widehat{\mu}_{n,\mathbf{h}}^\ast \coloneqq \bigotimes_{j=1}^d \widehat{\mu}_{j,n,\mathbf{h}}^\ast\) be the associated probability measure. By Corollary~\ref{cor:finiteMomentRestricted}, each \(\widehat{\mu}_{j,n,\mathbf{h}}^\ast\) has finite first moment. Hence, by the factorization property of \(W_1\) for product measures with \(\ell^1\)-cost and using identity \eqref{eq:identity} \hyperref[eq:identity]{of the Main Article},
\[\begin{split}
\Wl(\widehat{\mu}_{n,\mathbf{h}}^\ast,\, \mu_X) 
&= \Wl\!\left(\bigotimes_{j=1}^d \widehat{\mu}_{j,n,\mathbf{h}}^\ast,\, \bigotimes_{j=1}^d \mu_{X,j}\right)\\
&= \sum_{j=1}^d \Wl(\widehat{\mu}_{j,n,\mathbf{h}}^\ast,\, \mu_{X,j})\\
&= \sum_{j=1}^d\| \widehat{F}_{j,n,\mathbf{h}}^* - F_{X,j}\|_{L^1(\mathbb R)}.
\end{split}
\]
Taking expectations, 
\[
\mathsf{E}_{(\mu_X\ast\mu_\varepsilon)^{\otimes n}}\big[ \Wl(\widehat{\mu}_{n,\mathbf{h}}^\ast,\, \mu_X)  \big]
=\sum_{j=1}^d \mathsf{E}_{(\mu_X\ast\mu_\varepsilon)^{\otimes n}}\big [\| \widehat{F}_{j,n,\mathbf{h}}^* - F_{X,j}\|_{L^1(\mathbb R)}\big].
\]
From Theorem~\ref{thm:ubrate}, applied in dimension $d=1$ with the restricted projection space $\mathfrak F_{M_n}$, for each $j$ there exists a constant $C_j > 0$ such that
\[
\sup_{f_j \in \mathcal N_1^1(\alpha_j, L_j)} 
\mathsf{E}_{(\mu_X\ast\mu_\varepsilon)^{\otimes n}} \big[\|\widehat{F}_{j,n,\mathbf{h}}^* - F_{X,j}\|_{L^1(\mathbb R)}\big]
\le C_j \, n^{-E_j/2},
\]
where we define
\[
E_j \coloneqq \bigg[1 + \frac{(\beta_j - \frac12)_+}{\alpha_j + 1}\bigg]^{-1}.
\]
The choice $\varrho_{j,n} = O({n}^{-1/2})$ and $M_n$ as in Corollary~\ref{cor:finiteMomentRestricted} ensures that the projection and truncation errors do not affect the rate. Therefore,
\[
\sup_{\mu_X \in \mathcal P_{\text{prod}}} 
\mathsf{E}_{(\mu_X\ast\mu_\varepsilon)^{\otimes n}} \big[ \Wl(\widehat{\mu}_{n,\mathbf{h}}^\ast,\, \mu_{X}) \big]
\le \sum_{j=1}^d C_j \, n^{-E_j/2}.
\]
From the definition of $E$,
\[
\frac{1}{E} = 1 + \sum_{\ell=1}^d \frac{(\beta_\ell-\frac12)_+}{\alpha_\ell+1}
\ge 1 + \frac{(\beta_j-\frac12)_+}{\alpha_j+1} = \frac{1}{E_j}>0.
\]
Therefore, for $n \ge 1$, $n^{-E_j/2} \le n^{-E/2}$. Consequently,
$\sum_{j=1}^d C_j n^{-E_j/2}
\le n^{-E/2}\sum_{j=1}^d C_j
\lesssim n^{-E/2}$. Hence,
\[
\sup_{\mu_X \in \mathcal P_{\text{prod}}} 
\mathsf{E}_{(\mu_X\ast\mu_\varepsilon)^{\otimes n}} \big[ \Wl(\widehat{\mu}_{n,\mathbf{h}}^\ast,\, \mu_{X}) \big]
\lesssim \psi_n^{\boldsymbol{\alpha},\boldsymbol{\beta}}.
\]
This completes the proof of the upper bound.

\smallskip

Combining the lower and upper bounds, we obtain the minimax rate 
\[
\inf_{\hat{\mu}_n} \sup_{\mu_X \in \mathcal P_{\text{prod}}} 
\mathsf{E}_{(\mu_X\ast\mu_\varepsilon)^{\otimes n}} [W_1(\hat{\mu}_n,\, \mu_X)] \asymp \psi_n^{\boldsymbol{\alpha},\boldsymbol{\beta}}.
\]
This completes the proof.
\end{proof}
\end{supplement}

\end{document}